\documentclass[11pt]{amsart}
\usepackage{cellspace}
\usepackage{array}
\usepackage{tabularx}
\usepackage{booktabs}
\usepackage[foot]{amsaddr}
\usepackage{import}
\usepackage[english]{babel}
\usepackage[%pagebackref,
pdftex,hyperfootnotes]
{hyperref}
\usepackage[utf8]{inputenc}

\usepackage[total={19.5cm,25.7cm},top=1.35cm, left=1.2cm]{geometry}
\usepackage{rotating,multirow,pdflscape}
\usepackage{amssymb,latexsym,amsmath,amsthm}
\usepackage{mathrsfs}
\usepackage{mathtools,arydshln,mathdots}
\usepackage{bigints}
\makeatletter
\renewcommand*\env@matrix[1][*\c@MaxMatrixCols c]{%
  \hskip -\arraycolsep
  \let\@ifnextchar\new@ifnextchar
  \array{#1}}
\makeatother

\mathtoolsset{centercolon}
\usepackage[table]{xcolor}
\usepackage[toc,page]{appendix}

\usepackage{tocvsec2}
\usepackage[T1]{fontenc}
\usepackage[scaled]{helvet} % ss
\usepackage{courier} % tt
\usepackage{eulervm} % a better implementation of the euler package (not in gwTeX)
\normalfont

\newtheorem{coro}{{Corollary}}

\newtheorem{defi}{{ Definition}}
\newtheorem{teo}{Theorem}
\newtheorem{pro}{ Proposition }
\newtheorem{ma}{Lemma}

\newtheorem{rem}{Remark}

\renewcommand{\d}{\operatorname{d}}
\newcommand{\Exp}[1]{\operatorname{e}^{#1}}

\newcommand{\diag}{\operatorname{diag}}
\newcommand{\Ct}{\check{C}}
\newcommand{\Ht}{\check{H}}
\newcommand{\St}{\check{S}}

\newcommand{\C}{\mathbb{X}}

\newcommand{\R}{\mathbb{R}}

\newcommand{\prodint}[1]{\left\langle{#1}\right\rangle}

\begin{document}

 \title[ Transformation theory for   matrix biorthogonal  polynomials on the real line]{Transformation theory and Christoffel formulas for \\matrix biorthogonal  polynomials on the real line \\
(\MakeLowercase{with applications to the non}-A\MakeLowercase{belian} 2D T\MakeLowercase{oda lattice and noncommutative} KP \MakeLowercase{hierarchies})}

 \author[C Álvarez-Fernández]{Carlos Álvarez-Fernández$^{1,\maltese}$}
 \address{$^1$Departamento de Métodos Cuantitativos, Universidad Pontificia Comillas, C/ de Alberto Aguilera 23, 28015 Madrid, Spain }
 \email{calvarez@comillas.edu}

 \author[G Ariznabarreta]{Gerardo Ariznabarreta$^{2,\maltese,\dag}$}

 \address{$^2$Departamento de Física Teórica II (Métodos Matemáticos de la Física), Universidad Complutense de Madrid, Ciudad Universitaria, Plaza de Ciencias 1,  28040 Madrid, Spain}
 \email{gariznab@ucm.es}

\author[JC García-Ardila]{Juan C. García-Ardila$^{3,\ddag}$}
\address{$^3$Departamento de Matemáticas, Universidad Carlos III de Madrid, Avd/  Universidad 30, 28911 Leganés, Spain}
\email{jugarcia@math.uc3m.es}

\author[M Mañas]{Manuel Mañas$^{2,\maltese}$}
% \address{Departamento de Física Teórica II (Métodos Matemáticos de la Física), Universidad Complutense de Madrid, Ciudad Universitaria, Plaza de Ciencias 1, 28040 Madrid, Spain}
 \email{manuel.manas@ucm.es}

\author[F Marcellán]{Francisco Marcellán$^{3,4,\ddag}$}
\address{$^4$Instituto de Ciencias Matemáticas (ICMAT), C/ Nicolás Cabrera 13-15,
	28049 Canto Blanco, Spain}
\email{pacomarc@ing.uc3m.es}

	\thanks{$^\maltese$Thanks financial support from the Spanish ``Ministerio de Economía y Competitividad" research project [MTM2015-65888-C4-3-P],\emph{ Ortogonalidad, teoría de la aproximación y aplicaciones en física matemática}}
	\thanks{$^\dag$Thanks financial support from the Universidad Complutense de Madrid  Program ``Ayudas para Becas y Contratos Complutenses Predoctorales en España 2011"}
	\thanks{$^ \ddag$Thanks financial support from the Spanish ``Ministerio de Economía y Competitividad" research project [MTM2015-65888-C4-2-P],\emph{ Ortogonalidad, teoría de la aproximación y aplicaciones en física matemática}}

\enlargethispage{.13cm}
\begin{abstract}
In this paper transformations for matrix orthogonal polynomials in the real line are studied. The orthogonality is understood in a broad sense, and is given in terms of a  nondegenerate  continuous sesquilinear form,  which in turn is determined by a quasidefinite matrix of bivariate generalized functions with a well defined support.
The discussion of the orthogonality for such a sesquilinear  form includes, among others, matrix Hankel cases with linear functionals,  general matrix Sobolev orthogonality and discrete orthogonal polynomials with an infinite support.
The results are mainly concerned with the derivation of Christoffel type formulas, which allow to express the perturbed matrix biorthogonal  polynomials and its norms in terms of the original ones. The basic tool is the Gauss--Borel factorization of the Gram matrix, and particular attention is paid  to the non-associative character, in general, of the product of semi-infinite matrices.
	
The first transformation considered is that of Geronimus type, in where a right multiplication by  the inverse of a   matrix polynomial and an addition of adequate masses  is performed. The resolvent matrix  and connection formulas are given. Two different methods are developed. A spectral one, based on the spectral properties of the perturbing polynomial, and constructed in terms of the second kind functions.  This approach requires the perturbing matrix polynomial to have a  nonsingular leading term. Then, using spectral techniques and spectral jets,  Christoffel-Geronimus formulas for the transformed polynomials and norms are presented. For this type of transformations, the paper also proposes an alternative method, which does not require of spectral techniques,  that is valid also for singular leading coefficients. When the leading term is nonsingular a comparative of both methods is presented. The nonspectral method is applied to unimodular Christoffel perturbations, and a simple example for a degree one massless Geronimus perturbation is given.
A discussion on  Geronimus--Uvarov transformations is presented. These  transformations are rational perturbations of the  matrix of bivariate generalized functions together with an addition of appropriate masses, determined these last ones by the spectral properties of the polynomial denominator. As for the Geronimus case, two techniques are  applied, spectral   and  mixed spectral/nonspectral. Christoffel--Geronimus--Uvarov formulas are found  with both approaches and some applications are given. For example, to transformations preserving the symmetric character of the sesquilinear form, where the adjugate of the perturbing matrix polynomial is instrumental, and also to matrix hodographic  perturbations, with both polynomials of degree one.
To complete the list of transformations, additive perturbations and in particular matrix Uvarov transformations are considered.  Christoffel--Uvarov formulas are given for the perturbed biorthogonal families and its matrix norms, in terms of spectral jets of  the Christoffel--Darboux kernels. Applications to  discrete Sobolev matrix orthogonal polynomials are provided.
	
Finally, the transformation theory is  discussed in the context of the 2D non-Abelian Toda lattice and noncommutative KP hierarchies, understood as the theory of continuous transformations of quasidefinite sesquilinear forms. The interplay between transformations and   integrable flows is discussed. Miwa shifts, $\tau$-ratio matrix functions and Sato formulas are given. Bilinear identities, involving either Geronimus--Uvarov or Uvarov transformations,  first for the Baker functions, second  then  for the biorthogonal  polynomials and its second kind functions and finally for the $\tau$-ratio matrix functions are found.
	\end{abstract}

\maketitle
\thispagestyle{empty}

\newpage
\newgeometry{total={19cm,25cm},top=1.35cm, left=1.2cm}
\setcounter{tocdepth}{2}
\tableofcontents

\section{Introduction}
Perturbations of a linear functional $u$  in the linear space of polynomials with real coefficients have been extensively studied in the theory of orthogonal polynomials on the real line (scalar OPRL). In particular, when you deal with the positive definite case, i.e. linear functionals associated with probability measures supported in an infinite subset of the real line is considered, such perturbations provide an interesting information in the framework of Gaussian quadrature rules taking into account the perturbation yields new nodes and Christoffel numbers, see \cite{Gaut1,Gaut2}. Three perturbations have attracted the interest of the researchers. Christoffel perturbations, that  appear when you consider a new functional $\hat{u}= p(x) u$, where $p(x)$ is a polynomial, 
were studied in 1858 by the  German mathematician E. B.  Christoffel  in \cite{christoffel}, in the framework of Gaussian quadrature rules. He  found  explicit formulas relating the corresponding sequences of orthogonal polynomials with respect to two measures,  the Lebesgue measure  $\d\mu$ supported in the interval $(-1,1)$ and  $d\hat{\mu}(x)= p(x) d\mu(x)$, with $p(x)=(x-q_1)\cdots(x-q_N)$ a signed polynomial in the support of $\d\mu$,  as well as the distribution of their zeros as nodes in such quadrature rules. Nowadays, these are  called Christoffel formulas, and can be considered  a classical result in the theory of orthogonal polynomials which can be found in a number of  textbooks, see for example \cite{Chi,Sze,Gaut2}. Explicit relations between the corresponding sequences of orthogonal polynomials  have been extensively studied, see \cite{Gaut1}, as well as the connection between  the corresponding monic Jacobi matrices in the framework of the so-called Darboux transformations based on the LU factorization of such matrices \cite{Bue1}.
In the theory of orthogonal polynomials,  connection formulas between two families of orthogonal polynomials allow to express any polynomial of a given degree $n$ as a linear combination of all polynomials of degree less than or equal to $n$ in the second family. A noteworthy  fact regarding  the Christoffel finding is that in  some cases the number of terms does not grow with the degree $n$ but remarkably,  and on the contrary,  remain constant,  equal to the degree of the perturbing polynomial. See \cite{Gaut1,Gaut2} for more on the Christoffel type formulas.

Geronimus transformation appears when you are dealing with perturbed functionals $v$ defined by $p(x) v=u,$ where $p(x)$ is a polynomial. Such a kind of transformations were used by the Russian mathematician J. L. Geronimus, see \cite{Geronimus}, in order to have a nice proof of a result by W. Hahn \cite{Hann} concerning the characterization of classical orthogonal polynomials (Hermite, Laguerre, Jacobi, and Bessel) as those orthogonal polynomials whose first derivatives are also orthogonal polynomials,  for an English account of Geronimus' paper \cite{Geronimus} see \cite{Golinskii}.  Again, as happened  for the Christoffel transformation, within the Geronimus transformation one can find Christoffel type formulas, now in terms of the second kind functions, relating the corresponding sequences of orthogonal polynomials,  see for example the work of P. Maroni  \cite{Maro} for a  perturbation of the type $p(x)=x-a$. 
Despite that in \cite{Geronimus} no Christoffel type formula was derived, in the present paper, in order to distinguish these Christoffel type formulas from those for Christoffel transformations,  we refer to them as Christoffel--Geronimus formulas.
The relation between the monic Jacobi matrices has been studied by using the so-called Darboux transformation with parameter \cite{Bue1}. The more general problem related to linear functionals $u$ and $v$ satisfying $p(x)u= q(x)v,$ where $p(x), q(x)$ are polynomials has been analyzed in \cite{Uva} when $u,v$ are positive definite measures supported on the real line and in \cite{dini} when linear functionals are considered, see also \cite{Zhe}. The Russian mathematician V. B. Uvarov found \cite{Uva}  Christoffel type formulas, that allow for any pair of perturbing polynomials $p(x) $ and $q(x)$, to find the new orthogonal polynomials in terms of determinantal expressions of the original unperturbed second kind functions and orthogonal polynomials. On the other hand, the addition of a finite number of Dirac masses to a linear functional appears in the framework of the spectral analysis of fourth order linear differential operators with polynomial coefficients and with orthogonal polynomials as eigenfunctions. Therein you have the so called Laguerre-type, Legendre-type and Jacobi-type orthogonal polynomials introduced by H. L. Krall, see \cite{HKrall}, and later on deeply studied in \cite{AKrall}. A more general analysis from the point of view of the algebraic properties of the sequences of orthogonal polynomials associated to the linear functionals $u$ and $w= u + \sum_{n=0}^{N} M_{n} \delta(x- a_{n})$, the so-called general Uvarov transformation in \cite{Zhe}, has been done for the positive definite case in \cite{Uva}, and in more general cases in \cite{Marc1} and \cite{Chen}.

The case of Hermitian linear functionals in the linear space of Laurent polynomials with complex coefficients is intimately related with the theory of orthogonal polynomials on the unit circle (in short OPUC). Indeed, if $u$ is a positive definite Hermitian linear functional, then there exists a probability measure supported on an infinite subset of the unit circle  that provides an integral representation for such a functional (the solution of the trigonometric moment problem). In analogy with the case of OPRL, perturbations of such linear functionals have been analyzed in the literature. For the analogue of the Christoffel transformation, see \cite{Godoy1}, where extensions of the Christoffel determinantal type formulas were found using the original Szeg\H{o} polynomials and its  Christoffel--Darboux kernels.
For the Geronimus transformation see \cite{Godoy2} and for the Uvarov transformation, see \cite{Cachafeiro}.
 Later, in \cite{ismail}, alternative formulas \emph{á la Christoffel} were presented in terms determinantal expressions of the Szeg\H{o} polynomials and its reverse polynomials, and these results were extended in \cite{R. W. Ruedemann} for biorthogonal polynomials.
 The connection between the Hessenberg matrices ---GGT (Gragg--Geronimus--Tepliaev) matrices according to the terminology in \cite{Simon2}--- associated to the multiplication operator has been explored in \cite{Castillo}. For the CMV (Cantero--Moral--Velázquez) matrices, which constitute the representation of the multiplication operator in terms of the basis of orthonormal Laurent polynomials, see \cite{Cantero} where the connection with Darboux transformations and their applications to integrable systems has been  analyzed. Regarding the CMV ordering, orthogonal Laurent polynomials and Toda systems, see \S 4.4 of \cite{am2013}  where discrete Toda flows and its connection with Darboux transformations were discussed. For   a matrix version of this discussion see \cite{ari}.
 
Was  M. G. Krein in \cite{Krein} the first to discuss matrix orthogonal polynomials, for a review on the subject see \cite{DAS}.  The great activity in this scientific field has produced a vast  bibliography, treating among other things subjects like  inner products defined on the linear space of polynomials with matrix coefficients or  aspects as the existence  of the corresponding  sequences of matrix orthogonal polynomials in the real line,  see \cite{Dur5,Dur4,mir,Rod,van2}) and their applications in Gaussian quadrature for matrix-valued functions \cite{van}, scattering theory \cite{Ni,Ger} and system theory \cite{Fuh}. 
The  seminal paper \cite{Dur3} gave the key for further studies in this subject and, subsequently, some relevant advances has been achieved in the study of families of matrix orthogonal polynomials associated to second order linear differential operators as eigenfunctions and their structural properties \cite{Dur5,grupach,grupach2,Dur6}. In  \cite{Cas1}  sequences of orthogonal polynomials satisfying a first order linear matrix differential equation were found, which is a remarkable difference with the scalar scenario, where such a situation does not appear. The spectral problem for second order linear difference operators with polynomial coefficients has been considered in \cite{Al}. Therein four families of matrix orthogonal polynomials (as matrix relatives of Charlier, Meixner, Krawtchouk scalar polynomials and another one that seems not have any scalar relative) are obtained as illustrative examples of the method described therein. 

The Japanese mathematician M. Sato and the   Kyoto school  \cite{sato0,sato, date1,date2,date3} gave a  Lie group theoretical setting  for integrable hierarchies of nonlinear partial differential/difference equations, and in \cite{mulase}  a mathematical description of factorization problems, dressing procedure, and linear systems can be found. For the  2D lattice Toda lattice hierarchy see \cite{ueno-takasaki0,ueno-takasaki1,ueno}, and \S 3 in \cite{ueno}. Multicomponent versions of the Kadomtsev--Petvhiasvili (KP) hierarchy were analyzed in \cite{sato0,BtK2,Bogdanov,Dickey,manas-1,manas0}.  Block Hankel/Toeplitz reductions, discrete flows, additional symmetries and dispersionless limits for the multi-component 2D Toda lattice hierarchy were discussed  in \cite{manas1,manas2}, and the connection with multiple orthogonal polynomials was considered in  \cite{adler 2,amu}.
The papers \cite{adler-van moerbeke, adler-vanmoerbeke 0,adler-van moerbeke 1,adler-van moerbeke 1.1,adler-van moerbeke 2} 
 showed that the Gauss--Borel factorization problem is a keystone for the connection between integrable system  and orthogonal polynomials.
These papers clearly established  --from a group-theoretical setup-- why standard orthogonality of polynomials and integrability 
of nonlinear equations of Toda type where so close.
Non-Abelian versions of Toda equations and its relation with matrix orthogonal polynomials was studied, for example,  in \cite{mir,amu} (on the real line) and in \cite{mir2,ari0} (on the unit circle).

The transformations studied in this paper are also known in the Theory of  Orthogonal Polynomials, as well as in Integrable Systems Theory, as Darboux transformations,
 denomination  that was coined in  \cite{matveev}. The motivation, for that name, was the studies of the French mathematician G. Darboux that, when studying the  Sturm--Liouville theory in  \cite{darboux2},  explicitly  found these transformations, which he obtained by a simplification of a geometrical transformation founded previously by the French mathematician T. Moutard \cite{moutard}.
See \cite{Yoon} for a review of these transformations,  in the context of orthogonal polynomials, under the light of  the factorization of the Jacobi matrices. See also \cite{gru1,gru2} for the study  of bispectrality.
In the Differential Geometry, see \cite{eisenhart},  the Christoffel, Geronimus, Uvarov and linear spectral transformations are related to geometrical transformations like  the Laplace, Lévy, adjoint Lévy and the fundamental Jonas transformations. See \cite{dsm} for a discrete version of these geometrical transformations, and its connection with integrable systems. The interested reader may consult the excellent monographs \cite{matveev-book} and \cite{rogers-schief}.

The scheme  based on the Gauss--Borel factorization problem, used here for  transformations for matrix orthogonal polynomials or non-Abelian 2D Toda lattices,  has been applied  also in the following situations 
\begin{enumerate}
	\item Matrix orthogonal polynomials, its Christoffel transformations and the relation with non-Abelian Toda hierarchies were studied in \cite{alvarez2015Christoffel}, a paper that can be understood as the seed of the present one. 
	\item  Multiple orthogonal polynomials and multicomponent Toda \cite{am}.
	\item  For matrix orthogonal Laurent polynomials on the unit circle, CMV orderings,  and non-Abelian  lattices on the circle \cite{ari}.
	\item Multivariate orthogonal polynomials in several real variables and corresponding multispectral integrable Toda hierarchy  \cite{MVOPR,ari0}.
	\item Finally, multivariate orthogonal polynomials on the multidimensional unit torus, the  multivariate extension of the CMV ordering and integrable Toda hierarchies \cite{ari1}.
\end{enumerate}

\subsection{Objectives, methodology and results}
The main objective of this paper is to study the theory of transformations of matrix orthogonal polynomials in the real line and  to derive Christoffel type formulas for the perturbed polynomials and norms in terms of the primitive ones. The framework for matrix orthogonality is that of continuous sesquilinear forms in the space of matrix polynomials. We will extend the idea of  \emph{noyau-distribution} (that we call generalized kernel) introduced by the French mathematician L. Schwartz in \cite{Schwartz1} to this matrix context. These general non Hankel scenarios have been discussed in the scalar case in different contexts,  for an orthogonal polynomial approach see \cite{Bueno} while from an integrable systems point of view see \cite{adler,adler-van moerbeke}. We will consider an arbitrary nondegenerate continuous sesquilinear form given by a matrix of generalized kernels $u_{x,y}$ with a quasidefinite Gram matrix. This scheme contains the more usual choices of Gram matrices like those  of Hankel type, or those leading to matrix discrete orthogonality or matrix Sobolev orthogonality, but not only, as these examples correspond to matrices of generalized kernels supported by the diagonal $u_{x,y}=u_{x,x}$,  \cite{Hormander}, and our scheme is applicable to the general case, with support off the diagonal. The quasidefinite condition on the Gram matrix ensures for a block Gauss--Borel factorization and, consequently,  for the existence of two biorthogonal families of matrix polynomials.
The manipulation of semi-infinite matrices must be done with care. In particular, attention must be paid to the fact that, in general,  the product of three semi-infinite  matrices, $A,B$ and $C$ is not associative and, therefore, it could happen that  $(AB)C\neq A(BC)$, see \cite{cooke}. Basic objects in the theory of matrix orthogonal polynomials on the real line, like second kind functions and mixed  Christoffel--Darboux kernels will be instrumental in the finding
of Christoffel type formulas. 

The perturbations or transformations of our sesquilinear forms are of  rational type, with an additive term $v_{x,y}$.
For example, for the Geronimus-Uvarov  transformation, also known as linear spectral transformation for the scalar standard case \cite{Zhe}, we will consider the right multiplication of the matrix of generalized kernels by the inverse of a matrix polynomial and on its left by another matrix polynomial, $u_{x,y}\mapsto W_C(x)u_{x,y}(W_G(y))^{-1}+v_{x,y}$, with a Geronimus reduction when $W_C(x)=I_p$ is the identity matrix. As we deal with the inverse of matrix polynomials we need the spectrum of the denominator not to lay in the support of the generalized kernels, i.e. we will perturb matrices of  generalized kernels with well defined support.
The final objective is the finding of Christoffel type formulas for matrix versions of the Geronimus, Geronimus--Uvarov and Uvarov transformations. In doing so, we will use for the two first mentioned transformations two different methods, which happen to be equivalent when the leading coefficient of the perturbing polynomial in the denominator is nonsingular. The first one is based on
the spectral properties of the perturbing denominator matrix polynomial, and uses the second kind functions, Christoffel--Darboux kernels and its mixed version. For quasideterminantal Christoffel type expressions,  see Theorems \ref{teo:spectral} and \ref{teo:SCGU}.
The second method, which is based on the orthogonality relations, i.e., the Gauss--Borel factorization, allows to treat the singular leading term as well, and does not require of the second kind functions. The Christoffel type formulas, expressed in terms of quasideterminants,  derived within this method are given in  Theorems \ref{theorem:nonspectral} and \ref{theorem:nonspectralLST}.
These Geronimus and Geronimus--Uvarov allow for additive terms, the masses, constructed in the $y$-variable  in terms of the spectral elements, like eigenvalues and root polynomials, of the denominator polynomial, while in the $x$ variable do depend on linearly independent arbitrary vectorial generalized functions. Matrix Uvarov transformations $u_{x,y}\mapsto u_{x,y}+v_{x,y}$ are discussed in full generality by considering the addition of a matrix of generalized kernels $v_{x,y}$ with arbitrary linearly independent matrices of generalized functions in the $x$ variable accomplished with finite support arbitrary order matrices of generalized functions in the $y$-variable. The corresponding Christoffel type formulas are given in Theorem \ref{teo:uvarov}, and some applications to discrete Sobolev matrix orthogonal polynomials are presented.

An important motivation for the research contained in this paper is  the strong relation of the Gauss--Borel factorization of general Gram matrices and integrable systems of Toda type, see \cite{adler,adler-van moerbeke,manas3}. In \cite{alvarez2015Christoffel} we studied how the matrix Christoffel transformation accomplishes with the continuous Toda flows. Similar argumentation leads, in this paper, to the extension of matrix Geronimus, Geronimus--Uvarov and Uvarov transformations to the non-Abelian 2D  Toda lattice and noncommutative KP hierarchies. Using first order Christoffel and Geronimus perturbations we construct Miwa shifts and corresponding Sato type formula for the biorthogonal matrix polynomials and its second kind functions, for that aim we introduce two $\tau$-type matrix functions, see Theorem \ref{teo:Sato}.
We also discuss, in Theorem \ref{teo:bilinear}, three different sets of bilinear identities, the first one involving the Baker functions,  the second one giving an identity constructed in terms of the biorthogonal matrix polynomials and its second kind functions, and finally, the third one, involving the mentioned $\tau$-ratio matrices.

\subsection{Layout of the paper}
We continue this introduction with two introductory  subsections. One is focused on the spectral theory of matrix polynomials, we follow \cite{lan1}. The  other is a basic background on  matrix orthogonal polynomials, see \cite{DAS}. In the second section we extend the Geronimus transformations to the matrix realm, and find connection formulas for the biorthogonal polynomials and the Christoffel--Darboux kernels. These developments allow for the finding of the Christoffel--Geronimus formula for matrix perturbations of Geronimus type. As we said we present two different schemes. In the first one, which can be applied when the perturbing polynomial has a nonsingular leading coefficient,  we express the perturbed objects in terms  of spectral jets of the primitive second kind functions and Christoffel--Darboux kernels. We present a second approach, applicable even when the leading coefficient is singular. For each method we consider two different situations, the less interesting case of  biorthogonal polynomials of  degree less than the degree of the perturbing polynomial,  and the much more interesting situation whence the degrees of the families of biorthogonal polynomials are greater than or equal to the degree of the perturbing polynomial.
To end the section,  we compare spectral versus nonspectral methods and present a number of applications. In particular, we deal with
unimodular polynomial matrix perturbations and degree one matrix Geronimus transformations.
Matrix Geronimus--Uvarov transformations are discussed in \S 3,  where  we extend the ideas developed for the Geronimus case in \S 2.
 Mixed spectral/nonspectral Christoffel–Geronimus–Uvarov formulas are given  and some
 applications to Christoffel transformations with singular leading coefficients  and spectral symmetric transformations are discussed. As a last example, we explicitly treat
 degree one Geronimus--Uvarov transformations. In \S 4,  the matrix extension of Uvarov transformations is given within the framework of additive perturbations of generalized kernels. Matrix Christoffel–Uvarov formulas for finite discrete support additive perturbations are given and reductions to
 Uvarov perturbations supported by the diagonal $y = x$ are discussed. Finally, in \S 5 applications to non-Abelian Toda  and noncommutative KP hierarchies are considered. In particular, Sato formulas are given and bilinear identities derived.
For a table containing all transformations charting the  corresponding results see \S\ref{appendix}.

\subsection{On spectral theory of matrix polynomials}%\label{S:matrix polynomials}
Here we give some background material regarding the spectral theory of matrix polynomials \cite{lan1,Mark2}.
\begin{defi}
Let  $A_0, A_1\cdots,A_N\in \mathbb{C}^{p\times p}$ be  square  matrices of size $p\times p$ with complex entries.  Then
\begin{align}\label{eq:mp}
W(x)=A_N x^N+A_{N-1}x^{N-1}+\cdots +A_1x+A_0
\end{align}
is said to be  a  matrix polynomial of degree $N$, $\deg (W(x))=N$. The matrix polynomial  is said to be  monic when $A_N=I_p$,  where $I_p\in\mathbb C^{p\times p}$ denotes the identity matrix. The linear space  --a bimodule for the ring of matrices $\mathbb C^{p\times p}$-- of  matrix polynomials with coefficients in  $\mathbb{C}^{p\times p}$ will be denoted by
$\mathbb{C}^{p\times p}[x]$.
\end{defi}

\begin{defi}[Eigenvalues]
The spectrum,  or the set of eigenvalues, $\sigma(W(x))$ of a matrix polynomial $W$ is the zero set of  $\det W(x)$, i.e.
	\begin{align*}
	\sigma(W(x)):=\{x\in\mathbb C: \det W(x)=0\}.
	\end{align*}
\end{defi}

\begin{pro}
	A monic matrix polynomial $W(x)$, $\deg (W(x))=N$,  has $Np$ (counting multiplicities) eigenvalues or zeros, i.e., we can write
\begin{align*}
\det W(x)=\prod_{a=1}^q(x-x_a)^{\alpha_a},
 \end{align*}
with $Np=\alpha_1+\dots+\alpha_q$.
\end{pro}
\begin{rem}
Given the spectrum $\sigma(W(x))=\{x_1,\dots,x_q\}$, when we need to discuss generic properties  associated to an eigenvalue, and there is no need to specify which, for the sake of simplicity  will denote such an eigenvalue by $x_0$.  Thus, $x_0$ could be any of the eigenvalues $x_1$, $x_2,\dots,x_q$.
\end{rem}

\begin{rem}
	We use  $A^{\top}$ to denote   the   transpose of $A$, which constitutes an  antiautomorphism of order two in the ring of complex matrices.
From $\det \Big(\big(W(x)\big)^\top\Big)=\det (W(x))$  it follows that
	the spectrum, eigenvalues, and multiplicities of $W(x)$ and  $\big(W(x)\big)^\top$ coincide. Consequently,
\begin{align*}
	\operatorname{dim}\operatorname{Ker}\Big((W(x))^\top\Big)&=\operatorname{dim}\operatorname{Ker}(W(x)).
\end{align*}
\end{rem}

\begin{rem}
We will denote the dual space of $\mathbb C^p$, or space of covectors,  by $\big(\mathbb C^p	\big)^*$. Given a linear subspace $\mathcal V\subset \mathbb C^p$ we define its  orthogonal complement as $\mathcal V^\perp:=\{l\in \big(\mathbb C^p	\big)^*: l(v)=0juad \forall v\in\mathcal V\}$. Given a linear map $A:\mathbb C^p\to\mathbb C^p$, which we identify to its matrix $A\in\mathbb C^{p\times p}$, we have that the left null space $\operatorname{Ker}(A^\top)\cong \big(\operatorname{Im}(A)\big)^\perp$.  Indeed, any vector $v$ in the cokernel satisfies $v^\top A=0$, thus
$v^\top\in \big(\operatorname{Im}(A)\big)^\perp$.
\end{rem}

\begin{pro}\label{pro:partial multiplicity}
	Any  nonsingular matrix polynomial $W(x)\in \mathbb C^{p\times p}[x]$, $\det W(x)\neq 0$, can be represented
	\begin{align*}
W(x)=E_{x_0}(x)\operatorname{diag} ((x-x_0)^{\kappa_1},\dots ,(x-x_0)^{\kappa_m})F_{x_0}(x)
	\end{align*}at $x=x_0\in\mathbb C$,
	where $E_{x_0}(x)$ and $F_{x_0}(x)$ are nonsingular  matrices and $\kappa_1\leq\dots\leq\kappa_m$ are nonnegative integers. Moreover, $\{\kappa_1,\dots,\kappa_m\}$ are uniquely determined by $W(x)$ and they are known as partial multiplicities of $W(x)$ at $x_0$.
\end{pro}

\begin{defi}
	 For an  eigenvalue  $x_0$ of  a monic matrix polynomial $W(x)\in\mathbb C^{p\times p}[x]$, then:
	\begin{enumerate}
			\item  A non-zero vector $r_{0}\in \mathbb C^p$ is said to be a right eigenvector, with eigenvalue $x_0\in\sigma(W(x))$, whenever
		$W(x_0)r_{0}=0$, i.e., $r_{0}\in\operatorname{Ker} W(x_0)\neq \{0\}$.
			\item A non-zero covector $l_{0}\in \big(\mathbb C^p\big)^*$ is said to be an left eigenvector, with eigenvalue  $x_0\in\sigma(W(x))$, whenever
		$l_{0}W(x_0)=0$,$\big(l_{0}\big)^\top\in\big( \operatorname{Ker}(W(x_0))\big)^\perp=\operatorname{Ker}\big( (W(x_0))^\top\big)\neq \{0\}$.
	\item 	A sequence of vectors $\{r_{0},r_{1},\ldots, r_{m-1}\}$ is said to be a right Jordan chain of length $m$ corresponding to the eigenvalue $x_0\in\sigma(W(x))$, if  $r_{0}$ is an right eigenvector of $W(x_0)$ and
			\begin{align*}
		\sum_{s=0}^{j}\frac{1}{s!}	\frac{\operatorname{d}^sW}{\operatorname{d} x^s}	\Big|_{x=x_0}r_{j-s}&=0, & j&\in\{0,\ldots,m-1\}.
		\end{align*}
\item 		A sequence of covectors $\{l_{0},l_{a,1}\ldots, l_{m-1}\}$ is said to be a left Jordan chain of length $m$, corresponding to $x_0\in\sigma(W^\top)$, if
		$\{(l_{0})^\top,(l_{1})^\top,\ldots, (l_{m-1})^\top\}$ is a right Jordan chain of  length $m$ for the matrix polynomial  $\big(W(x)\big)^\top$.	
		\item A right root polynomial  at $x_0$ is a non-zero vector polynomial  $r(x)\in\mathbb C^p[x]$ such that $W(x)r(x)$ has a zero of  certain order at $x=x_0$, the order of this zero is called the order of the root polynomial. Analogously, a  left root polynomial is a non-zero covector polynomial  $l(x)\in\mathbb (\mathbb C^p)^*[x]$ such that $l(x_0)W(x_0)=0$.
		\item The maximal lengths, either of right or left Jordan chains corresponding to  the eigenvalue $x_0$, are called the multiplicity of the eigenvector $r_{0}$ or $l_{0}$ and are denoted by $m(r_{0})$ or $m(l_{0})$, respectively.
	\end{enumerate}
\end{defi}

\begin{pro}
Given an eigenvalue $x_0\in \sigma(W(x))$  of a monic matrix polynomial $W(x)$, multiplicities of right and left eigenvectors coincide and they are equal to the corresponding   partial multiplicities $\kappa_i$.
\end{pro}
	The above definition generalizes the  concept of Jordan chain for degree one matrix polynomials.
\begin{pro}
	The Taylor expansion of a right root polynomial $r(x)$, respectively of a  left root polynomial $l(x)$, at a given eigenvalue $x_0\in\sigma(W(x))$  of a monic matrix polynomial $W(x)$,
	\begin{align*}
	r(x)&=\sum_{j=0}^{\kappa-1}r_j (x-x_0)^j, &\text{ respectively $l(x)=\sum_{j=0}^{\kappa-1}l_j (x-x_0)^j$,}
		\end{align*}
provides  us with right    Jordan chain
\begin{align*}
&\{r_0,r_1,\dots,r_{\kappa-1}\},& \text{respectively,  left  Jordan chain $\{l_0,l_1,\dots,l_{\kappa-1}\}$.}
\end{align*}
\end{pro}
\begin{pro}
	\label{pro:canonical jordan}
	Given an eigenvalue $x_0\in\sigma(W(x))$ of a monic matrix polynomial $W(x)$, with multiplicity $s=\dim\operatorname{Ker}W(x_0)$, we can construct $s$ right root polynomials, respectively  left root polynomials, for $i\in\{1,\dots,s\}$,
	\begin{align*}
	r_i(x)=&\sum_{j=0}^{\kappa_i-1}r_{i,j}(x-x_0)^j, &\text{ respectively $l_i(x)=\sum_{j=0}^{\kappa_i-1}l_{i,j}(x-x_0)^j$,}
	\end{align*}
	where $r_i(x)$	are right root polynomials (respectively  $ l_i(x)$	are left root polynomials)  with the largest order $\kappa_i$  among all right root polynomials,  whose right eigenvector does not belong to
	$\mathbb C\{r_{0,1},\dots,r_{0,i-1}\}$  (respectively left root polynomials whose left eigenvector  does not belong to  $\mathbb C\{l_{0,1},\dots,l_{0,i-1}\}$).
\end{pro}

\begin{defi}[Canonical Jordan chains]\label{def:canonical Jordan}
	A canonical set of right Jordan chains (respectively left  Jordan chains) of the monic matrix polynomial $W(x)$ corresponding to the eigenvalue $x_0\in\sigma(W(x))$ is,
	in terms of the right  root polynomials (respectively left root polynomials) described in Proposition \ref{pro:canonical jordan}, the  following sets of vectors
	\begin{align*}
&	\{r_{1,0}\dots,r_{1,\kappa_1-1},\dots, r_{s,0}\dots,r_{s,\kappa_r-1}\}, & \text{ respectively, covectors	$\{l_{1,0}\dots,l_{1,\kappa_1-1},\dots, l_{s,0}\dots,l_{s,\kappa_r-1}\}$.}
	\end{align*}
\end{defi}

	\begin{pro}\label{pro3}
	For a monic matrix polynomial $W(x)$ the lengths $\{\kappa_1,\dots,\kappa_r\}$ of the Jordan chains in a canonical set of Jordan chains of $W(x)$ corresponding to the eigenvalue $x_0$, see Definition \ref{def:canonical Jordan}, are the nonzero  partial multiplicities of $W(x)$ at $x=x_0$ described in Proposition \ref{pro:partial multiplicity}.
	\end{pro}
	
\begin{defi}[Canonical Jordan chains and root  polynomials]\label{def:adapted}
		For each eigenvalue $x_a\in\sigma(W(x))$  of a monic matrix polynomial $W(x)$, with multiplicity $\alpha_a$ and  $s_a=\dim \operatorname{Ker} W(x_a)$, $a\in\{1,\dots,q\}$, we  choose a canonical set of right  Jordan chains, respectively left Jordan chains,
	\begin{align*}
	&\Big\{r_{j,0}^{(a)},\dots,r_{j,\kappa_{j}^{(a)}-1}^{(a)}\Big\}_{j=1}^{s_a}, &\text{respectively $\Big\{l_{j,0}^{(a)},\dots,l_{j,\kappa_{j}^{(a)}-1}^{(a)}\Big\}_{j=1}^{s_a}$,}
	\end{align*}
	and, consequently, with partial multiplicities satisfying  $\sum_{j=1}^{s_a}\kappa_j^{(a)}=\alpha_a$.
	Thus, we can  consider the following   right  root polynomials
	\begin{align}\label{vecmil}
		r_{j}^{(a)}(x)&=\sum_{l=0}^{\kappa_j^{(a)}-1}r_{j,l}^{(a)}(x-x_a)^l, &\text{respectively  left root polynomials $l_{j}^{(a)}(x)=\sum_{l=0}^{\kappa_j^{(a)}-1}l_{j,l}^{(a)}(x-x_a)^l$.}
	\end{align}
\end{defi}
\begin{defi}[Canonical Jordan pairs]
	We also define the corresponding canonical Jordan pair $(X_a,J_a)$ with $X_a$
 the matrix
	\begin{align*}
X_a:=\begin{bmatrix}
r_{1,0}^{(a)},\dots,r_{1,\kappa_{1}^{(a)}-1}^{(a)},\dots,r_{s_a,0}^{(a)},\dots,r_{s_a,\kappa_{s_a}^{(a)}-1}^{(a)}
\end{bmatrix}\in\mathbb C^{p\times\alpha_a},
	\end{align*}
	and  $J_a$ the matrix
	\begin{align*}
	J_a:=\diag(J_{a,1},\dots,J_{a,s_a})\in\mathbb C^{\alpha_a\times \alpha_a},
	\end{align*}
	where $J_{a,j}\in\mathbb C^{\kappa^{(a)}_j\times \kappa^{(a)}_j}$ are the Jordan blocks of the eigenvalue $x_a\in\sigma(W(x))$.
	Then, we say that $(X,J)$ with
	\begin{align*}
	X&:=\begin{bmatrix}
	X_1,\dots,X_q
	\end{bmatrix}\in\mathbb C^{p\times Np}, & J&:=\diag(J_1,\dots,J_q)\in\mathbb C^{Np\times Np},	\end{align*}
	is a canonical Jordan pair for $W(x)$.
\end{defi}
	
We have the important result, see \cite{lan1},
\begin{pro}
	The Jordan pairs of a monic matrix polynomial $W(x)$ satisfy
	\begin{align*}
	A_0X_a+A_1X_aJ_a+\dots+A_{N-1} X_a(J_a)^{N-1}+X_a(J_a)^N&=0_{p\times \alpha_a},\\
		A_0X+A_1XJ+\dots+A_{N-1} XJ^{N-1}+XJ^N&=0_{p\times Np}.
	\end{align*}
\end{pro}	
A key property, see Theorem 1.20 of \cite{lan1}, is
\begin{pro}\label{pro:Jordan pair0}
For any  Jordan pair $(X,J)$ of a monic matrix polynomial $W(x)=I_px^N+A_{N-1}x^{N-1}+\dots+A_0$  the matrix
	\begin{align*}
	\begin{bmatrix}
	X\\
	XJ\\
	\vdots\\
	XJ^{N-1}
	\end{bmatrix}\in\mathbb C^{Np\times Np}
	\end{align*}
	is nonsingular.
\end{pro}
\begin{defi}[Jordan triple]
Given 	\begin{align*}
	Y=\begin{bmatrix}
	Y_1\\\vdots\\Y_q
	\end{bmatrix}\in\mathbb C^{N p\times p},
	\end{align*}
	with  $Y_a\in \mathbb C^{\alpha_a\times p}$ ,
we say that  $(X,J,Y)$ is a Jordan triple whenever
	\begin{align*}
	\begin{bmatrix}
	X\\XJ\\\vdots\\
	XJ^{N-1}
	\end{bmatrix} Y=\begin{bmatrix}
	0_p\\\vdots\\0_p\\I_p
	\end{bmatrix}.
	\end{align*}
\end{defi}

Moreover, Theorem 1.23 of \cite{lan1}, gives the following characterization
\begin{pro}\label{pro:Jordan pair}
Two matrices $X\in\mathbb C^{p\times Np}$ and $J\in\mathbb C^{Np\times Np}$ constitute  a Jordan pair of a monic matrix polynomial $W(x)=I_px^N+A_{N-1}x^{N-1}+\dots+A_0$ if and only if the two following properties hold
\begin{enumerate}
	\item The matrix
	\begin{align*}
\begin{bmatrix}
	X\\
	XJ\\
	\vdots\\
	XJ^{N-1}
\end{bmatrix}
	\end{align*}
	is nonsingular.
	\item 	\begin{align*}
	A_0X+A_1XJ+\dots+A_{N-1} XJ^{N-1}+XJ^N&=0_{p\times Np}.
	\end{align*}
\end{enumerate}
\end{pro}
	
\begin{pro}\label{pro:adapted_root}
	Given a monic matrix polynomial $W(x)$ the  adapted  root polynomials given in Definition \ref{def:adapted} satisfy
	\begin{align}\label{eq:Wr}
\big(W(x) r_j^{(a)}(x)\big)^{(m)}_{x_a}&=0, &
\big(l_j^{(a)}(x)W(x) \big)^{(m)}_{x_a}&=0,&m&\in\{0,\dots,\kappa^{(a)}_j-1\}, &
	j&\in\{1\dots,s_a\}.
	\end{align}
	Here, 	given a function $f(x)$ we use the following notation for its derivatives evaluated at an  eigenvalue $x_a\in\sigma(W(x))$
	\begin{align*}
	(f)^{(m)}_{x_a}	:=\lim_{x\to x_a}\frac{\operatorname{d}^mf}{\operatorname{d} x^m}.
	\end{align*}
\end{pro}

In this paper we assume  that the partial multiplicities are ordered in an increasing way, i.e.,  $\kappa_1^{(a)}\leq \kappa_2^{(a)}\leq\cdots\leq \kappa_{s_a}^{(a)}$.
\begin{pro}\label{pro:lWr}
If $r_{i}^{(a)}$ and $l_j^{(a)}$ are right and left root polynomials corresponding to the eigenvalue  $x_a\in\sigma(W(x))$,  then
a polynomial
\begin{align*}
w_{i,j}^{(a)}(x)&=\sum_{m=0}^{d_{i,j}^{(a)}}w_{i,j;m}^{(a)}x^m\in\mathbb C[x], &d_{i,j}^{(a)}&:=\kappa^{(a)}_{\min(i,j)}+N-2,
\end{align*}
exists such that
\begin{align}\label{leftright}
l_i^{(a)}(x)W(x)r_j^{(a)}(x)=(x-x_a)^{\kappa^{(a)}_{\max(i,j)}}w_{i,j}^{(a)}(x).
\end{align}   	
\end{pro}
\begin{proof}
	From Proposition \ref{pro:adapted_root} it follows that  a covector polynomial $T_1(x)$  and a  vector polynomial  $T_2(x)$, both of degree $N$, exist such that
\begin{align}
\label{eq:Wraj}	l_i^{(a)}(x)W(x)&=(x-x_a)^{\kappa_i^{(a)}}T_1(x),& W(x)r_j^{(a)}(x)=(x-x_a)^{\kappa_j^{(a)}}T_2(x).
\end{align}
	Thus
\begin{align*}
l_i^{(a)}(x)W(x)r_j^{(a)}(x)&=(x-x_a)^{\kappa_i^{(a)}}T_1(x)r_j^{(a)}(x),& l_i^{(a)}(x)W(x)r_j^{(a)}(x)=(x-x_a)^{\kappa_j^{(a)}}l_i^{(a)}(x)T_2(x),
\end{align*}
	and the result is proved.
\end{proof}

\begin{defi}[Spectral jets]\label{def:spectral jets}
	Given a matrix  function $f(x)$ smooth in region $\Omega\subset \mathbb C$ with $x_a\in\overline{\Omega}$, a point in the closure of $\Omega$  we consider its  matrix spectral jets
	\begin{align*}
	\mathcal J^{(i)}_{f}(x_a)&:=\lim_{x\to x_a}\begin{bmatrix}
	f(x), \dots, \dfrac{f^{(\kappa^{(a)}_i-1)}(x)}{(\kappa^{(a)}_i-1)!}
	\end{bmatrix}\in\mathbb C^{p\times p\kappa^{(a)}_i},\\
	\mathcal J_{f}(x_a)&:=\begin{bmatrix}
	\mathcal J^{(1)}_{f}(x_a),\dots,	\mathcal J^{(s_a)}_{f}(x_a)
	\end{bmatrix}\in\mathbb C^{p\times p\alpha_a},\\
	\mathcal J_{f}&:=\begin{bmatrix}
	\mathcal J_{f}(x_1),\dots,	\mathcal J_{f}(x_q)
	\end{bmatrix}\in\mathbb C^{p\times Np^2},
	\end{align*}
	and given a Jordan pair  the root spectral jet vectors
	\begin{align*}
	\boldsymbol{\mathcal J}^{(i)}_{f}(x_a)&:=\lim_{x\to x_a}\begin{bmatrix}
	f(x_a)r^{(a)}_i(x_a), \dots, \dfrac{(f(x)r^{(a)}_i(x))^{(\kappa^{(a)}_i-1)}_{x_a}}{(\kappa^{(a)}_i-1)!}
	\end{bmatrix}\in\mathbb C^{p\times \kappa^{(a)}_i}\\
	\boldsymbol{\mathcal J}_{f}(x_a)&:=\begin{bmatrix}
	\boldsymbol{\mathcal J}^{(1)}_{f}(x_a),\dots,		\boldsymbol{\mathcal J}^{(s_a)}_{f}(x_a)
	\end{bmatrix}\in\mathbb C^{p\times \alpha_a},\\
	\boldsymbol{\mathcal J}_{f}&:=\begin{bmatrix}
	\boldsymbol{\mathcal J}_{f}(x_1),\dots,	\boldsymbol{\mathcal J}_{f}(x_q)
	\end{bmatrix}\in\mathbb C^{p\times Np}.
	\end{align*}
\end{defi}

\begin{defi}
	We consider the following jet matrices
	\begin{align*}
	\mathcal{ Q}_{n;i}^{(a)}&:=\boldsymbol{\mathcal J}^{(i)}_{I_px^n}(x_a)=\Bigg[(x_a)^nr^{(a)}_i(x_a),\big( x^nr^{(a)}_i(x)\big)_{x_a}^{(1)}, \dots,\frac{\big(x^nr^{(a)}_i(x)\big)_{x_a}^{(\kappa_i^{(a)}-1)}}{(\kappa_i^{(a)}-1)!}\Bigg]\in\mathbb C^{p\times \kappa^{(a)}_i},\\
	\mathcal{{ Q}}_{n}^{(a)}&:=\boldsymbol{\mathcal J}_{I_px^n}(x_a)=\Big[\mathcal Q_{n;1}^{(a)},\dots,\mathcal Q_{n;s_a}^{(a)}\Big]\in\mathbb C^{p\times \alpha_a},\\
	\mathcal{{ Q}}_{n}&:=\boldsymbol{\mathcal J}_{I_px^n}=\Big[\mathcal Q_{n}^{(1)},\dots,\mathcal Q_{n}^{(q)}\Big]\in\mathbb C^{p\times Np},\\
	\mathcal Q&:=\boldsymbol{\mathcal J}_{\chi_{[N]}}=\begin{bmatrix}
	\mathcal Q _0\\\vdots\\\mathcal Q_{N-1}
	\end{bmatrix}\in\mathbb C^{Np\times Np},
	\end{align*}
	where $(\chi_{[N]}(x))^\top :=\begin{bmatrix}
	I_p,\dots,I_px^{N-1}
	\end{bmatrix}\in\mathbb C^{p\times Np}[x]$.
\end{defi}
\begin{ma}[Root spectral jets and Jordan pairs]\label{lemma:pair}
	Given a canonical Jordan pair $(X,J)$ for the monic matrix polynomial $W(x)$ we have that
	\begin{align*}	
	\mathcal Q_n&=XJ^{n},&n&\in\{0,1,\dots\}.
	\end{align*}
	Thus, any polynomial $P_n(x)=\sum_{j=0}^nP_j x^j$ has as its spectral jet vector corresponding to $W(x)$ the following matrix
	\begin{align*}
	\boldsymbol{\mathcal J}_P=P_0X+P_1XJ+\dots+P_nXJ^{n-1}.
	\end{align*}
\end{ma}
\begin{proof}
	The computation
	\begin{align*}
	\frac{1}{m!}\big( x^nr^{(a)}_i(x)\big)_{x_a}^{(m)}&=\frac{1}{m!}\sum_{k=0}^m\binom{m}{k}(x^n)_{x_a}^{(m-k)}\big(r^{(a)}_i(x)\big)^{(k)}_{x_a}\\
	&=\frac{1}{m!}\sum_{k=0}^m\frac{m!}{k!(m-k)!}\frac{n!}{(n-m+k)!}(x_a)^{n-m+k}k! r_{i,k}^{(a)}\\
	&=\sum_{k=0}^m\binom{n}{m-k}(x_a)^{n-m+k}r^{(a)}_{i,k}\\&=\begin{bmatrix}	
	r^{(a)}_{i,0},\dots,r^{(a)}_{i,m}
	\end{bmatrix}
	\begin{bmatrix}
	(x_a)^{n-m}\binom{n}{m}\\
	\vdots\\
	x_a^{n}\binom{n}{n}
	\end{bmatrix},
	\end{align*}
	leads  to
	\begin{align*}
	\mathcal Q_{n;i}^{(a)}&=\begin{bmatrix}
	r^{(a)}_{i,0},\dots,r^{(a)}_{i,\kappa_i^{(a)}-1}
	\end{bmatrix}
	\begin{bmatrix}
	x_a^{n}&x_a^{n-1}\binom{n}{1}&\cdots&x_a^{n-\kappa_i^{(a)}+1}\binom{n}{\kappa_i^{(a)}-1}\\
	0       &  x_a^{n}&\cdots&               x_a^{n-\kappa_i^{(a)}+2}\binom{n}{\kappa_i^{(a)}-2}\\
	\vdots    &     \ddots    & \ddots     &\vdots\\
	0    & \dots        &            &x_a^{n}
	\end{bmatrix}\\
	&=\begin{bmatrix}
	r^{(a)}_{i,0},\dots,r^{(a)}_{i,\kappa_i^{(a)}-1}
	\end{bmatrix}\left(J_{a,i}\right)^n.
	\end{align*}
	Consequently, in terms of the Jordan pairs associated with the right root polynomials, we have $X_a\left(J_a\right)^n=\mathcal Q_n^{(a)}$ and
	\begin{align*}
	\mathcal Q_n=XJ^n,
	\end{align*}
	which is a nonsingular matrix, see Propositions \ref{pro:Jordan pair0} and \ref{pro:Jordan pair}. 	
\end{proof}
\begin{rem}
From Proposition \ref{pro:Jordan pair0} we conclude that the root spectral jet of $W$ is a zero rectangular matrix.
\end{rem}

\begin{defi}\label{def:B}
	If $W(x)=\sum\limits_{k=0}^{N}A_{k}x^k\in\mathbb C^{p\times p}[x]$  is a  matrix polynomial of degree $N$, we introduce
	the matrix
	\begin{align*}
	\mathcal B:=\begin{bmatrix}
	A_1 & A_2 & A_3 &\dots &A_{N-1} &A_N\\
	A_2  &A_3& \vdots&\iddots&A_N&0_p\\
	A_3 &\dots&A_{N-1}&\iddots&0_p&0_p\\\vdots&\iddots&\iddots&\iddots&&\vdots\\
	A_{N-1} &A_N&0_p&&&\\
	A_N&0_p& 0_p& & \dots&0_p\end{bmatrix}\in\mathbb C^{Np\times Np}.
	\end{align*}
\end{defi}

\begin{ma}\label{lemma:triple}
	Given a  Jordan triple $(X,J,Y)$ for the monic matrix polynomial $W(x)$ we have
	\begin{align*}	
	\mathcal Q&=\begin{bmatrix}
	X\\
	XJ\\
	\vdots\\
	XJ^{N-1}
	\end{bmatrix}, &
	(\mathcal B\mathcal Q)^{-1}=\begin{bmatrix}
	Y, JY,\dots, J^{N-1}Y
	\end{bmatrix}=:\mathcal R.
	\end{align*}
\end{ma}
\begin{proof}
	From Lemma \ref{lemma:pair} we deduce that
	\begin{align*}
	\mathcal Q=\begin{bmatrix}
	X\\
	XJ\\
	\vdots\\
	XJ^{N-1}
	\end{bmatrix}
	\end{align*}
	which is nonsingular, see Propositions \ref{pro:Jordan pair0} and \ref{pro:Jordan pair}. 	
	The biorthogonality condition (2.6)  of \cite{lan1}  for $\mathcal R$ and $\mathcal Q$ is
	\begin{align*}
	\mathcal R \mathcal B \mathcal Q =I_{Np},
	\end{align*}
	and if  $(X,J,Y)$ is a canonical Jordan triple, then
	\begin{align}\label{eq:RJY}
	\mathcal R&=\begin{bmatrix}
	Y, J Y,\dots, J^{N-1} Y
	\end{bmatrix}.
	\end{align}
\end{proof}

\begin{pro}
	The matrix $\mathcal R_n:=\begin{bmatrix}
	Y, J Y,\dots, J^{n-1} Y
	\end{bmatrix}\in\mathbb C^{Np\times np}$ has  full rank.
\end{pro}

Regarding the matrix $\mathcal B$, 
\begin{defi}\label{defi:V}
	Let  us  consider the bivariate matrix polynomial
	\begin{align*}
	{\mathcal V}(x,y):=\big((\chi(y))_{[N]}\big)^\top \mathcal B(\chi(x))_{[N]}\in\mathbb C^{p\times p}[x,y],
	\end{align*}
	where $A_j$ are the matrix coefficients of $W(x)$, see \eqref{eq:mp}.
\end{defi}

We consider the complete homogeneous symmetric  polynomials in two variables
\begin{align*}
h_n(x,y)=\sum_{j=0}^{n}x^jy^{n-j}.
\end{align*}
For example, the first four polynomials are
\begin{align*}
h_0(x,y)&=1, & h_1(x,y)&=x+y, & h_2(x,y)&=x^2+xy+y^2,& h_3(x,y)&=x^3+x^2y+xy^2+y^3.
\end{align*}
\begin{pro}\label{pro:symdefV}
	In terms of   complete homogeneous symmetric polynomials in two variables we can write
	\begin{align*}
	{\mathcal V}(x,y)&=\sum_{j=1}^{N}A_{j}h_{j-1}(x,y).
	\end{align*}
\end{pro}

\subsection{On orthogonal matrix polynomials}

The polynomial ring $\mathbb C^{p\times p}[x]$ is a free bimodule over the ring of matrices $\mathbb C^{p\times p}$ with a basis given by
$\{I_p,I_p x, I_p x^2,\dots\}$.  Important free bisubmodules are the sets $\mathbb C_m^{p\times p}[x]$ of matrix polynomials of degree less than or equal to $m$.  A basis, which has cardinality $m+1$, for $\mathbb C_m^{p\times p}[x]$  is $\{I_p,I_p x, \dots, I_p x^m\}$; as $\mathbb C$ has the invariant basis number (IBN) property so does $\mathbb C^{p\times p}$, see \cite{rowen}. Therefore, being $\mathbb C^{p\times p}$ an IBN ring,  the rank of the free module $\mathbb C_m^{p\times p}[x]$ is unique and equal to $m+1$, i.e. any other basis has the same cardinality. Its algebraic dual $\big(\mathbb C_m^{p\times p}[x]\big)^*$ is the set of  homomorphisms
$\phi :\mathbb C_m^{p\times p}[x]\rightarrow \mathbb C^{p\times p}$ which are, for the right module,  of the form
\begin{align*}
\langle \phi,P(x)\rangle&=\phi_0 p_0+\dots+\phi_m p_m,  & P(x)&=p_0+\dots +p_mx^m,
\end{align*}
where $\phi_k\in\mathbb C^{p\times p}$.
Thus, we can identify the dual of the right module with the corresponding left submodule. This dual is a free module with a unique rank, equal to $m+1$, and a dual basis $\{(I_p x^k)^*\}_{k=0}^m$ given by
\begin{align*}
\langle(I_px^k)^*,I_p x^l\rangle=\delta_{k,l}I_p.
\end{align*}
We have similar statements for the left module $\mathbb C_m^{p\times p}[x]$, being its dual a right module
\begin{align*}
\langle P(x),\phi\rangle&=P_0\phi_0 +\dots+P_m\phi_m , & \langle I_p x^l,(I_px^k)^*\rangle&=\delta_{k,l}I_p.
\end{align*}

 \begin{defi}[Sesquilinear form]\label{def:sesquilinear}
	A  sesquilinear  form  $\prodint{\cdot,\cdot}$  on the bimodule $\mathbb{C}^{p\times p}[x]$ is a continuous map
\begin{align*}
\begin{array}{cccc}
\prodint{\cdot,\cdot}: &\mathbb{C}^{p\times p}[x]\times\mathbb{C}^{p\times p}[x]&\longrightarrow &\mathbb{C}^{p\times p},\\
&(P(x), Q(x))&\mapsto& \prodint{P(x),Q(y)},
\end{array}
\end{align*}
such that for any triple $P(x),Q(x),R(x)\in  \mathbb{C}^{p\times p}[x]$ the following properties are fulfilled
\begin{enumerate}
	\item  $\prodint{AP(x)+BQ(x),R(y)}=A\prodint{P(x),R(y)}+B\prodint{Q(x),R(y)}$, $\forall A,B\in\mathbb{C}^{p\times p}$,
	\item $\prodint{P(x),AQ(y)+BR(y)}=\prodint{P(x),Q(y)}A^\top+\prodint{P(x),R(y)}B^\top$, $\forall A,B\in\mathbb{C}^{p\times p}$.
\end{enumerate}
\end{defi}
The reader probably noticed that, despite we deal with complex polynomials in a real variable, we have follow \cite{Gaut2} and chosen the transpose instead of the Hermitian conjugated. 
For any couple of matrix polynomials $P(x)=\sum\limits_{k=0}^{\deg P}p_kx^k$ and $Q(x)=\sum\limits_{l=0}^{\deg Q} q_lx^l$ the sesquilinear form is defined by
\begin{align*}
\prodint{P(x),Q(y)}=\sum_{\substack{k=1,\dots,\deg P\\
		l=1,\dots,\deg Q}}p_k G_{k,l}(q_l)^\top,
\end{align*}
where the coefficients are the values of the sesquilinear form on the basis of the module
\begin{align*}
G_{k,l}=\prodint{I_px^k ,I_py^l }.
\end{align*}
The corresponding semi-infinite matrix
\begin{align*}
G=\begin{bmatrix}
G_{0,0 } &G_{0,1}& \dots\\
G_{1,0} & G_{1,1} & \dots\\
\vdots & \vdots
\end{bmatrix}
\end{align*}
is  the  named as the Gram matrix of the sesquilinear form.

\subsubsection{Hankel  sesquilinear forms}
Now, we present a family of examples of sesquilinear forms in $\mathbb C^{p\times p}[x]$ that we call Hankel sesquilinear forms.
A first example is given by matrices with complex (or real) Borel measures in $\mathbb R$ as entries
\begin{align*}
\mu=\begin{bmatrix}
\mu_{1,1}&\dots &\mu_{1,p}\\
\vdots & &\vdots\\
\mu_{p,1} &\dots&\mu_{p,p}
\end{bmatrix},
\end{align*}
i.e.,  a  $p\times p$ matrix of Borel measures supported in $\R$.
Given any pair of matrix polynomials $P(x),Q(x)\in\mathbb{C}^{p\times p}[x]$  we introduce the following     sesquilinear form
\begin{align*}
\prodint{P(x),Q(x)}_\mu=\int_\R P(x)\d\mu(x)(Q(x))^{\top}.
\end{align*}

A more general sesquilinear form can be constructed in terms of generalized functions (or continuous linear functionals).  In \cite{Maroni1985espaces,Maroni1988calcul} a linear functional setting for orthogonal polynomials is given. We consider the space of polynomials $\mathbb C[x]$, with an appropriate topology,  as the space of fundamental functions, in the sense of \cite{gelfand-distribu1,gelfand-distribu2}, and take the space of generalized functions as the corresponding continuous linear functionals. It is remarkable that the topological dual space coincides with the algebraic dual space. On the other hand,  this space of   generalized functions  is the space of formal series with complex coefficients
$(\mathbb C[x])'=\mathbb C[\![x]\!]$.

In  this article we use generalized functions with a well defined support and, consequently,  the previously described setting requires of a suitable modification.
Following \cite{Schwartz,gelfand-distribu1,gelfand-distribu2}, let us recall that the space of distributions is a space of  generalized functions when the space of fundamental functions  is constituted  by the complex valued smooth functions of compact support $\mathcal D:=C_0^\infty(\mathbb R)$, the so called space of test functions.
 In this context,  the set of zeros of a  distribution $u\in\mathcal D'$is the region $\Omega\subset \mathbb R$ if for any fundamental function $f(x)$ with support in $\Omega$ we have $\langle u, f\rangle =0$.
  Its complement, a closed set, is what is called  support,  $\operatorname{supp} u$, of the distribution $u$.   Distributions of compact support, $u\in\mathcal E'$, are  generalized functions for which the  space of fundamental functions  is the topological space of complex valued smooth  functions $\mathcal E=C^\infty(\mathbb R)$.
  As $\mathbb C[x]\subsetneq \mathcal E$ we also know that $\mathcal E'\subsetneq (\mathbb C[x])'\cap \mathcal D'$.  The set of  distributions of compact support  is a first example of an appropriate framework for the consideration of polynomials and supports simultaneously. More general settings appear within the space of tempered distributions $\mathcal S'$, $\mathcal S'\subsetneq\mathcal D'$. The space of fundamental functions  is given by the Schwartz space $\mathcal S$ of complex valued fast decreasing functions, see
   \cite{Schwartz,gelfand-distribu1,gelfand-distribu2}.  We  consider the space of fundamental functions constituted by smooth functions of slow growth $\mathcal O_M\subset \mathcal E$,  whose elements are smooth functions with
    derivatives  bounded by polynomials.  As  $\mathbb C [x],\mathcal S\subsetneq \mathcal O_M$,  for the corresponding set of generalized functions  we find that  $\mathcal O_M'\subset (\mathbb C[x])'\cap \mathcal S'$.  Therefore, these distributions give a second appropriate  framework. Finally, for a third suitable framework, including the two previous ones, we need to introduce bounded distributions. Let us consider as  space of fundamental functions, the linear space $\mathcal B$ of bounded smooth functions, i.e.,  with all its derivatives in $L^\infty(\R)$,  being the corresponding space of generalized functions $\mathcal B'$ the bounded distributions. From $\mathcal D\subsetneq \mathcal B$ we conclude that bounded distributions are distributions $\mathcal B'\subsetneq \mathcal D'$. Then, we consider the space of fast decreasing distributions $\mathcal O_c'$ given by those distributions $u\in\mathcal D'$ such that for each positive integer $k$, we have
$\big(\sqrt{1+x^2}\big)^ku\in\mathcal B'$ is a bounded distribution.
Any polynomial $P(x)\in\mathbb C[x]$, with $\deg P=k$, can be written as
\begin{align*}
P(x)&=\Big(\sqrt{1+x^2}\Big)^k F(x), & F(x)&=\frac{P(x)}{\big(\sqrt{1+x^2)}\big)^k}\in\mathcal B.
\end{align*}
Therefore, given a fast decreasing distribution $u\in\mathcal O_c'$ we may consider
\begin{align*}
\langle u,P(x)\rangle =\left\langle\Big(\sqrt{1+x^2}\Big)^ku, F(x)\right\rangle
\end{align*}
which makes sense as $\big(\sqrt{1+x^2}\big)^ku\in\mathcal B', F(x)\in\mathcal B$.  Thus,  $\mathcal O'_c\subset  (\mathbb C[x])'\cap \mathcal D'$.
Moreover  it can be proven that $\mathcal O_M'\subsetneq \mathcal O_c'$, see
\cite{Maroni1985espaces}.
 Summarizing this discussion, we have found three  generalized function spaces suitable for the discussion of polynomials and supports simultaneously:
\begin{align*}
\mathcal E'\subset \mathcal O_M'\subset \mathcal O_c' \subset \big((\mathbb C[x])'\cap \mathcal D'\big).
\end{align*}

The linear functionals could have discrete  and, as the corresponding Gram matrix is required to be quasidefinite,   infinite support. Then, we are faced with discrete orthogonal polynomials, see for example \cite{Nikiforov1991Discrete}. Two classical  examples are those of  Charlier and the Meixner.
For $\mu>0$ we have the Charlier (or Poisson--Charlier) linear functional
\begin{align*}
u=\sum_{k=0}^\infty \frac{\mu^k}{k!} \delta(x-k),
\end{align*}
and $\beta>0$ and $0<c<1$, the Meixner linear functional is 
\begin{align*}
u=\sum_{k=0}^{\infty }\frac{ \beta(\beta+1)\dots(\beta+k-1)}{k!}c^k\delta(x-k).
\end{align*}
See \cite{Al} for matrix extensions of these discrete linear functionals and corresponding matrix orthogonal polynomials.

\begin{defi}[Hankel sesquilinear forms]\label{def:sesquilinear_hankel}
Given a matrix  of generalized functions as entries
\begin{align*}
u=\begin{bmatrix}
u_{1,1}&\dots &u_{1,p}\\
\vdots & &\vdots\\
u_{p,1} &\dots&u_{p,p}
\end{bmatrix},
\end{align*}
i.e., $u_{i,j}\in(\mathbb C[x])'$, then the associated  sesquilinear form  $\prodint{P(x),Q(x)}_u$ is
given by
\begin{align*}
\big(\prodint{P(x),Q(x)}_u\big)_{i,j}:=\sum_{k,l=1}^p \prodint{u_{k,l},P_{i,k}(x) Q_{j,l}(x)}.
\end{align*}
When $u_{k,l}\in\mathcal O_c'$, we write $u\in\big(\mathcal O_c'\big)^{p\times p}$ and say that we have a matrix of fast decreasing distributions. In this case the support is defined as $\operatorname{supp} (u):=\cup_{k,l=1}^N\operatorname{supp}(u_{k,l})$.
\end{defi}
Observe that in this Hankel case, we could also have continuous and discrete orthogonality.
\begin{pro}
In terms of  the moments
\begin{align*}
 m_n:=\begin{bmatrix}
\prodint{u_{1,1},x^n} & \dots &\prodint{u_{1,p},x^n}\\
\vdots & &\vdots\\
\prodint{u_{p,1},x^n} & \dots &\prodint{u_{p,p},x^n}
\end{bmatrix}
\end{align*}
the  Gram matrix of the sesquilinear form given in Definition \ref{def:sesquilinear_hankel} is the following moment matrix
\begin{align*}
G&:=\begin{bmatrix}
m_{0}&m_{1}& m_2&\cdots\\
m_{1}&m_{2}& m_3&\cdots\\
m_{2}&m_{3}& m_4&\cdots\\
\vdots    &\vdots      &\vdots&\\
\end{bmatrix},
\end{align*}
of Hankel type.
\end{pro}

\subsubsection{Matrices of generalized kernels and sesquilinear forms}
The previous examples all have in common the same Hankel block symmetry for the corresponding matrices. However,  there are sesquilinear forms which do not have this particular Hankel type symmetry. Let us stop for a moment at this point, and elaborate on bilinear and sesquilinear forms for polynomials. We first recall some facts regarding the scalar case with $p=1$, and bilinear forms instead of sesquilinear forms. Given   $u_{x,y}\in (\mathbb C[x,y])'=(\mathbb C[x,y])^*\cong\mathbb C[\![x,y]\!]$, we can consider the continuous  bilinear form $B(P(x),Q(y))=\langle u_{x,y}, P(x)\otimes Q(y)\rangle $. This gives a continuous  linear map $\mathcal L_u: \mathbb C[y]\to(\mathbb C [x])'$ such that $B(P(x),Q(y)))=\langle \mathcal L_u(Q(y)), P(x)\rangle$. The Gram matrix of this bilinear form has coefficients  $G_{k,l}=B(x^k,y^l)=\langle u_{x,y}, x^k\otimes y^{l}\rangle =\langle \mathcal L_u(y^l),x^k\rangle$.
Here we follow  Schwartz discussion on kernels and distributions \cite{Schwartz1},  see also \cite{Hormander}. A kernel  $u(x,y)$ is a complex valued locally integrable function,
that defines an integral operator $f(x)\mapsto g(x)=\int u(x,y) f(y)\d y $. Following \cite{Schwartz} we denote $(\mathcal D)_x$ and $(\mathcal D')_x$ the test functions and the corresponding distributions in the variable $x$, and similarly for the variable $y$.   We extend this construction considering   a  bivariate distribution in the variables $x,y$, $u_{x,y}\in (\mathcal D')_{x,y}$, that Schwartz called  \emph{noyau-distribution}, and as we use a wider range of generalized functions we will call generalized kernel. This  $u_{x,y}$ generates a continuous bilinear form
\begin{align*}
B_u\big(\phi(x),\psi(y)\big)
&=\langle u_{x,y}, \phi(x)\otimes \psi(y)\rangle.
\end{align*}
%The integral is just Schwartz notation to indicate functional and the corresponding variable, $x$ or $y$ where it is acting.
It also generates a continuous linear map $\mathcal L_u: (\mathcal D)_y\rightarrow (\mathcal D')_x$ with
\begin{align*}
\langle (\mathcal L_u (\psi(y)))_x,\phi(x)\rangle=\langle u_{x,y}, \phi(x)\otimes \psi(y)\rangle.
\end{align*}
The Schwartz kernel theorem states that every generalized kernel $u_{x,y}$ defines a continuous linear transformation $\mathcal L_u$ from $(\mathcal D)_y$ to $(\mathcal D')_x$, and to each of such continuous linear transformations  we can associate  one and only one  generalized kernel.
According to the prolongation scheme developed in \cite{Schwartz1}, the  generalized kernel  $u_{x,y}$ is such that $\mathcal L_u:(\mathcal E)_y\to(\mathcal E')_x$ if and only if the support of $u_{x,y}$ in $\R^2$ is compact.\footnote{Understood as a prolongation problem, see \S 5 in \cite{Schwartz1},  we have similar results if we require $\mathcal L_u:\mathcal O_M\to\mathcal O_c'$ or
	$\mathcal L_u:\mathcal O_c\to\mathcal O_c'$  or any other possibility that makes sense for polynomials and support. }

We can extended these ideas to the matrix scenario of this paper, where  instead of bilinear forms we have sesquilinear forms.
\begin{defi}
Given   a matrix of   generalized kernels
\begin{align*}
u_{x,y}:=\begin{bmatrix}
(u_{x,y})_{1,1}& \dots &(u_{x,y})_{1,p}\\
\vdots & & \vdots\\
(u_{x,y})_{p,1} & \dots & (u_{x,y})_{p,p}
\end{bmatrix}
\end{align*}
with $(u_{x,y})_{k,l}\in(\mathbb C[x,y])'$  or, if a notion of support is required, $(u_{x,y})_{k,l}\in(\mathcal E')_{x,y},(\mathcal O_M')_{x,y},(\mathcal O_c')_{x,y}$, provides a continuous sesquilinear form with  entries given by
\begin{align*}
\big(\langle P(x),Q(y\rangle_u\big)_{i,j}&=\sum_{k,l=1}^p\big\langle(u_{x,y})_{k,l}, P_{i,k}(x)\otimes Q_{j,l}(y)\big\rangle\\
&=\sum_{k,l=1}^p\big\langle \mathcal L_{ u_{k,l}}(Q_{j,l}(y)),P_{i,k}(x) \big\rangle,
\end{align*}
where $\mathcal L_{u_{k,l}}:\mathbb C[y]\to(\mathbb C[x])'$  --or  depending on the setting $\mathcal L_{u_{k,l}}:(\mathcal  E)_y\to(\mathcal E')_x$, $\mathcal L_{u_{k,l}}:(\mathcal  O_M)_y\to(\mathcal O'_c)_x$, for example--  is a continuous linear operator.  We can condensate it in a matrix form,  for  $u_{x,y}\in (\mathbb C^{p\times p}[x,y])'=(\mathbb C^{p\times p}[x,y])^*\cong \mathbb C^{p \times p}[\![x,y]\!]$,   a sesquilinear form  is given
	\begin{align*}
	\langle P(x), Q(y) \rangle_u&=\langle u_{x,y} , P(x)\otimes Q(y)\rangle\\
	&=\langle\mathcal L_u(Q(y)),P(x)\rangle,
	\end{align*}
	with $\mathcal L_u: \mathbb C^{p\times p}[y]\to(\mathbb C^{p\times p}[x])'$
 a continuous linear map. Or, in other scenarios   $\mathcal L_u:( (\mathcal  E)_y)^{p\times p}\to((\mathcal E')_x)^{p\times p}$ or
$\mathcal L_u: (( \mathcal O_M)_y)^{p\times p}\to((\mathcal O_c')_x)^{p\times p}$.
\end{defi}

\begin{rem}
All continuous sesquilinear forms are of this type. A matrix of generalized kernels $u_{x,y}$ supported by the diagonal of the form $u_{x,y}=u_{x,x}\delta(x-y)$,  gives the Hankel sesquilinear forms given in Definition \ref{def:sesquilinear_hankel}.
\end{rem}
If, instead of a matrix of bivariate distributions, we have a matrix of bivariate measures then we could write for the sesquilinear form
$\langle P(x),Q(y)\rangle=\iint P(x) \d\mu(x,y) (Q(y))^\top$,
where $\mu(x,y)$ is a matrix of bivariate measures. 

For the scalar case $p=1$, Adler and van Moerbeke discussed in \cite{adler-van moerbeke} different possibilities of non-Hankel Gram matrices.
Their Gram matrix has as coefficients
$G_{k,l}=\langle  u_l,x^k\rangle$,
for a infinite sequence of generalized functions $u_l$, that recovers the Hankel scenario for  $u_l=x^lu$. They studied in more detail the following cases
\begin{enumerate}
\item  Banded case: $u_{l+km}=x^{km} u_l$.
\item  Concatenated solitons: $u_l(x)=\delta(x-p_{l+1})-(\lambda_{l+1})^2\delta (x-q_{k+1})$.
\item Nested Calogero--Moser systems: $u_l(x)= \delta'(x-p_{l+1})+\lambda_{l+1}\delta(x-p_{l+1})$.
\item Discrete KdV soliton type: $u_l(x)=(-1)^k \delta^{(l)}(x-p)-\delta^{(l)}(x+p)$.
\end{enumerate}
We see that the three last weights are generalized functions. To compare with the Schwartz's  approach we  observe that
$\langle u_{x,y}, x^k\otimes y^l \rangle =\langle u_l, x^k \rangle$
and, consequently,  we deduce   $u_l=\mathcal L_u(y^l)$ (and for continuous kernels $u_l(x)=\int u(x,y)y^l\d y)$. The first case, has a banded structure and its Gram matrix fulfills
$\Lambda^m G=G(\Lambda^\top)^m$.
 In \cite{manas3} different examples are discussed for the matrix orthogonal polynomials, like bigraded Hankel matrices
$\Lambda^nG=G\big(\Lambda^\top\big)^m$,
where $n,m$ are positive integers,
can be realized as
$G_{k,l}=\langle u_l,  I_px^k\rangle$,
in terms of  matrices of linear functionals $u_l$ which satisfy the following periodicity condition
$u_{l+m}=u_l x^{n}$.
Therefore, given the linear functionals  $u_0,\dots,u_{m-1}$ we can recover all the others.

\subsubsection{Sesquilinear forms supported by the diagonal and Sobolev sesquilinear forms}
First we consider the scalar case
\begin{defi}
	A generalized kernel  $u_{x,y}$ is  supported by the diagonal $y=x$ if
	\begin{align*}
	\prodint{ u_{x,y}, \phi(x,y)}=\sum_{n,m}\prodint{ u^{(n,m)}_x,\frac{\partial^{n+m} \phi(x,y)}{\partial x^n\partial y^m}\Big|_{y=x}}
	\end{align*}
	for a locally finite sum and   generalized functions $u^{(n,m)}_x\in(\mathcal D')_x$.
\end{defi}
\begin{pro}[Sobolev bilinear forms]
	The  bilinear form corresponding to a generalized kernel supported by the diagonal is
\begin{align*}
B(\phi(x),\psi(x))=\sum_{n,m}\prodint{u^{(n,m)}_x,\phi^{(n)}(x)\psi^{(m)}(x)},
\end{align*}
which is of Sobolev type,
\end{pro}
For order zero $u^{(n,m)}_x$ generalized functions, i.e. for a set of Borel measures $\mu^{(n,m)}$, we have
\begin{align*}
B(\phi(x),\psi(x))=\sum_{n,m}\int \phi^{(n)}(x)\psi^{(m)}(x)\d\mu^{(n,m)}(x),
\end{align*}
which is of Sobolev type. Thus, in the scalar case, generalized kernels supported by the diagonal are just Sobolev bilinear forms.
The extension of these ideas to the matrix case is immediate, we only need to require to all generalized kernels to be supported by the diagonal.
\begin{pro}[Sobolev sesquilinear forms]
	A matrix of   generalized kernels supported by the diagonal
provides Sobolev sesquilinear forms
	\begin{align*}
	\big(\langle P(x),Q(x)\rangle_u\big)_{i,j}&=\sum_{k,l=1}^p\sum_{n,m}\prodint{u^{(n,m)}_{k,l}, P^{(n)}_{i,k}(x)Q^{(m)}_{j,l}(x)}.
	\end{align*}
	for a locally finite sum, in the of derivatives order $n,m$,  and  of generalized functions $u^{(n,m)}_x\in(\mathbb C[x] )'$. All Sobolev sesquilinear forms are obtained in this form.
\end{pro}

For a recent review on scalar  Sobolev orthogonal polynomials see \cite{Marcellan2014Sobolev}. Observe that with this general framework  we could consider matrix discrete Sobolev orthogonal polynomials, that will appear whenever the linear functionals $u^{(m,n)}$ have infinite discrete support, as far as $u$ is quasidefinite.

\subsubsection{Biorthogonality, quasidefiniteness and Gauss--Borel factorization}

\begin{defi}[Biorthogonal matrix polynomials]
	Given a  sesquilinear  form $\prodint{\cdot,\cdot}$, two sequences of matrix polynomials  $\big\{P_n^{[1]}(x)\big\}_{n=0}^\infty$ and $\big\{P_n^{[2]}(x)\big\}_{n=0}^\infty$ are said to be biorthogonal with respect to $\prodint{\cdot,\cdot}$ if
	\begin{enumerate}
		\item $\deg(P_n^{[1]}(x))=\deg(P_n^{[2]}(x))=n$ for all $n\in\{0,1,\dots\}$,
		\item $\prodint{P_n^{[1]}(x),P_m^{[2]}(y)}=\delta_{n,m}H_n$ for all $n,m\in\{0,1,\dots\}$,
	\end{enumerate}
	where $H_n$ are nonsingular matrices and $\delta_{n,m}$ is the Kronecker delta.
\end{defi}

\begin{defi}[Quasidefiniteness]
	A Gram matrix of a sesquilinear form $\langle\cdot,\cdot\rangle_u$ is said  to be quasidefinite whenever $\det G_{[k]}\neq 0$, $k\in\{0,1,\dots\}$.
Here $G_{[k]}$ denotes the truncation
	\begin{align*}
	G_{[k]}:=\begin{bmatrix}
	G_{0,0}&\dots & G_{0,k-1}\\
	\vdots & & \vdots\\
	G_{k-1,0} & \dots &G_{k-1,k-1}
		\end{bmatrix}.
	\end{align*}
	We say that  the bivariate generalized function $u_{x,y}$ is quasidefinite and the corresponding sesquilinear form is nondegenerate whenever its Gram matrix is quasidefinite.
	\end{defi}	

\begin{pro}[Gauss--Borel factorization, see \cite{ari}]\label{pro:fac}
	If the Gram matrix of a sesquilinear form  $\langle\cdot,\cdot\rangle_u$ is quasidefinite, then there exists a unique Gauss--Borel factorization given by 	\begin{align}\label{eq:gauss}
	G=(S_1)^{-1} H (S_2)^{-\top},
	\end{align}
	where $S_1,S_2$ are lower unitriangular block matrices and $H$ is a diagonal block matrix
	\begin{align*}
	S_i&=\begin{bmatrix}
	I_p&0_p&0_p&\dots\\
	(S_i)_{1,0}& I_p&0_p&\cdots\\
	(S_i)_{2,0}& (S_i)_{2,1}&I_p&\ddots\\
\vdots	&\vdots&\ddots&\ddots
	\end{bmatrix}, &i&=1,2,&
	H&=\diag(
		H_0,H_1, H_2,\dots),
	\end{align*}
	with $(S_i)_{n,m}$ and $H_n\in\mathbb C^{p\times p}$, $\forall n,m\in\{0,1,\dots\}$.
\end{pro}

\begin{rem}
	Despite that \eqref{eq:gauss} involve the product of three semi-infinite matrices, and therefore of series --\emph{infinite sums}-- could appear, this is not the case, as, given the order of  the factors, lower unitriangular, diagonal and upper unitriangular, all the sums are finite; following \cite{Cantero} we could say that is an \emph{admissible} product.
\end{rem}

\begin{rem}
The product of semi-infinite matrices is not associative in general, see \cite{cooke}.  Hence, special care is required in the   manipulations of the Gauss--Borel factorization problem.
\end{rem}

\begin{rem}[See \cite{cooke}]
	Lower unitriangular  semi-infinite matrices have a unique inverse matrix, which is again lower unitriangular.
	The product  of lower triangular semi-infinite matrices or the product of upper triangular semi-infinite matrices is associative.
\end{rem}

Following \cite{Cantero} we have
\begin{defi}
	Block of lower (upper) Hessenberg type matrices are  semi-infinite matrices with finite number of nonzero block superdiagonals (subdiagonals).
\end{defi}
Also we have the block version of Proposition 2.3  \cite{Cantero}
\begin{pro}[See \cite{Cantero}]\label{pro:associativity}
	The associative property $(AB)C=A(BC)$ of a product of three semi-infinite block matrices $A,B$ and $C$ hold whenever
	\begin{enumerate}
\item Both $A$ and $B$ are of block lower Hessenberg type.
\item Both $B$ and $C$ are block upper Hessenberg type.
\item  If $A$ is  block lower Hessenberg type and $C$ is  block upper Hessenberg type.
	\end{enumerate}
	\end{pro}

For $l\geq k$ we will also use the following bordered truncated Gram matrix
\begin{align*}
G_{[k,l]}^{[1]}&:=
\left[\begin{array}{ccc}
G_{0,0}&  \cdots & G_{0,k-1} \\
\vdots                        &   & \vdots \\
G_{k-2,0}  &  \cdots & G_{ k-2,k-1}\\[1pt]
\hline
G_{l,0}& \dots & G_{l,k-1}
\end{array}\right],
\end{align*}
where we have replaced the last row of blocks of the truncated Gram matrix $G_{[k]}$ by the row of blocks $\begin{bsmallmatrix}
G_{l,0}, \dots, G_{l, k-1} \end{bsmallmatrix}$. We also need a similar matrix but replacing the last block column of $G_{[k]}$ by a column of blocks as indicated
\begin{align*}
G_{[k,l]}^{[2]}&:=\left[
\begin{array}{ccc|c}
G_{0,0} &  \cdots & G_{0,k-2}&G_{0,l} \\
\vdots  &                      &  \vdots  & \vdots \\
G_{k-1,0}  &  \cdots & G_{k-1,k-2}&G_{k-1,l}
\end{array}\right].
\end{align*}

Using  last quasideterminants, see \cite{gelfand,olver}, we find
\begin{pro}\label{qd1}
	If the last quasideterminants of the truncated moment matrices are nonsingular, i.e.,
	\begin{align*}
	\det  \Theta_*(G_{[k]})\neq& 0, & k=1,2,\dots,
	\end{align*}
	then, the Gauss--Borel factorization  can be performed and the following expressions are fulfilled
	\begin{align*}
		H_{k}&=\Theta_*\begin{bmatrix}
		G_{0,0} & G_{0,1} &\dots &G_{0,k-1}\\
		G_{1,0} & G_{1,1} &\dots &G_{1,k-1}\\
		\vdots &  \vdots &  &\vdots\\
		G_{k-1,0} & G_{k-1,1} &\dots &G_{k-1,k-1}\\
		\end{bmatrix},\end{align*}
		\begin{align*}
		(S_1)_{k,l}&=\Theta_*\begin{bmatrix}
		G_{0,0} & G_{0,1} &\dots &G_{0,k-1}& 0_p\\
		G_{1,0} & G_{1,1} &\dots &G_{1,k-1}&0_p\\
		\vdots & \vdots& & \vdots  &\vdots\\
		G_{l-1,0} & G_{l,1} &\dots &G_{l-1,k-1}& 0_p\\
		G_{l,0} & G_{l,1} &\dots &G_{l,k-1} &I_p\\
		G_{l+1,0} & G_{l+1,1} &\dots &G_{l+1,k-1} &0_p\\
		\vdots &\vdots & &\vdots  &\vdots\\
		G_{k,0} & G_{k,1} &\dots &G_{k,k-1} &0_p\\	
		\end{bmatrix},\\	\big((S_2)^\top\big)_{k,l}&=\Theta_*\begin{bmatrix}
		G_{0,0} &G_{0,1}&\dots  &G_{0,l-1} & G_{0,l} & G_{0,l+1} &\dots & G_{0,k}\\
		G_{1,0} &G_{1,1}&\dots  &G_{1,l-1} & G_{1,l} & G_{1,l+1} &\dots & G_{1,k}\\
		\vdots & \vdots&  &\vdots&\vdots&\vdots&&\vdots\\
		G_{k-1,0} &G_{k-1,1}&\dots  &G_{k-1,l-1} & G_{k-1,l} & G_{k-1,l+1} &\dots & G_{k-1,k}\\
		0_p &0_p&\dots  &0_p &  I_p&0_p & \dots& 0_p
		\end{bmatrix},
	\end{align*}
	and for the inverse elements \cite{olver} the formulas
	\begin{align*}
	(S_1^{-1})_{k,l}&=\Theta_*(G^{[1]}_{[k,l+1]})\Theta_*(G_{ [l+1]})^{-1},\\
	(S_2^{-1})_{k, l}&=\big(\Theta_*(G_{[l+1]})^{-1}\Theta_*(G^{[2]}_{[k,l+1]})\big)^\top,
	\end{align*}
	hold true.
\end{pro}
We see that the matrices $H_k$ are quasideterminants, and following \cite{MVOPR,ari} we refer to them as quasitau matrices.

\subsubsection{Biorthogonal polynomials, second kind functions and Christoffel--Darboux kernels}
\begin{defi}
	We define  $\chi(x):=[I_p,I_p x,I_px^2,\dots]^\top$, and for $x\neq 0$,  $\chi^*(x):=[I_px^{-1},I_p x^{-2},I_px^{-3},\dots]^\top$.
\end{defi}
\begin{rem}
	Observe that the Gram matrix can be expressed as
	\begin{align}\label{eq:M_chi}
	G&=\prodint{\chi(x), \chi(y)}_u\\
	\notag &=\langle u_{x,y}, \chi(x)\otimes \chi(y)\rangle
	\end{align}
	and its block  entries  are
	\begin{align*}
	G_{k,l}=\prodint{ I_px^k, I_py^l}_{u}.
		\end{align*}
	If the sesquilinear form derives from a matrix of bivariate measures $\mu(x,y)=[\mu_{i.j}(x,y)]$ we have for the Gram matrix blocks
	\begin{align*}
G_{k,l}=\iint x^k\d\mu(x,y )y^l.
	\end{align*}
	which reduces for absolutely continuous measures with respect the Lebesgue measure $\d x\d y$ to  a matrix of weights $w(x,y)=[w_{i,j}(x,y)]$, and
	When the matrix of generalized kernels is  Hankel we recover the classical Hankel structure, and the Gram matrix is
	a moment matrix. For example, for a matrix of measures we will have $G_{k,l}=\int x^{k+l}\d\mu(x )$.
	
\end{rem}
	\begin{defi}\label{defi:bio2kind}
	Given a quasidefinite matrix of generalized kernels $u_{x,y}$ and the Gauss--Borel factorization \eqref{pro:fac} of  its Gram matrix,  the corresponding first and second families of matrix  polynomials  are
	\begin{align}\label{eq:bior}
	P^{[1]}(x)=\begin{bmatrix}
P^{[1]}_0(x)\\P^{[1]}_1(x)\\\vdots
	\end{bmatrix}&:=S_1\chi(x), &  P^{[2]}(y)=\begin{bmatrix}
	P^{[2]}_0(y)\\P^{[2]}_1(y)\\\vdots
	\end{bmatrix}&:=S_2\chi(y),
	\end{align}
respectively. The corresponding  first and second families of second kind functions  á la Gram are
		\begin{align}\label{eq:second_kind}
	C^{[1]}(x)=\begin{bmatrix}
	C^{[1]}_0(x)\\C^{[1]}_1(x)\\\vdots
	\end{bmatrix}&:=H\big(S_2\big)^{-\top}\chi^{*}(x), &  	C^{[2]}(y)=\begin{bmatrix}
	C^{[2]}_0(y)\\C^{[2]}_1(y)\\\vdots
	\end{bmatrix}&:=H^\top\big(S_1\big)^{-\top}\chi^{*}(y),
	\end{align}
	respectively, whenever the series involved converge.
\end{defi}
\begin{rem}
	The matrix polynomials $P_n^{[i]}(x)$ are monic and  $\deg P_n^{[i]}(x)=n$, $i=1,2$.
\end{rem}
\begin{rem}
	We see that the second kind functions  á la Gram implies non admissible products, in the sense of  \cite{Cantero}, however we will see that the series involved do converge, in some general situations, to certain  Cauchy transformations of the biorthogonal families of matrix polynomials.
\end{rem}

\begin{pro}[Biorthogonality]
	Given a quasidefinite matrix of generalized kernels $u_{x,y}$, the first and second  families of monic matrix polynomials $\big\{P_n^{[1]}(x)\big\}_{n=0}^\infty$ and $\big\{P_n^{[2]}(x)\big\}_{n=0}^\infty$
	are biorthogonal
	\begin{align}\label{eq:biorthogonal}
	\prodint{P^{[1]}_n(x),P^{[2]}_m(y)}_u&=\delta_{n,m}H_n,& n,m&\in\{0,1,\dots\}.
	\end{align}
\end{pro}
\begin{rem}
		The biorthogonal relations yield the orthogonality relations
		\begin{align}
		\label{eq:orthogonality1}\prodint{P^{[1]}_n(x),y^mI_p}_u&=0_p,&\prodint{x^mI_p, P^{[2]}_n(y)}_u&= 0_p,  &m&\in\{1,\dots n-1\},\\
		\label{eq:orthogonality2}\prodint{P^{[1]}_n(x),y^nI_p}_u&= H_n, & \prodint{x^nI_p, P^{[2]}_n(y)}_u&= H_n.
		\end{align}
\end{rem}

\begin{rem}[Symmetric generalized kernels]\label{rem:symmetric}
		If $u_{x,y}=(u_{y,x})^\top$, the Gram matrix is symmetric $G=G^\top$ and we are dealing with a  Cholesky block factorization with  $S_1=S_2$ and $H=H^\top$.
		Now $P^{[1]}_n(x)=P^{[2]}_n(x)=:P_n(x)$, and $\{P_n(x)\}_{n=0}^\infty$ is a set of monic orthogonal matrix polynomials.
		In this case $C_n^{[1]}(x)=C_n^{[2]}(x)=:C_n(x)$.
\end{rem}

The shift matrix is the following semi-infinite block matrix
\begin{align*}
\Lambda:=\begin{bmatrix}
0_p&I_p&0_p&0_p&\dots\\
0_p&0_p&I_p&0_p&\ddots\\
0_p&0_p&0_p&I_p&\ddots\\
0_p&0_p&0_p&0_p&\ddots\\
\vdots&\ddots&\ddots&\ddots&\ddots
\end{bmatrix}
\end{align*}
which satisfies the  spectral property
\begin{align*}%\label{eq:chi_eigen}
\Lambda\chi(x)=x\chi(x).
\end{align*}

Next results are useful to ensure that the manipulations related with the Gauss--Borel factorization and semi-infinite matrices do respect certain
associativity properties.
\begin{pro}\label{pro:Gauss--Borel factorization and associativity}
For the next first  three properties we assume a quasidefinite Gram matrix $G$ with Gauss--Borel factorization $G=(S_1)^{-1}H(S_2)^{-\top}$
\begin{enumerate}
	\item For any matrix polynomial $W(x)\in\mathbb C^{p\times p}[x]$
\begin{align*}
G W(\Lambda^\top)&=(S_1)^{-1}H\Big((S_2)^{-\top}W\big(\Lambda^\top\big)\Big),\\
W(\Lambda)G&=\Big(W(\Lambda)(S_1)^{-1}\Big)H(S_2)^{-\top}.
\end{align*}
\item Given a semi-infinite vector $X=\begin{bmatrix}
X_0,X_1,\dots
\end{bmatrix}^\top$, with $X_n\in\mathbb C^{p\times p}$ and, whenever the series involved in the products $(S_2)^{-\top}X$ and $X^\top (S_1)^{-1}$ converge, we have
\begin{align*}
G X&=(S_1)^{-1}H\big((S_2)^{-\top}X\big),& X^\top G &=\big(X^\top (S_1)^{-1}\big)H(S_2)^{-\top}.
\end{align*}
Consequently,
\begin{align*}
S_1(G X)&=H\big((S_2)^{-\top}X\big),& (X^\top G )(S_2)^{\top}&=\big(X^\top (S_1)^{-1}\big)H.
\end{align*}
	\item Let us now relax the quasidefinite condition and only assume that  $(GW(\Lambda^\top))X$ or $G(W(\Lambda^\top)X)$  involves only convergent series, then $(GW(\Lambda^\top))X=G(W(\Lambda^\top)X)$.  Similarly, if $X^\top(W(\Lambda)G)$ or $(X^\top W(\Lambda))G$ involves only convergent series, then $X^\top (W(\Lambda)G)=(X^\top W(\Lambda))G$.
	\item As above, we relax the quasidefinite condition and only assume that  $GX$, respectively  $X^\top G$, involves solely convergent series, then $W(\Lambda) (GX)=(W(\Lambda )G)X$, respectively  $(X^\top G)W(\Lambda^\top) =X^\top(G	W(\Lambda^\top ))$.
\end{enumerate}
\end{pro}
\begin{proof}
	We will use  the notation $L:=(S_1)^{-1}=(l_{i,j})$ and $U:=H(S_2)^{-\top}=(u_{i,j})$. Then, we write
	\begin{align*}
	G&=\begin{bmatrix}
	I_p & 0_p &0_p&\dots\\
	l_{1,0} & I_p & 0_p&\dots\\
	l_{2,0} & l_{2,1} & I_p &\\
	\vdots &\vdots &&\ddots&
	\end{bmatrix}\begin{bmatrix}
	u_{0,0} & u_{0,1} & u_{0,2}&\dots\\
	0_p &u_{1,1} &u_{1,2} &\dots\\
	0_p & 0_p&u_{2,2}&\dots\\
	\vdots & \vdots &&\ddots
	\end{bmatrix}
	\\
	&=\begin{bmatrix}
	u_{0,0} & u_{0,1} & u_{0,2}&\dots\\
	l_{1,0}	u_{0,0} & l_{1,0}u_{0,1}+u_{1,1} & l_{1,0}u_{0,2}+u_{1,2}&\dots\\
	l_{2,0}	u_{0,0} & l_{2,0}u_{0,1}+l_{2,1}u_{1,1} & l_{2,,0}u_{0,2}+l_{2,1}u_{1,2}+u_{2,2}&\dots\\
	\vdots & \vdots & \vdots&
	\end{bmatrix}.
	\end{align*}
	\begin{enumerate}
	\item It is enough to check it for $W(\Lambda ^\top)=\Lambda^\top$.
On the one hand
	\begin{align*}
	G\Lambda^\top&=\begin{bmatrix}
	u_{0,0} & u_{0,1} & u_{0,2}&\dots\\
	l_{1,0}	u_{0,0} & l_{1,0}u_{0,1}+u_{1,1} & l_{1,0}u_{0,2}+u_{1,2}&\dots\\
	l_{2,0}	u_{0,0} & l_{2,0}u_{0,1}+l_{2,1}u_{1,1} & l_{2,,0}u_{0,2}+l_{2,1}u_{1,2}+u_{2,2}&\dots\\
	\vdots & \vdots & \vdots&
	\end{bmatrix}\begin{bmatrix}
	0_p & 0_p & 0_p&\dots\\
	I_p & 0_p & 0_p &\dots\\
	0_p & I_p  &0_p&\dots\\
	\vdots &\ddots & \ddots&\ddots&
	\end{bmatrix}\\
&=	\begin{bmatrix}
 u_{0,1} & u_{0,2}&\dots\\
l_{1,0}u_{0,1}+u_{1,1} & l_{1,0}u_{0,2}+u_{1,2}&\dots\\
 l_{2,0}u_{0,1}+l_{2,1}u_{1,1} & l_{2,,0}u_{0,2}+l_{2,1}u_{1,2}+u_{2,2}&\dots\\
\vdots & \vdots & &
\end{bmatrix};
	\end{align*}
	i.e., we shift all columns one position to the left.
	On the other hand, we compute
	\begin{align*}
	L (U\Lambda ^\top)&=\begin{bmatrix}
	I_p & 0_p &0_p&\dots\\
	l_{1,0} & I_p & 0_p&\dots\\
	l_{2,0} & l_{2,1} & I_p &\\
	\vdots &\vdots &&\ddots&
	\end{bmatrix}\begin{bmatrix}
 u_{0,1} & u_{0,2}&\dots\\
u_{1,1} &u_{1,2} &\dots\\
 0_p&u_{3,3}&\dots\\
	\vdots & \ddots &\ddots
	\end{bmatrix}\\
	&=	\begin{bmatrix}
	u_{0,1} & u_{0,2}&\dots\\
	l_{1,0}u_{0,1}+u_{1,1} & l_{1,0}u_{0,2}+u_{1,2}&\dots\\
	l_{2,0}u_{0,1}+l_{2,1}u_{1,1} & l_{2,,0}u_{0,2}+l_{2,1}u_{1,2}+u_{2,2}&\dots\\
	\vdots & \vdots & &
	\end{bmatrix},
	\end{align*}
	and we obtain the desired result.
\item We first compute
$
UX=\begin{bsmallmatrix}
u_{0,0}X_0+ u_{0,1} X_1 +\cdots\\
u_{1,1} X_1 +u_{1,2}X_2+\cdots\\
u_{2,2}X_2+u_{2,3}X_3+\cdots\\
\vdots
\end{bsmallmatrix}$,
where we assumed that all the series involved,  $U_n:=\sum\limits_{m=0}^\infty u_{n,m}X_m$, do converge. Then, we have
\begin{align*}
L(UX)&=\begin{bmatrix}
I_p & 0_p &0_p&\dots\\
l_{1,0} & I_p & 0_p&\dots\\
l_{2,0} & l_{2,1} & I_p &\\
\vdots &\vdots &&\ddots&
\end{bmatrix}\begin{bmatrix}
U_0\\
U_1\\
U_2\\
\vdots
\end{bmatrix}=\begin{bmatrix}
U_0\\
l_{1,0}U_0+U_1\\
l_{2,0}U_0+l_{2,1}U_1+U_2\\
\dots
\end{bmatrix}
\end{align*}
We compute
\begin{align*}
GX&=\begin{bmatrix}
u_{0,0} & u_{0,1} & u_{0,2}&\dots\\
l_{1,0}	u_{0,0} & l_{1,0}u_{0,1}+u_{1,1} & l_{1,0}u_{0,2}+u_{1,2}&\dots\\
l_{2,0}	u_{0,0} & l_{2,0}u_{0,1}+l_{2,1}u_{1,1} & l_{2,,0}u_{0,2}+l_{2,1}u_{1,2}+u_{3,3}&\dots\\
\vdots & \vdots & \vdots&
\end{bmatrix}
\begin{bmatrix}
X_0\\X_1\\X_2\\
\vdots
\end{bmatrix}\\
&=\begin{bmatrix}
u_{0,0} X_0+ u_{0,1} X_1 +u_{0,2}X_2+\cdots\\
l_{1,0}u_{0,0} X_0+( l_{1,0}u_{0,1}+u_{1,1} )X_1 +(l_{1,0}u_{0,2}+u_{1,2})X_2+\cdots\\
l_{2,0}	u_{0,0} X_0+( l_{2,0}u_{0,1}+l_{2,1}u_{1,1}) X_1+ (l_{2,,0}u_{0,2}+l_{2,1}u_{1,2}+u_{2,2})X_3+\cdots\\
\vdots
\end{bmatrix}
\\
&=\begin{bmatrix}
U_0\\
l_{1,0}U_0+U_1\\
l_{2,0}U_0+l_{2,1}U_1+U_2\\
\dots
\end{bmatrix}
\end{align*}
where in the last identity we have used the fact that the series involved do converge and, therefore, a finite linear combination of convergent series gives the convergent series of the corresponding finite linear combination of its coefficients.
\item To illustrate the idea of the proof we just check the case $W(x)=I_px$.  Then, as we realized before, the matrix $G\Lambda^\top$ is obtained shifting left all columns in $G$
\begin{align*}
(G\Lambda^\top)X=\begin{bmatrix}
G_{0,1}  &G_{0,2} &\dots\\
G_{1,1} & G_{1,2} &\dots\\
\vdots &\vdots
\end{bmatrix}\begin{bmatrix}
X_0\\X_1\\\vdots
\end{bmatrix}=\begin{bmatrix}
G_{0,1} X_0+G_{0,2}X_1+\cdots\\
G_{1,1} X_0+G_{1,2}X_1+\cdots\\
\vdots
\end{bmatrix},
\end{align*}
and we suppose that all series $G_{j,1} X_0+G_{j,2}X_1+\cdots$ converge for $j\in\{0,1,\dots\}$.
But,
\begin{align*}
G(\Lambda^\top X)=\begin{bmatrix}
	G_{0,0}  &G_{0,1} &\dots\\
	G_{1,0} & G_{1,1} &\dots\\
	\vdots &\vdots
\end{bmatrix}\begin{bmatrix}
0_p\\X_0\\X_1\\\vdots
\end{bmatrix}=\begin{bmatrix}
G_{0,1} X_0+G_{0,2}X_1+\cdots\\
G_{1,1} X_0+G_{1,2}X_1+\cdots\\
\vdots
\end{bmatrix},
\end{align*}
and the statement follows.
\item Again we just check one case, $\Lambda (GX)=(\Lambda G)X$:
\begin{align*}
GX&=\begin{bmatrix}
G_{0,0}  &G_{0,1} &\dots\\
G_{0,1} & G_{1,1} &\dots\\
\vdots &\vdots
\end{bmatrix}\begin{bmatrix}
X_0\\X_1\\\vdots
\end{bmatrix}
=\begin{bmatrix}
G_{0,0}X_0+G_{0,1}X_1+\cdots\\
G_{1,0}X_0+G_{1,1}X_1+\cdots\\
\vdots
\end{bmatrix},
\end{align*}
so that
$\Lambda (GX)= \begin{bsmallmatrix}
G_{1,0}X_0+G_{1,1}X_1+\cdots\\
G_{2,0}X_0+G_{2,1}X_1+\cdots\\
\vdots
\end{bsmallmatrix}$. But,
\begin{align*}
(\Lambda G)X&=\begin{bmatrix}
G_{1,0}  &G_{1,1} &\dots\\
G_{2,1} & G_{2,1} &\dots\\
\vdots &\vdots
\end{bmatrix}\begin{bmatrix}
X_0\\X_1\\\vdots
\end{bmatrix}=\begin{bmatrix}
G_{1,0}X_0+G_{1,1}X_1+\cdots\\
G_{2,0}X_0+G_{2,1}X_1+\cdots\\
\vdots
\end{bmatrix},
\end{align*}
	as claimed.
\end{enumerate}
\end{proof}

\begin{pro}\label{pro:hankel}
	The symmetry of the  block Hankel  moment matrix reads
$\Lambda G=G\Lambda^\top$.
\end{pro}
Notice that this symmetry completely characterizes  Hankel block matrices.
\begin{defi}\label{def:jacobi}
	The  matrices
$	J_1:=S_1 \Lambda (S_1)^{-1}$ and $J_2:=S_2 \Lambda (S_2)^{-1}$
are the Jacobi matrices associated with the Gram matrix $G$.
\end{defi}
The reader must notice the abuse in the notation. But for the sake of simplicity we have used the same letter for Jacobi and Jordan matrices. The type of matrix will be clear from the context.
\begin{rem}
	 Let us observe that there is no problem with the definition related with  the associativity problems of semi-infinite matrices, as one can readily check that for any couple of  lower unitriangular matrices  $L_1,L_2$ we have  $(L_1\Lambda) L_2=L_1(\Lambda L_2)$. Moreover, if $L_3$ is a third lower unitriangular matrix, we have
	 $(L_1\Lambda L_2)L_3=(L_1\Lambda)(L_2L_3)$.
\end{rem}

\begin{pro}
	The biorthogonal polynomials are eigenvectors of the Jacobi matrices
	\begin{align*}
	J_1P^{[1]}(x)&=x P^{[1]}(x), & 	J_2P^{[2]}(x)&=x P^{[2]}(x).
	\end{align*}
	and the second kind functions á la Gram satisfy
	\begin{align*}
	\big(H (J_2)^\top H^{-1}\big)C^{[1]}(x)&=xC^{[1]}(x)-H_0 \begin{bmatrix}
	I_p\\0_p\\ \vdots
	\end{bmatrix}, &
	\big(  H^\top (J_1)^\top H^{-\top}\big)C^{[2]}(x)&=xC^{[2]}(x)-H_0^{\top} \begin{bmatrix}
	I_p\\0_p\\ \vdots
	\end{bmatrix}.
	\end{align*}
\end{pro}
\begin{proof}
	Let us check the first relation
	\begin{align*}
	J_1P^{[1]}(x)&=(S_1\Lambda (S_1)^{-1})(S_1 \chi(x))= \big((S_1\Lambda) (S_1)^{-1}\big)(S_1 \chi(x))\\
	&=	(S_1\Lambda) ((S_1)^{-1}S_1 \chi(x))=(S_1\Lambda)\chi(x).
	\end{align*}
	For the second kind functions we give a very detailed chain of relations, in where in every step we have carefully checked associative aspects of the operations performed (otherwise, it will take a couple of lines) a make heavy use of Propositions \ref{pro:associativity} and \ref{pro:Gauss--Borel factorization and associativity}
	\begin{align*}
	\big( H (J_2)^\top H^{-1}\big)C^{[1]}(x)&=\Big(H\big((S_2)^{-\top} \Lambda^\top (S_2)^\top \big) H^{-1}\Big)\big(	H\big(S_2\big)^{-\top}\chi^{*}(x)\big)=\big(\big(HS_2^{-\top} \Lambda^\top) \big((S_2)^{\top}H^{-1}\big)\big(	H\big(S_2\big)^{-\top}\chi^{*}(x)\big)\\
	&=(HS_2^{-\top} \Lambda^\top)  \Big(\big((S_2)^{\top}H^{-1}\big)\big(	H\big(S_2\big)^{-\top}\chi^{*}(x)\Big)=(HS_2^{-\top} \Lambda^\top)\chi^{*}(x)\\
	&=\big(HS_2^{-\top}\big)  (\Lambda^\top \chi^{*}(x)  )\\&=
	x\big(HS_2^{-\top}\big)\chi^*(x)- x\big(HS_2^{-\top}\big) \begin{bmatrix}
	I_p\\0_p\\ \vdots
	\end{bmatrix}\\&=
	xC^{[1]}(x)- x \begin{bmatrix}
	H_0\\0_p\\ \vdots
	\end{bmatrix}.   	\end{align*}
\end{proof}
\begin{rem}
	For the Hankel case we have
	\begin{align*}
	J_1C^{[1]}(x)&=xC^{[1]}(x)-H_0 \begin{bmatrix}
	I_p\\0_p\\ \vdots
	\end{bmatrix}, &
	J_2C^{[2]}(x)&=xC^{[2]}(x)-H_0^{\top} \begin{bmatrix}
	I_p\\0_p\\ \vdots
	\end{bmatrix}.
	\end{align*}
\end{rem}

\begin{pro}For Hankel type Gram matrices (i.e., associated with a matrix  of univariate generalized functionals) the two Jacobi  matrices are related by
	\begin{align*}
	H^{-1}J_1=J_2^{\top} H^{-1},
	\end{align*}
	being,  therefore,  a tridiagonal  matrix. This yields  the three term relation for biorthogonal polynomials
	and second kind functions, respectively.
\end{pro}
\begin{proof}
	The relation between the two Jacobi matrices follows from the Gauss--Borel factorization in the symmetry $\Lambda G = G \Lambda ^{\top}$. A consequence of
	this relation is the three-block-diagonal shape of these matrices. The
	three term  relations follow from the definitions of the Jacobi matrices in terms of the factorization matrices. There are not associative issues in this case and Propositions \ref{pro:associativity} and \ref{pro:Gauss--Borel factorization and associativity} ensure all the mentioned arguments.
\end{proof}

\begin{pro}
	We have the following last quasideterminantal expressions
	\begin{align*}
	P^{[1]}_n(x)&=\Theta_*
	\begin{bmatrix}
	G_{0,0}&G_{0,1}&\cdots&G_{0,n-1}&I_p\\
	G_{1,0}& G_{1,1}&\cdots&G_{1,n}-1&I_px\\
	\vdots&\vdots&&\vdots&\vdots\\
	G_{n-1,0}&G_{n-1,1}&\cdots&G_{n-1,n-1}&I_px^{n-1}\\
	G_{n,0}&G_{n,1}&\cdots&G_{n,n-1}&I_px^{n}
	\end{bmatrix}, \\
	(P^{[2]}_n(y))^\top&=\Theta_*
	\begin{bmatrix}
	G_{0,0}&G_{0,1}&\cdots&G_{0,n-1}&	G_{0,n}\\
	G_{1,0}& G_{1,1} &\cdots&G_{1,n-1}&G_{1,n}\\
	\vdots&\vdots&&\vdots&\vdots\\
	G_{n-1,0}&G_{n-1,1}&\cdots&G_{n-1,n-1}&G_{n-1,n}\\
	I_p&I_py&\cdots&I_py^{n-1}&I_py^{n}
	\end{bmatrix}.
	\end{align*}
\end{pro}

\begin{defi}[Christoffel--Darboux kernel,\cite{simon-cd,DAS}]
	Given two sequences of matrix  biorthogonal polynomials 
	\begin{align*}
	\text{ $\big\{P_k^{[1]}(x)\big\}_{k=0}^\infty$ and $\big\{P_k^{[2]}(y)\big\}_{k=0}^\infty$, }
	\end{align*}with respect to the sesquilinear form $\prodint{\cdot,\cdot}_u$, we define the $n$-th Christoffel--Darboux kernel matrix polynomial
\begin{align}\label{eq:CD kernel}
K_{n}(x,y):=\sum_{k=0}^{n}(P_k^{[2]}(y))^\top( H_k)^{-1}P^{[1]}_k(x),
\end{align}
and the  mixed Christoffel--Darboux kernel
\begin{align*}
	K^{(pc)}_n(x,y)&:=\sum_{k=0}^n\big(P_k^{[2]}(y)\big)^\top (H_k)^{-1}C_k^{[1]}(x).	
\end{align*}
\end{defi}
\begin{pro}
	\begin{enumerate}
		\item 	For a quasidefinite matrix of generalized kernels $u_{x,y}$, the corresponding  Christoffel--Darboux kernel gives the projection operator
		\begin{align}\label{eq:reproducing}
		\prodint{ K_n(x,z),\sum_{0\leq j\ll\infty} C_j P^{[2]}_j(y)}_u&=
	\Big(	\sum_{j=0}^nC_jP_j^{[2]}(z)\Big)^\top,& 	\prodint{ \sum_{0\leq j\ll\infty}C_jP^{[1]}_j(x),(K_n(z,y))^\top}_u&=
		\sum_{j=0}^nC_jP^{[1]}_j(z).
		\end{align}
		\item
		In particular, we have
		\begin{align}\label{eq:K-u}
		\prodint{ K_n(x,z),I_py^l}_u&=I_pz^l, & l\in&\{0,1,\dots,n\}.
		\end{align}
	\end{enumerate}
\end{pro}
\begin{proof}
	It follows from the biorthogonality \eqref{eq:biorthogonal}.
\end{proof}

\begin{pro}[Christoffel--Darboux formula]\label{pro:CD formula}
	When the  sesquilinear form is Hankel (now $u$ is  a matrix of univariate generalized functions  with its Gram matrix of block Hankel type) the Christoffel--Darboux kernel satisfies
	\begin{align*}
	(x- y)K_n(x,y)&=(P^{[2]}_{n}(y))^\top (H_n)^{-1}P^{[1]}_{n+1}(x)-(P^{[2]}_{n+1}(y))^\top (H_{n})^{-1}P^{[1]}_{n}(x),\\
 	\end{align*}
and the mixed Christoffel-Darboux kernel fulfills
	\begin{align*}
(x-{y})	K^{(pc)}_{n}(x,y)
	&=(P^{[2]}_{n}(y))^{\top}H_n^{-1}C_{n+1}^{[1]}(x)-(P^{[2]}_{n+1}(y))^{\top}H_n^{-1}C_{n}^{[1]}(x)
	+I_p.
	\end{align*}
\end{pro}
\begin{proof}
We only prove the second formula, for the first one proceeds similarly.
It is obviously a consequence of the three term  relation. Firstly, let us  notice that
	\begin{align*}
J_{2}^{\top}H^{-1}C^{[1]}(x)&=xH^{-1}C^{[1]}(x)-\begin{bmatrix}
I_p\\0_p\\\vdots
\end{bmatrix},&
	(P^{[2]}(y) )^{\top}J_{2}^{\top}H^{-1}&={y}(P^{[2]}(y) )^{\top}H^{-1}.
	\end{align*}
	Secondly, we have
	\begin{align*}
J_{2}^{\top}H^{-1}=\left[\begin{array}{c|c}
	\left[J_{2}^{\top}H^{-1}\right]_{[n]} &  \begin{matrix} 0 & 0 &\dots\\
	\vdots&\vdots&\\
	0&0&\dots\\
H_n^{-1}&0&\dots
	\end{matrix}  \\
	\hline
		\begin{matrix}
		0&\dots &0&H_n^{-1}\\
		0 & \dots &0 &0\\
		\vdots & &\vdots& \vdots
		\end{matrix}                         &  *
	\end{array}\right].
	\end{align*}

	Using this, we calculate the
	$	\big(P_{[n]}^{[2]}(y) \big)^{\top}
\left[J_{2}^{\top}H^{-1}\right]_{[n]}
	C^{[1]}_{[n]}(x)$,
	first by computing the action of middle matrix on its left and then on its right to get
	\begin{align*}
	x K^{(pc)}_{n-1}(x,y) - (P^{[2]}_{n-1}(y))^{\top}H_n^{-1}C_n^{[1]}(x)-P_0 =
	{y} K^{(pc)}_{n-1}(x,y) - (P^{[2]}_{n}(y))^{\top}H_n^{-1}C_{n-1}^{[1]}(x),
	\end{align*}
	and since $P_0=I_p$ the Proposition is  proven.
\end{proof}

Next,  we deal with the fact that our definition of second kind functions implies non admissible products and do involve series.
\begin{defi}\label{def:supports}
	For the  support of the matrix of generalized kernels $\operatorname{supp}( u_{x,y})\subset \mathbb C^2$ we consider the action of the component projections $\pi_1,\pi_2:\mathbb C^2\rightarrow \mathbb C$ on its first and second variables,  $(x,y)\overset{\pi_1} {\mapsto}x$, $(x,y)\overset{\pi_2} {\mapsto}y$, respectively, and introduce the projected supports $\operatorname{supp}_x(u):=\pi_1\big(\operatorname{supp} (u_{x,y})\big) $ and $\operatorname{supp}_y(u):=\pi_2\big(\operatorname{supp} (u_{x,y})\big)$, both subsets of  $ \mathbb C$. 
	We will assume that $r_x:=\sup\{|z|: z\in\operatorname{supp}_xu\})<\infty$ and $r_y:=\sup\{|z|: z\in\operatorname{supp}_yu\})<\infty$
		We also consider the disks about infinity, or annulus around the origin,  $D_x:=\{z\in\mathbb C: |z|> r_x\}$ and $D_y:=\{z\in\mathbb C: |z|> r_y\}$.
\end{defi}

\begin{pro} 
\label{pro:Cauchy1}	
Whenever the matrix of generalized kernels is such that $u_{x,y}\in\big((\mathcal O_c')_{x,y}\big)^{p\times p}$,  the second kind functions á la Gram can be expressed as
\begin{align*}
C_{n}^{[1]}(z)&=\left\langle P^{[1]}_n(x),\frac{I_p}{z-y}\right\rangle_u, &|z|>r_y,\\
 \big(C_{n}^{[2]}(z)\big)^\top&=\left\langle \frac{I_p}{z-x},P^{[2]}_n(y)\right\rangle_{u},& |z|>r_x.
\end{align*}
\end{pro}
\begin{proof}
	From \eqref{eq:second_kind}, the Gauss--Borel factorization \eqref{eq:gauss} and Proposition \ref{pro:Gauss--Borel factorization and associativity} we get
		\begin{align*}
		C^{[1]}(z)&=S_1\big(G\chi^{*}(z)\big), &  \big(C^{[2]}(z)\big)^\top&=\big(\big(\chi^{*}(z)\big)^\top G\big)\big(S_2\big)^\top,
		\end{align*}
		which recalling \eqref{eq:M_chi} can be written as
		\begin{align*}
		C^{[1]}(z)&=S_1\langle \chi(x),\chi(y)\rangle_u\chi^{*}(z), &  	 \big(C^{[2]}(z)\big)^\top&=\big(\chi^{*}(z)\big)^\top \langle \chi(x),\chi(y)\rangle_u\big(S_2\big)^\top.
		\end{align*}
	Using the properties of a sesquilinear form and \eqref{eq:bior} we find
		 \begin{align*}
		 C^{[1]}(z)&=\left\langle P^{[1]}(x),\chi(y)\right\rangle_u\chi^{*}(z), &  	 \big(C^{[2]}(z)\big)^\top&=\big(\chi^{*}(z)\big)^\top \langle\chi(x),P^{[2]}(y)\rangle_u.
		 \end{align*}
		 Finally, as  $z$ belongs to a disk about $\infty$ with empty intersection with $\operatorname{supp}_y (u)$, taking into account   the uniform convergence in any compact subset in this disk about infinity  of the geometric series $	\big(\chi^{*}(z)\big)^\top\chi(y)=\dfrac{I_p}{z-y}$,  we deduce	the result.
\end{proof}
When the matrix of generalized kernels is a matrix of bivariate measures, for $z$ outside a disk containing the projected support,  $\operatorname{supp}_y\mu(x,y)$, we get
\begin{align*}
C_{n}^{[1]}(z)&=\int P^{[1]}_n(x)d\mu(x,y)\frac{1}{z-y}, & \big(C_{n}^{[2]}(z)\big)^\top&=\int \frac{1}{z-x}d\mu(x,y)\big(P^{[2]}_n(y)\big)^\top,
\end{align*}
 and when the sesquilinear form is Hankel  it reduces to the following Hankel type second kind functions
\begin{align*}
C_{n}^{[1]}(z)&=\int P^{[1]}_n(x)d\mu(x)\frac{1}{z-x}, & \big(C_{n}^{[2]}(z)\big)^\top&=\int \frac{1}{z-x}d\mu(x)\big(P^{[2]}_n(x)\big)^\top.
\end{align*}

\begin{defi}[Second kind functions á la Cauchy]
	For  a  generalized kernels is such that $u_{x,y}\in\big((\mathcal O_c')_{x,y}\big)^{p\times p}$ we define two families of second kind functions á la Cauchy given by
	\begin{align*}
	C_{n}^{[1]}(z)&=\left\langle P^{[1]}_n(x),\frac{I_p}{z-y}\right\rangle_u, &z\not\in\operatorname{supp}_y(u),\\
	\big(C_{n}^{[2]}(z)\big)^\top&=\left\langle \frac{I_p}{z-x},P^{[2]}_n(y)\right\rangle_{u}, &z\not\in\operatorname{supp}_x(u).
	\end{align*}
\end{defi}
\begin{rem}
	The second kind functions á la Cauchy can be considered as an extension of their Gram version. Hereon we will consider the second kind functions in its Cauchy version.
\end{rem}

\section{Matrix Geronimus transformations}
Geronimus transformations for scalar orthogonal polynomials were first discussed in \cite{Geronimus}, where some determinantal formulas were found, see \cite{Zhe,Maro}. Geronimus perturbations of degree two of scalar bilinear forms have been very recently treated in \cite{Derevyagin} and in the general case in \cite{DereM}.
Here we discuss its matrix extension for general sesquilinear forms.
\begin{defi}
                  Given a matrix of generalized kernels $u_{x,y}=((u_{x,y})_{i,j})\in\big((\mathcal O_c')_{x,y}\big)^{p\times p}$ with a given support $\operatorname{supp} u_{x,y}$, and a  matrix polynomial $W(y)\in\mathbb C^{p\times p}[y]$ of degree $N$,  such that $ \sigma(W(y))\cap \operatorname{supp}_y(u)=\varnothing$,	a matrix of bivariate  generalized functions $\check u_{x,y}$ is said to be a matrix Geronimus transformation of the matrix of generalized kernels $u_{x,y}$ if
                  	\begin{align}\label{eq:geronimus}
                  	\check u_{x,y}W(y)=u_{x,y}.
                  	\end{align}
                  \end{defi}

\begin{pro}\label{pro:string}
	In terms of  sesquilinear forms  a Geronimus transformation  fulfills
	\begin{align*}
	\prodint{P(x), Q(y)(W(y))^\top}_{\check u}=\prodint{P(x), Q(y)}_{u},
	\end{align*}
	while, in terms of the corresponding Gram matrices, satisfies
	\begin{align*}
	\check G W(\Lambda^\top)=G.
	\end{align*}
\end{pro}
\begin{rem}
Given our non-Abelian scenario is, in principle, reasonable to propose the  alternative following Geronimus perturbation
	$\prodint{P(x),(W(y))^\top Q(y)}_{\check u}=\prodint{P(x), Q(y)}_{u}$,
	but we immediately realize that this perturbation will not respect the sesquilinearity and, consequently, has nonsense.
\end{rem}

We will  assume that the  perturbed moment matrix
has a Gauss--Borel factorization
$\check G=\check S_1^{-1} \check H (\check S_2)^{-\top}$,
	where $\check S_1,\check S_2$ are lower unitriangular block matrices and $\check H$ is a diagonal block matrix
	\begin{align*}
		\check S_i&=\begin{bmatrix}
	I_p&0_p&0_p&\dots\\
	(\check S_i)_{1,0}& I_p&0_p&\cdots\\
	(\check S_i)_{2,0}& (\check S_i)_{2,1}&I_p&\ddots\\
	&&\ddots&\ddots
	\end{bmatrix}, & i&=1,2,&
	\check H&=\diag
	(\check H_0, \check H_1, \check H_2,\dots).
	\end{align*}
Hence, the Geronimus transformation provides  the family of matrix biorthogonal polynomials
\begin{align*}
	\check P^{[1]}(x) & =\check S_1\chi(x), & \check P^{[2]}(y) & =\check S_2\chi(y),
\end{align*}
with respect to the perturbed sesquilinear form $\prodint{\cdot,\cdot}_{\check u}$.

Observe that the matrix generalized kernels $v_{x,y}$ such that $v_{x,y}W(y)=0_p$, can be added to a Geronimus transformed matrix of generalized kernels $\check u_{x,y}\mapsto \check u_{x,y}+v_{x,y}$, to  get a new Geronimus transformed matrix of generalized kernels. We call masses these type of terms.

\subsection{The resolvent and connection formulas}
\begin{defi}
	The resolvent matrix is
	\begin{align}\label{eq:def_Omega}
	\omega:=\check S_1 (S_1)^{-1}.
	\end{align}
\end{defi}

The key role of this resolvent matrix is determined by the following properties
\begin{pro}\label{conexw}
\begin{enumerate}
\item	The resolvent matrix can be also expressed as
\begin{align}\label{eq:resolvent_alternative}
\omega =	\check H \big(\check S_2\big)^{-\top} W(\Lambda^\top)\big( S_2\big)^{\top}	H^{-1},
\end{align}
where the products in the RHS  are associative.
\item
The resolvent matrix is a lower unitriangular block banded matrix ---with only the first $N$ block subdiagonals possibly not zero, i.e.,
\begin{align*}
\omega=\begin{bmatrix}
I_p &    0_p  &\dots        &   0_p   & 0_p&  \dots                       \\
\omega_{1,0} &I_p &\ddots&     0_p &    0_p           &    \ddots        \\
\vdots       &     \ddots         &\ddots&      \ddots         &      \ddots       \\
\omega_{N,0} &\omega_{N,1}  &   \dots   & I_p  & 0_p         & \ddots \\
0_p&\omega_{N+1,1}&\cdots& \omega_{N+1,N}&I_p& \ddots\\
\vdots&        \ddots      &\ddots&               &\ddots  &   \ddots
\end{bmatrix}.
\end{align*}	
\item The following connection formulas are satisfied
\begin{align}
	\check P^{[1]}(x) &=\omega P^{[1]}(x),\label{conex2}\\
\big(\check H^{-1}\omega H\big)^\top \check P^{[2]} (y) &=P^{[2]}(y)W^\top(y).\label{conex3}
\end{align}
\item For the last subdiagonal of the resolvent we have
\begin{align}\label{eq:omegaA}
\omega_{N+k,k}=\check H_{N+k}A_N(H_k)^{-1}.
\end{align}
\end{enumerate}
\end{pro}

\begin{proof}
\begin{enumerate}
	\item 	From  Proposition \ref{pro:string} and the Gauss--Borel factorization of $G$ and $\check G$ we get
	\begin{align*}
	\big( S_1\big)^{-1}  H \big( S_2\big)^{-\top}=\Big(\big(\check S_1\big)^{-1} \check H \big(\check S_2\big)^{-\top}\Big) W(\Lambda^\top),
	\end{align*}
 from Proposition \ref{pro:Gauss--Borel factorization and associativity} we deduce
	\begin{align*}
	\big( S_1\big)^{-1}  H \big( S_2\big)^{-\top}=\big(\check S_1\big)^{-1} \check H \Big(\big(\check S_2\big)^{-\top}W(\Lambda^\top)\Big),
	\end{align*}
and Proposition \ref{pro:associativity}, as well as the associativity and invertibility properties of lower and upper unitriangular matrices, we get
		\begin{align*}
\check S_1\big( S_1\big)^{-1}  H =	\check H \big(\check S_2\big)^{-\top} W(\Lambda^\top)\big( S_2\big)^{\top}.	
		\end{align*}
It is easily check that the RHS products are associative.
\item The resolvent matrix, being a product of lower unitriangular matrices,   is a lower unitriangular matrix. However, from \eqref{eq:resolvent_alternative} we deduce that is a matrix with all its subdiagonals with zero coefficients  but for the first $N$. Thus, it must have the described band structure.
\item From the definition we have \eqref{conex2}.  Let us notice  that \eqref{eq:resolvent_alternative} can be written as
\begin{align*}
\omega ^\top \check H ^{-\top}=	H^{-\top}S_2W^\top(\Lambda)\big(\check S_2\big)^{-1},
\end{align*}
so that
\begin{align*}
\omega ^\top \check H ^{-\top}\check P^{[2]}(y)=	H^{-\top}S_2W^\top(\Lambda)\chi(y),
\end{align*}	
and \eqref{conex3} follows.
\item It is a consequence of  \eqref{eq:resolvent_alternative}.
\end{enumerate}
\end{proof}

The connection formulas \eqref{conex2} and \eqref{conex3} can be written as
\begin{align}\label{conex1}
\check P^{[1]}_{n}(x)&=P^{[1]}_n(x)+\sum_{k=n-N}^{n-1}\omega_{n,k}P_k^{[1]}(x),\\\label{conex3'}
W(y)\big(P_n^{[2]}(y)\big)^\top(H_n)^{-1}&=\big(\check  P^{[2]}_n(y)\big)^\top (\check H_n)^{-1}+\sum_{k=n+1}^{n+N}\big(\check  P^{[2]}_k(y)\big)^\top (\check H_k)^{-1}\omega_{k,n}.
\end{align}

\begin{ma}
	We have that
		\begin{align}\label{eq:WLchi}
		W(\Lambda^\top)\chi^*(x)=\chi^*(x)W(x)-\begin{bmatrix}
		\mathcal B 	(\chi(x))_{[N]}
		\\
		0_{p}\\
		\vdots
		\end{bmatrix},
		\end{align}
		with $\mathcal B$ given in Definition \ref{def:B}.
	\end{ma}
	\begin{proof}It is a direct consequence	of
	$\big(\Lambda^\top\big)^n\chi(x)=x^n \chi^*(x)-\begin{bsmallmatrix}
	x^n I_p\\\vdots\\I_p\\0_p\\ \vdots
		\end{bsmallmatrix}$.
	\end{proof}
	
\begin{pro}
	The Geronimus transformation of the second kind functions satisfies
	\begin{align}
	\label{eq:conexionC1}
	\check C^{[1]}(x)W(x)-\begin{bmatrix}
\big(\check H\big(\check S_2\big)^{-\top}\big)_{[N]}	\mathcal B 	(\chi(x))_{[N]}
	\\
	0_{p}\\
	\vdots
	\end{bmatrix}&=	\omega C^{[1]}(x),\\
	\label{eq:conexionC2}
		\big(\check C^{[2]}(x)\big)^\top\check H^{-1}\omega  &= \big(C^{[2]}(x)\big)^\top H^{-1}.
	\end{align}
\end{pro}
\begin{proof}
	\begin{itemize}
		\item \textbf{á la Cauchy}
		To get \eqref{eq:conexionC1} we argue as follows
		\begin{align*}
		\check C^{[1]}(z)W(z)-\omega C^{[1]}(z)&=\prodint{\check P_1(x),\frac{I_p}{z-y}}_{\check u}W(z)-
		\prodint{\check P_1(x),\frac{I_p}{z-y}}_{ \check u W} & \text{use \eqref{conex2} and \eqref{eq:geronimus}}\\
		&=\prodint{\check P_1(x),\frac{W(z)-W(y)}{z-y}}_{\check u}.
		\end{align*}
		But, we have 
		\begin{align*}
		\frac{W(z)-W(y)}{z-y}&=I_p\frac{z^N-y^N}{z-y}+A_{N-1}\frac{z^{N-1}-y^{N-1}}{z-y}+\dots+A_1\\&=
		I_ph_{N-1}(z,y)+A_{N-1}h_{N-2}(z,y)+\dots+A_1\\&=
	(	\chi(y))^\top\begin{bmatrix}
\mathcal B (\chi(z))_N\\0
	\end{bmatrix}
		\end{align*}
		so that
			\begin{align*}
			\check C^{[1]}(z)W(z)-\omega C^{[1]}(z)
			&=\check S_1\prodint{\chi(x),\chi(y)}_{\check u}\begin{bmatrix}
			\mathcal B (\chi(z))_N\\0
			\end{bmatrix}\\
			&=\check S_1\check G\begin{bmatrix}
			\mathcal B (\chi(x))_N\\0
			\end{bmatrix}.
			\end{align*}
			and using the Gauss--Borel factorization the result follows. For  \eqref{eq:conexionC2} we have
	\begin{align*}
	\big(\check C^{[2]}(x)\big)^\top\check H^{-1}\omega  - \big(C^{[2]}(x)\big)^\top H^{-1}&=
	\prodint{\frac{I_p}{z-x},\check P^{[2]}(y)}_{\check u}\check H^{-1}\omega -	\prodint{\frac{I_p}{z-x}, P^{[2]}(y)}_{ u} H^{-1}\\
	&=	\prodint{\frac{I_p}{z-x},\big(\check H^{-1}\omega \big)^\top\check P^{[2]}(y)}_{\check u}-	\prodint{\frac{I_p}{z-x},H^{-\top} P^{[2]}(y)}_{ u} \\
		&=	\prodint{\frac{I_p}{z-x}, H^{-\top}P^{[2]}(y)(W(y))^\top}_{\check u}-	\prodint{\frac{I_p}{z-x},H^{-\top} P^{[2]}(y)}_{ u}\\&=0.
	\end{align*}
		\item \textbf{á la Gram} This proof is included to show how  the connection formulas can be derived directly from the Gauss--Borel factorization.
		From 	\eqref{eq:second_kind},  and taking care of the potential associative  problems with semi-infinite matrices, we find the chain of relations
	(recall that we assume that $(S_2)^{-\top}\chi_2^*$ involve only convergent series)
\begin{align*}
\omega C^{[1]}(x)&=\Big(\check S_1 (S_1)^{-1}\Big)\Big(H\big(S_2\big)^{-\top}\chi^*(x)\Big)=\check S_1 \Big((S_1)^{-1}\big(H\big(S_2\big)^{-\top}\chi^*(x)\big)\Big)
\\&=\St_1 \big(G\chi^*(x)\big)=\St_1\big(\big(\check G W(\Lambda^\top)\big)\chi^{*}(x)\big).
\end{align*}
But,  taking into account \eqref{eq:WLchi} and the assumption that $\check G\chi^*(x)$ do involve only convergent series, we conclude that
$\check G\big(W(\Lambda^\top)\chi^*(x)\big)$ involve only convergent series and therefore
$\big(\check G W(\Lambda^\top)\big)\chi^*(x)=\check G\big(W(\Lambda^\top)\chi^*(x)\big)$. Consequently, we derive
\begin{align*}
\omega C^{[1]}(x)&=
\check S_1\Big(\check G\big(W(\Lambda^\top)\chi^*(x)\big)\Big)
\\&=\check S_1 \left(\check G\left(\chi^*(x)W(x)
-\begin{bmatrix}
\mathcal B 	(\chi(x))_{[N]}
\\
0_{p}\\
\vdots
\end{bmatrix}\right)\right)
\\&=\check S_1 \left(\big((\check S_1)^{-1}\check H(\check S_2)^{-\top}\big)\left(\chi^*(x)W(x)
-\begin{bmatrix}
\mathcal B 	(\chi(x))_{[N]}
\\
0_{p}\\
\vdots
\end{bmatrix}\right)\right)
\\&=\check H\big(\check S_2\big)^{-\top}\chi^*(x)W(x)
-\check H\big(\check S_2\big)^{-\top}\begin{bmatrix}
\mathcal B 	(\chi(x))_{[N]}
\\
0_{p}\\
\vdots
\end{bmatrix}
\\&
=\check C^{[1]}(x)W(x)-\check H\big(\check S_2\big)^{-\top}\begin{bmatrix}
\mathcal B 	(\chi(x))_{[N]}
\\
0_{p}\\
\vdots
\end{bmatrix},
\end{align*}

and \eqref{eq:conexionC1} follows.
Finally, \eqref{eq:conexionC2} is a consequence of
\begin{align*}
\big(	\check C^{[2]}(y)\big)^\top\check H^{-1}\omega H&=\big(	\check C^{[2]}(y)\big)^\top\check H^{-1}\check S_1 \big( S_1\big)^{-1} H=\big(\chi^*(y)\big)^\top \big( S_1\big)^{-1} H\\&=\big(C^{[2]}(y)\big)^\top.
\end{align*}
	\end{itemize}
\end{proof}

Observe that the corresponding entries are
\begin{align}
\big(C_n^{[2]}(y)\big)^\top(H_k)^{-1}=\big(\Ct^{[2]}_n(y)\big)^\top(\check H_n)^{-1}+\sum_{k=n+1}^{n+N}\big(\Ct^{[2]}_k(y)\big)^\top(\check H_k)^{-1}\omega_{n,k}.  \label{Cauchy2}
\end{align}

\subsection{Geronimus transformations and  Christoffel--Darboux kernels}
\begin{defi}\label{def:omeganN}
The  resolvent wing is the matrix
	\begin{align*}
	\Omega{[n]}= \begin{cases}
	\begin{bmatrix}
	\omega_{n,n-N}&\dots&\dots&\omega_{n,n-1}\\
	0_p	&\ddots &&\vdots\\
	\vdots	&\ddots &\ddots&\vdots\\
	0_p	&\dots&0_p&\omega_{n+N-1,n-1}
	\end{bmatrix}\in \mathbb{C}^{Np\times Np}, &  n\geq N, \\
	\begin{bmatrix}
	\omega_{n,0}&\dots &\dots&\omega_{n,n-1}\\
	\vdots      &       &&\vdots\\
	\omega_{N,0}&&&\omega_{N,n-1}\\
	0_p&\ddots\\
	\vdots	&\ddots&\ddots&\vdots\\
	0_p&   \dots   &0_p&\omega_{n+N-1,n-1}
	\end{bmatrix}\in \mathbb{C}^{Np\times np},& n<N.
	\end{cases}
	\end{align*}
\end{defi}

\begin{teo}\label{teoconex}
For $m=\min(n,N)$, the perturbed and original Christoffel--Darboux kernels  are related by the following connection formula   	
		\begin{align}\label{K}
		\check{K}_{n-1}(x,y)&=
		W(y)K_{n-1}(x,y)-\begin{bmatrix}
		\big(	\check P_{n}^{[2]}(y)\big)^\top\Ht^{-1}_n,\dots,\big(\check P_{n+N-1}^{[2]}(y)\big)^\top\Ht^{-1}_{n+N-1}
		\end{bmatrix}
		\Omega{[n]}
		\begin{bmatrix}
		P_{n-m}^{[1]}(x)\\
		\vdots\\
		P_{n-1}^{[1]}(x)
		\end{bmatrix}.
		\end{align}
For $n\geq N$, the connection formula for the mixed Christoffel--Darboux kernels is
\begin{align}
		\label{N}
		\check K_{n-1}^{(pc)}(x,y)W(x)&=W(y)K^{(pc)}_{n-1}(x,y)-\begin{bmatrix}
		\big(	\check P_{n}^{[2]}(y)\big)^\top\Ht^{-1}_n,\dots,\big(\check P_{n+N-1}^{[2]}(y)\big)^\top\Ht^{-1}_{n+N-1}
		\end{bmatrix}
		\Omega{[n]}
		\begin{bmatrix}
		C_{n-N}^{[1]}(x)\\
		\vdots\\
		C_{n-1}^{[1]}(x)
		\end{bmatrix}+\mathcal V(x,y),
		\end{align}
		where $\mathcal V(x,y)$ was introduced in Definition \ref{defi:V}.%, see also Proposition \ref{pro:symdefV}.
\end{teo}

\begin{proof}
		For the first connection formula \eqref{K} we consider the pairing
		\begin{align*}
		\mathcal K_{n-1}(x,y):=
		\begin{bmatrix}
		\big(\check P_{0}^{[2]}(y)\big)^\top(\Ht_0)^{-1},\cdots,\big(\check P_{n-1}^{[2]}(y)\big)^\top(\Ht_{n-1})^{-1}
		\end{bmatrix}
		\omega_{[n]}
		\begin{bmatrix}
		P_{0}^{[1]}(x)\\
		\vdots\\
		P_{n-1}^{[1]}(x)
		\end{bmatrix},
		\end{align*}
and compute it in two different ways. From \eqref{conex1} we get 
\begin{align*}
	\omega_{[n]}
	\begin{bmatrix}
	P_{0}^{[1]}(x)\\
	\vdots\\
	P_{n-1}^{[1]}(x)
	\end{bmatrix}=	\begin{bmatrix}
	\check P_{0}^{[1]}(x)\\
	\vdots\\
\check 	P_{n-1}^{[1]}(x)
	\end{bmatrix},
\end{align*}
and, therefore,
$\mathcal K_{n-1}(x,y)=\check K_{n-1}(x,y)$.
Relation   \eqref{conex3'} leads to
	\begin{align*}
	\mathcal K_{n-1}(x,y)= W(y)	K_{n-1}(x,y)-\begin{bmatrix}
	\big(	\check P_{n}^{[2]}(y)\big)^\top(\Ht_n)^{-1},\dots,\big(\check P_{n+N-1}^{[2]}(y)\big)^\top(\Ht_{n+N-1})^{-1}
	\end{bmatrix}
	\Omega{[n]}
	\begin{bmatrix}
	P_{n-m}^{[1]}(x)\\
	\vdots\\
	P_{n-1}^{[1]}(x)
	\end{bmatrix},
	\end{align*}
and \eqref{K} is proven.

To derive \eqref{N} we  consider the  pairing
	\begin{equation*}
	\mathcal K^{(pc)}_{n-1}(x,y):=\begin{bmatrix}
\big(	\check P_{0}^{[2]}(y)\big)^\top(\Ht_0)^{-1},\dots,\big(\check P_{n-1}^{[2]}(y)\big)^\top(\Ht_{n-1})^{-1}
	\end{bmatrix}
	\omega_{[n]}
	\begin{bmatrix}
	C_{0}^{[1]}(x)\\
	\vdots\\
	C_{n-1}^{[1]}(x)
	\end{bmatrix},
	\end{equation*}	
which, as before,	can be computed in two different forms. On the one hand, using  \eqref{eq:conexionC1} we get
	\begin{align*}
	\mathcal K^{(pc)}_{n-1}(x,y)=	&\begin{bmatrix}
	\big(	\check P_{0}^{[2]}(y)\big)^\top(\Ht_0)^{-1},\dots,\big(\check P_{n-1}^{[2]}(y)\big)^\top(\Ht_{n-1})^{-1}
	\end{bmatrix}\left(
	\begin{bmatrix}
	\Ct_{0}^{[1]}(x)W(x)\\
	\vdots\\
	\Ct_{n-1}^{[1]}(x)W(x)
	\end{bmatrix}
	-\big(\check H\big(\check S_2\big)^{-\top}\big)_{[n,N]}\mathcal B 	(\chi(x))_{[N]}
	\right)\\
	&=\check K^{(pc)}_{n-1}(x,y)W(x)-
\big((\chi(y))_{[n]}\big)^\top\
\big(\big(\check S_2\big)^{\top}\check H^{-1}\big)_{[n]}\big(\check H\big(\check S_2\big)^{-\top}\big)_{[n,N]}\mathcal B 	(\chi(x))_{[N]},
	\end{align*}
	where $\big(\check H\big(\check S_2\big)^{-\top}\big)_{[n,N]}	$ is the truncation to the $n$ first block rows and first $N$ block  columns of
	$\check H\big(\check S_2\big)^{-\top}$. This simplifies  for $n\geq N$  to
		\begin{align*}
		\mathcal K^{(pc)}_{n-1}(x,y)
		&=\check K^{(pc)}_{n-1}(x,y)W(x)-
		\big((\chi(y))_{[N]}\big)^\top
		\mathcal B 	(\chi(x))_{[N]}.
		\end{align*}
	On the other hand,  from  \eqref{conex3'} we conclude
	\begin{align*}
	\mathcal K^{(pc)}_{n-1}(x,y)= W(y)	K^{(pc)}_{n-1}(x,y)-\begin{bmatrix}
	\big(	\check P_{n}^{[2]}(y)\big)^\top(\Ht_n)^{-1},\dots,\big(\check P_{n+N-1}^{[2]}(y)\big)^\top(\Ht_{n+N-1})^{-1}
	\end{bmatrix}
	\Omega{[n]}
	\begin{bmatrix}
	C_{n-N}^{[1]}(x)\\
	\vdots\\
	C_{n-1}^{[1]}(x)
	\end{bmatrix},
	\end{align*}
	and, consequently, we obtain
	\begin{multline*}
	\check K^{(pc)}_{n-1}(x,y)W(x)
	=	W(y)K^{(pc)}_{n-1}(x,y)-\begin{bmatrix}
	\big(	\check P_{n}^{[2]}(y)\big)^\top(\Ht_n)^{-1},\dots,\big(\check P_{n+N-1}^{[2]}(y)\big)^\top(\Ht_{n+N-1})^{-1}
	\end{bmatrix}
	\Omega{[n]}
	\begin{bmatrix}
	C_{n-N}^{[1]}(x)\\
	\vdots\\
	C_{n-1}^{[1]}(x)
	\end{bmatrix}\\+
	\big((\chi(y))_{[N]}\big)^\top\mathcal B 	(\chi(x))_{[N]}.
	\end{multline*}
\end{proof}

\subsection{Spectral jets and relations for the perturbed polynomials and its second kind functions }
For the time being we will assume that the perturbing polynomial is monic, $W(x)=I_p x^N+\sum\limits_{k=0}^{N-1}A_{k}x^k\in\mathbb C^{p\times p}[x]$.

\begin{defi}
	Given a perturbing monic matrix polynomial $W(y)$ the most general mass term will have the form
	\begin{align}\label{eq:v_general}
	v_{x,y}&:=\sum_{a=1}^{q}\sum_{j=1}^{s_a}\sum_{m=0}^{\kappa_j^{(a)}-1}\frac{(-1)^{m}}{m!}\big(\xi^{[a]}_{j,m}\big)_x\otimes\delta^{(m)}(y-x_a)l_{j}^{(a)}(y),
	\end{align}
	expressed in terms of derivatives of Dirac linear functionals and adapted left root polynomials $l_{j}^{(a)}(x)$ of $W(x)$, and for vectors of generalized functions $\big(\xi^{[a]}_{j,m}\big)_x\in\big(( \mathbb C[x])'\big)^p$ .
Discrete Hankel  masses appear when these terms are supported by the diagonal with
	\begin{align}\label{eq:v_diagonal}
	v_{x,x}&:=\sum_{a=1}^{q}\sum_{j=1}^{s_a}\sum_{m=0}^{\kappa_j^{(a)}-1}(-1)^{m}\delta^{(m)}(x-x_a)\frac{\xi^{[a]}_{j,m}}{m!}l_{j}^{(a)}(x),
	\end{align}
with $\xi^{[a]}_{j,m}\in\mathbb{C}^p$.
\end{defi}

\begin{rem}
	Observe that  the  Hankel masses \eqref{eq:v_diagonal} are particular cases of  \eqref{eq:v_general} with
	\begin{align*}%\label{eq:v_diagonal}
	v_{x,y}&:=\sum_{a=1}^{q}\sum_{j=1}^{s_a}	\sum_{m=0}^{\kappa_j^{(a)}-1}(-1)^{m}\frac{\xi^{[a]}_{j,m}}{m!}\sum_{k=0}^m \binom{m}{k} \delta^{(m-k)}(x-x_a)\otimes
	\delta^{(k)}(y-x_a)
l_{j}^{(a)}(y),
	\end{align*}
	so that, with the particular choice in \eqref{eq:v_general}
	\begin{align*}
	\big(\xi^{[a]}_{j,k}\big)_x=%\frac{(-1)^k}{k!}
	\sum_{n=0}^{\kappa_j^{(a)}-1-k}(-1)^{n}\frac{\xi^{[a]}_{j,k+n}}{n!}\delta^{(n)}(x-x_a),
	\end{align*}
	we get the diagonal case.
\end{rem}

\begin{rem}
	For the sesquilinear forms we have\begin{align*}
	\prodint{P(x),Q(y)
	}_{\check u}=
		\prodint{P(x), Q(y)(W(y))^{-\top}}_u+\sum_{a=1}^{q}\sum_{j=1}^{s_a}\sum_{m=0}^{\kappa_j^{(a)}-1}\prodint{P(x),\big(\xi^{[a]}_{j,m}\big)_x}\frac{1}{m!} \Big(l_{j}^{(a)}(y)\big(Q(y)\big)^\top\Big)^{(m)}_{x_a}.
	\end{align*}
\end{rem}

Observe that the distribution $v_{x,y}$ is associated with the eigenvalues and  left root vectors of the perturbing polynomial $W(x)$. Needless to say that, when $W(x)$ has a singular leading coefficient, this spectral part could even disappear, for example if $ W(x)$ is unimodular; i.e., with constant determinant, not depending on $x$. Notice that, in general,  we have $Np\geq\sum_{a=1}^q\sum_{i=1}^{s_a}\kappa^{(a)}_j$ and we can  not ensure  the equality, up to  for the nonsingular leading coefficient case.

\begin{defi}
	Given a set of generalized functions  $(\xi^{[a]}_{i,m})_x$, we introduce the matrices
	 \begin{align*}
	 \prodint{
	 	\check P^{[1]}_{n}(x),(\xi^{[a]}_i)_x}&	:=
	 \begin{bmatrix}
	 \prodint{
	 	\check P^{[1]}_{n}(x),\big(\xi^{[a]}_{i,0}\big)_x
	 },\prodint{
	 \check P^{[1]}_{n}(x),\big(\xi^{[a]}_{i,1}\big)_x
	},\dots,\prodint{
	\check P^{[1]}_{n}(x),\big(\xi^{[a]}_{i,\kappa^{(a)}_i-1}\big)_x
}
\end{bmatrix}\in\mathbb C^{ p\times \kappa_i^{(a)}},\\
 \prodint{
 	\check P^{[1]}_{n}(x),(\xi^{[a]})_x}	&:=
 \begin{bmatrix}
 \prodint{
 	\check P^{[1]}_{n}(x),\big(\xi^{[a]}_{1}\big)_x
 },\prodint{
 \check P^{[1]}_{n}(x),\big(\xi^{[a]}_{2}\big)_x
},\dots,\prodint{
\check P^{[1]}_{n}(x),\big(\xi^{[a]}_{s_a}\big)_x
}
\end{bmatrix}\in\mathbb C^{ p\times \alpha_a},
\\
\prodint{
	\check P^{[1]}_{n}(x),(\xi)_x}	&:=
\begin{bmatrix}
\prodint{
	\check P^{[1]}_{n}(x),\big(\xi^{[1]}\big)_x
},\prodint{
\check P^{[1]}_{n}(x),\big(\xi^{[2]}\big)_x
},\dots,\prodint{
\check P^{[1]}_{n}(x),\big(\xi^{[q]}_{s_a}\big)_x
}
\end{bmatrix}\in\mathbb C^{ p\times Np}.
\end{align*}
\end{defi}

\begin{defi}
The exchange matrix is
	\begin{align*}
	\eta_i^{(a)}=\begin{bmatrix}
	0 &0&\dots &0&1\\
		0 &0 &\dots &1&0\\
		\vdots& &\iddots&&\vdots\\
			0 &1 &\dots &0&0\\	
			1 &0 &\dots &0&0
				\end{bmatrix}\in\mathbb C^{\kappa_i^{(a)}\times \kappa_i^{(a)}}.
	\end{align*}
\end{defi}
\begin{defi}
	The left Jordan chain matrix is given by
	\begin{align*}	
\mathcal L_i^{(a)}&:=	\begin{bmatrix}
	l^{(a)}_{i,0}&l^{(a)}_{i,1} &l^{(a)}_{i,2} &\dots&l^{(a)}_{i,\kappa_i^{(a)}-1}\\
0_{1\times p}	&l^{(a)}_{i,0}&l^{(a)}_{i,1} &\dots &l^{(a)}_{i,\kappa_i^{(a)}-2}\\
0_{1\times p}	&    0_{1\times p}  &l^{(a)}_{i,0}             &  \dots&l^{(a)}_{i,\kappa_i^{(a)}-3}\\
\vdots & \ddots& \ddots &\ddots&\vdots \\
0_{1\times p}	&0_{1\times p}&&&l^{(a)}_{i,0}
	\end{bmatrix}\in\mathbb C^{ \kappa_i^{(a)}\times p\kappa_i^{(a)}}.
	\end{align*}
	For $z\neq x_a$, we also introduce the $p\times p$ matrices
	\begin{align}\label{eq:defC}
	\check{\mathcal C}_{n;i}^{(a)}(z):=	\prodint{
		\check P^{[1]}_{n}(x),(\xi^{[a]}_i)_x}	\eta_i^{(a)}
	\mathcal L^{(a)}_i
	\begin{bmatrix}
	\frac{I_p}{(z-x_a)^{\kappa_i^{(a)}}}\\
	\vdots\\
	\frac{I_p}{z-x_a}
	\end{bmatrix},
	\end{align}
		where $i=1,\dots,s_a$.
	\end{defi}
\begin{rem}
Assume that the mass matrix  is as in \eqref{eq:v_diagonal}. Then, in terms of
	\begin{align}\label{eq:chi}
	\mathcal X_{i}^{(a)}&:=
	\begin{bmatrix}
	\xi^{[a]}_{i,\kappa_i^{(a)}-1}&\xi^{[a]}_{i,\kappa_i^{(a)}-2} &\xi^{[a]}_{i,\kappa_i^{(a)}-3} &\dots&\xi^{[a]}_{i,0}\\
	0_{p\times 1}	&\xi^{[a]}_{i,\kappa_i^{(a)}-1}&\xi^{[a]}_{i,\kappa_i^{(a)}-2}&\dots&\xi^{[a]}_{i,1}\\
	0_{p\times 1}	&   0_{p\times 1}      &\xi^{[a]}_{i,\kappa_i^{(a)}-1}          & \dots &\xi^{[a]}_{i,2}\\
	\vdots & \ddots& \ddots &\ddots&\vdots \\
	0_{p\times 1}	&   0_{p\times 1}      &&&\xi^{[a]}_{i,\kappa_i^{(a)}-1}
	\end{bmatrix}\in\mathbb C^{p\kappa_i^{(a)}\times \kappa_i^{(a)}},
	\end{align}
	we can write
\begin{align}\label{eq:PX-JPX}
	\prodint{
		\check P^{[1]}_{n}(x),(\xi^{[a]}_i)_x}	\eta_i^{(a)}=\mathcal J_{\check P^{[1]}_n}^{(i)}(x_a)	\mathcal X^{(a)}_{i}.
\end{align}
Consequently,
	\begin{align*}
		\check{\mathcal C}_{n;i}^{(a)}(z):=
	\mathcal J_{\check P^{[1]}_n}^{(i)}(x_a)
		\mathcal X^{(a)}_{i}\mathcal L^{(a)}_i
		\begin{bmatrix}
		\frac{I_p}{(z-x_a)^{\kappa_i^{(a)}}}\\
		\vdots\\
		\frac{I_p}{z-x_a}
		\end{bmatrix}.
	\end{align*}
	Observe that $\mathcal X^{(a)}_i\mathcal L_i^{(a)}\in\mathbb C^{p\kappa^{(a)}_i\times p\kappa^{(a)}_i}$ is a block upper triangular matrix, with blocks in $\mathbb C^{p\times p}$.
\end{rem}

\begin{pro}\label{pro:checkCCauchy}
Fo $ z\not\in \operatorname{supp}_y(\check u)=\operatorname{supp}_y(u)\cup \sigma(W(y))$, the following expression
\begin{align*}
\check C_n^{[1]}(z)
&=\left\langle \check P^{[1]}_n(x),\frac{I_p}{z-y}\right\rangle_{ uW^{-1}}+\sum_{a=0}^{q}\sum_{i=1}^{s_a}\check{\mathcal C}_{n;i}^{(a)}(z)
\end{align*}	
holds.
\end{pro}
\begin{proof}
From Proposition \ref{pro:Cauchy1}	we have
\begin{align*}
\Ct_n^{[1]}(z)&=\prodint{
	 \check P^{[1]}_n(x),\frac{I_p}{z-y}
	 }_{\check u}\\
&=
\prodint{\check P^{[1]}_n(x),\frac{I_p}{z-y}}_{ uW^{-1}}
+\sum_{a=0}^{q}\sum_{i=1}^{s_a}\sum_{m=0}^{\kappa_i^{(a)}-1}
\prodint{
	\check P^{[1]}_{n}(x),\big(\xi^{[a]}_{j,m}\big)_x
	}
\left(\frac{1}{m!}\frac{l_i^{(a)}(x)}{z-x}\right)^{(m)}_{x_a}.
\end{align*}
Now, taking into account that
\begin{align*}
\left(\frac{1}{m!}\frac{l_i^{(a)}(x)}{z-x}\right)^{(m)}_{x=x_a}=\sum_{k=0}^m\left(
\frac{l_i^{(a)}(x)}{(m-k)!}\right)^{(m-k)}_{x_a}\frac{1}{(z-x_a)^{k+1}},
\end{align*}
we deduce the result.
\end{proof}

\begin{ma}\label{lemma:trabajandoC}
		Let $r^{(a)}_j(x)$ be  right root polynomials  of the monic matrix polynomial $W(x)$ given in \eqref{vecmil}, then
		\begin{align*}%\label{arbitrario}
		\mathcal L^{(a)}_i
		\begin{bmatrix}
		\dfrac{I_p}{(x-x_a)^{\kappa_i^{(a)}}}\\
		\vdots\\
		\dfrac{I_p}{x-x_a}
		\end{bmatrix}W(x)r^{(a)}_j(x)&=
		\begin{bmatrix}
		\dfrac{1}{(x-x_a)^{\kappa_i^{(a)}}}\\
		\vdots\\
		\dfrac{1}{x-x_a}
		\end{bmatrix}l_i^{(a)}(x)W(x)r_j^{(a)}(x)
		+(x-x_a)^{\kappa^{(a)}_j}T(x), & T(x)&\in\mathbb{C}^{\kappa^{(a)}_j}[x].
		\end{align*}		
\end{ma}
	\begin{proof}
	 Notice that we can write
\begin{align*}%\label{arbitrario}
\mathcal L^{(a)}_i
	\begin{bmatrix}
	\dfrac{I_p}{(x-x_a)^{\kappa_i^{(a)}}}\\
	\vdots\\
	\dfrac{I_p}{x-x_a}
	\end{bmatrix}W(x)r^{(a)}_j(x)&=\begin{bmatrix}
l^{(a)}_{i,0}&l^{(a)}_{i,1} &l^{(a)}_{i,2} &\cdots&l^{(a)}_{i,\kappa_i^{(a)}-1}\\
0_{1\times p}	&l^{(a)}_{i,0}&l^{(a)}_{i,1} &\cdots &l^{(a)}_{i,\kappa_i^{(a)}-2}\\
0_{1\times p}	&    0_{1\times p}  &l^{(a)}_{i,0}             &  &l^{(a)}_{i,\kappa_i^{(a)}-3}\\
\vdots &\ddots &  \ddots&\ddots&\vdots \\
0_{1\times p}	&0_{1\times p}&&&l^{(a)}_{i,0}
\end{bmatrix}	\begin{bmatrix}
\dfrac{I_p}{(x-x_a)^{\kappa_i^{(a)}}}\\
\vdots\\
\dfrac{I_p}{x-x_a}
\end{bmatrix}W(x)r^{(a)}_j(x)\\&=	\begin{bmatrix}
\dfrac{l_i^{(a)}(x)}{(x-x_a)^{\kappa_i^{(a)}}}\\
\dfrac{l_i^{(a)}(x)}{(x-x_a)^{\kappa_i^{(a)}-1}}	-l^{(a)}_{i,\kappa^{(a)}_i-1}\\
\vdots\\
\dfrac{l_i^{(a)}(x)}{x-x_a}	-l^{(a)}_{i,1}- \cdots-l^{(a)}_{i,\kappa^{(a)}_i-1}(x-x_a)^{\kappa_i^{(a)}-2}
\end{bmatrix}W(x)r^{(a)}_j(x).
\end{align*}
Now,  \eqref{eq:Wraj} yields the result.
\end{proof}

\begin{ma}
The function $\check{\mathcal C}^{(a)}_{n;i}(x)W(x)r^{(b)}_j(x)\in\mathbb C^p[x]$ satisfies
{\small\begin{align}\label{eq:CWr}
\check{\mathcal C}_{n;i}^{(a)}(x)W(x)r^{(b)}_j(x)=\begin{cases}
	\prodint{
		\check P^{[1]}_{n}(x),(\xi^{[a]}_i)_x}	\eta_i^{(a)}
\begin{bmatrix}
(x-x_a)^{\kappa^{(a)}_{\max(i,j)}-\kappa^{(a)}_i}\\
\vdots\\
(x-x_a)^{\kappa^{(a)}_{\max(i,j)}-1}
\end{bmatrix}w_{i,j}^{(a)}(x)+(x-x_a)^{\kappa^{(a)}_j}T^{(a,a)}(x), & \text{ if }a=b,\\
(x-x_b)^{\kappa_j^{(b)}}T^{(a,b)}(x), &\text{ if } a\neq b,
\end{cases}
\end{align}}
where the $\mathbb C^p$-valued function $T^{(a,b)}(x)$  is analytic at $x=x_b$ and, in particular, $T^{(a,a)}(x) \in\mathbb C^p[x]$ .
\end{ma}
\begin{proof}
	First, for the function $\check{\mathcal C}^{(a)}_{n;i}(x)W(x)r^{(b)}_j(x)\in\mathbb C^p[x]$, with  $a\neq b$, and recalling \eqref{eq:Wraj}, we have
	\begin{align*}
	\check{\mathcal C}^{(a)}_{n;i}(x)W(x)r^{(b)}_j(x)&=
	\prodint{
		\check P^{[1]}_{n}(x),(\xi^{[a]}_i)_x}	\eta_i^{(a)}
	\mathcal L^{(a)}_i
	\begin{bmatrix}
	\dfrac{I_p}{(x-x_a)^{\kappa_i^{(a)}}}\\
	\vdots\\
	\dfrac{I_p}{x-x_a}
	\end{bmatrix}W(x)r_j^{(b)}(x)\\&=(x-x_b)^{\kappa_j^{(b)}}T^{(a,b)}(x),
	\end{align*}
	where the $\mathbb C^p$-valued function $T^{(a,b)}(x)$ is analytic at $x=x_b$.
Secondly, from  \eqref{eq:defC} and Lemma \ref{lemma:trabajandoC} we deduce that
\begin{align*}
\check{\mathcal C}_{n;i}^{(a)}(x)W(x)r^{(a)}_j(x)&=
	\prodint{
		\check P^{[1]}_{n}(x),(\xi^{[a]}_i)_x}	\eta_i^{(a)}
\mathcal L^{(a)}_i
\begin{bmatrix}
\dfrac{I_p}{(x-x_a)^{\kappa_i^{(a)}}}\\
\vdots\\
\dfrac{I_p}{x-x_a}
\end{bmatrix}W(x)r^{(a)}_j(x)
\\
&=
\begin{multlined}[t]
	\prodint{
		\check P^{[1]}_{n}(x),(\xi^{[a]}_i)_x}	\eta_i^{(a)}
\begin{bmatrix}
\dfrac{I_p}{(x-x_a)^{\kappa_i^{(a)}}}\\
\vdots\\
\dfrac{I_p}{x-x_a}
\end{bmatrix}l_i^{(a)}(x)W(x)r_j^{(a)}(x)
\\	+(x-x_a)^{\kappa^{(a)}_j}
	\prodint{
		\check P^{[1]}_{n}(x),(\xi^{[a]}_i)_x}	\eta_i^{(a)}
T^{(a,a)}(x),
\end{multlined}
\end{align*}
for some $T^{(a,a)}(x)\in \mathbb C^p[x]$. Therefore, from  Proposition \ref{pro:lWr} we get
\begin{multline*}
\check{\mathcal C}_{n;i}^{(a)}(x)W(x)r^{(a)}_j(x)=
	\prodint{\check P^{[1]}_{n}(x),(\xi^{[a]}_i)_x}	\eta_i^{(a)}
\begin{bmatrix}
(x-x_a)^{\kappa^{(a)}_{\max(i,j)}-\kappa^{(a)}_i}\\
\vdots\\
(x-x_a)^{\kappa^{(a)}_{\max(i,j)}-1}
\end{bmatrix}\\\times\Big(w^{(a)}_{i,j;0}+w^{(a)}_{i,j;1}(x-x_a)+\cdots+w^{(a)}_{i,j;\kappa^{(a)}_{\min(i,j)}+N-2}
(x-x_a)^{\kappa^{(a)}_{\min(i,j)}+N-2}\Big).
\\	+(x-x_a)^{\kappa^{(a)}_j}	\prodint{
	\check P^{[1]}_{n}(x),(\xi^{[a]}_i)_x}	\eta_i^{(a)}
T^{(a,a)}(x),
\end{multline*}
and the result follows.
\end{proof}

We evaluate now the spectral jets of the second kind functions $\check C^{[1]}(z)$  á la Cauchy, thus we must take limits of derivatives precisely in points of the spectrum of $W(x)$, which do not lay in the region of definition but on the border of it. Notice that these operations are not available for the second kind functions á la Gram.

\begin{ma}
	For $m=0,\dots,\kappa^{(a)}_j-1$, the following relations hold
		\begin{align}\label{eq:C}
		\Big(	\check C_{n}^{[1]}(z)W(z)r^{(a)}_j(z)\Big)^{(m)}_{x_a}
		=\sum_{i=1}^{s_a}\Big(\check{\mathcal C}_{n;i}^{(a)}(z)W(z)r_j^{(a)}(z)\Big)^{(m)}_{x_a}.
		\end{align}
\end{ma}
\begin{proof}
	
	For $z\not\in\operatorname{supp}_y(u)\cup \sigma(W(y))$,	 a consequence of Proposition \ref{pro:checkCCauchy} is that
	\begin{align*}
	\Big(	\check C_{n}^{[1]}(z)W(z)r^{(a)}_j(z)\Big)^{(m)}_{x_a}
	=\bigg(\left\langle \check P^{[1]}_n(x),\frac{I_p}{z-y}\right\rangle_{uW^{-1}} W(z)r_j^{(a)}(z)\bigg)^{(m)}_{x_a}+\sum_{b=0}^{q}\sum_{i=1}^{s_b}\Big(\check{\mathcal C}_{n;i}^{(b)}(z)W(z)r_j^{(a)}(z)\Big)^{(m)}_{x_a}.
	\end{align*}
	But, as $\sigma(W(y))\cap \operatorname{supp}_y(u)=\varnothing$, the derivatives of the Cauchy kernel $1/(z-y)$ are analytic functions at  $z=x_a$. Therefore,
	\begin{align*}
	\bigg(\left\langle \check P^{[1]}_n(x),\frac{I_p}{z-y}\right\rangle_{uW^{-1}} W(z)r_j^{(a)}(z)\bigg)^{(m)}_{x_a}&=\left\langle \check P^{[1]}_n(x), \bigg(\frac{W(z)r_j^{(a)}(z)}{z-y} \bigg)^{(m)}_{x_a}\right\rangle_{uW^{-1}}\\&=\left\langle \check P^{[1]}_n(x), \sum_{k=0}^m\binom{m}{k}\big(W(z)r_j^{(a)}(z) \big)^{(k)}_{x_a}\frac{(-1)^{m-k}(m-k)!}{(x_a-y)^{m-k+1}}\right\rangle_{uW^{-1}}\\
	&=0_{p\times 1},
	\end{align*}
	for $m=1,\dots,\kappa^{(a)}_j-1$, where in the last equation we have used \eqref{eq:Wraj}. Equation \eqref{eq:CWr} shows that $	\check{\mathcal C}_{n;i}^{(b)}(x)W(x)r^{(a)}_j(x)$ for  $b\neq a$ has a zero at $z=x_a$ of order $\kappa^{(a)}_j$ and, consequently,
	\begin{align*}
\big(	\check{\mathcal C}_{n;i}^{(b)}(x)W(x)r^{(a)}_j(x)\big)^{(m)}_{x_a}&=0, &b\neq a,
	\end{align*}
for $m=0,\dots,\kappa^{(a)}_j-1$.
\end{proof}

\begin{defi}\label{def:Wji}
	Given the functions $w^{(a)}_{i,j;k}$ introduced in Proposition \ref{pro:lWr}, let us introduce  the matrix $\mathcal{W}_{j,i}^{(a)}\in\mathbb C^{\kappa^{(a)}_{j}\times \kappa^{(a)}_{i}}$
	\begin{align*}
	\mathcal{W}_{j,i}^{(a)}&:=\begin{cases}
	\eta_j^{(a)}\left[\begin{array}{c|c}
	0_{\kappa^{(a)}_{j}\times (\kappa^{(a)}_{i}-\kappa^{(a)}_{j})} &\begin{matrix}	w^{(a)}_{i,j;0}&w^{(a)}_{i,j;1}&\cdots
	&w^{(a)}_{i,j;\kappa^{(a)}_j-1}\\
	0_{p}&w^{(a)}_{i,j:0} &\cdots &w^{(a)}_{i,j;\kappa^{(a)}_j-2}\\
	\vdots& & \ddots&    \vdots \\
	0_p& 0_p&   & w^{(a)}_{i,j:0} \\
	\end{matrix}
	\end{array}\right], & i\geq j,\\
	\eta_j^{(a)}\begin{bmatrix}
	w^{(a)}_{i,j;\kappa_j^{(a)}-\kappa_i^{(i)}}&w^{(a)}_{i,j;\kappa_j^{(a)}-\kappa_i^{(i)}+1}&\cdots&w^{(a)}_{i,j,\kappa_j^{(a)}-1}\\
	\vdots&&&\vdots\\w^{(a)}_{i,j;0}&w^{(a)}_{i,j;1}&\cdots & w^{(a)}_{i,j;\kappa_i^{(a)}-1}\\
	0_p&w^{(a)}_{i,j;0}&&\vdots\\\vdots&&\ddots&\\
	0_p&0_p&& w^{(a)}_{i,j;0}
	\end{bmatrix}, & i\leq j.
	\end{cases}
	\end{align*}
	and the matrix
	$\mathcal W^{(a)}_{j}\in\mathbb C^{\kappa_j^{(a)}\times \alpha_a}$ given by
	\begin{align*}
	\mathcal W^{(a)}_{j}:=\Big[\mathcal W^{(a)}_{j,1},\dots, \mathcal W^{(a)}_{j,s_a}\Big].
	\end{align*}
	We also consider the matrices  $\mathcal W^{(a)}\in\mathbb C^{\alpha_a\times \alpha_a}$ and $\mathcal W\in\mathbb C^{Np\times Np}$
	\begin{align}\label{eq:T}
	\mathcal W^{(a)}&:=\begin{bmatrix}
	\mathcal W_1^{(a)}\\\vdots\\	\mathcal W_{s_a}^{(a)}
	\end{bmatrix},& \mathcal W:=\diag(\mathcal W^{(1)},\dots,\mathcal W^{(q)}).
	\end{align}
\end{defi}

\begin{pro} The following relations among the spectral jets, introduced in Definition \ref{def:spectral jets}, of the perturbed polynomials and second kind functions
	\begin{align}\label{eq:CscC}
	\boldsymbol{\mathcal J}^{(j)}_{\check{C}^{[1]}_nW}(x_a)&=\sum_{i=1}^{s_a}\boldsymbol{\mathcal J}^{(j)}_{\check {\mathcal C}_{n;i}W}(x_a),&
	\boldsymbol{\mathcal J}_{\check{C}^{[1]}_nW}(x_a)&=\sum_{i=1}^{s_a}\boldsymbol{\mathcal J}_{\check {\mathcal C}_{n;i}W}(x_a),\\
	\label{eq:CPXW}
	\boldsymbol{\mathcal J}^{(j)}_{\check {\mathcal C}_{n;i}W}(x_a)&=	\prodint{\check P^{[1]}_{n}(x),(\xi^{[a]}_i)_x}	\mathcal W_{i,j}^{(a)},&
	\boldsymbol{\mathcal J}_{\check {\mathcal C}_{n;i}W}(x_a)&=	\prodint{\check P^{[1]}_{n}(x),(\xi^{[a]}_i)_x}	\mathcal W_{i}^{(a)},\\
	\label{eq:checkCP}
	\boldsymbol{\mathcal J}_{\check C^{[1]}_n W}(x_a)
	&=\prodint{\check P^{[1]}_{n}(x),(\xi^{[a]})_x}\mathcal W^{(a)}, &
	\boldsymbol{\mathcal J}_{\check C^{[1]}_n W}
	&=\prodint{\check P^{[1]}_{n}(x),(\xi)_x}\mathcal W,
	\end{align}
	are satisfied.
\end{pro}
\begin{proof}
	Equation \eqref{eq:CscC} is a direct consequence of \eqref{eq:C}. According to \eqref{eq:CWr}  for $m=0,\dots, \kappa_j^{(a)}-1$, we have
	\begin{align*}
	\big(\check{\mathcal C}_{n;i}^{(a)}(x)W(x)r^{(a)}_j(x)\big)_{x=x_a}^{(m)}=
	\prodint{\check P^{[1]}_{n}(x),(\xi^{[a]}_i)_x}\eta_i^{(a)}
	\begin{bmatrix}
	\Big((x-x_a)^{\kappa^{(a)}_{\max(i,j)}-\kappa^{(a)}_i}w_{i,j}^{(a)}(x)\Big)^{(m)}_{x_a}\\
	\vdots\\
	\Big((x-x_a)^{\kappa^{(a)}_{\max(i,j)}-1}w_{i,j}^{(a)}(x)\Big)^{(m)}_{x_a}
	\end{bmatrix},
	\end{align*}
	and collecting all these equations in a matrix form we get \eqref{eq:CPXW}. Finally, we notice that from \eqref{eq:CscC}  and \eqref{eq:CPXW}  we deduce
	\begin{align*}
	\boldsymbol{\mathcal J}^{(j)}_{\check C^{[1]}_nW}(x_a)&=\sum_{i=1}^{s_a}	\prodint{\check P^{[1]}_{n}(x),(\xi^{[a]}_i)_x}\mathcal W_{i,j}^{(a)},&
	\boldsymbol{\mathcal J}_{\check C^{[1]}_nW}(x_a)&=\sum_{i=1}^{s_a}	\prodint{\check P^{[1]}_{n}(x),(\xi^{[a]}_i)_x}\mathcal W_{i}^{(a)}.
	\end{align*}
	Now,  using \eqref{eq:T} we can  write the second equation as 
	\begin{align*}
	\boldsymbol{\mathcal J}_{\check C^{[1]}_nW}(x_a)&=\sum_{i=1}^{s_a}	\prodint{\check P^{[1]}_{n}(x),(\xi^{[a]}_i)_x}\mathcal W^{(a)}_i \\
	&=	\prodint{\check P^{[1]}_{n}(x),(\xi^{[a]})_x}\mathcal W^{(a)}.
	\end{align*}
	A similar argument leads to the second relation in \eqref{eq:checkCP}.
\end{proof}

\begin{defi}%\label{def:Tji}
	For the Hankel masses, we also consider the matrices $\mathcal T_i^{(a)}\in\mathbb C^{p\kappa^{(a)}_i\times \alpha_a}$, $\mathcal T^{(a)}\in\mathbb C^{p\alpha_a\times \alpha_a}$ and $\mathcal T\in\mathbb C^{Np^2\times Np}$ given by
			\begin{align*}%\label{eq:T}
			\mathcal T_i^{(a)}&:=\mathcal X_i^{(a)}\eta_i^{(a)}\mathcal W_i^{(a)}, &
			\mathcal T^{(a)}&:=\begin{bmatrix}
				\mathcal T_1^{(a)}\\\vdots\\	\mathcal T_{s_a}^{(a)}
			\end{bmatrix},& \mathcal T:=\diag(\mathcal T^{(1)},\dots,\mathcal T^{(q)}).
			\end{align*}
\end{defi}

\begin{rem} For  masses as in \eqref{eq:v_diagonal},  relations \eqref{eq:CPXW} and \eqref{eq:checkCP} reduce to
	\begin{align*}
\boldsymbol{\mathcal J}^{(j)}_{\check {\mathcal C}_{n;i}W}(x_a)&=\mathcal J^{(i)}_{\check P^{[1]}_n}(x_a)\mathcal X^{(a)}_i\eta_i^{(a)} \mathcal W_{i,j}^{(a)},&
\boldsymbol{\mathcal J}_{\check {\mathcal C}_{n;i}W}(x_a)&=\mathcal J^{(i)}_{\check P^{[1]}_n}(x_a)\mathcal T^{(a)}_i ,\\
\boldsymbol{\mathcal J}_{\check C^{[1]}_n W}(x_a)
&=\mathcal J_{\check P^{[1]}_n}(x_a)\mathcal T^{(a)}, &
\boldsymbol{\mathcal J}_{\check C^{[1]}_n W}
&=\mathcal J_{\check P^{[1]}_n}\mathcal T.
	\end{align*}
\end{rem}

\subsection{Spectral Christoffel--Geronimus formulas}
We assume a mass term as in \eqref{eq:v_general}.
\subsubsection{Discussion for $n\geq N$}
\begin{rem}\label{rem:nonsingular}
	First, we perform a preliminary comment.  Later on, see Corollary \ref{corollary:non_singularity},  with the aid of a nonspectral approach, we will see that
\begin{align*}
\begin{vmatrix}
\boldsymbol{\mathcal J}_{C^{[1]}_{n-N}}-\prodint{ P^{[1]}_{n-N}(x),(\xi)_x}\mathcal W\\ 	\vdots\\ \boldsymbol{\mathcal J}_{C^{[1]}_{n-1}}-\prodint{ P^{[1]}_{n-1}(x),(\xi)_x}\mathcal W
\end{vmatrix}&\neq 0, & n&\geq N.
\end{align*}
\end{rem}

\begin{pro}\label{pro:resolventCP}
If    $n\geq N$,
  the matrix coefficients of  the  connection matrix satisfy
\begin{align*}
\big[\omega_{n,n-N},\dots,\omega_{n,n-1}\big]	
&=
-\big(\boldsymbol{\mathcal J}_{C^{[1]}_{n}}-\prodint{ P^{[1]}_{n}(x),(\xi)_x}\mathcal W\big)
\begin{bmatrix}
\boldsymbol{\mathcal J}_{C^{[1]}_{n-N}}-\prodint{ P^{[1]}_{n-N}(x),(\xi)_x}\mathcal W\\ 	\vdots\\ \boldsymbol{\mathcal J}_{C^{[1]}_{n-1}}-\prodint{ P^{[1]}_{n-1}(x),(\xi)_x}\mathcal W
\end{bmatrix}^{-1}.
\end{align*}
\end{pro}
\begin{proof}
	From the connection formula  \eqref{eq:conexionC1}, for  $n\geq N$
\begin{align*}
\check C_n^{[1]}(x)W(x)=\sum_{k=n-N}^{n-1}\omega_{n,k}C_k^{[1]}(x)+C_n^{[1]}(x),
\end{align*}
and we conclude that
\begin{align*}
\boldsymbol{\mathcal J}_{ \check C^{[1]}_nW}=\big[\omega_{n,n-N},\dots,\omega_{n,n-1}\big]	
\begin{bmatrix}
\boldsymbol{\mathcal J}_{C^{[1]}_{n-N}}\\ 	\vdots\\ \boldsymbol{\mathcal J}_{C^{[1]}_{n-1}}
	\end{bmatrix}
	+\boldsymbol{\mathcal J}_{C^{[1]}_{n}}.
\end{align*}
Similarly, using the equation \eqref{conex1}, we get
\begin{align}\label{eq:PomegaP}
\prodint{ \check P^{[1]}_{n}(x),(\xi)_x}\mathcal W=\big[\omega_{n,n-N},\dots,\omega_{n,n-1}\big]	
\begin{bmatrix}
\prodint{ P^{[1]}_{n-N}(x),(\xi)_x}\mathcal W\\ 	\vdots\\ \prodint{ P^{[1]}_{n-1}(x),(\xi)_x}\mathcal W
\end{bmatrix}
+\prodint{ P^{[1]}_{n}(x),(\xi)_x}\mathcal W.
\end{align}
Now, from \eqref{eq:checkCP} we deduce
\begin{align*}
\big[\omega_{n,n-N},\dots,\omega_{n,n-1}\big]	
\begin{bmatrix}
\boldsymbol{\mathcal J}_{C^{[1]}_{n-N}}\\ 	\vdots\\ \boldsymbol{\mathcal J}_{C^{[1]}_{n-1}}
\end{bmatrix}
+\boldsymbol{\mathcal J}_{C^{[1]}_{n}}
&=\big[\omega_{n,n-N},\dots,\omega_{n,n-1}\big]	
\begin{bmatrix}
\prodint{ P^{[1]}_{n-N}(x),(\xi)_x}\mathcal W\\ 	\vdots\\ \prodint{ P^{[1]}_{n-1}(x),(\xi)_x}\mathcal W
\end{bmatrix}
+\prodint{ P^{[1]}_{n}(x),(\xi)_x}\mathcal W,
\end{align*}
that is to say
\begin{align*}
\big[\omega_{n,n-N},\dots,\omega_{n,n-1}\big]	
\begin{bmatrix}
\boldsymbol{\mathcal J}_{C^{[1]}_{n-N}}-\prodint{ P^{[1]}_{n-N}(x),(\xi)_x}\mathcal W\\ 	\vdots\\ \boldsymbol{\mathcal J}_{C^{[1]}_{n-1}}
-\prodint{ P^{[1]}_{n-1}(x),(\xi)_x}\mathcal W
\end{bmatrix}
&=
-\Big(\boldsymbol{\mathcal J}_{C^{[1]}_{n}}-\prodint{ P^{[1]}_{n}(x),(\xi)_x}\mathcal W\Big).
\end{align*}
Hence, using Remark \ref{rem:nonsingular} we get the result.
\end{proof}
\begin{rem}
		In the next results, the jets of the Christoffel--Darboux kernels are considered with respect to the first variable $x$, and we treat the $y$-variable as a parameter.
\end{rem}
\begin{teo}[Spectral Christoffel--Geronimus formulas]\label{teo:spectral}
When $n\geq N$, for monic Geronimus perturbations, with  masses as described in \eqref{eq:v_general},
	we have the following last quasideterminantal  expressions for the perturbed biorthogonal matrix polynomials and its matrix norms
	\begin{align*}
	\check P^{[1]}_{n}(x)&=
		\Theta_*\begin{bmatrix}\boldsymbol{\mathcal J}_{C^{[1]}_{n-N}}-\prodint{ P^{[1]}_{n-N}(x),(\xi)_x}\mathcal W& P_{n-N}^{[1]}(x)\\ 	\vdots&\vdots\\
\boldsymbol{\mathcal J}_{C^{[1]}_{n}}-\prodint{ P^{[1]}_{n}(x),(\xi)_x}\mathcal W& P^{[1]}_n(x)
		\end{bmatrix},\\
			\check H_{n}&=
			\Theta_*\begin{bmatrix}	\boldsymbol{\mathcal J}_{C^{[1]}_{n-N}}-\prodint{ P^{[1]}_{n-N}(x),(\xi)_x}\mathcal W& H_{n-N}\\
	\boldsymbol{\mathcal J}_{C^{[1]}_{n-N+1}}-\prodint{ P^{[1]}_{n-N+1}(x),(\xi)_x}\mathcal W& 0_p\\	\vdots&\vdots\\
	\boldsymbol{\mathcal J}_{C^{[1]}_{n}}-\prodint{ P^{[1]}_{n}(x),(\xi)_x}\mathcal W& 0_p
			\end{bmatrix},\\
				\big(	\check P _{n}^{[2]}(y)\big)^\top&=	-\Theta_*
				\begin{bmatrix}
				\boldsymbol{\mathcal J}_{C^{[1]}_{n-N}}-	\prodint{ P^{[1]}_{n-N}(x),(\xi)_x}\mathcal W&	H_{n-N}\\
							\boldsymbol{\mathcal J}_{C^{[1]}_{n-N+1}}-	\prodint{ P^{[1]}_{n-N+1}(x),(\xi)_x}\mathcal W& 0_p\\
				\vdots &\vdots\\	
				\boldsymbol{\mathcal J}_{C^{[1]}_{n-1}}-	\prodint{ P^{[1]}_{n-1}(x),(\xi)_x}\mathcal W&0_p\\
				W(y)\big(	\boldsymbol{\mathcal J}_{ K^{(pc)}_{n-1}}(y)-
				\prodint{ K_{n-1}(x,y),(\xi)_x}\mathcal W\big)+\boldsymbol{\mathcal J}_{ \mathcal V}(y)&0_p
				\end{bmatrix}.
	\end{align*}
\end{teo}

\begin{proof}
	First, we consider  the expressions for  $	\check P^{[1]}_{n}(x)$ and      $\check H_{n}$.
	Using relation \eqref{conex1} we have
	\begin{align*}
	\check P^{[1]}_{n}(x)&=P^{[1]}_n(x)+\big[\omega_{n,n-N},\dots,\omega_{n,n-1}\big]\begin{bmatrix}P_{n-N}^{[1]}(x)\\\vdots\\
P_{n-1}^{[1]}(x)	
	\end{bmatrix},
	\end{align*}
	 from Proposition \ref{pro:resolventCP} we obtain
\begin{align*}
	\check P^{[1]}_{n}(x)&=P^{[1]}_{n}(x)-\big(\boldsymbol{\mathcal J}_{C^{[1]}_{n}}-\prodint{ P^{[1]}_{n}(x),(\xi)_x}\mathcal W\big)
	\begin{bmatrix}\boldsymbol{\mathcal J}_{C^{[1]}_{n-N}}-\prodint{ P^{[1]}_{n-N}(x),(\xi)_x}\mathcal W\\ 	\vdots\\ \boldsymbol{\mathcal J}_{C^{[1]}_{n-1}}-\prodint{ P^{[1]}_{n-1}(x),(\xi)_x}\mathcal W
	\end{bmatrix}^{-1}\begin{bmatrix}P_{n-N}^{[1]}(x)\\\vdots\\
P_{n-1}^{[1]}(x)	
\end{bmatrix},
\end{align*}
and the result follows.
To get the transformation for the $H$'s we proceed as follows. From \eqref{eq:omegaA} we deduce
\begin{align} \label{eq:HG}
\check H_{n}=\omega_{n,n-N}H_{n-N}.
\end{align}
But, according to Proposition \ref{pro:resolventCP}, we have
\begin{align*}
\omega_{n,n-N}
&= -\big(\boldsymbol{\mathcal J}_{C^{[1]}_{n}}-\prodint{ P^{[1]}_{n}(x),(\xi)_x}\mathcal W\big)
\begin{bmatrix}\boldsymbol{\mathcal J}_{C^{[1]}_{n-N}}-\prodint{ P^{[1]}_{n-N}(x),(\xi)_x}\mathcal W\\ 	\vdots\\ \boldsymbol{\mathcal J}_{C^{[1]}_{n-1}}-\prodint{ P^{[1]}_{n-1}(x),(\xi)_x}\mathcal W
\end{bmatrix}^{-1}\begin{bmatrix}
I_p\\0_p\\\vdots\\0_p
\end{bmatrix}.
\end{align*}
Hence,
\begin{align*}
\check H_{n}=-\big(\boldsymbol{\mathcal J}_{C^{[1]}_{n}}-\prodint{ P^{[1]}_{n}(x),(\xi)_x}\mathcal W\big)
\begin{bmatrix}\boldsymbol{\mathcal J}_{C^{[1]}_{n-N}}-\prodint{ P^{[1]}_{n-N}(x),(\xi)_x}\mathcal W\\ 	\vdots\\ \boldsymbol{\mathcal J}_{C^{[1]}_{n-1}}-\prodint{ P^{[1]}_{n-1}(x),(\xi)_x}\mathcal W
\end{bmatrix}^{-1}\begin{bmatrix}
H_{n-N}\\0_p\\\vdots\\0_p
\end{bmatrix}.
\end{align*}

We now prove the result for $\Big(\check P^{[2]}_{n}(y)\Big)^\top$.	On  one hand, 	according to Definition \ref{eq:CD kernel}	we rewrite   \eqref{N} as
	\begin{align*}
	\sum_{k=0}^{n-1}\big(\check P_k^{[2]}(y)\big)^\top \check H_k^{-1}\check C_k^{[1]}(x)W(x)&=\begin{multlined}[t][0.7\textwidth]
	W(y)K^{(pc)}_{n-1}(x,y)\\-\begin{bmatrix}
	\big(	\check P_{n}^{[2]}(y)\big)^\top\Ht^{-1}_n,\dots,\big(\check P_{n+N-1}^{[2]}(y)\big)^\top\Ht^{-1}_{n+N-1}
	\end{bmatrix}
	\Omega{[n]}
	\begin{bmatrix}
	C_{n-N}^{[1]}(x)\\
	\vdots\\
	C_{n-1}^{[1]}(x)
	\end{bmatrix}+\mathcal V(x,y).
	\end{multlined}
	\end{align*}
	Therefore, the corresponding spectral jets do satisfy
	\begin{align*}
	\sum_{k=0}^{n-1}\big(\check P_k^{[2]}(y)\big)^\top \check H_k^{-1}\boldsymbol{\mathcal J}_{\check C^{[1]}_kW}&=\begin{multlined}[t][0.7\textwidth]	W(y)
	\boldsymbol{\mathcal J}_{ K^{(pc)}_{n-1}}(y)\\-\begin{bmatrix}
	\big(	\check P_{n}^{[2]}(y)\big)^\top\Ht^{-1}_n,\dots,\big(\check P_{n+N-1}^{[2]}(y)\big)^\top\Ht^{-1}_{n+N-1}
	\end{bmatrix}
	\Omega{[n]}\begin{bmatrix}
	\boldsymbol{\mathcal J}_{ C^{[1]}_{n-N}}\\ 	\vdots\\	\boldsymbol{\mathcal J}_{ C^{[1]}_{n-1}}
	\end{bmatrix}+\boldsymbol{\mathcal J}_{ \mathcal V}(y),
	\end{multlined}
	\end{align*}
	and, recalling \eqref{eq:checkCP}, we conclude that
	\begin{multline}\label{eq:WKpc}
	\sum_{k=0}^{n-1}\big(\check P_k^{[2]}(y)\big)^\top \check H_k^{-1}	\prodint{  \check P^{[1]}_{k} (x),(\xi)_x}\mathcal W=	W(y)
	\boldsymbol{\mathcal J}_{ K^{(pc)}_{n-1}}(y)\\-\begin{bmatrix}
	\big(	\check P_{n}^{[2]}(y)\big)^\top\Ht^{-1}_n,\dots,\big(\check P_{n+N-1}^{[2]}(y)\big)^\top\Ht^{-1}_{n+N-1}
	\end{bmatrix}
	\Omega{[n]}\begin{bmatrix}
	\boldsymbol{\mathcal J}_{ C^{[1]}_{n-N}}\\ 	\vdots\\	\boldsymbol{\mathcal J}_{ C^{[1]}_{n-1}}
	\end{bmatrix}+\boldsymbol{\mathcal J}_{ \mathcal V}(y).
	\end{multline}
	On the other hand, from \eqref{K} we realize that
	\begin{multline*}
	\sum_{k=0}^{n-1}\big(\check P_k^{[2]}(y)\big)^\top \check H_k^{-1}\prodint{ P^{[1]}_{k}(x),(\xi)_x}\mathcal W\\=
	W(y)\prodint{ K_{n-1}(x,y),(\xi)_x}\mathcal W-\begin{bmatrix}
	\big(	\check P_{n}^{[2]}(y)\big)^\top\Ht^{-1}_n,\dots,\big(\check P_{n+N-1}^{[2]}(y)\big)^\top\Ht^{-1}_{n+N-1}
	\end{bmatrix}
	\Omega{[n]}
	\begin{bmatrix}
	\prodint{ P^{[1]}_{n-N}(x),(\xi)_x}\mathcal W\\ 	\vdots\\\prodint{ P^{[1]}_{n-1}(x),(\xi)_x}\mathcal W
	\end{bmatrix},
	\end{multline*}
	which can be subtracted to 	\eqref{eq:WKpc} to get
	\begin{multline*}
	W(y)\big(	\boldsymbol{\mathcal J}_{ K^{(pc)}_{n-1}}(y)-\prodint{ K_{n-1}(x,y),(\xi)_x}\mathcal W\big)+\boldsymbol{\mathcal J}_{ \mathcal V}(y)\\=\begin{bmatrix}
	\big(	\check P_{n}^{[2]}(y)\big)^\top\Ht^{-1}_n,\dots,\big(\check P_{n+N-1}^{[2]}(y)\big)^\top\Ht^{-1}_{n+N-1}
	\end{bmatrix}
	\Omega{[n]}\begin{bmatrix}\boldsymbol{\mathcal J}_{C^{[1]}_{n-N}}-\prodint{ P^{[1]}_{n-N}(x),(\xi)_x}\mathcal W\\ 	\vdots\\ \boldsymbol{\mathcal J}_{C^{[1]}_{n-1}}-\prodint{ P^{[1]}_{n-1}(x),(\xi)_x}\mathcal W
	\end{bmatrix}.
	\end{multline*}	
	Hence, we obtain the  formula
	\begin{multline}\label{eq:PomeganN}
	\begin{bmatrix}
	\big(	\check P _{n}^{[2]}(y)\big)^\top\Ht^{-1}_n,\dots,\big(\check P_{n+N-1}^{[2]}(y)\big)^\top\Ht^{-1}_{n+N-1}
	\end{bmatrix}
	\Omega{[n]}\\=	\Big(	W(y)\big(	\boldsymbol{\mathcal J}_{ K^{(pc)}_{n-1}}(y)-\prodint{ K_{n-1}(x,y),(\xi)_x}\mathcal W\big)+\boldsymbol{\mathcal J}_{ \mathcal V}(y)\Big)
\begin{bmatrix}\boldsymbol{\mathcal J}_{C^{[1]}_{n-N}}-\prodint{ P^{[1]}_{n-N}(x),(\xi)_x}\mathcal W\\ 	\vdots\\ \boldsymbol{\mathcal J}_{C^{[1]}_{n-1}}-\prodint{ P^{[1]}_{n-1}(x),(\xi)_x}\mathcal W
\end{bmatrix}^{-1}.
	\end{multline}	
	Now, for $n\geq N$, from Definition \ref{def:omeganN}  and the fact that   $\omega_{n,n-N}=\check H_n\big(H_{n-N}\big)^{-1}$, we get
	\begin{align*}
	\big(	\check P _{n}^{[2]}(y)\big)^\top=	\Big(	W(y)\big(	\boldsymbol{\mathcal J}_{ K^{(pc)}_{n-1}}(y)-\prodint{ K_{n-1}(x,y),(\xi)_x}\mathcal W\big)+\boldsymbol{\mathcal J}_{ \mathcal V}(y)\Big)
	\begin{bmatrix}\boldsymbol{\mathcal J}_{C^{[1]}_{n-N}}-\prodint{ P^{[1]}_{n-N}(x),(\xi)_x}\mathcal W\\ 	\vdots\\ \boldsymbol{\mathcal J}_{C^{[1]}_{n-1}}-\prodint{ P^{[1]}_{n-1}(x),(\xi)_x}\mathcal W
	\end{bmatrix}^{-1}\begin{bmatrix}
	H_{n-N}\\0_p\\\vdots\\0_p
	\end{bmatrix},
	\end{align*}
	and the result follows.	
\end{proof}

\begin{rem}[Spectral Christoffel--Geronimus formulas with discrete masses supported by the diagonal]
	When $n\geq N$, for monic Geronimus perturbations, with  masses as described in \eqref{eq:v_diagonal},
	we have the following last quasideterminantal  expressions for the perturbed biorthogonal matrix polynomials and its matrix norms
	\begin{align*}
	\check P^{[1]}_{n}(x)&=
	\Theta_*\begin{bmatrix}\boldsymbol{\mathcal J}_{C^{[1]}_{n-N}}-{\mathcal J}_{P^{[1]}_{n-N}}\mathcal T& P_{n-N}^{[1]}(x)\\ 	\vdots&\vdots\\
	\boldsymbol{\mathcal J}_{C^{[1]}_{n}}-{\mathcal J}_{P^{[1]}_{n}}\mathcal T& P^{[1]}_n(x)
	\end{bmatrix},\\
	\check H_{n}&=
	\Theta_*\begin{bmatrix}	\boldsymbol{\mathcal J}_{C^{[1]}_{n-N}}-{\mathcal J}_{P^{[1]}_{n-N}}\mathcal T& H_{n-N}\\
	\boldsymbol{\mathcal J}_{C^{[1]}_{n-N+1}}-{\mathcal J}_{P^{[1]}_{n-N+1}}\mathcal T & 0_p\\	\vdots&\vdots\\
	\boldsymbol{\mathcal J}_{C^{[1]}_{n}}-{\mathcal J}_{P^{[1]}_{n}}\mathcal T & 0_p
	\end{bmatrix},\\
	\big(	\check P _{n}^{[2]}(y)\big)^\top&=	-\Theta_*
	\begin{bmatrix}
	\boldsymbol{\mathcal J}_{C^{[1]}_{n-N}}-	\mathcal J_{P^{[1]}_{n-N}}\mathcal T&	H_{n-N}\\
	\boldsymbol{\mathcal J}_{C^{[1]}_{n-N+1}}-	\mathcal J_{P^{[1]}_{n-N+1}}\mathcal T& 0_p\\
	\vdots &\vdots\\	
	\boldsymbol{\mathcal J}_{C^{[1]}_{n-1}}-	\mathcal J_{P^{[1]}_{n-1}}\mathcal T&0_p\\
	W(y)\big(	\boldsymbol{\mathcal J}_{ K^{(pc)}_{n-1}}(y)-\mathcal J_{K_{n-1}}(y)\mathcal T\big)+\boldsymbol{\mathcal J}_{ \mathcal V}(y)&0_p
	\end{bmatrix}.
	\end{align*}
\end{rem}

\subsubsection{Discussion for $n< N$}
Later on, in the context of nonspectral methods, we will derive Corollary \ref{corollary:non_singularity2}, which is applicable in our monic polynomial perturbation  scenario. Thus, we know that  $	\begin{bsmallmatrix}
\boldsymbol{\mathcal J}_{C^{[1]}_0}-\prodint{ P^{[1]}_{0}(x),(\xi)_x}\mathcal W\\
\vdots\\
\boldsymbol{\mathcal J}_{C^{[1]}_n}-\prodint{ P^{[1]}_{n}(x),(\xi)_x}\mathcal W
\end{bsmallmatrix}$ is full rank.

\begin{pro}\label{pro:LUN}
The truncations $\omega_{[N]}$, $\check H_{[N]}$, and $(\check S_2)_{[N]}$ yield the Gauss--Borel factorization
	\begin{align*}
\big(\omega_{[N]}\big)^{-1}	\check H_{[N]}\big((\check S_2)_{[N]}\big)^{-\top}=-		\begin{bmatrix}
	\boldsymbol{\mathcal J}_{C^{[1]}_0}-\prodint{ P^{[1]}_{0}(x),(\xi)_x}\mathcal W\\
	\vdots\\			\boldsymbol{\mathcal J}_{C^{[1]}_{N-1}}-\prodint{ P^{[1]}_{N-1}(x),(\xi)_x}\mathcal W
	\end{bmatrix}\mathcal R,
	\end{align*}
	where $\mathcal R$ is given in \eqref{eq:RJY}.
\end{pro}
\begin{proof}
		From \eqref{eq:conexionC1} we deduce
		\begin{align*}
		\big(\check C^{[1]}(z)\big)_{[N]}W(z)-
		\big(\check H\big(\check S_2\big)^{-\top}\big)_{[N]}\mathcal B 	(\chi(x))_{[N]}
		=	\omega_{[N]} \big(C^{[1]}(z)\big)_{[N]},
		\end{align*}
		so that
		\begin{align*}
		\begin{bmatrix}
		\boldsymbol{\mathcal J}_{\check C^{[1]}_0W}\\
		\vdots\\
		\boldsymbol{\mathcal J}_{\check C^{[1]}_{N-1}W}
		\end{bmatrix}-
		\big(\check H\big(\check S_2\big)^{-\top}\big)_{[N]}\mathcal B \mathcal Q=\omega_{[N]}		\begin{bmatrix}
		\boldsymbol{\mathcal J}_{ C^{[1]}_0}\\
		\vdots\\
		\boldsymbol{\mathcal J}_{ C^{[1]}_{N-1}}
		\end{bmatrix}.
		\end{align*}
			Recalling  \eqref{eq:checkCP} we deduce
			\begin{align*}
			\begin{bmatrix}
		\prodint{ \check P^{[1]}_{0}(x),(\xi)_x}\mathcal W\\
			\vdots\\
		\prodint{ \check P^{[1]}_{N-1}(x),(\xi)_x}\mathcal W
			\end{bmatrix}-
			\big(\check H\big(\check S_2\big)^{-\top}\big)_{[N]}\mathcal B \mathcal Q=\omega_{[N]}		\begin{bmatrix}
			\boldsymbol{\mathcal J}_{ C^{[1]}_0}\\
			\vdots\\
			\boldsymbol{\mathcal J}_{ C^{[1]}_{N-1}}
			\end{bmatrix}.
			\end{align*}
			Therefore, using  \eqref{conex1}, we conclude
			\begin{align}\label{eq:WomegaCP}
			-			\big(\check H\big(\check S_2\big)^{-\top}\big)_{[N]}=\omega_{[N]}		\begin{bmatrix}
			\boldsymbol{\mathcal J}_{C^{[1]}_0}-\prodint{ P^{[1]}_{0}(x),(\xi)_x}\mathcal W\\
			\vdots\\			\boldsymbol{\mathcal J}_{C^{[1]}_{N-1}}-\prodint{ P^{[1]}_{N-1}(x),(\xi)_x}\mathcal W
			\end{bmatrix}\big(\mathcal B \mathcal Q\big)^{-1},
			\end{align}
			and the result is proven.
\end{proof}

From Proposition \ref{pro:LUN} and  the corresponding explicit quasideterminantal  expressions for its solution we get
\begin{ma}\label{lemma:<N}
	For  $n\in\{0,1,\dots,N-1\}$ and $0\leq k<n$,
		\begin{align*}
	\check H_{n}&=-\Theta_*\begin{bmatrix}
	\big(\boldsymbol{\mathcal J}_{C^{[1]}_0}-\prodint{ P^{[1]}_{0}(x),(\xi)_x}\mathcal W\big)\mathcal R_n\\
	\vdots\\			\big(\boldsymbol{\mathcal J}_{C^{[1]}_{n-1}}-\prodint{ P^{[1]}_{n-1}(x),(\xi)_x}\mathcal W\big)\mathcal R_n
	\end{bmatrix}, \\
	\omega_{n,k}&=\Theta_*\left[\begin{array}{c|c}
	\begin{matrix}
		\big(\boldsymbol{\mathcal J}_{C^{[1]}_0}-\prodint{ P^{[1]}_{0}(x),(\xi)_x}\mathcal W\big)\mathcal R_n\\
\vdots\\
\big(\boldsymbol{\mathcal J}_{C^{[1]}_{n}}-\prodint{ P^{[1]}_{n}(x),(\xi)_x}\mathcal W\big)\mathcal R_n\\	
	\end{matrix}& e_k
	\end{array}\right],\\	\big((\check S_2)^\top\big)_{n,k}&=\Theta_*\begin{bmatrix}
		\big(\boldsymbol{\mathcal J}_{C^{[1]}_0}-\prodint{ P^{[1]}_{0}(x),(\xi)_x}\mathcal W\big)\mathcal R_{n+1}\\
\vdots\\
		\big(\boldsymbol{\mathcal J}_{C^{[1]}_{n-1}}-\prodint{ P^{[1]}_{n-1}(x),(\xi)_x}\mathcal W\big)\mathcal R_{n+1}\\[5pt]
(e_k)^\top
	\end{bmatrix}.
	\end{align*}
Here we have used  the matrices $e_k=\begin{bmatrix} 0_p, \dots, 0_p,I_p,0_p,\dots,0_p)
\end{bmatrix}^\top\in\mathbb C^{(n+1)p\times p}$ with all its $p\times p$ blocks  being the zero matrix $0_p$, but for the $k$-th block which is the identity matrix  $I_p$.
\end{ma}

	\begin{teo}[Spectral Christoffel--Geronimus formulas]\label{teo:spectralN}
		For $n< N$ and monic Geronimus perturbations,  with masses as described in \eqref{eq:v_general},
		we have the following last quasideterminant expressions for the perturbed biorthogonal matrix polynomials
			\begin{align}\label{eq:checkP1}
			\check P^{[1]}_{n}(x)&=
			\Theta_*\begin{bmatrix}	\big(	\boldsymbol{\mathcal J}_{C^{[1]}_{0}}-\prodint{ P^{[1]}_{0}(x),(\xi)_x}\mathcal W\big)\mathcal R_n& P_{0}^{[1]}(x)\\ 	\vdots&\vdots\\
			\big(	\boldsymbol{\mathcal J}_{C^{[1]}_{n}}-\prodint{ P^{[1]}_{n}(x),(\xi)_x}\mathcal W\big)\mathcal R_n& P^{[1]}_n(x)
			\end{bmatrix},
\\
		\big(\check P^{[2]}_n(y)\big)^\top&=\Theta_*
		\begin{bmatrix}
		\big(\boldsymbol{\mathcal J}_{C^{[1]}_{0}}-\prodint{ P^{[1]}_{0}(x),(\xi)_x}\mathcal W\big)\mathcal R_{n+1}\\
		\vdots\\			\big(\boldsymbol{\mathcal J}_{C^{[1]}_{n-1}}-\prodint{ P^{[1]}_{n-1}(x),(\xi)_x}\mathcal W\big)\mathcal R_{n+1}\\
	(\chi(y))^\top_{[n+1]}
		\end{bmatrix}.
		 \label{eq:checkP2N}
			\end{align}
	\end{teo}
\begin{proof}
	For $n<N$ , \eqref{conex1} and \eqref{eq:bior} imply
	\begin{align*}
	\check P^{[1]}_{n}(x)&=P^{[1]}_n(x)+\big[\omega_{n,0},\dots,\omega_{n,n-1}\big]\begin{bmatrix}P_{0}^{[1]}(x)\\\vdots\\
	P_{n-1}^{[1]}(x)	
	\end{bmatrix},\\
	\big(\check P^{[2]}_n(y)\big)^\top&=I_py ^n+\begin{bmatrix}
	I_p,\dots,I_py^{n-1}
	\end{bmatrix}\begin{bmatrix}
	\big( (\check S_2)^\top\big)_{0,n}\\\vdots\\
	\big( (\check S_2)^\top\big)_{n-1,n}
	\end{bmatrix}.
	\end{align*}
	Then,  Lemma \ref{lemma:<N} gives the stated result.
\end{proof}

	\begin{rem}[Spectral Christoffel--Geronimus formulas]
	For    masses as described in \eqref{eq:v_diagonal},
		we have
		\begin{align*}
		\check P^{[1]}_{n}(x)&=
		\Theta_*\begin{bmatrix}	\big(	\boldsymbol{\mathcal J}_{C^{[1]}_{0}}-\mathcal J_{P^{[1]}_{0}}\mathcal T\big)\mathcal R_n& P_{0}^{[1]}(x)\\ 	\vdots&\vdots\\
		\big(	\boldsymbol{\mathcal J}_{C^{[1]}_{n}}-\mathcal J_{P^{[1]}_{n}}\mathcal T\big)\mathcal R_n& P^{[1]}_n(x)
		\end{bmatrix},
		\\
		\big(\check P^{[2]}_n(y)\big)^\top&=\Theta_*
		\begin{bmatrix}
		\big(\boldsymbol{\mathcal J}_{C^{[1]}_{0}}-\mathcal J_{P^{[1]}_{0}}\mathcal T\big)\mathcal R_{n+1}\\
		\vdots\\			\big(\boldsymbol{\mathcal J}_{C^{[1]}_{n-1}}-\mathcal J_{P^{[1]}_{n-1}}\mathcal T\big)\mathcal R_{n+1}\\
		(\chi(y))^\top_{[n+1]}
		\end{bmatrix}.
		\end{align*}
	\end{rem}

\subsection{Nonspectral Christoffel--Geronimus formulas}
We now present  an alternative \emph{orthogonality relations} approach for the derivation of Christoffel type formulas,   that avoids the use of  the second kind functions and of the spectral structure of the perturbing polynomial.
A key feature of these results  is that they hold even for perturbing matrix polynomials with singular leading coefficient.

\begin{defi}\label{def:R}
For a given perturbed matrix of  generalized kernels  $\check u_{x,y}= u_{x,y}\big(W(y)\big)^{-1}+v_{x,y}$, with   $v_{x,y}W(y)=0_{p}$, we define a semi-infinite block matrix
	\begin{align*}
	R&:=	\prodint{P^{[1]}(x),\chi(y)}_{\check u}
	\\
	&=	\prodint{P^{[1]}(x),\chi(y)}_{u W^{-1}}+	\prodint{P^{[1]}(x),\chi(y)}_{ v}.
	\end{align*}
\end{defi}
\begin{rem}
		Its blocks are
	$R_{n,l}=\prodint{P^{[1]}_n(x),I_py^l}_{\check u}\in\mathbb C^{p\times p}$.
		Observe that for a Geronimus  perturbation of a Borel measure $\d\mu(x,y)$, with general masses as in \eqref{eq:v_general} we have
			\begin{align*}
			R_{n,l}=\int P^{[1]}_n(x)\d\mu(x,y)(W(y))^{-1}y^l+\sum_{a=1}^{q}\sum_{i=1}^{s_a}\sum_{m=0}^{\kappa_j^{(a)}-1}\frac{1}{m!}
			\langle P^{[1]}_n(x),
			\big(\xi^{[a]}_{i,m}\big)_x\rangle
			\big(l_{j}^{(a)}(y)y^l\big)^{(m)}_{x_a},
			\end{align*}	
		that, when the masses are discrete and supported by the diagonal $y=x$, reduces to
			\begin{align*}
			R_{n,l}=\int P^{[1]}_n(x)\d\mu(x,y)(W(y))^{-1}y^l+\sum_{a=1}^{q}\sum_{i=1}^{s_a}\sum_{m=0}^{\kappa_j^{(a)}-1}\frac{1}{m!}\Big(P^{[1]}_n(x)x^l\xi^{[a]}_{i,m}l_{j}^{(a)}(x)\Big)^{(m)}_{x_a}.
			\end{align*}
\end{rem}

\begin{pro}
The following relations hold true
	\begin{align}\label{eq:RSM}
	R&=S_1\check G,\\
	\omega R&=\check H \big(\check S_2\big)^{-\top},\label{OmegaR0}\\
	RW(\Lambda^\top) &=	H\big(S_2\big)^{-\top}.\label{RW}
	\end{align}
\end{pro}
\begin{proof}
	 \eqref{eq:RSM} follows  from Definition \ref{def:R}. Indeed,
	\begin{align*}
	R&=	\prodint{P^{[1]}(x),\chi(y)}_{\check u}\\
	&=S_1	\prodint{\chi(x),\chi(y)}_{\check u}.
	\end{align*}
	To deduce \eqref{OmegaR0} we recall  \eqref{eq:def_Omega}, \eqref{eq:RSM}, and the Gauss factorization of the perturbed matrix of moments
	\begin{align*}
	\omega R&= \big(\check S_1 (S_1)^{-1} \big) (S_1\check G)\\&=\check S_1\check G\\&=\check S_1\Big( (\check S_1)^{-1}\check H (\check S_2)^{-\top}\Big)
	\end{align*}
	Finally, to get \eqref{RW}, we use \eqref{eq:resolvent_alternative} together with \eqref{OmegaR0}, which implies
$\omega=\omega R W(\Lambda^\top) \big(S_2)^\top H^{-1}$,
	and as the resolvent is unitriangular with a unique inverse matrix \cite{cooke},  we obtain the result.		
\end{proof}

From \eqref{OmegaR0} it immediately follows that
\begin{pro}\label{OmegaR}
	The matrix $R$ fulfills
	\begin{align*}
	(\omega R)_{n,l}=\begin{cases}
	0_p, &l\in\{0,\dots,n-1\},\\
	\check H_n, & n=l.
	\end{cases}
	\end{align*}
\end{pro}

\begin{pro}
	The matrix
	$\begin{bsmallmatrix}
	R_{0,0} & \dots & R_{0,n-1}\\
	\vdots & & \vdots\\
	R_{n-1,0} & \dots & R_{n-1,n-1}
	\end{bsmallmatrix}$
	is nonsingular.
\end{pro}
\begin{proof}
	From \eqref{eq:RSM} we conclude for the corresponding truncations that $R_{[n]}=(S_1)_{[n]}\check G_{[n]}$ is nonsingular, as we are assuming, to ensure the orthogonality, that $\check G_{[n]}$ is nonsingular for all $n\in\{1,2,\dots\}$.
\end{proof}

\begin{defi}
	Let us introduce the polynomials $	r^K_{n,l}(z)\in\mathbb C^{p\times p}[z]$, $l\in\{0,\dots,n-1\}$, given by
	\begin{align*}
	r^K_{n,l}(z):&=\prodint{W(z)K_{n-1}(x,z),I_py^l}_{\check u}-I_pz^l\\ &=
\prodint{W(z)K_{n-1}(x,z),I_py^l}_{ uW^{-1}}+\prodint{W(z)K_{n-1}(x,z),I_py^l}_{v}-I_pz^l.
	\end{align*}
\end{defi}

\begin{pro}
	For $l\in\{0,1,\dots,n-1\}$	 and  $m=\min(n,N)$ we have
	\begin{align*}
	r^K_{n,l}(z)=\Big[(\check P_{n}^{[2]}(z))^\top(\check H_{n})^{-1},\dots,(\check P_{n-1+N}^{[2]}(z))^\top(\check H_{n-1+N})^{-1}\Big]
\Omega{[n]}
	\begin{bmatrix}
	R_{n-m,l}\\
	\vdots\\
	R_{n-1,l}
	\end{bmatrix}.
	\end{align*}
\end{pro}
\begin{proof}
	It follows from  \eqref{K}, Definition \ref{def:R}, and  \eqref{eq:K-u}.
\end{proof}

\begin{defi}
	For $n\geq N$,  given the matrix
	\begin{align*}
	\begin{bmatrix}
	R_{n-N,0} & \dots & R_{n-N,n-1}\\
	\vdots & & \vdots\\
	R_{n-1,0} & \dots & R_{n-1,n-1}
	\end{bmatrix}\in\mathbb C^{Np\times n p},
	\end{align*}
	we  construct a submatrix of it by selecting $Np$ columns among all the $np$ columns. For that aim, we use indexes $(i,a)$  labeling the columns, where $i$ runs through $\{0,\dots,n-1\}$ and indicates the block, and $a\in\{1,\dots, p\}$ denotes the corresponding column in  that block; i.e., $(i,a)$ is an index selecting the $a$-th column of the $i$-block. Given a set of $N$ different couples $I=\{(i_r,a_r)\}_{r=1}^{N}$, with a lexicographic ordering, we define the corresponding \emph{square}
	submatrix $R_n^{\square}:=\big[\mathfrak c_{(i_1,a_1)},\dots, \mathfrak c_{(i_{Np},a_{Np})}\big]$. Here $\mathfrak c_{(i_r,a_r)}$ denotes the $a_r$-th column of the matrix
	\begin{align*}
	\begin{bmatrix}
	R_{n-N,i_r} \\
	\vdots \\
	R_{n-1,i_r}
	\end{bmatrix}.
	\end{align*}
	The set of indexes $I$ is said poised if $R_n^{\square}$ is nonsingular.
	We also use the notation
	where $r_n^{\square}:=\big[\tilde {\mathfrak c}_{(i_1,a_1)},\dots, \tilde{\mathfrak c}_{(i_{Np},a_{Np})}\big]$. Here $\tilde {\mathfrak c}_{(i_r,a_r)}$ denotes the $a_r$-th column of the matrix $	R_{n,i_r}$.
	Given  a poised set of indexes we define  $(r^K_{n}(y))^\square$ as the matrix built up by taking  from the matrices $r^K_{n,i_r}(y)$  the columns $a_r$.
\end{defi}	

\begin{ma}
	For $n\geq N$, there exists at least a poised set.
\end{ma}
\begin{proof}
	For $n\geq N$, we consider the rectangular block matrix
	\begin{align*}
	\begin{bmatrix}
	R_{n-N,0} & \dots & R_{n-N,n-1}\\
	\vdots & & \vdots\\
	R_{n-1,0} & \dots & R_{n-1,n-1}
	\end{bmatrix}\in\mathbb C^{Np\times n p}.
	\end{align*}
	As the truncation $R_{[n]}$ is nonsingular, this matrix is  full rank, i.e., all its $Np$ rows are linearly independent. Thus, there must be $Np$ independent columns and the desired result follows.
\end{proof}

\begin{ma}\label{lemma:Euclides}
	Whenever the leading coefficient $A_N$ of the perturbing polynomial $W(y)$ is nonsingular, we can  decompose any monomial $I_p y^l$ as
	\begin{align*}
	I_p y^l=\alpha_l(y)(W(y))^\top+\beta_l(y),
	\end{align*}
	where $\alpha_l(y),\beta_l(y)=\beta_{l,0}+\cdots+\beta_{l,N-1}y^{N-1}\in\mathbb C^{p\times p}[y] $, with $\deg  \alpha_l(y)  \leq l-N$.
\end{ma}
\begin{pro}\label{pro:first_poised}
	Let us assume that the matrix polynomial $W(y)=A_Ny^N+\dots +A_0$ has a nonsingular leading coefficient and $n\geq N$. Then, the set $\{0,1,\dots,N-1\}$ is poised.
\end{pro}
\begin{proof}
	From Proposition \ref{OmegaR} we deduce
\begin{align*}
	\big[\omega_{n,n-N},\dots,\omega_{n,n-1}\big]\begin{bmatrix}
	R_{n-N,l}\\\vdots\\
	R_{n-1,l}
	\end{bmatrix}=-R_{n,l},
\end{align*}
for $l\in\{0,1,\dots,n-1\}$.
In particular, the resolvent vector $\big[\omega_{n,n-N},\dots,\omega_{n,n-1}\big]$ is a solution of the linear system
\begin{align}\label{eq:system}
\big[\omega_{n,n-N},\dots,\omega_{n,n-1}\big]\begin{bmatrix}
R_{n-N,0}&\dots & R_{n-N,N-1}\\\vdots& &\vdots\\
R_{n-1,0}&\dots &R_{n-1,N-1}
\end{bmatrix}=-\begin{bmatrix}
R_{n,0},\dots,R_{n,N-1}
\end{bmatrix}.
\end{align}
We will show now that this is the unique solution to this linear system. Let us proceed by contradiction and assume that there is another solution, say $\big[\tilde\omega_{n,n-N},\dots,\tilde\omega_{n,n-1}\big]$. Consider then the monic matrix polynomial
\begin{align*}
\tilde P_n(x)=P^{[1]}_n(x)+\tilde\omega_{n,n-1}P^{[1]}_{n-1}(x)+\dots+\tilde\omega_{n,n-N}P^{[1]}_{n-N}(x).
\end{align*}
Because $\big[\tilde\omega_{n,n-N},\dots,\tilde\omega_{n,n-1}\big]$ solves \eqref{eq:system} we know that
\begin{align*}
\langle \tilde P_n(x), I_py^l\rangle_{\check u}&=0_p, & l\in\{0,\dots, N-1\}.
\end{align*}
 %and, hence, from
 Lemma  \ref{lemma:Euclides} implies  the following relations for $\deg \alpha_l(y)<m$,
\begin{align*}
%R_{m,l}&=
\langle P^{[1]}_m(x), I_py^l\rangle_{\check u}&=\langle P^{[1]}_m(x), \alpha_l(y)\rangle_{\check u W}+\langle P^{[1]}_m(x), \beta_l(y)\rangle_{\check u}\\
&=\langle P^{[1]}_m(x), \alpha_l(y)\rangle_{ u }+\langle P^{[1]}_m(x), \beta_l(y)\rangle_{\check u}\\
&=\langle P^{[1]}_m(x), \beta_l(y)\rangle_{\check u}.
\end{align*}
But $\deg\alpha_l(y)\leq l-N$, so that the previous equation will hold at least for $l-N<m$; i.e., $l<m+N$.
Consequently,  for $l\in\{0,\dots,n-1\}$, we find
\begin{align*}
\langle\tilde P_n(x), I_p y^l\rangle_{\check u}&= \langle P^{[1]}_n(x), I_p y^l\rangle_{\check u}+\tilde\omega_{n,n-1}\langle  P^{[1]}_{n-1}(x), I_p y^l\rangle_{\check u}+\dots+\tilde\omega_{n,n-N}\langle P^{[1]}_{n-N}(y)
, I_p y^l\rangle_{\check u}\\
&= \langle P^{[1]}_n(x), \beta_l(y)\rangle_{\check u}+\tilde\omega_{n,n-1}\langle  P^{[1]}_{n-1}(x), \beta_l(y)\rangle_{\check u}+\dots+\tilde\omega_{n,n-N}\langle P^{[1]}_{n-N}(x)
, \beta_l(y)\rangle_{\check u}\\
&=\sum_{k=0}^{N-1}\big(R_{n,k}+\tilde{\omega}_{n,n-1}R_{n-1,k}+\dots+ \tilde{\omega}_{n,n-N}R_{n-N,k} \big)    (\beta_{l,k})^\top\\
&=0_p.
\end{align*}
Therefore,  from the uniqueness of the biorthogonal families,  we deduce
$\tilde P_n(x)=\check P^{[1]}_n(x)$,
and, recalling \eqref{conex1},  there is a unique solution of \eqref{eq:system}. Thus,
\begin{align*}
\begin{bmatrix}
R_{n-N,0}& \dots & R_{n-N,n-1}\\\vdots& &\vdots\\
R_{n-1,0}&\dots &R_{n-1,n-1}
\end{bmatrix}
\end{align*}
is nonsingular, and $I=\{0,\dots,N-1\}$ is a poised set.
\end{proof}	

\begin{pro}\label{pro:como_es_Omega}
	For $n< N$, we can write
	\begin{align*}
	\big[\omega_{n,0},\dots,\omega_{n,n-1}\big]=-\big[R_{n,0},\dots,R_{n,n-1}\big]
	\begin{bmatrix}
	R_{0,0} & \dots & R_{0,n-1}\\
	\vdots & & \vdots\\
	R_{n-1,0} & \dots & R_{n-1,n-1}
	\end{bmatrix}^{-1}.
	\end{align*}
For $n\geq N$,  given poised set, which always exists, we have
	\begin{align*}
	\big[\omega_{n,n-N},\dots,\omega_{n,n-1}\big]=-r_n^{\square}
	(R_n^{\square})^{-1}.
	\end{align*}
\end{pro}
\begin{proof}
	It follows from Proposition \ref{OmegaR}.
\end{proof}

\begin{teo}[Non-spectral Christoffel--Geronimus formulas]\label{theorem:nonspectral}
	Given a matrix Geronimus transformation the corresponding  perturbed polynomials, $\{\check P^{[1]}_{n}(x)\}_{n=0}^\infty$ and $\{\check P^{[2]}_{n}(y)\}_{n=0}^\infty$, and matrix norms $\{\check H_{n}\}_{n=0}^\infty$  can be expressed as follows.
		For $n\geq  N$,
		\begin{align*}
		\check P^{[1]}_n(x)&=\Theta_*\begin{bmatrix}[c|c]
		R_n^{\square}&\begin{matrix}P^{[1]}_{n-N}(x)\\\vdots \\P^{[1]}_{n-1}(x)\end{matrix}\\\hline
		r_n^{\square} & P^{[1]}_{n}(x)
		\end{bmatrix},& \big(	\check P_n^{[2]}(y)\big)^\top A_N=-\Theta_*\left[\begin{array}{c|c}R_n^{\square} & \begin{matrix}
		H_{n-N}\\ 0_p\\\vdots \\0_p
		\end{matrix}\\\hline
		(r^K_{n}(y))^\square& 0_p
		\end{array}	
		\right],\end{align*}
		and two alternative expressions 
	\begin{align*}
			\check H_n&=\Theta_*\left[\begin{array}{c|c}
				R_n^{\square}&\begin{matrix}R_{n-N,n}\\\vdots \\R_{n-1,n}\end{matrix}\\\hline
				r_n^{\square} & R_{n,n}
				\end{array}\right]=\Theta_*\left[\begin{array}{c|c}R_n^{\square} & \begin{matrix}
				H_{n-N}\\ 0_p\\\vdots \\0_p
				\end{matrix}\\\hline
				r_{n}^\square& 0_p
				\end{array}	
				\right]
		\end{align*}
	 For $n< N$, we have $ \check H_n=\Theta_*\big(R_{[n+1]}\big)$ and
	 \begin{align*}
	 \check P^{[1]}_n(x)&=\Theta_*\begin{bmatrix}
	 R_{0,0} & \dots & R_{0,n-1} & P^{[1]}_0(x)\\
	 \vdots & & \vdots &\vdots\\
	 R_{n,0} & \dots & R_{n,n-1} & P^{[1]}_{n}(x)\
	 \end{bmatrix},&
	\big(\check P^{[2]}_n(y)\big)^\top&=\Theta_*\begin{bmatrix}
	 R_{0,0} & \dots & R_{0,n-1} &R_{0,n}\\
	 \vdots & & \vdots &\vdots\\
	 R_{n-1,0} & \dots & R_{n-1,n-1} & R_{n-1,n}\\
	 I_p & \dots & I_p y^{n-1} & I_p y^n
	 \end{bmatrix}.
	 \end{align*}			
\end{teo}
\begin{proof}
For $m=\min(n,N)$, from the connection formula \eqref{conex2} we have
		\begin{align*}
		\check P^{[1]}_{n}(x)=	[\omega_{n,n-m},\dots,\omega_{n,n-1}]\begin{bmatrix}
		P^{[1]}_{n-m}(x)\\\vdots\\P^{[1]}_{n-1}(x)
		\end{bmatrix}+P^{[1]}_{n}(x),
		\end{align*}
		and from Proposition \ref{OmegaR} we deduce
		\begin{align*}
		\check H_n=	\
		[\omega_{n,n-m},\dots,\omega_{n,n-1}]\begin{bmatrix}
		R_{n-m,n}\\\vdots\\R_{n-1,n}
		\end{bmatrix}+R_{n,n},
		\end{align*}
		and use \eqref{eq:HG}.
		Then, recalling Proposition \ref{pro:como_es_Omega} we obtain the desired formulas for $\check P^{[1]}_n(x)$ and $\check H_n$.
		\item
		For $n\geq N$, we have
		\begin{align*}
		r^K_{n,l}(y)=\begin{bmatrix}(\check P_{n}^{[2]}(y))^\top(\check H_{n})^{-1},\dots,(\check P_{n-1+N}^{[2]}(y))^\top(\check H_{n-1+N})^{-1}\end{bmatrix}\Omega{[n]}
		\begin{bmatrix}
		R_{n-N,l}\\
		\vdots\\
		R_{n-1,l}
		\end{bmatrix},
		\end{align*}
	so that
		\begin{align*}
		(r^K_n(y))^\square(R_n^{\square})^{-1}=
	\begin{bmatrix}(\check P_{n}^{[2]}(y))^\top(\check H_{n})^{-1},\dots,(\check P_{n-1+N}^{[2]}(y))^\top(\check H_{n-1+N})^{-1}\end{bmatrix}
		\Omega{[n]}.
		\end{align*}
		In particular, recalling \eqref{eq:omegaA}, we deduce that
		\begin{align*}
		(\check P_{n}^{[2]}(y))^\top A_N=(r^K_n(y))^\square(R_n^{\square})^{-1}\begin{bmatrix}
		H_{n-N}\\ 0_p\\\vdots \\0_p
		\end{bmatrix}.
		\end{align*}

			For $n<N$, 	we can write	\eqref{OmegaR0} as $R\big(\check S_2\big)^\top=\omega^{-1} \check H$, which implies
				\begin{align*}
				\begin{bmatrix}
				\big((S_2)^\top\big)_{0,n}\\\vdots\\
				\big((S_2)^\top\big)_{n-1,n}
				\end{bmatrix}=-\big(R_{[n]}\big)^{-1}\begin{bmatrix}
				R_{0,n}\\\vdots\\R_{n-1,n}
				\end{bmatrix},
				\end{align*}
		giving in turn
				\begin{align*}
			\big(P^{[2]}_n(y)\big)^\top&=y^n+\begin{bmatrix}
			I_p,\dots,I_py^{n-1}
			\end{bmatrix}\begin{bmatrix}
			\big((S_2)^\top\big)_{0,n}\\\vdots\\
			\big((S_2)^\top\big)_{n-1,n}
			\end{bmatrix}\\&=
			y^n-\begin{bmatrix}
			I_p,\dots,I_py^{n-1}
			\end{bmatrix}\big(R_{[n]}\big)^{-1}\begin{bmatrix}
			R_{0,n}\\\vdots\\R_{n-1,n}
			\end{bmatrix}.
				\end{align*}	
\end{proof}
Observe that  the Gauss--Borel factorization
$R_{[n]}=\big(\omega_{[n]}\big)^{-1} \check H_{[n]}\big((\check S_2)_{[n]}\big)^{-\top}$,
 gives the formulas in the previous theorem,  for $n<N$.
\subsection{Spectral versus nonspectral}
\begin{defi}
	We introduce the  truncation given by taking only the first $N$ columns of a given semi-infinite matrix
	\begin{align*}
	R^{(N)}:=\begin{bmatrix}
	R_{0,0} & R_{0,1} &\dots &R_{0,N-1}\\
	R_{1,0} & R_{0,1} &\dots &R_{1,N-1}\\
	\vdots &\vdots && \vdots
	\end{bmatrix}.
	\end{align*}
\end{defi}
Then, we can connect the spectral methods and the nonspectral techniques as follows
\begin{pro}\label{pro:specvsnon}
The following relation takes place
		\begin{align*}
		\boldsymbol{\mathcal J}_{C^{[1]}}-\prodint{ P^{[1]}(x),(\xi)_x}\mathcal W=-R^{(N)}
		\mathcal B 	\mathcal Q.
		\end{align*}
\end{pro}

\begin{proof}
From \eqref{eq:conexionC1} we deduce that
	\begin{align*}
	\check C^{[1]}(x)W(x)-\check H\big(\check S_2\big)^{-\top}\left[\begin{array}{c}
	\mathcal B 	(\chi(x))_{[N]}
	\\
	0_{p}\\
	\vdots
	\end{array}\right]&=	\omega C^{[1]}(x).\\
	\end{align*}
Taking the corresponding root spectral jets, we obtain
	\begin{align*}
\boldsymbol{\mathcal J}_{\check C^{[1]}W}-\check H\big(\check S_2\big)^{-\top}\left[\begin{array}{c}
		\mathcal B 	\mathcal Q
		\\
		0_{p}\\
		\vdots
	\end{array}\right]&=	\omega \boldsymbol{\mathcal J}_{C^{[1]}},\\
\end{align*}
that, together with \eqref{eq:checkCP}, gives
		\begin{align*}
		\omega \Big(\boldsymbol{\mathcal J}_{C^{[1]}}-\prodint{ P^{[1]}(x),(\xi)_x}\mathcal W\Big)=-\check H\big(\check S_2\big)^{-\top}\left[\begin{array}{c}
		\mathcal B 	\mathcal Q
		\\
		0_{p}\\
		\vdots
		\end{array}\right].
		\end{align*}
Now, relation \eqref{OmegaR0} implies
	\begin{align*}
	\omega \Big(\boldsymbol{\mathcal J}_{C^{[1]}}-\prodint{ P^{[1]}(x),(\xi)_x}\mathcal W+R^{(N)}\mathcal B\mathcal Q\Big)=0.
	\end{align*}
But, given that $\omega$ is a lower unitriangular matrix, and therefore with an inverse, see \cite{cooke}, the unique solution to $\omega X=0$, where $X$ is a semi-infinite matrix, is $X=0$.
\end{proof}
We now discuss  an important fact, which  ensures that the spectral Christoffel--Geronimus formulas presented in previous sections make sense
\begin{coro}\label{corollary:non_singularity}
If the leading coefficient $A_N$  is nonsingular and  $n\geq N$,	then
\begin{align*}
\begin{bmatrix}
\boldsymbol{\mathcal J}_{ C^{[1]}_{n-N}}-\prodint{ P^{[1]}_{n-N}(x),(\xi)_x}\mathcal W\\
\vdots\\
\boldsymbol{\mathcal J}_{ C^{[1]}_{n-1}}-\prodint{ P^{[1]}_{n-1}(x),(\xi)_x}\mathcal W\
\end{bmatrix}
\end{align*}
is nonsingular.
\end{coro}
\begin{proof}
From Proposition \ref {pro:specvsnon} one deduces the following formula
\begin{align}\label{eq:spectral non spectral}
\begin{bmatrix}
\boldsymbol{\mathcal J}_{ C^{[1]}_{n-N}}-\prodint{ P^{[1]}_{n-N}(x),(\xi)_x}\mathcal W\\
\vdots\\
\boldsymbol{\mathcal J}_{ C^{[1]}_{n-1}}-\prodint{ P^{[1]}_{n-1}(x),(\xi)_x}\mathcal W
\end{bmatrix}
=-\begin{bmatrix}
R_{n-N,0}& \dots& R_{n-N,N-1}\\\vdots& &\vdots\\
R_{n-1,0}&\dots &R_{n-1,N-1}
\end{bmatrix}
\mathcal B\mathcal Q.
\end{align}
Now,  Proposition \ref{pro:first_poised} and Lemma \ref{lemma:triple} lead to the result.
\end{proof}
 We stress at this point that \eqref{eq:spectral non spectral} connects the spectral and the nonspectral  methods. Moreover,  when we border with a further block row we  obtain
 \begin{align*}
\begin{bmatrix}
\boldsymbol{\mathcal J}_{ C^{[1]}_{n-N}}-\prodint{ P^{[1]}_{n-N}(x),(\xi)_x}\mathcal W\\
\vdots\\
\boldsymbol{\mathcal J}_{ C^{[1]}_{n}}-\prodint{ P^{[1]}_{n}(x),(\xi)_x}\mathcal W
\end{bmatrix}
 =-\begin{bmatrix}
 R_{n-N,0}&\dots & R_{n-N,N-1}\\\vdots& &\vdots\\
 R_{n,0}&\dots &R_{n,N-1}
 \end{bmatrix}
 \mathcal B\mathcal Q.
 \end{align*}
 \begin{coro}\label{corollary:non_singularity2}
If the leading coefficient $A_N$  is nonsingular and  $n< N$, then
\begin{align*}
\begin{bmatrix}
\boldsymbol{\mathcal J}_{ C^{[1]}_{0}}-\prodint{ P^{[1]}_{0}(x),(\xi)_x}\mathcal W\\
\vdots\\
\boldsymbol{\mathcal J}_{ C^{[1]}_{n-1}}-\prodint{ P^{[1]}_{n-1}(x),(\xi)_x}\mathcal W
\end{bmatrix}\mathcal R_n
\end{align*}
is nonsingular, and
\begin{align*}
\begin{bmatrix}
\boldsymbol{\mathcal J}_{ C^{[1]}_{0}}-\prodint{ P^{[1]}_{0}(x),(\xi)_x}\mathcal W\\
\vdots\\
\boldsymbol{\mathcal J}_{ C^{[1]}_{n-1}}-\prodint{ P^{[1]}_{n-1}(x),(\xi)_x}\mathcal W
\end{bmatrix}
\end{align*}
is full rank.
 \end{coro}

\begin{proof}
	 From Proposition \ref{pro:specvsnon} we know that when
 the leading coefficient $A_N$ is nonsingular	the following relation holds
 \begin{align*}
\big( \boldsymbol{\mathcal J}_{C^{[1]}}-\prodint{ P^{[1]}_{n-1}(x),(\xi)_x}\mathcal W
\big)\mathcal R=-R^{(N)},
 \end{align*}
where $\mathcal R=(Y,JY,\dots, J^{N-1}Y)$ is  a  Jordan triple  $(X,Y,J)$ of the perturbing polynomial $W(x)$. Thus, for $n<N$ we have
in terms of the truncation $\mathcal R_n=(Y,JY,\dots, J^{n-1}Y)$ , which as we have proved is a full rank matrix,
\begin{align*}
\begin{bmatrix}
\boldsymbol{\mathcal J}_{ C^{[1]}_{0}}-\prodint{ P^{[1]}_{0}(x),(\xi)_x}\\
\vdots\\
\boldsymbol{\mathcal J}_{ C^{[1]}_{n-1}}-\prodint{ P^{[1]}_{n-1}(x),(\xi)_x}
\end{bmatrix}\mathcal R_n
 =-\begin{bmatrix}
 R_{0,0}&\dots & R_{0,n-1}\\\vdots& &\vdots\\
 R_{n,0}&\dots &R_{n,n-1}
 \end{bmatrix}.
 \end{align*}
 Thus, as $\begin{vmatrix}
 R_{0,0}&\dots & R_{0,n-1}\\\vdots& &\vdots\\
 R_{n,0}&\dots &R_{n,n-1}
 \end{vmatrix}\neq 0$, we  deduce  the result.
\end{proof}

\subsection{Applications}
\subsubsection{Unimodular Christoffel perturbations and nonspectral techniques}\label{s:unimodular}
The spectral methods apply to those Geronimus transformations with a perturbing polynomial $W(y)$ having a nonsingular leading coefficient $A_N$.
This was also the case for the techniques developed in \cite{alvarez2015Christoffel} for matrix Christoffel transformations, where the perturbing polynomial had a nonsingular leading coefficient. However, we have shown that despite we can extend the use of the spectral techniques to the study of matrix Geronimus transformations, we also
have a  nonspectral approach applicable even for singular leading coefficients. For example, some cases that have appeared several times in the literature --see \cite{Dur6}-- are unimodular perturbations and,  consequently,  with $ W(y)$ having a singular leading coefficient. In this case, we have that $(W(y))^{-1}$ is a matrix polynomial, and we can consider the Geronimus transformation associated with the matrix polynomial $(W(y))^{-1}$ --as the spectrum is empty $\sigma(W(y))=\varnothing$,   no masses appear-- as a Christoffel transformation with perturbing matrix polynomial $W(y)$ of the original matrix of
 generalized kernels
\begin{align}\label{eq:unimodular}
\check u_{x,y}=u_{x,y} \big((W(y))^{-1}\big)^{-1}=u_{x,y} W(y).
\end{align}
We can apply Theorem \ref{theorem:nonspectral} with
\begin{align*}
R&=	\prodint{P^{[1]}(x),\chi(y)}_{u W}, &
	R_{n,l}&=\prodint{P^{[1]}_n(x),I_py^l}_{uW}\in\mathbb C^{p\times p}.
	\end{align*}
For example, when the matrix of generalized kernels is  a matrix of measures $\mu$, we can write
\begin{align*}
R_{n,l}&=\int P^{[1]}_n(x)\d\mu(x,y)W(y)y^l .
\end{align*}
Here $W(x)$ is a Christoffel perturbation and $\deg((W(x))^{-1})$ gives you the number of original orthogonal polynomials required for the Christoffel type formula.
	Theorem \ref{theorem:nonspectral}  can be nicely applied to get $\check P^{[1]}_n(x)$ and $\check H_n$.
However, it only gives Christoffel--Geronimus formulas for $\big(\check P^{[2]}_n(y)\big)^\top A_N$ and given that $A_N$ is singular, we only partially recover $\check P^{[2]}_n(y)$. This problem disappears whenever we
 have symmetric generalized kernels $u_{x,y}=(u_{y,x})^\top$, see Remark \ref{rem:symmetric}, as then  $ P_n^{[1]}(x)= P_n^{[2]}(x)=:P_n(x)$ and
biorthogonality collapses to orthogonality of $\{P_n(x)\}_{n=0}^\infty$. From \eqref{eq:unimodular}, we need to require
\begin{align*}
u_{x,y}W(y)= (W(x))^\top (u_{y,x})^\top,
\end{align*}
that when the initial matrix of kernels is itself symmetric $u_{x,y}= (u_{y,x})^\top$ reads
$u_{x,y}W(y)= (W(x))^\top u_{x,y}$.
Now, if we are dealing with Hankel matrices of generalized kernels $u_{x,y}=u_{x,x}$ we find
$u_{x,x,}W(x)= (W(x))^\top u_{x,x}$,
that for the scalar case reads $u_{x,x}=u_0I_p$ with $u_0$ a generalized function we need $W(x)$ to be a symmetric matrix polynomial.
For this scenario, if $\{p_n(x)\}_{n=0}^\infty$ denotes the set of monic orthogonal polynomials associated with $u_{0}$, we have
$R_{n,l}=\big\langle u_{0} , p_n(x)W(x)x^l\big\rangle$.

For example, if we take $p=2$, with the unimodular perturbation given by
\begin{align*}
W(x)=\begin{bmatrix}
(A_2)_{1,1}x^2+(A_1)_{1,1}x+(A_0)_{1,1} &(A_2)_{1,2}x^2+(A_1)_{1,2}x+(A_0)_{1,2}\\(A_2)_{1,2}x^2+(A_1)_{1,2}x+(A_0)_{1,2}&(A_2)_{2,2}x^2+(A_1)_{2,2}x+(A_0)_{2,2}
\end{bmatrix}
\end{align*}
we have, that the inverse  is the following matrix polynomial
\begin{align*}
(W(x))^{-1}=\frac{1}{\det W(x)}\begin{bmatrix}
(A_2)_{2,2}x^2+(A_1)_{2,2}x+(A_0)_{2,2} &-(A_2)_{1,2}x^2-(A_1)_{1,2}x-(A_0)_{1,2}\\-(A_2)_{1,2}x^2-(A_1)_{1,2}x-(A_0)_{1,2}&(A_2)_{1,1}x^2+(A_1)_{1,1}x+(A_0)_{1,1}
\end{bmatrix},
\end{align*}
where $\det W(x)$ is a constant,
and the inverse has also degree 2. Therefore, for $n\in\{2,3,\dots\}$,
we have the following expressions for the perturbed matrix orthogonal polynomials
\begin{align*}
\check P_n(x)&=\Theta_*\begin{bmatrix}
\big\langle u_{0} , p_{n-2}(x)x^k (A_2x^2+A_1x+A_0)\big\rangle & \big\langle u_{0} , p_{n-2}(x)x^l (A_2x^2+A_1x+A_0)\big\rangle &p_{n-2}(x) I_p\\
\big\langle u_{0} , p_{n-1}(x)x^k(A_2x^2+A_1x)\big\rangle & \big\langle u_{0} , p_{n-1}(x)x^l (A_2x^2+A_1x+A_0)\big\rangle & p_{n-1}(x) I_p\\
\big\langle u_{0} , p_n(x)x^k A_2x^2\big\rangle & \big\langle u_{0} , p_n(x)x^l (A_2x^2+A_1x)\big\rangle &p_{n}(x)I_p
\end{bmatrix},
\end{align*}
and the corresponding matrix norms or quasitau matrices are
{\small
	\begin{align*}
\check H_n&=\Theta_*\begin{bmatrix}
\big\langle u_{0} , p_{n-2}(x)x^k (A_2x^2+A_1x+A_0)\big\rangle & \big\langle u_{0} , p_{n-2}(x)x^l (A_2x^2+A_1x+A_0)\big\rangle &\big\langle u_{0} , p_{n-2}(x)x^n (A_2x^2+A_1x+A_0)\big\rangle\\
\big\langle u_{0} , p_{n-1}(x)x^k(A_2x^2+A_1x)\big\rangle & \big\langle u_{0} , p_{n-1}(x)x^l (A_2x^2+A_1x+A_0)\big\rangle & \big\langle u_{0} , p_{n-1}(x)x^n (A_2x^2+A_1x+A_0)\big\rangle\\
\big\langle u_{0} , p_n(x)x^k A_2x^2\big\rangle & \big\langle u_{0} , p_n(x)x^l (A_2x^2+A_1x)\big\rangle &\big\langle u_{0} , p_{n}(x)x^n (A_2x^2+A_1x+A_0)\big\rangle
\end{bmatrix}.
\end{align*}
}
Here the natural numbers $k$ and $l$ satisfy $0\leq k<l\leq n-1$ and are among those  (we know that they do exist) that fulfil
\begin{align*}
\det\begin{bmatrix}
\big\langle u_{0} , p_{n-2}(x)x^k (A_2x^2+A_1x+A_0)\big\rangle & \big\langle u_{0} , p_{n-2}(x)x^l (A_2x^2+A_1x+A_0)\big\rangle \\
\big\langle u_{0} , p_{n-1}(x)x^k(A_2x^2+A_1x)\big\rangle & \big\langle u_{0} , p_{n-1}(x)x^l (A_2x^2+A_1x+A_0)\big\rangle
\end{bmatrix}\neq 0.
\end{align*}

Observe  that the case of size $p=2$ unimodular matrix polynomials is particularly simple, because the degree of the perturbation and its inverse coincide. However, for bigger sizes this is not the case.  For a better understanding,
 let us recall that unimodular matrices always factorize in terms of elementary matrix polynomials and elementary matrices, which  are of the following form
\begin{enumerate}
	\item Elementary matrix polynomials: $e_{i,j}(x)=I_p+E_{i,j}p(x)$ with $i\neq j$ and $E_{i,j}$ the matrix with a 1 at the $(i,j)$ entry and zero elsewhere,
	and $p(x)\in\mathbb C[x]$.
\item Elementary matrices: \begin{enumerate}
 	\item  $I_p+(c-1)E_{i,i}$ with $c\in\mathbb C$.
	\item $\eta^{(i,j)}=I_p-E_{i,i}-E_{j,j}+E_{i,j}+E_{j,i}$: the identity matrix with the $i$-th and $j$-th rows interchanged.
\end{enumerate}
\end{enumerate}
The inverses of these matrices are elementary again
\begin{align*}
(e_{i,j}(x))^{-1}&=I_p-p(x)E_{i,j},\\
(I_p+(c-1)E_{i,i})^{-1}&=I_p+(c^{-1}-1)E_{i,i},\\
(\eta^{(i,j)})^{-1}&=\eta^{(i,j)},
\end{align*}
and the inverse of a general unimodular matrix polynomial can be computed immediately once its factorization in terms of elementary matrices is  given.
However, the degree of the matrix polynomial and its inverse requires a separate analysis.

If our perturbation $W(x)=I_p+p(x)E_{i,j}$ is an elementary matrix polynomial, with $\deg p(x)=N$, then we have that $(W(x))^{-1}=I_p-p(x)E_{i,j}$ and $\deg W(x)=\deg ((W(x))^{-1})=N$.  If we assume a departing matrix of generalized kernels $u_{x,y}$,  for $n\geq N$, the first family of perturbed polynomials will be
\begin{align*}
\check P^{[1]}_n(x)&=\Theta_*\begin{bmatrix}
\big\langle P^{[1]}_{n-N}(x),y^{k_1} (I_p+p(y)E_{i,j})\big\rangle_u & \dots& \big\langle P^{[1]}_{n-N}(x),y^{k_N} (I_p+p(y)E_{i,j})\big\rangle_u &P^{[1]}_{n-N}(x) \\
\vdots & & \vdots&\vdots\\
\big\langle P^{[1]}_n(x),y^{k_1} (I_p+p(y)E_{i,j})\big\rangle_u &\dots & \big\langle P^{[1]}_n(x),y^{k_N}(I_p+p(y)E_{i,j})\big\rangle_u &P^{[1]}_{n}(x)
\end{bmatrix}.
\end{align*}
Here, the  sequence of different integers $\{k_1,\dots,k_N\}\subset\{1,\dots,n-1\}$ is such that
\begin{align*}
\det\begin{bmatrix}
\big\langle P^{[1]}_{n-N}(x),y^{k_1} (I_p+p(y)E_{i,j})\big\rangle_u & \dots& \big\langle P^{[1]}_{n-N}(x),y^{k_N} (I_p+p(y)E_{i,j})\big\rangle_u \\
\vdots & & \vdots\\
\big\langle  P^{[1]}_{n-1}(x),y^{k_1} (I_p+p(y)E_{i,j})\big\rangle_u &\dots & \big\langle P^{[1]}_{n-1}(x),y^{k_N}(I_p+p(y)E_{i,j})\big\rangle_u
\end{bmatrix}\neq 0.
\end{align*}

A bit more complex situation appears when we have the product of different elementary matrix polynomials, for example
\begin{align*}
W(x)=\big(I_p+p^{(1)}_{i_1,j_1}(x)E_{i_1,j_1}\big)\big(I_p+p^{(2)}_{i_2,j_2}(x)E_{i_2,j_2}\big),
\end{align*}
which has two possible forms depending on whether  $j_1\neq i_2$ or $j_1= i_2$
\begin{align*}
W(x)=\begin{cases}
I_p+p^{(1)}_{i_1,j_1}(x)E_{i_1,j_1}+p^{(2)}_{i_2,j_2}(x)E_{i_2,j_2}, &j_1\neq i_2, \\
I_p+p^{(1)}_{i_1,j_1}(x)E_{i_1,j_1}+p^{(2)}_{j_2,j_2}(x)E_{j_2,j_2}+p^{(1)}_{i_1,j_1}(x)p^{(2)}_{j_1,j_2}(x)E_{i_1,j_2},&j_1= i_2,
\end{cases}
\end{align*}
so that
\begin{align*}
\deg (W(x))=\begin{cases}
(1-\delta_{i_1,i_2}\delta_{j_1,j_2})\max\big(\deg (p^{(1)}_{i_1,j_1}(x)),\deg(p^{(2)}_{i_2,j_2}(x))\big)
+\delta_{i_1,i_2}\delta_{j_1,j_2}\deg (p^{(1)}_{i_1,j_1}(x)+p^{(2)}_{i_2,j_2}) , &j_1\neq i_2, \\
\deg (p^{(1)}_{i_1,j_1}(x))+\deg(p^{(1)}_{j_1,j_2}(x)), &j_1= i_2.
\end{cases}
\end{align*}
For the inverse, we find
\begin{align*}
(W(x))^{-1}=\begin{cases}
I_p-p^{(1)}_{i_1,j_1}(x)E_{i_1,j_1}-p^{(2)}_{i_2,j_2}(x)E_{i_2,j_2}, &j_2\neq i_1, \\
I_p-p^{(1)}_{i_1,j_1}(x)E_{i_1,j_1}-p^{(2)}_{i_2,i_1}(x)E_{i_2,i_1}+p^{(1)}_{i_1,j_1}(x)p^{(2)}_{i_2,i_1}(x)E_{i_2,j_1}, &j_2= i_1,
\end{cases}
\end{align*}
and
\begin{align*}
\deg \big((W(x))^{-1}\big)=\begin{cases}
(1-\delta_{i_1,i_2}\delta_{j_1,j_2})\max(\deg (p^{(1)}_{i_1,j_1}(x)),\deg(p^{(2)}_{i_2,j_2}(x))+\delta_{i_1,i_2}\delta_{j_1,j_2}\deg (p^{(1)}_{i_1,j_1}(x)+p^{(2)}_{i_2,j_2}), &j_2\neq i_1, \\
\deg (p^{(1)}_{i_1,j_1}(x))+\deg(p^{(2)}_{i_2,i_1}(x)), &j_2= i_1.
\end{cases}
\end{align*}
Thus, if either $j_1\neq i_2$ and $ j_2\neq i_1$, or when $j_1= i_2$ and $ j_2= i_1$,
 the degrees $W(x)$ and $(W(x))^{-1}$ coincide, for $j_1= i_2$ and $ j_2\neq i_1$ we find
$\deg W(x)>\deg((W(x))^{-1})$ and when $j_1\neq  i_2$ and $ j_2= i_1$ we have $\deg W(x)<\deg((W(x))^{-1})$.
Consequently, the degrees of  unimodular matrix polynomials can be bigger than, equal to  or smaller than the degrees of  its   inverses.

We will be interested in unimodular perturbations $W(x)$ that factorize in terms of
$K$ elementary polynomial factors $\{e_{i_m,j_m}(x)\}_{m=1}^K$ and $L$ exchange factors $\{\eta^{(l_n,q_n)}\}_{n=1}^L$.
We will use the following notation for elementary polynomials and elementary matrices
\begin{align*}
(i,j)_{p_{i,j}(x)}&:=E_{i,j}p_{i,j}(x)   &    [l,q]&:=\eta_{l,q},
\end{align*}
suited to take products among them, according to the product table
\begin{align*}
(i_1,j_1)_{p_{i_1,j_1}}(i_2,j_2)_{p_{i_2,j_2}}&=\delta_{j_1,i_2}(i_1,j_2)_{p_{i_1,j_1}p_{i_2,j_2}}, \\
[l,q](i,j)_{p_{i,j}}&=(1-\delta_{l,i})(1-\delta_{q,i})(i,j)_{p_{i,j}}+\delta_{l,i}(q,j)_{p_{i,j}}+\delta_{q,i}(l,j)_{p_{i,j}},\\
(i,j)_{p_{i,j}}[l,q]&=(1-\delta_{l,j})(1-\delta_{q,j})(i,j)_{p_{i,j}}+\delta_{l,j}j(i,)_{p_{i,j}}+\delta_{q,j}(i,l)_{p_{i,j}}.
\end{align*}
Bearing this in mind,  we denote all the possible permutations of a vector with $K$ entries, having $i$ out of these
equal to 1 and the rest equal to zero, by $\sigma_{i}^K=\big\{{\sigma}_{i,j}^{K}\big\}_{j=1}^{|\sigma_{i}^K|}$ with
${\sigma}_{i,j}^{K}=\begin{pmatrix} ({\sigma}_{i,j}^{K})_1,
\dots, ({\sigma}_{i,j}^{K})_K \end{pmatrix}\in (\mathbb Z_2)^K$ where $({\sigma}_{i,j}^{K})_r \in \mathbb Z_2:=\{1,0\}$ and
$|\sigma_{i}^K|=\begin{pmatrix} K \\ i \end{pmatrix}$ we can rewrite a given unimodular perturbation as a sum. Actually,
any unimodular polynomial that factorizes in terms of $K$ elementary polynomials $e_{i,j}(x)$
and $L$ elementary matrices $\eta^{(l,q)}$, in a given order, can be expanded  into a sum of $2^K$ terms
\begin{align*}
W(x)&=e_{i_1,j_1}(x)\cdots e_{i_r,j_r}(x)\eta^{(l_1,q_1)} \cdots \eta^{(l_t,q_t)}
e_{i_{r+1},j_{r+1}}(x)\cdots \eta^{(l_L,q_L)} \cdots e_{i_K,j_K}(x)\\
&=\sum_{i=0}^{K}\sum_{j=1}^{|\sigma_{i}^K|}
(i_1,j_1)^{(\sigma_{i,j}^K)_{1}}_{p_{i_1,j_1}} \cdots (i_r,j_r)^{(\sigma_{i,j}^K)_{r}}_{p_{i_r,j_r}}
[l_1,q_1]\cdots[l_t,q_t](i_{r+1},j_{r+1})^{(\sigma_{i,j}^K)_{r+1}}_{p_{i_{r+1},j_{r+1}}}\dots [l_L,q_L]\dots
(i_K,j_K)^{(\sigma_{i,j}^K)_{K}}_{p_{i_K,j_K}},
\end{align*}
where  $(i,j)_{p_{i,j}}^0=\mathbb{I}_p$.
Notice that although in the factorization of $W$ we have assumed that it starts and ends with elementary polynomials, the result would still
be valid if it started and/or ended with an interchange elementary matrix $\eta$.
We notationally simplify these type of expressions by considering   the sequences of couples  of natural numbers
$\{i_1,j_1\}\,\{(i_2,j_2\}),\dots,\{i_k,j_k\}\big\}$,
where $\{n,m\}$ stands either for $(n,m)_{p_{m,n}}$ or $[m,n]$, and identifying  paths.
We say that two couples of naturals $\{k,l\}$ and $\{n,m\}$ are linked if $l=n$.
When we deal with a couple $[n,m]$ the order is not of the natural numbers is not relevant, for example $(k,l)$ and $[l,m]$ are linked as well
as $(k,l)$ and $[m,l]$ are linked.
A path of length $l$ is a subset of $I$ of the form
\begin{align*}
\big\{\{a_1,a_2\},\{a_2,a_3\}, \{a_3,a_4\},\dots,\{a_{l-1},a_{l}\},\{a_l,a_{l+1}\}\big\}_l.
\end{align*}
The order of the sequence is respected for the construction of each path. Thus,  the element  $(a_i,a_{i+1})$, as an element  of the sequence  $I$, is previous to the element $(a_{i+1}, a_{i+2})$ in the sequence. A path is proper if it does not belong to a longer path.
Out of the $2^K$ terms that appear only paths remain. In order to know
the degree of the unimodular polynomial one must check the factors of   the proper paths, and look for the maximum  degree involved in those factors .
For a better  understanding  let us work out  a couple of significant examples.  These examples deal with non symmetric matrices and, therefore, we have complete Christoffel type expressions for $\check P^{[1]}_n(x)$ and $\check H_n$, but also  the mentioned penalty for $P^{[2]}_n(x)$.
Firstly, let us consider a polynomial with $K=5$, $L=0$ and $p=6$,
\begin{align*}
W(x)&=e_{1,2}(x)e_{2,3}(x)e_{3,6}(x)e_{4,3}(x)e_{3,5}(x)
\end{align*}
in terms of sequences of couples the  paths for this unimodular polynomial has the following structure
\begin{gather*}
\{\varnothing\}_{i=5}, \\
\{\varnothing\}_{i=4},\\
\underline{\{(1,2),(2,3),(3,6)\}_{i=3}},\underline{\{(1,2),(2,3),(3,5)\}_{i=3}},\\
\underline{\{(4,3),(3,5)\}_{i=2}},\{(2,3),(3,5)\}_{i=2},\{(2,3),(3,6)\}_{i=2},\{(1,2)(2,3)\}_{i=2},\\ \{(1,2)\}_{i=1},\{(2,3)\}_{i=1},\{(3,6)\}_{i=1},\{(4,3)\}_{i=1},\{(3,5)\}_{i=1},\\
\{I_6\}_{i=0},
\end{gather*}
where $\{I_6\}_{i=0}$ indicates that the product not involving couples produces the identity matrix (in general will be a product of interchanging matrices) and we have underlined the proper paths.
Thus
\begin{align*}
W(x)&=e_{1,2}(x)e_{2,3}(x)e_{3,6}(x)e_{4,3}(x)e_{3,5}(x)\\
&=(1,6)_{p_{1,2}p_{2,3}p_{3,6}}+(1,5)_{p_{1,2}p_{2,3}p_{3,5}} +
(4,5)_{p_{4,3}p_{3,5}}+(2,5)_{p_{2,3}p_{3,5}}+(2,6)_{p_{2,3}p_{3,6}}+(1,3)_{p_{1,2}p_{2,3}} \\
&+(1,2)_{p_{1,2}}+(2,3)_{p_{2,3}}+(3,6)_{p_{3,6}}+(4,3)_{p_{4,3}}+(3,5)_{p_{3,5}} +I_5\\
&=\begin{bmatrix}
1 & p_{1,2}(x) & p_{1,2}(x)p_{2,3}(x)  & 0 &p_{1,2}(x)p_{2,3}(x)p_{3,5}(x) &p_{1,2}(x)p_{2,3}(x)p_{3,6}(x)\\
0  & 1& p_{2,3}(x) & 0 &p_{2,3}(x)p_{3,5}(x) &p_{2,3}(x)p_{3,6}(x)\\
0 & 0& 1& 0&p_{3,5}(x) & p_{3,6}(x)\\
0 &0 &p_{4,3}(x) & 1 &p_{4,3}(x) p_{3,5}(x)&0\\
0 & 0 & 0 & 0 &1&0\\
0 & 0 & 0 & 0 &0&1
\end{bmatrix}.
\end{align*}
Its inverse is
\begin{align*}
(W(x))^{-1}&=(e_{3,5}(x))^{-1}(e_{4,3}(x))^{-1}(e_{3,6}(x))^{-1}(x)(e_{2,3}(x))^{-1}(e_{1,2}(x))^{-1},
\end{align*}
and the paths are
\begin{gather*}
\{\varnothing \}_{i=5},\\
\{\varnothing \}_{i=4},\\
\{\varnothing\}_{i=3},\\
\underline{\{(4,3),(3,6)\}_{i=2}},\\
\underline{\{(3,5)\}_{i=1}}, \{(4,3)\}_{i=1},\{(3,6)\}_{i=1},\underline{\{(2,3)\}_{i=1}},\underline{\{(1,2)\}_{i=1}},\\
\{I_6\}_{i=0}.
\end{gather*}
Thus,
\begin{align*}
(W(x))^{-1}&=(4,6)_{p_{4,3}p_{3,6}}+
(3,5)_{-p_{3,5}}+(4,3)_{-p_{4,3}}+(3,6)_{-p_{3,6}}+(2,3)_{-p_{2,3}}+(1,2)_{-p_{1,2}} +
I_6\\&=\begin{bmatrix}
1 &- p_{1,2}(x) & 0 & 0 &0 &0\\
0  & 1& -p_{2,3}(x) & 0 &0 &0\\
0 & 0& 1& 0&-p_{3,5}(x) & -p_{3,6}(x)\\
0 &0 &-p_{4,3}(x) & 1 &0&p_{4,3}(x) p_{3,6}(x)\\
0 & 0 & 0 & 0 &1&0\\
0 & 0 & 0 & 0 &0&1
\end{bmatrix}.
\end{align*}
Then, looking at the proper paths, we find
{\small \begin{align*}
	\deg  W(x)&=\max\big(\deg p_{1,2}(x) +\deg p_{2,3}(x)+\deg p_{3,6}(x), \deg p_{1,2}(x) +\deg p_{2,3}(x)+\deg p_{3,5}(x),\deg p_{4,3}(x) +\deg p_{3,5}(x)\big),\\
\deg((W(x))^{-1})&=\max \big(
\deg  p_{1,2}(x),\deg  p_{2,3}(x),\deg  p_{3,6}(x)+\deg  p_{4,3}(x), \deg p_{3,5}(x)
\big).
\end{align*}}
For example, if we assume that
\begin{align*}
\deg p_{1,2}(x)&=2,  &\deg p_{2,3}(x)&=1, & \deg p_{3,6}(x)&=2, & \deg p_{4,3}(x)&=1, & \deg p_{3,5}(x)&=3,
\end{align*}
we get for the corresponding unimodular matrix polynomial and its inverse
\begin{align*}
\deg (W(x))&=6, & \deg\big( (W(x))^{-1}\big)&=3,
\end{align*}
so that, for example, the first family of perturbed biorthogonal polynomials, for $n\geq 3$ is
\begin{align}\label{eq:example}
\check P^{[1]}_n(x)&=\Theta_*\begin{bmatrix}
\big\langle P^{[1]}_{n-3}(x),y^{k_1} W(y)\big\rangle_u & \big\langle P^{[1]}_{n-3}(x),x^{k_2} W(y)\big\rangle_u & \big\langle P^{[1]}_{n-3}(x),y^{k_3} W(y)\big\rangle_u &P^{[1]}_{n-3}(x) \\
\big\langle P^{[1]}_{n-2}(x),y^{k_1}W(y)\big\rangle_u & \big\langle P^{[1]}_{n-2}(x),y^{k_2}W(y)\big\rangle_u& \big\langle P^{[1]}_{n-2}(x),y^{k_3}W(y)\big\rangle_u &P^{[1]}_{n-2}(x)\\
\big\langle P^{[1]}_{n-1}(x),y^{k_1}W(y)\big\rangle_u & \big\langle P^{[1]}_{n-1}(x),y^{k_2}W(y)\big\rangle_u& \big\langle P^{[1]}_{n-1}(x),y^{k_3}W(y)\big\rangle_u &P^{[1]}_{n-1}(x)\\
\big\langle P^{[1]}_n(x),y^{k_1} W(y)\big\rangle_u &\big\langle P^{[1]}_{n}(x),y^{k_2}W(y)\big\rangle_u & \big\langle P^{[1]}_n(x),y^{k_3}W(y)\big\rangle_u &P^{[1]}_{n}(x)
\end{bmatrix}.
\end{align}
Here, the  sequence of different integers $\{k_1,k_2,k_3\}\subset\{1,\dots,n-1\}$ is such that
\begin{align*}
\det\begin{bmatrix}
\big\langle P^{[1]}_{n-3}(x),y^{k_1} W(y)\big\rangle_u & \big\langle P^{[1]}_{n-3}(x),y^{k_2} W(y)\big\rangle_u & \big\langle P^{[1]}_{n-3}(x),y^{k_3} W(y)\big\rangle_u \\
\big\langle P^{[1]}_{n-2}(x),y^{k_1}W(y)\big\rangle_u & \big\langle P^{[1]}_{n-2}(x),y^{k_2}W(y)\big\rangle_u& \big\langle P^{[1]}_{n-2}(x),y^{k_3}W(y)\big\rangle_u \\
\big\langle P^{[1]}_{n-1}(x),y^{k_1}W(y)\big\rangle_u & \big\langle P^{[1]}_{n-1}(x),y^{k_2}W(y)\big\rangle_u& \big\langle P^{[1]}_{n-1}(x),y^{k_3}W(y)\big\rangle_u
\end{bmatrix}\neq 0.
\end{align*}

Let us now work out a polynomial with $K=L=4$ and $p=5$.
The unimodular matrix polynomial we consider is
\begin{align*}
W(x)&=e_{2,1}(x)\eta^{(1,4)}\eta^{(5,4)}e_{5,1}(x)\eta^{(3,2)}e_{2,3}(x)\eta^{(3,1)}e_{1,5}(x).
\end{align*}
The paths are
{\small\begin{gather*}
\{\varnothing\}_{i=4},\\
\{\varnothing \}_{i=3},\\
\underline{\{(2,1),[1,4],[5,4],(5,1),[3,2],[3,1]\}_{i=2}},\underline{\{[1,4],[5,4],[3,2],(2,3),[3,1],(1,5)\}_{i=2}}, \\
\{(2,1),[1,4],[5,4],[3,2],[3,1]\}_{i=1},\{[1,4],[5,4],(5,1),[3,2],[3,1]\}_{i=1},\{[1,4],[5,4],[3,2],(2,3),[3,1]\}_{i=1},\{[1,4],[5,4],[3,2],[3,1],(1,5)\}_{i=1},\\
\{[1,4],[5,4],[3,2],[3,1]\}_{i=0},
\end{gather*}}
so that
\begin{align}\label{eq:example2}
W(x)
&= (2,3)_{p_{2,1}p_{5,1}}+(3,5)_{p_{2,3}p_{1,5}}+
(2,5)_{p_{2,1}}+(1,3)_{p_{5,1}}+(3,1)_{p_{2,3}}+(2,5)_{p_{1,5}}+[1,4][5,4][3,2][3,1]\\
&=\begin{bmatrix}
0& 0 & p_{5,1}(x)  &0 &1\\
1  & 0& p_{2,1}(x)p_{5,1}(x) & 0 &p_{2,1}(x)+p_{1,5}(x)\\
p_{2,3}(x) & 1& 0& 0&p_{2,3}(x)  p_{1,5}(x)\\
0 & 0 & 1 & 0 &0\\
0 & 0 &  0  &1&0
\end{bmatrix}.
\end{align}
The inverse matrix is
\begin{align*}
(W(x))^{-1}&=(e_{1,5}(x))^{-1}\eta^{(3,1)}(e_{2,3}(x))^{-1}\eta^{(3,2)}(e_{5,1}(x))^{-1}\eta^{(5,4)}\eta^{(1,4)}(e_{2,1}(x))^{-1},
\end{align*}
with paths given by
{\small \begin{gather*}
\{\varnothing\}_{i=4},\\
\{\varnothing\}_{i=3},\\
\underline{\{(1,5),[3,1],[3,2],(5,1),[5,4],[1,4]\}_{i=2}},\underline{\{[3,1],(2,3),[3,2],[5,4],[1,4],(2,1)\}_{i=2}},\\
\{[3,1],[3,2],[5,4],[1,4],(2,1)\}_{i=1},\{[3,1],[3,2],(5,1),[5,4],[1,4])\}_{i=1},\{[3,1],(2,3),[3,2],[5,4],[1,4])\}_{i=1},
\{(1,5),[3,1],[3,2],[5,4],[1,4]\}_{i=1},\\
\{[3,1],[3,2],[5,4],[1,4]\}_{i=0},
\end{gather*}}
and, consequently,
\begin{align*}
(W(x))^{-1}&=(1,4)_{p_{1,5}p_{5,1}}+(2,1)_{p_{2,3}p_{2,1}} +
(1,1)_{-p_{2,1}}+(5,4)_{-p_{5,1}}+(2,2)_{-p_{2,3}}+(1,1)_{-p_{1,5}}
+[3,1][3,2][5,4][1,4]\\
&=\begin{bmatrix}
-p_{2,1}(x)-p_{1,5}(x)&1&0&p_{1,5}(x)p_{5,1}(x)&0\\
p_{2,3}(x)p_{2,1}(x)&-p_{2,3}(x)&1&0&0\\
0&0&0&1&0\\
0&0&0&0&1\\
1&0&0&-p_{5,1}(x)&0
\end{bmatrix}.
\end{align*}
Proper paths, which we have  underlined, give the degrees of the polynomials
 \begin{align*}
	\deg  W(x)&=\max\big(\deg p_{2,1}(x) +\deg p_{5,1}(x), \deg p_{1,5}(x) +\deg p_{2,3}(x)\big),\\
	\deg((W(x))^{-1})&=\max \big(
	\deg  p_{1,5}(x)+\deg  p_{5,1}(x),\deg  p_{2,3}(x)+\deg  p_{2,1}(x)
	\big).
	\end{align*}
	For example, if we assume that
	\begin{align*}
	\deg p_{2,1}(x)&=2,  &\deg p_{5,1}(x)&=1, & \deg p_{1,5}(x)&=2, & \deg p_{2,3}(x)&=1,
	\end{align*}
we find $\deg  W(x)=\deg((W(x))^{-1})=3$ and formula \eqref{eq:example} is applicable for $W(x)$ as given in \eqref{eq:example2}.

If we seek for  symmetric unimodular polynomials of the form
\begin{align*}
W(x)=V(x)\big(V(x)\big)^\top,
\end{align*}
where $V(x)$ is a unimodular matrix polynomial. For example, we put $p=4$, and consider
\begin{align*}
V(x)=\begin{bmatrix}
1 & p_{1,2}(x)p_{3,2}(x) & p_{1,2}(x)& 0\\
0 & p_{3,2} (x) & 1 & 0\\
0 & 1& 0 &0\\
0 & 0 & p_{4,3}(x)& 1
\end{bmatrix},
\end{align*}
in such a way  the perturbing symmetric unimodular  matrix  polynomial is
\begin{align*}
W(x)=\begin{bmatrix}
1+(p_{1,2}(x))^2(p_{3,2}(x))^2 &  p_{1,2}(x)(p_{3,2}(x))^2+p_{1,2}(x) & p_{1,2}(x)p_{3,2}(x) & p_{1,2}(x)p_{4,3}(x)\\
p_{1,2}(x)(p_{3,2}(x))^2+p_{1,2}(x) & 1+(p_{3,2}(x))^2 & p_{3,2} (x) & p_{4,3} (x)\\
p_{1,2}(x)p_{3,2}(x) & p_{3,2} (x) & 1 &0\\
p_{1,2}(x)p_{4,3}(x)&  p_{4,3} (x) &0& 1+(p_{4,3}(x))^2
\end{bmatrix}.
\end{align*}
Let us assume that
\begin{align*}
\deg p_{1,2}(x)&=3, &\deg p_{3,2}(x)&=1, & \deg p_{4,3}(x)&=1,
\end{align*}
then
\begin{align*}
\deg W(x) &= 8, & \deg\big(
(W(x))^{-1}\big)&=4.
\end{align*}
Now, we take a scalar matrix of linear functionals $u=u_{0} I_p$, with $u_{0}\in\big(\R[x]\big)'$  positive definite, and assume that the polynomials $p_{1,2}(x), p_{2,3}(x),p_{3,4}(x)\in\mathbb R[x]$. Then, we obtain matrix orthogonal polynomials $\{P_n(x)\}_{n=0}^\infty$ for the matrix of linear functionals $W(x)u_0$, which in terms of the sequence of scalar orthogonal polynomials $\{p_n(x)\}_{n=0}^\infty$ of the linear functional $u_0$ are, for $n\geq 4$
{\small
	\begin{align*}%\label{eq:example}
P_n(x)=\Theta_*\begin{bmatrix}
\big\langle u_{0},p_{n-4}(x)x^{k_1} W(x)\big\rangle & \big\langle u_{0},p_{n-4}(x)x^{k_2} W(x)\big\rangle & \big\langle u_{0},p_{n-4}(x)x^{k_3} W(x)\big\rangle  &\big\langle u_{0},p_{n-4}(x)x^{k_4} W(x)\big\rangle
&p_{n-4}(x) I_p\\
\big\langle u_{0},p_{n-3}(x)x^{k_1} W(x)\big\rangle & \big\langle u_{0},p_{n-3}(x)x^{k_2} W(x)\big\rangle & \big\langle u_{0},p_{n-3}(x)x^{k_3} W(x)\big\rangle  &\big\langle u_{0},p_{n-3}(x)x^{k_4} W(x)\big\rangle
&p_{n-3}(x) I_p\\
\big\langle u_{0},p_{n-2}(x)x^{k_1} W(x)\big\rangle & \big\langle u_{0},p_{n-2}(x)x^{k_2} W(x)\big\rangle & \big\langle u_{0},p_{n-2}(x)x^{k_3} W(x)\big\rangle  &\big\langle u_{0},p_{n-2}(x)x^{k_4} W(x)\big\rangle
&p_{n-2}(x) I_p\\
\big\langle u_{0},p_{n-1}(x)x^{k_1} W(x)\big\rangle & \big\langle u_{0},p_{n-1}(x)x^{k_2} W(x)\big\rangle & \big\langle u_{0},p_{n-1}(x)x^{k_3} W(x)\big\rangle  &\big\langle u_{0},p_{n-1}(x)x^{k_4} W(x)\big\rangle
&p_{n-1}(x) I_p\\
\big\langle u_{0},p_{n}(x)x^{k_1} W(x)\big\rangle & \big\langle u_{0},p_{n}(x)x^{k_2} W(x)\big\rangle & \big\langle u_{0},p_{n}(x)x^{k_3} W(x)\big\rangle  &\big\langle u_{0},p_{n}(x)x^{k_4} W(x)\big\rangle
&p_{n}(x) I_p
\end{bmatrix}.
\end{align*}}
The set $\{k_1,k_2,k_3,k_4\}\subset\{1,\dots,n-1\}$ is such that
\begin{align*}
\det\begin{bmatrix}
\big\langle u_{0},p_{n-4}(x)x^{k_1} W(x)\big\rangle & \big\langle u_{0},p_{n-4}(x)x^{k_2} W(x)\big\rangle & \big\langle u_{0},p_{n-4}(x)x^{k_3} W(x)\big\rangle  &\big\langle u_{0},p_{n-4}(x)x^{k_4} W(x)\big\rangle
\\
\big\langle u_{0},p_{n-3}(x)x^{k_1} W(x)\big\rangle & \big\langle u_{0},p_{n-3}(x)x^{k_2} W(x)\big\rangle & \big\langle u_{0},p_{n-3}(x)x^{k_3} W(x)\big\rangle  &\big\langle u_{0},p_{n-3}(x)x^{k_4} W(x)\big\rangle
\\
\big\langle u_{0},p_{n-2}(x)x^{k_1} W(x)\big\rangle & \big\langle u_{0},p_{n-2}(x)x^{k_2} W(x)\big\rangle & \big\langle u_{0},p_{n-2}(x)x^{k_3} W(x)\big\rangle  &\big\langle u_{0},p_{n-2}(x)x^{k_4} W(x)\big\rangle
\\
\big\langle u_{0},p_{n-1}(x)x^{k_1} W(x)\big\rangle & \big\langle u_{0},p_{n-1}(x)x^{k_2} W(x)\big\rangle & \big\langle u_{0},p_{n-1}(x)x^{k_3} W(x)\big\rangle  &\big\langle u_{0},p_{n-1}(x)x^{k_4} W(x)\big\rangle
\end{bmatrix}\neq 0.
\end{align*}

\subsubsection{Degree one matrix Geronimus transformations}\label{S:degreeone}

We consider a degree one perturbing polynomial of the form
\begin{align*}
W(x)=xI_p-A,
\end{align*}
and assume, for the sake of simplicity, that all $\xi$ are taken zero, i.e., there  are no masses.
Observe that in this case a Jordan pair $(X,J)$ is such that $A=XJX^{-1}$, and 
 Lemma \ref{lemma:pair}  implies that the  root spectral jet of a polynomial $P(x)=\sum_kP_kx^k\in\mathbb C^{p\times p}[x]$
is $\boldsymbol{\mathcal J}_P=P(A)X$, where we understand a right evaluation, i.e., $P(A):=\sum_{k}P_k A^k$.  An similar argument, for $\sigma(A)\cap \operatorname{supp}_y(u)=\varnothing$,  yields
\begin{align*}
\boldsymbol{\mathcal J}_{C^{[1]}_n}=\left\langle P^{[1]}(x), (A-I_py)^{-1}X\right\rangle_u,
\end{align*}
expressed in terms of the resolvent $(A-I_py)^{-1}$ of $A$.  Formally, it can be  written
\begin{align*}
\boldsymbol{\mathcal J}_{C^{[1]}_n}=C^{[1]}_n(A)X,
\end{align*}
where  we again understand a right evaluation in the Taylor series of the Cauchy transform.
Moreover, we  also need the root spectral jet of the mixed Christoffel--Darboux kernel
\begin{align*}
\boldsymbol{\mathcal J}_{K_{n-1}^{(pc)}}(y)&=\sum_{k=0}^{n-1}
\big(P^{[2]}_k(y)\big)^\top
\big(H_k\big)^{-1}C^{[1]}_k(A)X
\\&=:K_{n-1}^{(pc)}(A,y)X,
\end{align*}
that for a Hankel generalized  kernel $u_{x,y}$,
using  the Christoffel--Darboux formula for mixed kernels, reads
\begin{align*}
\boldsymbol{\mathcal J}_{K_{n-1}^{(pc)}}(y)
&=\Big(\big(P^{[2]}_{n-1}(y)\big)^\top
\big(H_{n-1}\big)^{-1}C^{[1]}_{n}(A)-\big(P^{[2]}_{n}(y)\big)^\top
\big(H_{n-1}\big)^{-1}C^{[1]}_{n-1}(A)+I_p\Big)(A-I_py)^{-1}X.
\end{align*}
We also have $\mathcal V(x,y)=I_p$ so that $\boldsymbol{\mathcal J}_{\mathcal V}=X$.

Thus, for $n\geq 1$ we have
\begin{align*}
\check P^{[1]}_n(x)&=\Theta_*\begin{bmatrix}
C^{[1]}_{n-1}(A)X & P^{[1]}_{n-1} (x)\\
C^{[1]}_{n}(A)X & P^{[1]}_{n} (x)
\end{bmatrix}\\&=P^{[1]}_n(x)-C^{[1]}_n(A)\big(C^{[1]}_{n-1}(A)\big)^{-1} P^{[1]}_{n-1}(x),\\
\check H_n&=\Theta_*\begin{bmatrix}
C^{[1]}_{n-1}(A)X & H_{n-1}\\
C^{[1]}_{n}(A)X & 0_p
\end{bmatrix}\\&=-C^{[1]}_n(A)\big(C^{[1]}_{n-1}(A)\big)^{-1} H_{n-1},\\
\big(\check P^{[2]}_n(y)\big)^\top&=\Theta_*\begin{bmatrix}
C^{[1]}_{n-1}(A)X & H_{n-1}\\
(I_py-A)\big(K_{n-1}^{(pc)}(A,y)+I_p\big)X & 0_p
\end{bmatrix}\\&=
\big((I_py-A)K_{n-1}^{(pc)}(A,y)+I_p\big)\big(C^{[1]}_{n-1}(A)\big)^{-1}H_{n-1}.
 \end{align*}
 For a Hankel  matrix of bivariate generalized functionals, i.e., with  a Hankel  Gram matrix  so that the Christoffel--Darboux formula holds, we have
\begin{align*}
\big(\check P^{[2]}_n(y)\big)^\top&=
-(I_py-A)\Big(\big(P^{[2]}_{n-1}(y)\big)^\top
\big(H_{n-1}\big)^{-1}C^{[1]}_{n}(A)-\big(P^{[2]}_{n}(y)\big)^\top
\big(H_{n-1}\big)^{-1}C^{[1]}_{n-1}(A)\Big)(I_py-A)^{-1} H_{n-1}.
\end{align*}

\section{Matrix Geronimus--Uvarov transformations}\label{sGU}
 In 1969, Uvarov in \S 1 of \cite{Uva}   considered for the first time a massless Geronimus--Uvarov transformation for scalar orthogonal polynomials --called linear spectral transformations in \cite{Zhe}-- finding (as Uvarov called them) general  Christoffel formulas for these transformations. The results in \S 1 of \cite{Uva} are a detailed version of the results presented in \cite{Uva0} in 1959 in Russian. See \cite{Zhe,Gaut2} for more details. We now consider a transformation generated by two  matrix polynomials $W_C(x),W_G(x)\in\mathbb C^{p\times p}[x]$, that we call  Christoffel and Geronimus polynomials, respectively. This can be understood as a composition of a Geronimus transformation as treated  in the previous section and a Christoffel transformation as discussed in \cite{alvarez2015Christoffel}.
\begin{defi}\label{def:Geronimus-Uvarov}
	Given two   matrix polynomials $W_C(x),W_G(y)$ of degrees $N_C,N_G$,  and a matrix of generalized kernels $u_{x,y}\in (\mathcal O_c')^{p\times p}$ such that $\sigma(W_G(y)))\cap \operatorname{supp}_y(u)=\varnothing$, a matrix Geronimus--Uvarov transformation $\hat u_{x,y}$ of  $ u_{x,y}$ is a matrix of generalized kernels such that
	\begin{align}
	\hat u_{x,y} W_G(y)=W_C(x)u_{x,y}.
	\end{align}
\end{defi}

\begin{pro}\label{pro:stringLST}
	The perturbed Gram  matrix $\hat G:=\langle \chi(x),\chi(y)\rangle_{\hat u}$ and the original one $G$ satisfy
	\begin{align*}
	\hat  G W_G(\Lambda^\top)=W_C(\Lambda)G.
	\end{align*}
The sequilinear forms are related by
\begin{align*}
\prodint{P(x),  Q(y)\big(W_G(y)\big)^\top}_{\hat u}=\prodint{ P(x)W_C(x), Q(y)}_u.
\end{align*}
\end{pro}

As we did for the Geronimus transformation, we   assume that the  perturbed moment matrix
admits a Gauss--Borel   factorization
$\hat  G=\hat  S_1^{-1} \hat  H (\hat  S_2)^{-\top}$,
where $\hat  S_1,\hat  S_2$ are lower unitriangular block matrices and $\hat  H$ is a diagonal block matrix
\begin{align*}
\hat S_i&=\begin{bmatrix}
I_p&0_p&0_p&\dots\\
(\hat  S_i)_{1,0}& I_p&0_p&\cdots\\
(\hat  S_i)_{2,0}& (\hat  S_i)_{2,1}&I_p&\ddots\\
&&\ddots&\ddots
\end{bmatrix}, & i&=1,2,&
\hat  H&=\diag (\hat  H_0,\hat  H_1,\hat  H_2,\dots).
\end{align*}
Consequently, the Geronimus--Uvarov transformation provides a new family of matrix biorthogonal polynomials
\begin{align*}
\hat  P^{[1]}(x)&=\hat  S_1\chi(x), & \hat  P^{[2]}(y)&=\hat  S_2\chi(y),
\end{align*}
with respect to the perturbed sesquilinear form $\prodint{\cdot,\cdot}_{\hat u}$.

\subsection{The resolvent and connection formulas for the matrix Geronimus--Uvarov transformation}
\begin{defi}
	The resolvent matrix is given  by
	\begin{align}\label{eq:def_OmegaLST}
	\omega:=\hat S_1 W_C(\Lambda)(S_1)^{-1}.
	\end{align}
\end{defi}

\begin{pro}\label{conexwLST}
	\begin{enumerate}
		\item	The resolvent matrix can be also expressed as
		\begin{align}\label{eq:resolvent_alternativeLST}
		\omega =	\hat  H \big(\hat  S_2\big)^{-\top} W_G(\Lambda^\top)\big( S_2\big)^{\top}	H^{-1}.
		\end{align}
		\item
		The resolvent matrix is a  block banded matrix  ---with only the first $N_G$ block subdiagonals the main diagonal and the $N_C$ block superdiagonals  possibly not zero, i.e.,
		\begin{align*}
		\omega=\begin{bmatrix}
	\omega_{0,0} &    \omega_{0,1}  &\dots     &\dots   &\omega_{0,N_C-1}&   I_p   & 0_p&  \dots             &&         \\
		\omega_{1,0} &\omega_{1,1}&\dots&  \dots&   \omega_{1,N_C-1}&    \omega_{1,N_C}         &   I_p &\ddots        \\
		\vdots       &  \vdots   &         &&&      \ddots       &  &      \ddots       \\
		\omega_{N_G,0} &\omega_{N_G,1}  & &&&&  \omega_{N_G,N_C+N_G-1}   & I_p  & 0_p         &  \\
		0_p&\omega_{N_G+1,1}&& &&& \omega_{N_G+1,N_G+N_C-1}&\omega_{N_G+1,N_G+N_C}&I_p& \ddots\\
		\vdots&       & \ddots      &&&&\ddots&      \ddots         &\ddots  &   \ddots
		\end{bmatrix}.
		\end{align*}	
		\item The following connection formulas are satisfied
		\begin{align}
		\hat P^{[1]}(x)W_C(x) &=\omega P^{[1]}(x),\label{conex2LST}\\
		\big(\hat  P^{[2]} (y) \big)^\top \hat H^{-1}\omega &=W_G(y)\big(P^{[2]}(y)\big)^\top H^{-1}.\label{conex3LST}
		\end{align}
		\item For the last resolvent subdiagonal we have
		\begin{align}\label{eq:omegaALST}
		\omega_{N_G+k,k}=\hat H_{N_G+k}A_{G,N_G}(H_k)^{-1},
		\end{align}
		where $A_{G,N_G}$ is the leading coefficient of the perturbing matrix polynomial $W_G(x)$.
	\end{enumerate}
\end{pro}

\begin{proof}
	\begin{enumerate}
		\item 	  Proposition \ref{pro:stringLST} and the Gauss--Borel  factorizations of the Gram matrices  $G$ and $\hat G$ lead to
		\begin{align*}
	W_C(\Lambda)	\big( S_1\big)^{-1}  H \big( S_2\big)^{-\top}=\big(\hat  S_1\big)^{-1} \hat  H \big(\hat  S_2\big)^{-\top} W_G(\Lambda^\top),
		\end{align*}
		so that
		\begin{align*}
		\hat  S_1 W_C(\Lambda)\big( S_1\big)^{-1}  H =	\hat  H \big(\hat  S_2\big)^{-\top} W_G(\Lambda^\top)\big( S_2\big)^{\top}	,
		\end{align*}
		and the result follows. Propositions \ref{pro:associativity} and \ref{pro:Gauss--Borel factorization and associativity} ensure all the mentioned steps.\footnote{	A	detailed version is
			\begin{align*}
		&&	W_C(\Lambda)\Big(	\big( S_1\big)^{-1}  H \big( S_2\big)^{-\top}\Big)&=\Big(\big(\hat  S_1\big)^{-1} \hat  H \big(\hat  S_2\big)^{-\top} \Big)W_G(\Lambda^\top) & &\overset{\text P\ref{pro:Gauss--Borel factorization and associativity}}{\Rightarrow }&
			\Big(	W_C(\Lambda)	\big( S_1\big)^{-1}\Big)  H \big( S_2\big)^{-\top}&=\big(\hat  S_1\big)^{-1} \Big(\hat  H \big(\hat  S_2\big)^{-\top} W_G(\Lambda^\top)\Big)\\ \
			&\Rightarrow &
			\hat	S_1\Big(	\Big(	W_C(\Lambda)	\big( S_1\big)^{-1}\Big)  H \big( S_2\big)^{-\top}\Big)&=\hat S_1\Big(\big(\hat  S_1\big)^{-1} \Big(\hat  H \big(\hat  S_2\big)^{-\top} W_G(\Lambda^\top)\Big)\Big)& &\overset{\text{P\ref{pro:associativity}}}{\Rightarrow} &
			\Big(	\hat	S_1	\Big(	W_C(\Lambda)	\big( S_1\big)^{-1}\Big)  H \big( S_2\big)^{-\top}&=\Big(\hat  H \big(\hat  S_2\big)^{-\top} W_G(\Lambda^\top)\Big)\\&\Rightarrow &
			\Big(	\Big(	\hat	S_1	\Big(	W_C(\Lambda)	\big( S_1\big)^{-1}\Big)  H \big( S_2\big)^{-\top}\Big)(S_2)^\top&=\Big(\hat  H \big(\hat  S_2\big)^{-\top} W_G(\Lambda^\top)\Big)(S_2)^\top & &\overset{\text{P\ref{pro:associativity}}}{\Rightarrow}  &
			\Big(	\hat	S_1	\Big(	W_C(\Lambda)	\big( S_1\big)^{-1}\Big)  H&=\Big(\hat  H \big(\hat  S_2\big)^{-\top} W_G(\Lambda^\top)\Big)(S_2)^\top.
			\end{align*}
			By P\ref{pro:associativity} or P\ref{pro:Gauss--Borel factorization and associativity} we indicated that is due to the corresponding Propositions.}
		\item From its definition, the resolvent matrix is a lower generalized Hessenberg block  matrix, with $N_C$ nonzero superdiagonals. However, from \eqref{eq:resolvent_alternativeLST} we deduce that it is an upper generalized block Hessenberg matrix with $N_G$ nonzero subdiagonals. As a conclusion, we get the band structure. Propositions \ref{pro:associativity} and \ref{pro:Gauss--Borel factorization and associativity} ensure all the mentioned steps.
		\item From the resolvent's definition we deduce \eqref{conex2LST}, and  \eqref{eq:resolvent_alternativeLST} gives
		 \eqref{conex3LST}.
		\item Direct observation from \eqref{eq:resolvent_alternativeLST}.
	\end{enumerate}
\end{proof}

%Connection formulas \eqref{conex2LST} and \eqref{conex3LST} can be written as
%\begin{align}\label{conex1LST}
%\check P^{[1]}_{n}(x)W_C(x)&=P^{[1]}_{n+N_C}(x)+\sum_{k=n-N_G}^{n+N_C-1}\omega_{n,k}P_k^{[1]}(x),\\\label{conex3'LST}
%W_G(x)\big(P_n^{[2]}(x)\big)^\top(H_n)^{-1}&=\sum_{k=n-N_C}^{n-N_C+N_G}\big(\hat  P^{[2]}_k(x)\big)^\top (\hat H_k)^{-1}\omega_{k,n}.
%%\big(H^\top\omega^\top \check H^{-\top} \big)_{n,k} \\&=\sum_{k=n}^{n+N}(H_n)^\top\big(\omega_{k,n}\big)^\top (\check H_k)^{-\top} \check P^{[2]}_k (x).\notag
%\end{align}

\begin{pro}
	The matrix Geronimus--Uvarov transformation of the second kind functions satisfy
	\begin{gather}
	\label{eq:conexionC1LST}
	\hat C^{[1]}(x)W_G(x)-\begin{bmatrix}
	\big(\hat H\big(\hat S_2\big)^{-\top}\big)_{[N_G]}	\mathcal B_G 	(\chi(x))_{[N_G]}
	\\
	0_{p}\\
	\vdots
	\end{bmatrix}=	\omega C^{[1]}(x),\\
	\label{eq:conexionC2LST}
\big(	\hat C^{[2]}(x)\big)^\top\hat  H^{-1}\omega  = W_C(x)\big(C^{[2]}(x)\big)^\top H^{-1}-\begin{bmatrix}
(\chi(x))^\top_{[N_C]}(\mathcal B_C) ^\top\big( S_1\big)^{-1}_{[N_C]}, 0_p,\dots	
\end{bmatrix}.
	\end{gather}
\end{pro}
\begin{proof}
	We proceed with the proof of \eqref{eq:conexionC1LST}
	\begin{align*}
	\hat C^{[1]}(z)W_G(z)-\omega C^{[1]}(z)&= \prodint{P^{[1]}(x),\frac{I_p}{z-y}}_{\hat u}W_G(z)-\prodint{\hat P^{[1]}(x)W_C(x),\frac{I_p}{z-y}}_{ u}\\
	&= \prodint{P^{[1]}(x),\frac{W_G(z)-W_G(y)}{z-y}}_{\hat u}
	\end{align*}
	and the results follows as in he Geronimus case.

For \eqref{eq:conexionC2LST} we have
	\begin{align*}
	\big(	\hat C^{[2]}(z)\big)^\top\hat  H^{-1}\omega  - W_C(z)\big(C^{[2]}(z)\big)^\top H^{-1}&=
	\prodint{\frac{I_p}{z-x},\hat P^{[2]}(y)}_{\hat u}\hat H^{-1}\omega-W_C(z)\prodint{\frac{I_p}{z-x},P^{[2]}(y)}_u\\
	&=
	\prodint{\frac{I_p}{z-x},\big(\hat H^{-1}\omega\big)^\top\hat P^{[2]}(y)}_{\hat u}-\prodint{\frac{W_C(z)}{z-x},P^{[2]}(y)}_u\\
	&=
	\prodint{\frac{I_p}{z-x},H^{-\top}P^{[2]}(y)(W_G(y))^\top}_{\hat u}-\prodint{\frac{W_C(z)}{z-x},P^{[2]}(y)}_u
	\\
	&=-\prodint{\frac{W_C(x)-W_C(z)}{x-z},P^{[2]}(y)}_u.
	\end{align*}
	Again we using the ideas of the proof for the Geronimus case the result follows.
\end{proof}

\begin{rem}\label{rem}
	We can understand the matrix Geronimus--Uvarov transformation as a composition of a  Geronimus transformation and  a Christoffel transformation, in this order. Indeed, we can write it in terms of the corresponding matrices  of generalized kernels
$u_{x,y} \mapsto \check u_{x,y} \mapsto \hat u_{x,y} $
where
\begin{align*}
\check u_{x,y}W_G(y)&=u_{x,y}, & \hat u_{x,y}&=W_C (x) \check u_{x,y}.
\end{align*}
At the level of the Gram matrices we will have
\begin{align*}
\check G W_G(\Lambda^\top)&= G, & \hat G&=W_C(\Lambda)  \check G.
\end{align*}
 In this paper,  the Geronimus--Uvarov transformation can be performed in two steps, i.e.,  $\check G$ has a Gauss--Borel factorization.\footnote{Recall that $G$ and $\hat G$ do have such factorization by assumption.} This leads to the following relations
\begin{align*}
\big(\check S_1\big)^{-1} \check H\big(\check{ S_2}\big)^{-\top}W_G(\Lambda^\top)&=\big( S_1\big)^{-1}  H\big({ S_2}\big)^{-\top},&
\big(\hat  S_1\big)^{-1} \hat  H\big(\hat{ S_2}\big)^{-\top}&=W_C(\Lambda)\big( \check S_1\big)^{-1}  \check H\big({ \check S_2}\big)^{-\top},
\end{align*}
and to the associated Geronimus and Christoffel resolvent matrices
\begin{align*}
\omega_G&:=\check S_1 \big(S_1)^{-1}=\check H \big(\check S_2\big)^{-\top}W_G(\Lambda^\top)S_2H^{-1},\\
\omega_C&:= \hat S_1 W_C(\Lambda)\big(\check S_1\big)^{-1} = \hat H \big(\hat S_2\big)^{-\top} \big(\check S_2\big)^{\top} \check H^{-1}.
\end{align*}
The Geronimus resolvent matrix  $\omega_G$ is a lower unitriangular  block semi-infinite matrix with only the first $N_G$ subdiagonals non-zero, and  $\omega_C$ is an upper triangular  block semi-infinite matrix with only the $N_C$  first superdiagonals non-zero. The resolvent $\omega$ results from  the composition of both transformations, so that it factors as
$\omega=\omega_C\omega_G$.
\end{rem}

\subsection{Matrix Geronimus--Uvarov transformation and Christoffel--Darboux kernels}

\begin{defi}\label{def:omeganNLST}
	We introduce the  resolvent \emph{wings}   matrices
	\begin{align*}
	\Omega^G{[n]}:= \begin{cases}
	\begin{bmatrix}
	\omega_{n,n-N_G}&\dots&\dots&\omega_{n,n-1}\\
	0_p	&\ddots &&\vdots\\
	\vdots	&\ddots &\ddots&\vdots\\
	0_p	&\dots&0_p&\omega_{n+N_G-1,n-1}
	\end{bmatrix}\in \mathbb{C}^{N_Gp\times N_Gp}, &  n\geq N_G, \\
	\begin{bmatrix}
	\omega_{n,0}&\dots &\dots&\omega_{n,n-1}\\
	\vdots      &       &&\vdots\\
	\omega_{N_G,0}&&&\omega_{N_G,n-1}\\
	0_p&\ddots\\
	\vdots	&\ddots&\ddots&\vdots\\
	0_p&   \dots   &0_p&\omega_{n+N_G-1,n-1}
	\end{bmatrix}\in \mathbb{C}^{N_Gp\times np},& n<N_G,
	\end{cases}
	\end{align*}
		\begin{align*}
		\Omega^C{[n]}:=
		\begin{cases}
			\begin{bmatrix}
%			0_p &0_p&\dots &0_p&0_p\\
%			\vdots & & &\vdots&\vdots\\
%			0_p &0_p&\dots &0_p&0_p\\
			I_p&0_p&\dots &0_p&0_p\\
			\omega_{n-N_C+1,n}& I_p &\ddots&0_p&0_p\\
			\vdots&\ddots&\ddots&&\vdots\\
			\vdots&&\ddots&I_p&0_p\\
			%		\omega_{n,n+1}&\cdots&\omega_{n-1,n+N-1} &I\\
			\omega_{n-1,n}&\dots&&\omega_{n-1,n+N_C-2} &I_p
			\end{bmatrix}\in \mathbb{C}^{N_Cp\times N_Cp}, & n\geq N_C,\\[40pt]
		\begin{bmatrix}
		\omega_{0,n}& \dots& \omega_{0,N_C-1}&I_p &\dots&0\\
		\vdots&&&&\ddots&\vdots\\
		%&	\vdots&&&I_p&0\\
		%		\omega_{n,n+1}&\cdots&\omega_{n-1,n+N-1} &I\\
		\omega^{[1]}_{n-1,n}&\dots&&\dots&\omega_{n-1,n+N_C-2} &I_p
		\end{bmatrix}\in\mathbb C^{np\times N_Cp}, & n< N_C.
		\end{cases}
		\end{align*}
\end{defi}

\begin{teo}[Matrix Geronimus--Uvarov transformation and connection formulas for the Christoffel--Darboux kernels]\label{teoconexLST}
	For $m_G=\min(n,N_G)$ and $m_C=\min(n,N_C)$, the perturbed and original Christoffel--Darboux kernels  are connected through	
	\begin{multline}\label{eq:KLST}
	\hat K_{n-1}(x,y)W_C(x)		=	W_G(y)	K_{n-1}(x,y)\\-	\begin{bmatrix}
	\big(\hat P_{n-m_C}^{[2]}(y)\big)^\top(\hat H_{n-m_C})^{-1},\cdots,\big(\hat P_{n+N_G-1}^{[2]}(y)\big)^\top(\hat H_{n+N_G-1})^{-1}
	\end{bmatrix}\begin{bmatrix}0_{m_Cp\times  m_G p} &
	-\Omega^C{[n]} \\
	\Omega^G{[n]}	& 0_{N_Gp\times N_Cp}
	\end{bmatrix}	\begin{bmatrix}	P_{n-m_G}^{[1]}(x)\\
	\vdots\\
	P_{n+N_C-1}^{[1]}(x)
	\end{bmatrix}.
	\end{multline}
	For $n\geq N_G$, the mixed Christoffel--Darboux kernels satisfy
		\begin{multline}\label{eq:NLST}
		\hat K_{n-1}^{(pc)}(x,y)W_G(x)		=	W_G(y)	K_{n-1}^{(pc)}(x,y)+\mathcal V_G(x,y)\\-	\begin{bmatrix}
		\big(\hat P_{n-m_C}^{[2]}(y)\big)^\top(\hat H_{n-m_C})^{-1},\cdots,\big(\hat P_{n+N_G-1}^{[2]}(y)\big)^\top(\hat H_{n+N_G-1})^{-1}
		\end{bmatrix}\begin{bmatrix}0_{m_Cp\times  N_G p} &
		-\Omega^C{[n]} \\
		\Omega^G{[n]}	& 0_{N_Gp\times N_Cp}
		\end{bmatrix}	\begin{bmatrix}	C_{n-N_G}^{[1]}(x)\\
		\vdots\\
		C_{n+N_C-1}^{[1]}(x)
		\end{bmatrix}.
		\end{multline}
\end{teo}

\begin{proof}
	For the first connection formulas \eqref{eq:KLST} we consider the pairing
	\begin{align*}
	\mathcal K_{n-1}(x,y):=
	\begin{bmatrix}
	\big(\hat P_{0}^{[2]}(y)\big)^\top(\hat H_0)^{-1},\cdots,\big(\hat P_{n-1}^{[2]}(y)\big)^\top(\hat H_{n-1})^{-1}
	\end{bmatrix}
	\omega_{[n]}
	\begin{bmatrix}
	P_{0}^{[1]}(x)\\
	\vdots\\
	P_{n-1}^{[1]}(x)
	\end{bmatrix},
	\end{align*}
	and compute it in two different ways. From \eqref{conex2LST} we get that
	\begin{align*}
	\omega_{[n]}
	\begin{bmatrix}
	P_{0}^{[1]}(x)\\
	\vdots\\
	P_{n-1}^{[1]}(x)
	\end{bmatrix}=	\begin{bmatrix}
	\hat P_{0}^{[1]}(x)\\
	\vdots\\
	\hat 	P_{n-1}^{[1]}(x)
	\end{bmatrix}W_C(x)- \tilde \Omega^C{[n]}	\begin{bmatrix}
	P_{n}^{[1]}(x)\\
	\vdots\\
	P_{n+N_C-1}^{[1]}(x)
	\end{bmatrix},
	\end{align*}
	where we have used the notation
	\begin{align*}
\tilde\Omega ^C[n]:=	\begin{cases}
\begin{bmatrix}
	0_{(n-N_c)p\times N_cp}\\
	\Omega^C{[n]}
	\end{bmatrix}, & n\geq N_C,\\
	\Omega^C{[n]},&n<N_C.
\end{cases}
	\end{align*}
Therefore, for $m_C=\min(N_C,n)$, we can write
	\begin{align*}
	\mathcal K_{n-1}(x,y)=\hat K_{n-1}(x,y)W_C(x)-	\begin{bmatrix}
	\big(\hat P_{n-m_C}^{[2]}(y)\big)^\top(\hat H_{n-m_C})^{-1},\dots,\big(\hat P_{n-1}^{[2]}(y)\big)^\top(\hat H_{n-1})^{-1}
	\end{bmatrix}\Omega^C{[n]}	\begin{bmatrix}
	P_{n}^{[1]}(x)\\
	\vdots\\
	P_{n+N_C-1}^{[1]}(x)
	\end{bmatrix}.
	\end{align*}
For $m_G=\min(n,N_G)$,	relation   \eqref{conex3LST} yields
	\begin{align*}
	\mathcal K_{n-1}(x,y)= W_G(y)	K_{n-1}(x,y)-\begin{bmatrix}
	\big(	\hat P_{n}^{[2]}(y)\big)^\top(\hat H_n)^{-1},\dots,\big(\hat  P_{n+N_G-1}^{[2]}(y)\big)^\top(\hat H_{n+N_G-1})^{-1}
	\end{bmatrix}
	\Omega^G{[n]}
	\begin{bmatrix}
	P_{n-m_G}^{[1]}(x)\\
	\vdots\\
	P_{n-1}^{[1]}(x)
	\end{bmatrix}.
	\end{align*}
	Hence, we obtain		
	\begin{multline*}
W_G(y)	K_{n-1}(x,y)-\begin{bmatrix}
\big(	\hat P_{n}^{[2]}(y)\big)^\top(\hat H_n)^{-1},\dots,\big(\hat  P_{n+N_G-1}^{[2]}(y)\big)^\top(\hat H_{n+N_G-1})^{-1}
\end{bmatrix}
\Omega^G{[n]}
\begin{bmatrix}
P_{n-m_G}^{[1]}(x)\\
\vdots\\
P_{n-1}^{[1]}(x)
\end{bmatrix}\\	=\hat K_{n-1}(x,y)W_C(x)-	\begin{bmatrix}
		\big(\hat P_{n-m_C}^{[2]}(y)\big)^\top(\hat H_{n-m_C})^{-1},\dots,\big(\hat P_{n-1}^{[2]}(y)\big)^\top(\hat H_{n-1})^{-1}
		\end{bmatrix}\Omega^C{[n]}	\begin{bmatrix}
		P_{n}^{[1]}(x)\\
		\vdots\\
		P_{n+N_C-1}^{[1]}(x)
		\end{bmatrix},
		\end{multline*}
		and \eqref{eq:KLST} follows.

Next, we consider the  pairing
	\begin{align*}
	\mathcal K^{(pc)}_{n-1}(x,y):=\begin{bmatrix}
	\big(	\hat P_{0}^{[2]}(y)\big)^\top(\hat H_0)^{-1},\dots,\big(\hat  P_{n-1}^{[2]}(y)\big)^\top(\hat H_{n-1})^{-1}
	\end{bmatrix}
	\omega_{[n]}
	\begin{bmatrix}
	C_{0}^{[1]}(x)\\
	\vdots\\
	C_{n-1}^{[1]}(x)
	\end{bmatrix},
	\end{align*}	
	which, we calculate   in two different ways. First, using  \eqref{eq:conexionC1LST} we get
%	\begin{align*}
%\omega_{[n]}
%\begin{bmatrix}
%C_{0}^{[1]}(x)\\
%\vdots\\
%C_{n-1}^{[1]}(x)
%\end{bmatrix}
%	\end{align*}

	\begin{align*}
	\mathcal K^{(pc)}_{n-1}(x,y)=	&\begin{multlined}[t][0.8\textwidth]
	\begin{bmatrix}
	\big(	\hat P_{0}^{[2]}(y)\big)^\top(\hat H_0)^{-1},\dots,\big(\hat  P_{n-1}^{[2]}(y)\big)^\top(\hat H_{n-1})^{-1}
	\end{bmatrix}\left(
	\begin{bmatrix}
	\hat C_{0}^{[1]}(x)W_G(x)\\
	\vdots\\
	\hat C_{n-1}^{[1]}(x)W_G(x)
	\end{bmatrix}
	-\big(\hat H\big(\hat S_2\big)^{-\top}\big)_{[n,N_G]}\mathcal B_G 	(\chi(x))_{[N_G]}
%	\hat H\big(\hat  S_2\big)^{-\top}\begin{bmatrix}
%	\mathcal B_G 	(\chi(z))_{[N_G]}\\
%	0_p\\\vdots
%	\end{bmatrix}\right.\\\left.- \tilde \Omega^C{[n]}	\begin{bmatrix}
%	C_{n}^{[1]}(x)\\
%	\vdots\\
%	C_{n+N_C-1}^{[1]}(x)
%	\end{bmatrix}
	\right)
	\end{multlined}\\
	&=\begin{multlined}[t][0.8\textwidth]\hat K^{(pc)}_{n-1}(x,y)W_G(x)-
	\big((\chi(y))_{[n]}\big)^\top\
	\big(\big(\hat S_2\big)^{\top}\hat H^{-1}\big)_{[n]}\big(\hat H\big(\hat S_2\big)^{-\top}\big)_{[n,N_G]}\mathcal B_G 	(\chi(x))_{[N_G]}\\
	-	\begin{bmatrix}
	\big(	\hat P_{0}^{[2]}(y)\big)^\top(\hat H_0)^{-1},\dots,\big(\hat  P_{n-1}^{[2]}(y)\big)^\top(\hat H_{n-1})^{-1}
	\end{bmatrix}\tilde \Omega^C{[n]}	
	\begin{bmatrix}
	C_{n}^{[1]}(x)\\
	\vdots\\
	C_{n+N_C-1}^{[1]}(x)
	\end{bmatrix},
	\end{multlined}
	\end{align*}
	where, again, $\big(\hat  H\big(\hat  S_2\big)^{-\top}\big)_{[n,N_G]}	$ denotes the truncation to the $n$ first block rows and first $N_G$ block  columns of
	$\hat  H\big(\hat  S_2\big)^{-\top}$. This simplifies,  for $n\geq N_G$,  to
	\begin{multline*}
	\mathcal K^{(pc)}_{n-1}(x,y)
	=\hat K^{(pc)}_{n-1}(x,y)W_G(x)-
	\big((\chi(y))_{[N_G]}\big)^\top\
	\mathcal B_G 	(\chi(x))_{[N_G]}\\	-	\begin{bmatrix}
	\big(	\hat P_{n-m_C}^{[2]}(y)\big)^\top(\hat H_{n-m_C})^{-1},\dots,\big(\hat  P_{n-1}^{[2]}(y)\big)^\top(\hat H_{n-1})^{-1}
	\end{bmatrix} \Omega^C{[n]}	\begin{bmatrix}
	C_{n}^{[1]}(x)\\
	\vdots\\
	C_{n+N_C-1}^{[1]}(x)
	\end{bmatrix}.
	\end{multline*}
	Second,  from  \eqref{conex3LST}
	\begin{align*}
	\mathcal K^{(pc)}_{n-1}(x,y)= W_G(y)	K^{(pc)}_{n-1}(x,y)-\begin{bmatrix}
	\big(	\hat P_{n}^{[2]}(y)\big)^\top(\hat H_n)^{-1},\dots,\big(\hat  P_{n+N_G-1}^{[2]}(y)\big)^\top(\hat H_{n+N_G-1})^{-1}
	\end{bmatrix}
	\Omega^G{[n]}
	\begin{bmatrix}
	C_{n-m_G}^{[1]}(x)\\
	\vdots\\
	C_{n-1}^{[1]}(x)
	\end{bmatrix},
	\end{align*}
	and, consequently, we obtain
	\begin{multline*}
\hat K^{(pc)}_{n-1}(x,y)W_G(x)-
	\mathcal V_G(x,y)	-	\begin{bmatrix}
	\big(	\hat P_{n-m_C}^{[2]}(y)\big)^\top(\hat H_{n-m_C})^{-1},\dots,\big(\hat  P_{n-1}^{[2]}(y)\big)^\top(\hat H_{n-1})^{-1}
	\end{bmatrix} \Omega^C{[n]}	\begin{bmatrix}
	C_{n}^{[1]}(x)\\
	\vdots\\
	C_{n+N_C-1}^{[1]}(x)
	\end{bmatrix}\\=W_G(y)	K^{(pc)}_{n-1}(x,y)-\begin{bmatrix}
	\big(	\hat P_{n}^{[2]}(y)\big)^\top(\hat H_n)^{-1},\dots,\big(\hat  P_{n+N_G-1}^{[2]}(y)\big)^\top(\hat H_{n+N_G-1})^{-1}
	\end{bmatrix}
	\Omega^G{[n]}
	\begin{bmatrix}
	C_{n-N_G}^{[1]}(x)\\
	\vdots\\
	C_{n-1}^{[1]}(x)
	\end{bmatrix}.
	\end{multline*}
\end{proof}

\subsection{On the first family of  linearly spectral transformed second kind functions}
In the next  sections we will assume that the perturbing polynomials are monic, $W_G(x)=I_p x^N+\sum\limits_{k=0}^{N-1}A_{G,k}x^k$, $W_C(x)=I_p x^N+\sum\limits_{k=0}^{N-1}A_{C,k}x^k$. The corresponding spectral data, eigenvalues, algebraic multiplicities, partial multiplicities, left and right root polynomials and Jordan pairs and triples will have a subindex $C$ or $G$ to indicate to which polynomial $W_C(x)$ or $W_G(x)$ they are linked to.

\begin{pro}
	The most general matrix Geronimus--Uvarov transformation is given by
	\begin{align}\label{uva}%\label{new}
	\hat u_{x,y}&:=W_C(x)u_{x,y}(W_G(y))^{-1}+v_{x,y}, & v_{x,y}&:=\sum_{a=1}^{q_G}\sum_{j=1}^{s_{G,a}}\sum_{m=0}^{\kappa_{G,j}^{(a)}-1}\frac{(-1)^{m}}{m!}W_C(x)\big(\xi^{[a]}_{j,m}\big)_x\otimes
	\delta^{(m)}(y-x_{G,a})l_{G,j}^{(a)}(y),
	\end{align}
	expressed in terms of the spectrum $\sigma(W_G(y))=\{x_{G,a}\}_{a=1}^{q_G}$, number of  Jordan blocks $s_{G,a}$, partial multiplicities $\kappa_{G,j}^{(a)}$,  and corresponding  adapted left root polynomials $l_{G,j}^{(a)}(y)$ of the matrix polynomial $W_G(y)$ and $\big(\xi^{[a]}_{j,m}\big)_x\in\big(\mathbb{C}^p[x]\big)'$.
\end{pro}
\begin{rem}
Observe that the functional $v_{x,y}$ is associated with the eigenvalues and  left root vectors of the perturbing polynomial $W_G(x)$. Also notice that we have introduced  $W_C(x)$ in this term. In general,  we have $N_Gp\geq\sum_{a=1}^q\sum_{i=1}^{s_a}\kappa^{(a)}_j$ and we can  not ensure  the equality, up to  for the nonsingular leading coefficient case. 
\end{rem}

\begin{rem}
We could also consider  diagonal supported masses, with
\begin{align}\label{eq:v_diagonal-2}
v_{x,x}=\sum_{a=1}^{q_G}\sum_{j=1}^{s_{G,a}}\sum_{m=0}^{\kappa_{G,j}^{(a)}-1}\frac{(-1)^{m}}{m!}W_C(x)\xi^{[a]}_{j,m}\delta^{(m)}(x-x_{G,a})l_{G,j}^{(a)}(x),
\end{align}
	For these sesquilinear forms we have\begin{align*}
	\prodint{P(x),Q(y)
	}_{\hat u}=
	\prodint{P(x)W_C(x), Q(y)(W_G(y))^{-\top}}_u+\sum_{a=1}^{q_G}\sum_{j=1}^{s_{G,a}}\sum_{m=0}^{\kappa_{G,j}^{(a)}-1}\prodint{P(x)W_C(x),\big(\xi^{[a]}_{j,m}\big)_x}\frac{1}{m!} \Big(l_{G,j}^{(a)}(y)\big(Q(y)\big)^\top\Big)^{(m)}_{x_{G,a}}.
	\end{align*}
\end{rem}

\begin{defi}
	For $z\neq x_{G,a}$, we  introduce the $p\times p$ matrices
	\begin{align}\label{eq:defCLST}
	\hat{\mathcal C}_{n;i}^{(a)}(z):=\prodint{\hat  P^{[1]}_n(x)W_C(x),\big(\xi^{[a])}_i\big)_x}\eta_{G,i}^{(a)}
\mathcal L^{(a)}_{G,i}
	\begin{bmatrix}
	\frac{I_p}{(z-x_{G,a})^{\kappa_i^{(a)}}}\\
	\vdots\\
	\frac{I_p}{z-x_{G,a}}
	\end{bmatrix},
	\end{align}
	where $i=1,\dots,s_{G,a}$.
\end{defi}
\begin{pro} For Hankel masses as in \eqref{eq:v_diagonal-2} we have
	\begin{align*}
\prodint{\hat  P^{[1]}_n(x)W_C(x),\big(\xi^{[a])}_i\big)_x}&=
	\mathcal J_{\hat  P^{[1]}_nW_C}^{(i)}(x_{G,a})
	\mathcal X^{(a)}_i\eta^{(a)}_{G,i},&
	\prodint{  P^{[1]}_n(x),\big(\xi^{[a])}_i\big)_x}&=
	\mathcal J_{  P^{[1]}_n}^{(i)}(x_{G,a})
	\mathcal X^{(a)}_i\eta^{(a)}_{G,i}.
	\end{align*}
\end{pro}

\begin{pro}\label{pro:checkCCauchyLST}
	For $z\not\in\operatorname{supp}_y(\hat u)=\operatorname{supp}_y(u)\cup \sigma(W_G(y))$, the following expression
	\begin{align*}
	\hat C_n^{[1]}(z)
	&=\left\langle \hat  P^{[1]}_n(x),\frac{I_p}{z-y}\right\rangle_{ W_CuW_G^{-1}}+\sum_{a=0}^{q_G}\sum_{i=1}^{s_{G,a}}\hat{\mathcal C}_{n;i}^{(a)}(z)
	\end{align*}	
	holds true.
\end{pro}
\begin{proof}
	From \eqref{pro:Cauchy1}	we have
	\begin{align*}
	\hat C_n^{[1]}(z)&=\left\langle \hat  P^{[1]}_n(x),\frac{I_p}{z-y}\right\rangle_{\hat  u}\\
	&=\prodint{\hat  P^{[1]}_n(x),\frac{I_p}{z-y}}_{ W_CuW_G^{-1}}+\sum_{a=0}^{q_G}\sum_{i=1}^{s_{G,a}}\sum_{m=0}^{\kappa_{G,i}^{(a)}-1}
\prodint{\hat P^{[1]}_{n}(x)W_C(x),\big(\xi^{[a]}_{j,m}\big)_x}
\frac{1}{m!}	\left(\frac{l_{G,i}^{(a)}(x)}{(z-x)}\right)^{(m)}_{x_{G,a}},
	\end{align*}
and	we deduce the result.
\end{proof}

\begin{ma}
	The function $\hat {\mathcal C}^{(a)}_{n;i}(x)W_G(x)r^{(b)}_{G,j}(x)\in\mathbb C^p[x]$ satisfies
	{\small\begin{align}\label{eq:CWrLST}
		\hat {\mathcal C}_{n;i}^{(a)}(x)W_G(x)r^{(b)}_{G,j}(x)=\begin{cases}
		\prodint{\hat  P^{[1]}_n(x)W_C(x),\big(\xi^{[a])}_i\big)_x}\eta_{G,i}^{(a)}
		\begin{bmatrix}
		(x-x_{G,a})^{\kappa^{(a)}_{G,\max(i,j)}-\kappa^{(a)}_{G,i}}\\
		\vdots\\
		(x-x_{G,a})^{\kappa^{(a)}_{G,\max(i,j)}-1}
		\end{bmatrix}w_{G,i,j}^{(a)}(x)+(x-x_{G,a})^{\kappa^{(a)}_{G,j}}T^{(a,a)}(x), &a=b,\\
		(x-x_{G,b})^{\kappa_{G,j}^{(b)}}T^{(a,b)}(x), & a\neq b,
		\end{cases}
		\end{align}}
	where the $\mathbb C^p$-valued function $T^{(a,b)}(x)$  is analytic at $x=x_{G,b}$ and, in particular, $T^{(a,a)}(x) \in\mathbb C^p[x]$.
\end{ma}
\begin{proof}
The function $\hat{\mathcal C}^{(a)}_{n;i}(x)W_G(x)r^{(b)}_{G,j}(x)\in\mathbb C^p[x]$, with  $a\neq b$, as  \eqref{eq:Wraj} inform us about, is such that
	\begin{align*}
	\hat{\mathcal C}^{(a)}_{n;i}(x)W_G(x)r^{(b)}_{G,j}(x)&=	\prodint{\hat  P^{[1]}_n(x)W_C(x),\big(\xi^{[a])}_i\big)_x}\eta_i^{(a)}\mathcal L^{(a)}_{G,i}
	\begin{bmatrix}
	\dfrac{I_p}{(x-x_{G,a})^{\kappa_{G.i}^{(a)}}}\\
	\vdots\\
	\dfrac{I_p}{x-x_{G,a}}
	\end{bmatrix}W_G(x)r_{G,j}^{(b)}(x)\\&=(x-x_{G,b})^{\kappa_{G,j}^{(b)}}T^{(a,b)}(x),
	\end{align*}
	where   $T^{(a,b)}(x)$ is $\mathbb C^p$-valued analytic function  at $x=x_{G,b}$.
	 \eqref{eq:defCLST} and Lemma \ref{lemma:trabajandoC} yield
	\begin{align*}
	\hat{\mathcal C}_{n;i}^{(a)}(x)W_G(x)r^{(a)}_{G,j}(x)&=
	\prodint{\hat  P^{[1]}_n(x)W_C(x),\big(\xi^{[a])}_i\big)_x}\eta_{G,i}^{(a)}
	\mathcal L^{(a)}_i
	\begin{bmatrix}
	\dfrac{I_p}{(x-x_{G,a})^{\kappa_{G,i}^{(a)}}}\\
	\vdots\\
	\dfrac{I_p}{x-x_{G,a}}
	\end{bmatrix}W_G(x)r^{(a)}_{G,j}(x)
	\\
	&=
	\begin{multlined}[t]
	\prodint{\hat  P^{[1]}_n(x)W_C(x),\big(\xi^{[a])}_i\big)_x}\eta_{G,i}^{(a)}
	\begin{bmatrix}
	\dfrac{1}{(x-x_{G,a})^{\kappa_{G,i}^{(a)}}}\\
	\vdots\\
	\dfrac{1}{x-x_{G,a}}
	\end{bmatrix}l_i^{(a)}(x)W_G(x)r_{G,j}^{(a)}(x)
	\\	+(x-x_{G,a})^{\kappa^{(a)}_{G,j}}
	\prodint{\hat  P^{[1]}_n(x)W_C(x),\big(\xi^{[a])}_i\big)_x}\eta_{G,i}^{(a)}T^{(a,a)}(x),
	\end{multlined}
	\end{align*}
	for some $T^{(a,a)}(x)\in \mathbb C^p[x]$. Hence, using  Proposition \ref{pro:lWr} we get
	\begin{multline*}
	\hat{\mathcal C}_{n;i}^{(a)}(x)W_G(x)r^{(a)}_{G,j}(x)=
	\prodint{\hat  P^{[1]}_n(x)W_C(x),\big(\xi^{[a])}_i\big)_x}\eta_{G,i}^{(a)}
	\begin{bmatrix}
	(x-x_{G,a})^{\kappa^{(a)}_{G,\max(i,j)}-\kappa^{(a)}_{G,i}}\\
	\vdots\\
	(x-x_{G,a})^{\kappa^{(a)}_{G,\max(i,j)}-1}
	\end{bmatrix}\\\times\Big(w^{(a)}_{G,i,j;0}+w^{(a)}_{G,i,j;1}(x-x_{G,a})+\cdots+w^{(a)}_{G,i,j;\kappa^{(a)}_{G,\min(i,j)}+N_G-2}
	(x-x_{G,a})^{\kappa^{(a)}_{G,\min(i,j)}+N_G-2}\Big).
	\\	+(x-x_{G,a})^{\kappa^{(a)}_{G,j}}
	\prodint{\hat  P^{[1]}_n(x)W_C(x),\big(\xi^{[a])}_i\big)_x}\eta_{G,i}^{(a)}
	T^{(a,a)}(x),
	\end{multline*}
	and the result follows.
\end{proof}

\begin{ma}
	The following relations hold
	\begin{align}\label{eq:CLST}
	\Big(	\hat  C_{n}^{[1]}(z)W_G(z)r^{(a)}_{G,j}(z)\Big)^{(m)}_{x_{G,a}}
	=\sum_{i=1}^{s_{G,a}}\Big(\hat{\mathcal C}_{n;i}^{(a)}(z)W_G(z)r_{G,j}^{(a)}(z)\Big)^{(m)}_{x_{G,a}},
	\end{align}
	for $m=0,\dots,\kappa^{(a)}_{G,j}-1$.
\end{ma}
\begin{proof}
	
	For $z\not\in\operatorname{supp}(u)\cup \sigma(W_G)$,	Proposition \ref{pro:checkCCauchyLST} leads to
	\begin{align*}
	\Big(	\hat  C_{n}^{[1]}(z)W_G(z)r^{(a)}_{G,j}(z)\Big)^{(m)}_{x_{G,a}}
	=\bigg(\left\langle \hat  P^{[1]}_n(x),\frac{I_p}{z-y}\right\rangle_{W_CuW_G^{-1}} W_G(z)r_{G,j}^{(a)}(z)\bigg)^{(m)}_{x_{G,a}}+\sum_{b=0}^{q_G}\sum_{i=1}^{s_{G,b}}\Big(\hat{\mathcal C}_{n;i}^{(b)}(z)W_G(z)r_{G,j}^{(a)}(z)\Big)^{(m)}_{x_{G,a}}.
	\end{align*}
Since $\sigma(W_G(y)))\cap \operatorname{supp}_y(u)=\varnothing$, the derivatives of the Cauchy kernel $1/(z-y)$ are analytic at  $z=x_{G,a}$. Hence,
	\begin{align*}
	\bigg(\left\langle \hat P^{[1]}_n(x),\frac{I_p}{z-y}\right\rangle_{W_CuW_G^{-1}} W_G(z)r_{G,j}^{(a)}(z)\bigg)^{(m)}_{x_{G,a}}&=\left\langle \hat  P^{[1]}_n(x), \bigg(\frac{W_G(z)r_{G,j}^{(a)}(z)}{z-y} \bigg)^{(m)}_{x_{G,a}}\right\rangle_{W_CuW_G^{-1}}\\&=\left\langle \hat P^{[1]}_n(x), \sum_{k=0}^m\binom{m}{k}\big(W_G(z)r_{G,j}^{(a)}(z) \big)^{(k)}_{x_{G,a}}\frac{(-1)^{m-k}(m-k)!}{(x_{G,a}-y)^{m-k+1}}\right\rangle_{W_CuW_G^{-1}}\\
	&=0_{p\times 1},
	\end{align*}
	for $m=1,\dots,\kappa^{(a)}_{G,j}-1$. From  \eqref{eq:CWrLST} we know that $	\hat {\mathcal C}_{n;i}^{(b)}(x)W_G(x)r^{(a)}_{G,j}(x)$, $b\neq a$, has a zero at $z=x_{G,a}$ of order $\kappa^{(a)}_{G,j}$,  so that
	\begin{align*}
	\big(	\hat {\mathcal C}_{n;i}^{(b)}(x)W_G(x)r^{(a)}_{G,j}(x)\big)^{(m)}_{x_{G,a}}&=0, &b\neq a,
	\end{align*}
	for $m=0,\dots,\kappa^{(a)}_{G,j}-1$.
\end{proof}
 
\begin{pro} The following formulas
	\begin{align}\label{eq:CscCLST}
	\boldsymbol{\mathcal J}^{(j)}_{G,\hat{C}^{[1]}_nW_G}(x_{G,a})&=\sum_{i=1}^{s_{G,a}}\boldsymbol{\mathcal J}^{(j)}_{\hat  {\mathcal C}_{n;i}W_G}(x_{G,a}),&
	\boldsymbol{\mathcal J}_{G,\hat {C}^{[1]}_nW_G}(x_{G,a})&=\sum_{i=1}^{s_{G,a}}\boldsymbol{\mathcal J}_{G,\hat {\mathcal C}_{n;i}W_G}(x_{G,a}),\\
	\label{eq:CPXWLST}
	\boldsymbol{\mathcal J}^{(j)}_{G,\hat  {\mathcal C}_{n;i}W_G}(x_{G,a})&=	\prodint{\hat  P^{[1]}_n(x)W_C(x),\big(\xi^{[a])}_i\big)_x}\mathcal W_{G,i,j}^{(a)},&
	\boldsymbol{\mathcal J}_{G,\hat {\mathcal C}_{n;i}W_G}(x_{G,a})&=	\prodint{\hat  P^{[1]}_n(x)W_C(x),\big(\xi^{[a])}_i\big)_x} \mathcal W_{G,i}^{(a)},\\
	\label{eq:checkCPLST}
	\boldsymbol{\mathcal J}_{G,\hat C^{[1]}_n W_G}(x_{G,a})
	&=	\prodint{\hat  P^{[1]}_n(x)W_C(x),\big(\xi^{[a])}\big)_x}\mathcal W_G^{(a)}, &
	\boldsymbol{\mathcal J}_{G,\hat C^{[1]}_n W_G}
	&=\prodint{\hat  P^{[1]}_n(x)W_C(x),(\xi)_x}\mathcal W_G,
	\end{align}
	hold true.
\end{pro}
\begin{proof}
	Formula \eqref{eq:CscCLST} is deduced from \eqref{eq:CLST}. To show \eqref{eq:CPXWLST},  let us notice that according to \eqref{eq:CWrLST},  for $m=0,\dots, \kappa_{G,j}^{(a)}-1$, we know that
	\begin{align*}
	\big(\hat {\mathcal C}_{n;i}^{(a)}(x)W_G(x)r^{(a)}_{G,j}(x)\big)_{x=x_{G,a}}^{(m)}=
\prodint{\hat  P^{[1]}_n(x)W_C(x),\big(\xi^{[a])}\big)_x}\eta_{G,i}^{(a)}
	\begin{bmatrix}
	\Big((x-x_{G,a})^{\kappa^{(a)}_{G,\max(i,j)}-\kappa^{(a)}_{G,i}}w_{G,i,j}^{(a)}(x)\Big)^{(m)}_{x_{G,a}}\\
	\vdots\\
	\Big((x-x_{G,a})^{\kappa^{(a)}_{G,\max(i,j)}-1}w_{G,i,j}^{(a)}(x)\Big)^{(m)}_{x_{G,a}}
	\end{bmatrix}.
	\end{align*}
For \eqref{eq:checkCPLST}, from \eqref{eq:CPXWLST} and \eqref{eq:CscCLST} we get
	\begin{align*}
	\boldsymbol{\mathcal J}^{(j)}_{G,\hat  C^{[1]}_nW_G}(x_{G,a})&=\sum_{i=1}^{s_{G,a}}\prodint{\hat  P^{[1]}_n(x)W_C(x),\big(\xi^{[a])}\big)_x} \mathcal W_{G,i,j}^{(a)},&
	\boldsymbol{\mathcal J}_{G,\hat  C^{[1]}_nW_G}(x_{G,a})&=\sum_{i=1}^{s_{G,a}}\prodint{\hat  P^{[1]}_n(x)W_C(x),\big(\xi^{[a])}\big)_x} \mathcal W_{G,i}^{(a)}.
	\end{align*}
\end{proof}

\subsection{Spectral Christoffel--Geronimus--Uvarov formulas}
We assume that the leading terms of both perturbing polynomials, $A_{G,N_G}$ and $A_{C,N_C}$, are nonsingular.
\subsubsection{Discussion for $n\geq N_G$}
\begin{rem}
In Corollary \ref{coro:CPTR},  with the aid of a mixed spectral/nonspectral approach, we will check that for $A_{G,N}$ nonsingular we have
	\begin{align*}
	\begin{vmatrix}
	\boldsymbol{\mathcal J}_{C,P^{[1]}_{n-N_G}} &	
	\boldsymbol{\mathcal J}_{G,C^{[1]}_{n-N_G}}-
 \prodint{  P^{[1]}_{n-N_G} (x),(\xi)_x} \mathcal W_G
	\\ 	\vdots & \vdots \\ \boldsymbol{\mathcal J}_{C,P^{[1]}_{n+N_C-1}}&\boldsymbol{\mathcal J}_{G, C^{[1]}_{n+N_C-1}}-	 \prodint{  P^{[1]}_{n+N_C-1} (x),(\xi)_x} \mathcal W_G
	\end{vmatrix}\neq 0.
	\end{align*}
\end{rem}

\begin{pro}\label{pro:resolventCPLST}
	For   $n\geq N_G$,  the resolvent can be expressed as follows
{		\begin{multline*}
		\big[\omega_{n,n-N_G},\dots,\omega_{n,n+N_C-1}\big]	
		=\\
		-\begin{bmatrix} \boldsymbol{\mathcal J}_{C,P^{[1]}_{n+N_C}},
		\boldsymbol{\mathcal J}_{G,C^{[1]}_{n+N_C}}-		\prodint{  P^{[1]}_{n+N_C} (x),(\xi)_x} \mathcal W_G
		\end{bmatrix}\begin{bmatrix}
		\boldsymbol{\mathcal J}_{C,P^{[1]}_{n-N_G}} &	
		\boldsymbol{\mathcal J}_{G,C^{[1]}_{n-N_G}}-
		\prodint{  P^{[1]}_{n-N_G} (x),(\xi)_x} \mathcal W_G
		\\ 	\vdots & \vdots \\ \boldsymbol{\mathcal J}_{C,P^{[1]}_{n+N_C-1}}&\boldsymbol{\mathcal J}_{G, C^{[1]}_{n+N_C-1}}-	 \prodint{  P^{[1]}_{n+N_C-1} (x),(\xi)_x} \mathcal W_G		\end{bmatrix}^{-1}.
		\end{multline*}}
\end{pro}
\begin{proof}
The connection formula  \eqref{eq:conexionC1LST} gives for  $n\geq N_G$
	\begin{align*}
	\hat  C_n^{[1]}(x)W_G(x)=\sum_{k=n-N_G}^{n+N_C}\omega_{n,k}C_k^{[1]}(x).
	\end{align*}
Thus,	we deduce that
	\begin{align*}
	\boldsymbol{\mathcal J}_{G, \hat  C^{[1]}_nW_G}=\big[\omega_{n,n-N_G},\dots,\omega_{n,n+N_C-1}\big]	
	\begin{bmatrix}
	\boldsymbol{\mathcal J}_{G,C^{[1]}_{n-N_G}}\\ 	\vdots\\ \boldsymbol{\mathcal J}_{G,C^{[1]}_{n+N_C-1}}
	\end{bmatrix}
	+\boldsymbol{\mathcal J}_{C^{[1]}_{n+N_C}}.
	\end{align*}
Now, from  \eqref{conex2LST}, we conclude
	\begin{align}\label{eq:PomegaPLST}
		\prodint{  \hat P^{[1]}_{n} (x)W_C(x),(\xi)_x}=\big[\omega_{n,n-N_G},\dots,\omega_{n,n+N_C-1}\big]	
	\begin{bmatrix}
\prodint{  P^{[1]}_{n-N_G} (x),(\xi)_x}\\ 	\vdots\\ \prodint{  P^{[1]}_{n+N_C-1} (x),(\xi)_x}
	\end{bmatrix}
	+\prodint{  P^{[1]}_{n+N_C} (x),(\xi)_x}.
	\end{align}
Recalling  \eqref{eq:checkCPLST} we obtain
	\begin{multline*}
	\big[\omega_{n,n-N_G},\dots,\omega_{n,n+N_C-1}\big]	
	\begin{bmatrix}
	\boldsymbol{\mathcal J}_{G,C^{[1]}_{n-N_C}}\\ 	\vdots\\ \boldsymbol{\mathcal J}_{G,C^{[1]}_{n+N_C-1}}
	\end{bmatrix}
	+\boldsymbol{\mathcal J}_{G,C^{[1]}_{n+N_C}}
	\\=\big[\omega_{n,n-N_G},\dots,\omega_{n,n+N_C-1}\big]	
	\begin{bmatrix}
	\prodint{  P^{[1]}_{n-N_G} (x),(\xi)_x}\mathcal W_G\\ 	\vdots\\ \prodint{  P^{[1]}_{n+N_C-1} (x),\big(\xi^{[a])}\big)_x}\mathcal W_G
	\end{bmatrix}
	+\prodint{  P^{[1]}_{n+N_C} (x),(\xi)_x}\mathcal W_G.
	\end{multline*}
In other words
	\begin{align*}
	\big[\omega_{n,n-N_G},\dots,\omega_{n,n+N_C-1}\big]	
	\begin{bmatrix}
	\boldsymbol{\mathcal J}_{G,C^{[1]}_{n-N_G}}-\prodint{  P^{[1]}_{n-N_G} (x),(\xi)_x}\mathcal W_G.\\ 	\vdots\\ \boldsymbol{\mathcal J}_{G, C^{[1]}_{n+N_C-1}}-\prodint{  P^{[1]}_{n+N_C-1} (x),(\xi)_x}\mathcal W_G.
	\end{bmatrix}
	&=
	-\Big(\boldsymbol{\mathcal J}_{G,C^{[1]}_{n+N_C}}-\prodint{  P^{[1]}_{n+N_C} (x),(\xi)_x}\mathcal W_G\Big).
	\end{align*}
From  \eqref{conex2LST} we also get
	\begin{align}\label{eq:JP}
	\big[\omega_{n,n-N_G},\dots,\omega_{n,n+N_C-1}\big]	
	\begin{bmatrix}
	\boldsymbol{\mathcal J}_{C,P^{[1]}_{n-N_G}}\\ 	\vdots\\ \boldsymbol{\mathcal J}_{C,P^{[1]}_{n+N_C-1}}
	\end{bmatrix}
	+\boldsymbol{\mathcal J}_{C,P^{[1]}_{n+N_C}}=0.
	\end{align}
	Therefore, we conclude that
		\begin{multline*}
		\big[\omega_{n,n-N_G},\dots,\omega_{n,n+N_C-1}\big]	
		\begin{bmatrix}
		\boldsymbol{\mathcal J}_{C,P^{[1]}_{n-N_G}} &	\boldsymbol{\mathcal J}_{G,C^{[1]}_{n-N_G}}-
		\prodint{  P^{[1]}_{n-N_G} (x),(\xi)_x}\mathcal W_G\\ 	\vdots & \vdots \\ \boldsymbol{\mathcal J}_{C,P^{[1]}_{n+N_C-1}}&\boldsymbol{\mathcal J}_{G, C^{[1]}_{n+N_C-1}}-\prodint{  P^{[1]}_{n+N_C-1} (x),(\xi)_x}\mathcal W_G.
		\end{bmatrix}
		\\=
		-\begin{bmatrix} \boldsymbol{\mathcal J}_{C,P^{[1]}_{n+N_C}},\;\,
		\boldsymbol{\mathcal J}_{G,C^{[1]}_{n+N_C}}-\prodint{  P^{[1]}_{n+N_C} (x),(\xi)_x}\mathcal W_G.
		\end{bmatrix}.
		\end{multline*}
\end{proof}

\begin{teo}[Spectral Christoffel--Geronimus--Uvarov formulas]\label{teo:SCGU}
	For monic Geronimus--Uvarov perturbations, with  masses as described in \eqref{uva}, when $n\geq N_G$,
	we have the following last quasideterminant expressions for the perturbed biorthogonal matrix polynomials and its matrix norms
	\begin{align*}
	\hat P^{[1]}_{n}(x)W_C(x)&=
	\Theta_*
	\begin{bmatrix}
	\boldsymbol{\mathcal J}_{C,P^{[1]}_{n-N_G}} &	\boldsymbol{\mathcal J}_{G,C^{[1]}_{n-N_G}}-
	\prodint{  P^{[1]}_{n-N_G} (x),(\xi)_x}\mathcal W_G&P_{n-N_G}^{[1]}(x)\\ 	\vdots & \vdots &\vdots \\ \boldsymbol{\mathcal J}_{C,P^{[1]}_{n+N_C}}&\boldsymbol{\mathcal J}_{G, C^{[1]}_{n+N_C}}-\prodint{  P^{[1]}_{n+N_C} (x),(\xi)_x}\mathcal W_G& P^{[1]}_{n+N_C}(x)
	\end{bmatrix},\\
	\hat H_{n}&=\Theta_*
	\begin{bmatrix}
	\boldsymbol{\mathcal J}_{C,P^{[1]}_{n-N_G}} &	\boldsymbol{\mathcal J}_{G,C^{[1]}_{n-N_G}}-
	\prodint{  P^{[1]}_{n-N_G} (x),(\xi)_x}\mathcal W_G&H_{n-N_G}\\
		\boldsymbol{\mathcal J}_{C,P^{[1]}_{n-N_G+1}} &	\boldsymbol{\mathcal J}_{G,C^{[1]}_{n-N_G+1}}-\prodint{  P^{[1]}_{n-N_G+1} (x),(\xi)_x}\mathcal W_G&0_p\\ 	\vdots & \vdots &\vdots \\ \boldsymbol{\mathcal J}_{C,P^{[1]}_{n+N_C}}&\boldsymbol{\mathcal J}_{G, C^{[1]}_{n+N_C}}-\prodint{  P^{[1]}_{n+N_C} (x),(\xi)_x}\mathcal W_G & 0_p
	\end{bmatrix},\\
	\big(	\hat  P _{n}^{[2]}(y)\big)^\top&=-\Theta_*\begin{bmatrix} 	\boldsymbol{\mathcal J}_{ C, P^{[1]}_{n-N_G}}&
	\boldsymbol{\mathcal J}_{ G, C^{[1]}_{n-N_G}}	-\prodint{  P^{[1]}_{n-N_G} (x),(\xi)_x}\mathcal W_G& H_{n-N_G}\\
	\boldsymbol{\mathcal J}_{ C, P^{[1]}_{n-N_G+1}}&
	\boldsymbol{\mathcal J}_{ G, C^{[1]}_{n-N_G+1}}	-\prodint{  P^{[1]}_{n-N_G+1} (x),(\xi)_x}\mathcal W_G& 0_p\\ 	\vdots & \vdots\\\boldsymbol{\mathcal J}_{C, P^{[1]}_{n+N_C-1}}&	\boldsymbol{\mathcal J}_{G, C^{[1]}_{n+N_C-1}}\prodint{  P^{[1]}_{n+N_C-1} (x),(\xi)_x}\mathcal W_G&0_p\\
	W_G(y) \boldsymbol{\mathcal J}_{C, K_{n-1}}(y) &  W_G(y)\big(	\boldsymbol{\mathcal J}_{ G,K^{(pc)}_{n-1}}(y)-
	\prodint{  K_{n-1}(x,y),(\xi)_x}\mathcal W_G
	\big)+\boldsymbol{\mathcal J}_{G, \mathcal V}(y) &0_p
	\end{bmatrix}.
	\end{align*}
\end{teo}

\begin{proof}
Relation \eqref{conex2LST} gives
	\begin{align}\label{eq:conexion}
	\hat  P^{[1]}_{n}(x)W_C(x)&=P^{[1]}_{n+N_C}(x)+\big[\omega_{n,n-N_G},\dots,\omega_{n,n+N_C-1}\big]\begin{bmatrix}P_{n-N_G}^{[1]}(x)\\\vdots\\
	P_{n+N_C-1}^{[1]}(x)	
	\end{bmatrix},
	\end{align}
	and, applying Proposition \ref{pro:resolventCPLST}, we get
{	\begin{multline*}
	\hat  P^{[1]}_{n}(x)W_C(x)=	P^{[1]}_{n+N_C}(x)-\begin{bmatrix} \boldsymbol{\mathcal J}_{C,P^{[1]}_{n+N_C}},
	\boldsymbol{\mathcal J}_{G,C^{[1]}_{n+N_C}}-	\prodint{  P^{[1]}_{n+N_C} (x),(\xi)_x}\mathcal W_G
	\end{bmatrix}\\\times\begin{bmatrix}
	\boldsymbol{\mathcal J}_{C,P^{[1]}_{n-N_G}} &	\boldsymbol{\mathcal J}_{G,C^{[1]}_{n-N_G}}-	\prodint{  P^{[1]}_{n-N_G} (x),(\xi)_x}\mathcal W_G\\ 	\vdots & \vdots \\ \boldsymbol{\mathcal J}_{C,P^{[1]}_{n+N_C-1}}&\boldsymbol{\mathcal J}_{G, C^{[1]}_{n+N_C-1}}-
	\prodint{  P^{[1]}_{n+N_G-1} (x),(\xi)_x}\mathcal W_G
	\end{bmatrix}^{-1}\begin{bmatrix}P_{n-N_G}^{[1]}(x)\\\vdots\\
	P_{n+N_C-1}^{[1]}(x)	
	\end{bmatrix},
	\end{multline*}}
	and the result is proven.
From \eqref{eq:omegaALST} we deduce
	\begin{align*} \label{eq:HG2}
	\hat H_{n}=\omega_{n,n-N_G}H_{n-N_G}.
	\end{align*}
According to Proposition \ref{pro:resolventCPLST} we conclude
	\begin{multline*}
	\omega_{n,n-N_C}
	=-\begin{bmatrix} \boldsymbol{\mathcal J}_{C,P^{[1]}_{n+N_C}},
	\boldsymbol{\mathcal J}_{G,C^{[1]}_{n+N_C}}-	\prodint{  P^{[1]}_{n+N_C} (x),(\xi)_x}\mathcal W_G
	\end{bmatrix}\\\times\begin{bmatrix}
	\boldsymbol{\mathcal J}_{C,P^{[1]}_{n-N_G}} &	\boldsymbol{\mathcal J}_{G,C^{[1]}_{n-N_G}}-
		\prodint{  P^{[1]}_{n-N_G} (x),(\xi)_x}\mathcal W_G\\ 	\vdots & \vdots \\ \boldsymbol{\mathcal J}_{C,P^{[1]}_{n+N_C-1}}&\boldsymbol{\mathcal J}_{G, C^{[1]}_{n+N_C-1}}-
	\prodint{  P^{[1]}_{n+N_G-1} (x),(\xi)_x}\mathcal W_G
	\end{bmatrix}^{-1}\begin{bmatrix}
	I_p\\0_p\\\vdots\\0_p
	\end{bmatrix}.
	\end{multline*}
	For the second family of biorthogonal matrix polynomials we proceed as follows.	First, 	recalling to Definition \ref{eq:CD kernel}	we write   \eqref{eq:NLST}  in the  following way
		\begin{multline}%\label{eq:NLST}
		\sum_{k=0}^{n-1}\big(\hat  P_k^{[2]}(y)\big)^\top \hat  H_k^{-1}\hat  C_k^{[1]}(x)W_G(x)		=	W_G(y)	K_{n-1}^{(pc)}(x,y)+\mathcal V_G(x,y)\\-	\begin{bmatrix}
		\big(\hat P_{n-m_C}^{[2]}(y)\big)^\top(\hat H_{n-m_C})^{-1},\dots,\big(\hat P_{n+N_G-1}^{[2]}(y)\big)^\top(\hat H_{n+N_G-1})^{-1}
		\end{bmatrix}\begin{bmatrix}0_{m_Cp\times  N_G p} &
		-\Omega^C{[n]} \\
		\Omega^G{[n]}	& 0_{N_Gp\times N_Cp}
		\end{bmatrix}	\begin{bmatrix}	C_{n-N_G}^{[1]}(x)\\
		\vdots\\
		C_{n+N_C-1}^{[1]}(x)
		\end{bmatrix}.
		\end{multline}
The corresponding spectral jets fulfill
	\begin{multline*}
	\sum_{k=0}^{n-1}\big(\hat  P_k^{[2]}(y)\big)^\top \hat  H_k^{-1}\boldsymbol{\mathcal J}_{G, \hat C^{[1]}_kW_G}=W_G(y)
	\boldsymbol{\mathcal J}_{G, K^{(pc)}_{n-1}}(y) 	
	+\boldsymbol{\mathcal J}_{ G,\mathcal V}(y)\\-\begin{bmatrix}
	\big(\hat P_{n-m_C}^{[2]}(y)\big)^\top(\hat H_{n-m_C})^{-1},\dots,\big(\hat P_{n+N_G-1}^{[2]}(y)\big)^\top(\hat H_{n+N_G-1})^{-1}
	\end{bmatrix}\begin{bmatrix}0_{m_Cp\times  N_G p} &
	-\Omega^C{[n]} \\
	\Omega^G{[n]}	& 0_{N_Gp\times N_Cp}
	\end{bmatrix}\begin{bmatrix}
	\boldsymbol{\mathcal J}_{ G, C^{[1]}_{n-N_G}}\\ 	\vdots\\	\boldsymbol{\mathcal J}_{G, C^{[1]}_{n+N_C-1}}
	\end{bmatrix}.
	\end{multline*}
From \eqref{eq:checkCPLST} we conclude
		\begin{multline}\label{eq:WKpcLST}
		\sum_{k=0}^{n-1}\big(\hat  P_k^{[2]}(y)\big)^\top \hat  H_k^{-1}\prodint{\hat  P^{[1]}_{k} (x),(\xi)_x}\mathcal W_G=W_G(y)
		\boldsymbol{\mathcal J}_{G, K^{(pc)}_{n-1}}(y)+\boldsymbol{\mathcal J}_{ G,\mathcal V}(y)\\
		-\begin{bmatrix}
		\big(\hat P_{n-m_C}^{[2]}(y)\big)^\top(\hat H_{n-m_C})^{-1},\dots,\big(\hat P_{n+N_G-1}^{[2]}(y)\big)^\top(\hat H_{n+N_G-1})^{-1}
		\end{bmatrix}\begin{bmatrix}0_{m_Cp\times  N_G p} &
		-\Omega^C{[n]} \\
		\Omega^G{[n]}	& 0_{N_Gp\times N_Cp}
		\end{bmatrix}\begin{bmatrix}
		\boldsymbol{\mathcal J}_{ G, C^{[1]}_{n-N_G}}\\ 	\vdots\\	\boldsymbol{\mathcal J}_{G, C^{[1]}_{n+N_C-1}}
		\end{bmatrix},
		\end{multline}
while, from \eqref{eq:KLST}, we realize that
{\small	\begin{multline*}
	\sum_{k=0}^{n-1}\big(\hat P_k^{[2]}(y)\big)^\top \hat H_k^{-1}	\prodint{  \hat P^{[1]}_{k} (x),(\xi)_x}\mathcal W_G=
	W_G(y) \prodint{  K_{n-1}(x,y),(\xi)_x}\mathcal W_G
\\	-\begin{bmatrix}
	\big(\hat P_{n-m_C}^{[2]}(y)\big)^\top(\hat H_{n-m_C})^{-1},\dots,\big(\hat P_{n+N_G-1}^{[2]}(y)\big)^\top(\hat H_{n+N_G-1})^{-1}
	\end{bmatrix}\begin{bmatrix}0_{m_Cp\times  N_G p} &
	-\Omega^C{[n]} \\
	\Omega^G{[n]}	& 0_{N_Gp\times N_Cp}
	\end{bmatrix}
	\begin{bmatrix}
		\prodint{   P^{[1]}_{n-N_G} (x),(\xi)_x}\mathcal W_G\\ 	\vdots\\		\prodint{   P^{[1]}_{n+N_C-1} (x),(\xi)_x}\mathcal W_G
	\end{bmatrix},
	\end{multline*}}
	which can be subtracted to 	\eqref{eq:WKpcLST} to get
{\small	\begin{multline*}
	W_G(y)\big(	\boldsymbol{\mathcal J}_{ G,K^{(pc)}_{n-1}}(y)-\prodint{  K_{n-1}(x,y),(\xi)_x}\mathcal W_G\big)+\boldsymbol{\mathcal J}_{G, \mathcal V}(y)\\=\begin{bmatrix}
	\big(\hat P_{n-m_C}^{[2]}(y)\big)^\top(\hat H_{n-m_C})^{-1},\dots,\big(\hat P_{n+N_G-1}^{[2]}(y)\big)^\top(\hat H_{n+N_G-1})^{-1}
	\end{bmatrix}\begin{bmatrix}0_{m_Cp\times  N_G p} &
	-\Omega^C{[n]} \\
	\Omega^G{[n]}	& 0_{N_Gp\times N_Cp}
	\end{bmatrix}
\\\times	\begin{bmatrix}
	\boldsymbol{\mathcal J}_{ G, C^{[1]}_{n-N_G}}	-	\prodint{   P^{[1]}_{n-N_G} (x),(\xi)_x}\mathcal W_G\\ 	\vdots\\	\boldsymbol{\mathcal J}_{G, C^{[1]}_{n+N_C-1}}-	\prodint{   P^{[1]}_{n+N_C-1} (x),(\xi)_x}\mathcal W_G
	\end{bmatrix}.
	\end{multline*}	}
 \eqref{eq:KLST} also implies
		\begin{multline}\label{eq:WKpcLST2}
	W_G(y) \boldsymbol{\mathcal J}_{C, K_{n-1}}(y)\\
		=\begin{bmatrix}
		\big(\hat P_{n-m_C}^{[2]}(y)\big)^\top(\hat H_{n-m_C})^{-1},\dots,\big(\hat P_{n+N_G-1}^{[2]}(y)\big)^\top(\hat H_{n+N_G-1})^{-1}
		\end{bmatrix}\begin{bmatrix}0_{m_Cp\times  N_G p} &
		-\Omega^C{[n]} \\
		\Omega^G{[n]}	& 0_{N_Gp\times N_Cp}
		\end{bmatrix}\begin{bmatrix}
		\boldsymbol{\mathcal J}_{ C, P^{[1]}_{n-N_G}}\\ 	\vdots\\	\boldsymbol{\mathcal J}_{C, P^{[1]}_{n+N_C-1}}
		\end{bmatrix}.
		\end{multline}

	Hence, we obtain
	\begin{multline}\label{eq:PomeganNLST}
\begin{multlined}[b][0.7\textwidth]
\begin{bmatrix} W_G(y) \boldsymbol{\mathcal J}_{C, K_{n-1}}(y),
	W_G(y)\big(	\boldsymbol{\mathcal J}_{ G,K^{(pc)}_{n-1}}(y)-	\prodint{  K_{n-1}(x,y),(\xi)_x}\mathcal W_G\big)+\boldsymbol{\mathcal J}_{G, \mathcal V}(y)	
\end{bmatrix}\\\times\begin{bmatrix} 	\boldsymbol{\mathcal J}_{ C, P^{[1]}_{n-N_G}}&
	\boldsymbol{\mathcal J}_{ G, C^{[1]}_{n-N_G}}	-	\prodint{   P^{[1]}_{n-N_G} (x),(\xi)_x}\mathcal W_G\\ 	\vdots & \vdots\\\boldsymbol{\mathcal J}_{C, P^{[1]}_{n+N_C-1}}&	\boldsymbol{\mathcal J}_{G, C^{[1]}_{n+N_C-1}}-	\prodint{   P^{[1]}_{n+N_C-1} (x),(\xi)_x}\mathcal W_G
	\end{bmatrix}^{-1}
\end{multlined}
\\=\begin{bmatrix}
		\big(\hat P_{n-m_C}^{[2]}(y)\big)^\top(\hat H_{n-m_C})^{-1}&\cdots&\big(\hat P_{n+N_G-1}^{[2]}(y)\big)^\top(\hat H_{n+N_G-1})^{-1}
		\end{bmatrix}\begin{bmatrix}0_{m_Cp\times  N_G p} &
		-\Omega^C{[n]} \\
		\Omega^G{[n]}	& 0_{N_Gp\times N_Cp}
		\end{bmatrix}.
	\end{multline}	
	Now, from Definition \ref{def:omeganNLST}  and the fact   $\omega_{n,n-N_G}=\hat H_n\big(H_{n-N_G}\big)^{-1}$, we get
{	\begin{multline*}
	\big(	\hat P _{n}^{[2]}(y)\big)^\top=		\begin{bmatrix} W_G(y) \boldsymbol{\mathcal J}_{C, K_{n-1}}(y),
	W_G(y)\big(	\boldsymbol{\mathcal J}_{ G,K^{(pc)}_{n-1}}(y)-	\prodint{  K_{n-1}(x,y),(\xi)_x}\mathcal W_G\big)+\boldsymbol{\mathcal J}_{G, \mathcal V}(y)	
	\end{bmatrix}\\\times\begin{bmatrix} 	\boldsymbol{\mathcal J}_{ C, P^{[1]}_{n-N_G}}&
	\boldsymbol{\mathcal J}_{ G, C^{[1]}_{n-N_G}}	-	\prodint{   P^{[1]}_{n-N_G} (x),(\xi)_x}\mathcal W_G\\ 	\vdots & \vdots\\\boldsymbol{\mathcal J}_{C, P^{[1]}_{n+N_C-1}}&	\boldsymbol{\mathcal J}_{G, C^{[1]}_{n+N_C-1}}-	\prodint{   P^{[1]}_{n+N_C-1} (x),(\xi)_x}\mathcal W_G
	\end{bmatrix}^{-1}\begin{bmatrix}
	H_{n-N_G}\\0_p\\\vdots\\0_p
	\end{bmatrix},
	\end{multline*}}
	and the result follows.	
\end{proof}
\begin{rem} For a mass term $v_{x,y}$ as in \eqref{eq:v_diagonal-2} the previous Christoffel--Geronimus--Uvarov formulas read
		\begin{align*}
		\hat P^{[1]}_{n}(x)W_C(x)&=
		\Theta_*
		\begin{bmatrix}
		\boldsymbol{\mathcal J}_{C,P^{[1]}_{n-N_G}} &	\boldsymbol{\mathcal J}_{G,C^{[1]}_{n-N_G}}-{\mathcal J}_{G,P^{[1]}_{n-N_G}}\mathcal T_G&P_{n-N_G}^{[1]}(x)\\ 	\vdots & \vdots &\vdots \\ \boldsymbol{\mathcal J}_{C,P^{[1]}_{n+N_C}}&\boldsymbol{\mathcal J}_{G, C^{[1]}_{n+N_C}}-{\mathcal J}_{G,P^{[1]}_{n+N_C}}\mathcal T_G & P^{[1]}_{n+N_C}(x)
		\end{bmatrix},\\
		\hat H_{n}&=\Theta_*
		\begin{bmatrix}
		\boldsymbol{\mathcal J}_{C,P^{[1]}_{n-N_G}} &	\boldsymbol{\mathcal J}_{G,C^{[1]}_{n-N_G}}-{\mathcal J}_{G,P^{[1]}_{n-N_G}}\mathcal T_G&H_{n-N_G}\\
		\boldsymbol{\mathcal J}_{C,P^{[1]}_{n-N_G+1}} &	\boldsymbol{\mathcal J}_{G,C^{[1]}_{n-N_G+1}}-{\mathcal J}_{G,P^{[1]}_{n-N_G+1}}\mathcal T_G&0_p\\ 	\vdots & \vdots &\vdots \\ \boldsymbol{\mathcal J}_{C,P^{[1]}_{n+N_C}}&\boldsymbol{\mathcal J}_{G, C^{[1]}_{n+N_C}}-{\mathcal J}_{G,P^{[1]}_{n+N_C}}\mathcal T_G & 0_p
		\end{bmatrix},\\
		\big(	\hat  P _{n}^{[2]}(y)\big)^\top&=-\Theta_*\begin{bmatrix} 	\boldsymbol{\mathcal J}_{ C, P^{[1]}_{n-N_G}}&
		\boldsymbol{\mathcal J}_{ G, C^{[1]}_{n-N_G}}	-\mathcal J_{ G,P^{[1]}_{n-N_G}}\mathcal T_G & H_{n-N_G}\\
		\boldsymbol{\mathcal J}_{ C, P^{[1]}_{n-N_G+1}}&
		\boldsymbol{\mathcal J}_{ G, C^{[1]}_{n-N_G+1}}	-\mathcal J_{ G,P^{[1]}_{n-N_G+1}}\mathcal T_G & 0_p\\ 	\vdots & \vdots\\\boldsymbol{\mathcal J}_{C, P^{[1]}_{n+N_C-1}}&	\boldsymbol{\mathcal J}_{G, C^{[1]}_{n+N_C-1}}-\mathcal J_{G, P^{[1]}_{n+N_C-1}}\mathcal T_G&0_p\\
		W_G(y) \boldsymbol{\mathcal J}_{C, K_{n-1}}(y) &  W_G(y)\big(	\boldsymbol{\mathcal J}_{ G,K^{(pc)}_{n-1}}(y)-\mathcal J_{G,K_{n-1}}(y)\mathcal T_G\big)+\boldsymbol{\mathcal J}_{G, \mathcal V}(y) &0_p
		\end{bmatrix}.
		\end{align*}
\end{rem}

\subsubsection{Discussion for $n<N_G$}\label{susection:<N}
\begin{pro}
	We have that
	\begin{align*}
	\omega\begin{bmatrix}
	\boldsymbol{\mathcal J}_{C,P^{[1]}},\boldsymbol{\mathcal J}_{G,C^{[1]}}-
\prodint{   P^{[1]} (x),(\xi)_x}\mathcal W_G
	\end{bmatrix}=-			\hat H\big(\hat S_2\big)^{-\top}\begin{bmatrix}
	0_{N_Cp}&\mathcal B_G {\mathcal Q}_{G}\\
	0_{N_C}p &0_{N_Gp}\\
	\vdots &\vdots
	\end{bmatrix}.
	\end{align*}
\end{pro}
\begin{proof}
	From \eqref{eq:conexionC1LST} we deduce
	\begin{align*}
	\boldsymbol{\mathcal J}_{G,\hat C^{[1]}W_G}-
\hat H\big(\hat S_2\big)^{-\top}\mathcal B_G \boldsymbol{\mathcal J}_{G,\chi}=\omega	
	\boldsymbol{\mathcal J}_{ G,C^{[1]}}.
	\end{align*}
	Recalling  \eqref{eq:checkCPLST} we obtain
	\begin{align*}
\prodint{   \hat P^{[1]}(x)W_C(x),(\xi)_x}\mathcal W_G
-
\hat H\big(\hat S_2\big)^{-\top}\mathcal B_G\boldsymbol{\mathcal J}_{G,\chi}=\omega	
	\boldsymbol{\mathcal J}_{G, C^{[1]}}.
	\end{align*}
	Therefore, using  \eqref{conex2LST} we conclude
	\begin{align}\label{eq:WomegaCPLST}
\omega
\big(	\boldsymbol{\mathcal J}_{G,C^{[1]}}-\prodint{   P^{[1]}(x),(\xi)_x}\mathcal W_G\big)	=-			\hat H\big(\hat S_2\big)^{-\top}\mathcal B_G \boldsymbol{\mathcal J}_{G,\chi}.
	\end{align}
	Observe that \eqref{conex2LST} also implies
	\begin{align*}
	\omega \boldsymbol{\mathcal J}_{C,P^{[1]}}=0.
	\end{align*}
\end{proof}
	Given a block matrix $A$ we denote by $A_{[N],[M]}$ the truncation obtained by taking the first $N$ block rows and the first $M$ first block columns.
\begin{ma}
	The following relation
	\begin{align}\label{eq:WomegaCPLST2}
	\omega_{[N_G],[N_C+N_G]}
	\begin{bmatrix}
	\boldsymbol{\mathcal J}_{C,P^{[1]}},\boldsymbol{\mathcal J}_{G,C^{[1]}}-\prodint{   P^{[1]} (x),(\xi)_x}\mathcal W_G
	\end{bmatrix}_{[N_C+N_G]}	\diag(I_{N_C},\mathcal R_{G})=
\begin{bmatrix} 0_{N_Gp\times N_Cp},
	-\hat H_{[N_G]}\big(\hat S_2\big)_{[N_G]}^{-\top}
\end{bmatrix}
	\end{align}
	holds.
\end{ma}
\begin{ma}
	The matrix	$\big(\boldsymbol{\mathcal J}_{C,P^{[1]}}\big)_{[N_C]}$ is nonsingular.
\end{ma}
\begin{proof}
	It follows immediately from
	\begin{align*}
	\big(\boldsymbol{\mathcal J}_{C,P^{[1]}}\big)_{[N_C]}=\big(S_1\big)_{[N_C]} \mathcal R_C,
	\end{align*}
	and the fact that $\mathcal R_C$ is nonsingular.
\end{proof}

 \begin{ma}
 	The equation
 { \small\begin{multline*}
  -\begin{bmatrix}
  \boldsymbol{\mathcal J}_{C,P^{[1]}},\boldsymbol{\mathcal J}_{G,C^{[1]}}-\prodint{   P^{[1]}(x),(\xi)_x}\mathcal W_G
  \end{bmatrix}_{[N_C+N_G]}	\diag(I_{N_C},\mathcal R_{G})\\=\begin{multlined}[t][0.75\textwidth]
  	\left[\begin{array}{c|c} I_{[N_C]} &0_{pN_C\times pN_G}\\\hline
  \multicolumn{2}{c}{	\omega_{[N_G],[N_C+N_G]} }
  \end{array}\right]^{-1}\begin{bmatrix}
  \big(\boldsymbol{\mathcal J}_{C,P^{[1]}}\big)_{[N_C]} & 0_{pN_C\times pN_G}\\
  0_{pN_G\times pN_C} & \hat H_{[N_G]}
  \end{bmatrix}\\\times
  \begin{bmatrix} I_{[N_C]} & \big(\big(\boldsymbol{\mathcal J}_{C,P^{[1]}}\big)_{[N_C]}\big)^{-1}\big(\boldsymbol{\mathcal J}_{G,C^{[1]}}-
\prodint{   P^{[1]}(x),(\xi)_x}\mathcal W_G\big)_{[N_C],[N_G]}\mathcal R_G\\0_{N_Gp\times N_Cp}&\big(\hat S_2\big)_{[N_G]}^{-\top}
  \end{bmatrix}
  \end{multlined}
  \end{multline*}}
  is fulfilled.
 \end{ma}
\begin{proof}
\eqref{eq:WomegaCPLST2} can be written as
 \begin{multline*}
 -	\left[\begin{array}{c|c} I_{[N_C]} &0_{pN_C\times pN_G}\\\hline
 \multicolumn{2}{c}{	\omega_{[N_G],[N_C+N_G]} }
 \end{array}\right]\begin{bmatrix}
 \boldsymbol{\mathcal J}_{C,P^{[1]}},\boldsymbol{\mathcal J}_{G,C^{[1]}}-\prodint{   P^{[1]}(x),(\xi)_x}\mathcal W_G
 \end{bmatrix}_{[N_C+N_G]}	\diag(I_{N_C},\mathcal R_{G})\\=	\begin{bmatrix}
 \big(\boldsymbol{\mathcal J}_{C,P^{[1]}}\big)_{[N_C]}& 0_{pN_C\times pN_G}\\
 0_{pN_G\times pN_C} & \hat H_{[N_G]}
 \end{bmatrix}
 \begin{bmatrix} I_{[N_C]} & \big(\big(\boldsymbol{\mathcal J}_{C,P^{[1]}}\big)_{[N_C]}\big)^{-1} \big(\boldsymbol{\mathcal J}_{G,C^{[1]}}-
\prodint{   P^{[1]} (x),(\xi)_x}\mathcal W_G\big)_{[N_C],[N_G]}\mathcal R_G\\0_{N_Gp\times N_Cp}&\big(\big(\hat S_2\big)_{[N_G]}\big)^{-\top}
 \end{bmatrix},
 \end{multline*}
and the result follows.
\end{proof}

\begin{ma}\label{lemma:<NLST}
	For  $n\in\{1,\dots,N_G-1\}$ we find the following expressions
	\begin{align*}
	\hat H_{n}&=-\Theta_*\begin{bmatrix} \boldsymbol{\mathcal J}_{C,P_0^{[1]}} &
	\big(\boldsymbol{\mathcal J}_{G,C^{[1]}_0}-\prodint{   P^{[1]}_{0} (x),(\xi)_x}\mathcal W_G\big)\mathcal R_{G,n}\\
	\vdots &\vdots\\	\boldsymbol{\mathcal J}_{C,P_{N_C+n-1}^{[1]}}& 		\big(\boldsymbol{\mathcal J}_{G,C^{[1]}_{n-1}}-
\prodint{   P^{[1]}_{n+N_C-1} (x),(\xi)_x}\mathcal W_G\big)\mathcal R_{G,n}
	\end{bmatrix}, \\
	\omega_{n,k}&=\Theta_*\left[\begin{array}{c|c}
	\begin{matrix}
	\boldsymbol{\mathcal J}_{C,P_0^{[1]}} &\big(\boldsymbol{\mathcal J}_{G,C^{[1]}_0}-\prodint{   P^{[1]}_{0} (x),(\xi)_x}\mathcal W_G\big)\mathcal R_{G,n}\\
	\vdots&\vdots\\
\boldsymbol{\mathcal J}_{C,P_{N_C+n}^{[1]}} &	\big(\boldsymbol{\mathcal J}_{G,C^{[1]}_{N_C+n}}-\prodint{   P^{[1]}_{n+N_C} (x),(\xi)_x}\mathcal W_G\big)\mathcal R_{G,n}\\	
	\end{matrix}& e_k
	\end{array}\right],& 0&\leq k<N_C+n,\\	\big((\hat S_2)^\top\big)_{n,k}&=\Theta_*\begin{bmatrix}
		\boldsymbol{\mathcal J}_{C,P_0^{[1]}}  &\big(\boldsymbol{\mathcal J}_{G,C^{[1]}_0}-\prodint{   P^{[1]}_{0} (x),(\xi)_x}\mathcal W_G\big)\mathcal R_{G,n+1}\\
\vdots &	\vdots\\
\boldsymbol{\mathcal J}_{C,P_{N_C+n-1}^{[1]}}& 		\big(\boldsymbol{\mathcal J}_{G,C^{[1]}_{n-1}}-
\prodint{   P^{[1]}_{n+N_C-1} (x),(\xi)_x}\mathcal W_G\big)\mathcal R_{G,n+1}\\[5pt]
\multicolumn{2}{c}{	(e_{N_C+k})^\top}
	\end{bmatrix},& 0&\leq k<N_G.
	\end{align*}
\end{ma}

\begin{teo}[Spectral Christoffel--Geronimus--Uvarov  formulas]\label{teo:SCGUN}
	For  $n< N$ and monic Geronimus--Uvarov perturbations, with masses as described in \eqref{uva},
	we have the following last quasideterminant expressions for the perturbed biorthogonal matrix polynomials
	\begin{align}\label{eq:checkP1LST}
	\hat  P^{[1]}_{n}(x)W_C(x)&=
	\Theta_*\begin{bmatrix}
	\boldsymbol{\mathcal J}_{C,P_0^{[1]}} &\big(\boldsymbol{\mathcal J}_{G,C^{[1]}_0}-\prodint{   P^{[1]}_{0} (x),(\xi)_x}\mathcal W_G\big)\mathcal R_{G,n}&P_n^{[1]}(x)\\
	\vdots&\vdots&\vdots\\
	\boldsymbol{\mathcal J}_{C,P_{N_C+n}^{[1]}} &	\big(\boldsymbol{\mathcal J}_{G,C^{[1]}_{N_C+n}}-\prodint{   P^{[1]}_{n+N_C} (x),(\xi)_x}\mathcal W_G\big)\mathcal R_{G,n}&P^{[1]}_{n+N_C+n}(x)
	\end{bmatrix},
	\\
	\big(\hat P^{[2]}_n(y)\big)^\top&=\Theta_*\begin{bmatrix}
	\boldsymbol{\mathcal J}_{C,P_0^{[1]}}  &\big(\boldsymbol{\mathcal J}_{G,C^{[1]}_0}-\prodint{   P^{[1]}_{0} (x),(\xi)_x}\mathcal W_G\big)\mathcal R_{G,n+1}\\
	\vdots &	\vdots\\
	\boldsymbol{\mathcal J}_{C,P_{N_C+n-1}^{[1]}}& 		\big(\boldsymbol{\mathcal J}_{G,C^{[1]}_{n-1}}-\prodint{   P^{[1]}_{n+N_C-1} (x),(\xi)_x}\mathcal W_G\big)\mathcal R_{G,n+1}\\
	0_{p\times pN_C} & \big(\chi_{[N_G]}(y)\big)^\top
	\end{bmatrix}.
	\end{align}
\end{teo}
\begin{proof}
	From \eqref{conex2LST} and \eqref{eq:bior} for $n<N_G$ we have
	\begin{align*}
	\hat P^{[1]}_{n}(x)W_C(x)&=P^{[1]}_{N_C+n}(x)+\big[\omega_{n,0},\dots,\omega_{n,N_C+n-1}\big]\begin{bmatrix}P_{0}^{[1]}(x)\\\vdots\\
	P_{N_C+n-1}^{[1]}(x)	
	\end{bmatrix},\\
	\big(\hat P^{[2]}_n\big)^\top(y)&=I_py^n+\begin{bmatrix}
	I_p,\dots,I_py^{n-1}
	\end{bmatrix}\begin{bmatrix}
	\big( (\hat S_2)^\top\big)_{0,n}\\\vdots\\
	\big( (\hat S_2)^\top\big)_{n-1,n}
	\end{bmatrix}.
	\end{align*}
	Then,  Lemma \ref{lemma:<N} proves the result.
\end{proof}

\begin{rem}
	For masses  as in \eqref{eq:v_diagonal-2} the previous formulas reduces to
		\begin{align*}%\label{eq:checkP1LST}
		\hat  P^{[1]}_{n}(x)W_C(x)&=
		\Theta_*\begin{bmatrix}
		\boldsymbol{\mathcal J}_{C,P_0^{[1]}} &\big(\boldsymbol{\mathcal J}_{G,C^{[1]}_0}-\mathcal J_{G,P^{[1]}_{0}}\mathcal T\big)\mathcal R_{G,n}&P_n^{[1]}(x)\\
		\vdots&\vdots&\vdots\\
		\boldsymbol{\mathcal J}_{C,P_{N_C+n}^{[1]}} &	\big(\boldsymbol{\mathcal J}_{G,C^{[1]}_{N_C+n}}-\mathcal J_{G,P^{[1]}_{N_C+n}}\mathcal T\big)\mathcal R_{G,n}&P^{[1]}_{n+N_C}(x)
		\end{bmatrix},
		\\
		\big(\hat P^{[2]}_n(y)\big)^\top&=\Theta_*\begin{bmatrix}
		\boldsymbol{\mathcal J}_{C,P_0^{[1]}}  &\big(\boldsymbol{\mathcal J}_{G,C^{[1]}_0}-\mathcal J_{G,P^{[1]}_{0}}\mathcal T\big)\mathcal R_{G,n+1}\\
		\vdots &	\vdots\\
		\boldsymbol{\mathcal J}_{C,P_{N_C+n-1}^{[1]}}& 		\big(\boldsymbol{\mathcal J}_{G,C^{[1]}_{n-1}}-\mathcal J_{G,P^{[1]}_{N_C+n-1}}\mathcal T\big)\mathcal R_{G,n+1}\\
		0_{p\times pN_C} & \big(\chi_{[N_G]}(y)\big)^\top
		\end{bmatrix}.
		\end{align*}
\end{rem}

\subsection{Mixed spectral/nonspectral Christoffel--Geronimus--Uvarov formulas}
We consider the Geronimus--Uvarov transformation as a composition of a first Geronimus transformation, applying nonspectral techniques, and then a Christoffel transformation, for which we can apply spectral techniques. In this situation, we still require  the leading coefficient $A_{C,N_C}$ to be nonsingular, i.e. similar to monic, but we free $W_G(x)$ of such a condition.
Thus, we consider
\begin{align}\label{uvaLST}%\label{new}
\check u_{x,y}&:= u_{x,y}(W_G(y))^{-1}+v_{x,y},
\end{align}
with $v_{x,y}$ a matrix of bivariate generalized functions such that $v_{x,y}W_G(y)=0_{ p}$
so that, after a Christoffel transformation, the  Geronimus--Uvarov transformation is achieved
\begin{align*}
\hat u_{x,y}=W_C(x)\check u_{x,y}.
\end{align*}
As for the Geronimus case, we consider

\begin{defi}\label{def:RLST}
	For a given perturbed matrix of functionals $\check u$ we define a semi-infinite block matrix of the form
	\begin{align*}
	R&:=	\prodint{P^{[1]}(x),\chi(y)}_{\check u}\\
	&=	\prodint{P^{[1]}(x),\chi(y)}_{u W_G^{-1}}+	\prodint{P^{[1]}(x),\chi(y)}_{ v_G}.
	\end{align*}
\end{defi}
\begin{pro}
	The equations
	\begin{align}\label{eq:RSMLST}
	R&=S_1\check G,\\
	\omega R&=\hat  H \big(\hat S_2\big)^{-\top}\label{OmegaR0LST}
	\end{align}
	are fulfilled.
\end{pro}
\begin{proof}
	Relation \eqref{eq:RSMLST} is a  consequence of its definition.
	Let us show \eqref{OmegaR0LST}
	 \begin{align*}
	\omega R&= \hat S_1 W_C(\Lambda) \big(S_1\big)^{-1} S_1 \check G\\
	&=\hat S_1W_C(\Lambda )\check G\\
	&=\hat S_1 \hat G\\
	&=\hat H \big(\hat S_2\big)^{-\top}.
	\end{align*}
\end{proof}

Thus, we conclude
\begin{pro}\label{OmegaRLST}
	The matrix $R$ satisfies 
	\begin{align*}
	(\omega R)_{n,l}=\begin{cases}
	0_p, &l\in\{0,\dots,n-1\},\\
	\hat H_n, & n=l.
	\end{cases}
	\end{align*}
Moreover, the matrix
	$\begin{bsmallmatrix}
	R_{0,0} & \dots & R_{0,n-1}\\
	\vdots & & \vdots\\
	R_{n-1,0} & \dots & R_{n-1,n-1}
	\end{bsmallmatrix}$
	is nonsingular.
\end{pro}

As for the pure Geronimus situation we consider
\begin{defi}
We introduce the matrix polynomials $r^K_{n,l}(y)\in\mathbb C^{p\times p}[y]$, $l\in\{0,\dots,n-1\}$, given by
	\begin{align*}
	r^K_{n,l}(z):&=\prodint{W_G(z)K_{n-1}(x,z),I_py^l}_{\check u}-I_pz^l\\ &=
\prodint{W_G(z)K_{n-1}(x,z),I_py^l}_{ uW_G^{-1}}+\prodint{W(z)K_{n-1}(x,z),I_py^l}_{v_G}-	I_p	z^l.
	\end{align*}
\end{defi}

\begin{pro}
	For $l\in\{0,1,\dots,n-1\}$	and $n\geq N_G$ we have
	\begin{align}\label{eq:gammaLST}
	r^K_{n,l}(z)=	\begin{bmatrix}
	\big(\hat P_{n-m_C}^{[2]}(z)\big)^\top(\hat H_{n-m_C})^{-1},\dots,\big(\hat P_{n+N_G-1}^{[2]}(z)\big)^\top(\hat H_{n+N_G-1})^{-1}
	\end{bmatrix}\begin{bmatrix}0_{m_Cp\times  N_G p} &
	-\Omega^C{[n]} \\
	\Omega^G{[n]}	& 0_{N_Gp\times N_Cp}
	\end{bmatrix}	\begin{bmatrix}	R_{n-N_G,l}\\
	\vdots\\
	R_{n+N_C-1,l}
	\end{bmatrix}.
	\end{align}
\end{pro}
\begin{proof}
	It follows from  \eqref{eq:KLST}, Definition \ref{def:RLST} and  \eqref{eq:K-u}.
\end{proof}

\begin{defi}\label{def:Phi}
	For $n\geq N_G$, let us assume that the matrix
\begin{align*}
\Phi_n:=	\begin{bmatrix}
	\boldsymbol{\mathcal J}_{C,P^{[1]}_{n-N_G}}&R_{n-N_G,0} & \dots &  R_{n-N_G,n-1} \\ 	\vdots&\vdots&&\vdots\\ \boldsymbol{\mathcal J}_{C,P^{[1]}_{n+N_C-1}}&	R_{n+N_C-1,0}&\dots &R_{n+N_C-1,n-1}
	\end{bmatrix}\in\mathbb C^{(N_C+N_G)p\times (N_C+n)p}
\end{align*}
is full rank, i.e., $\operatorname{rank}( \Phi_n)=(N_C+N_G)p$. Then, we know that   nonsingular square submatrices $\Phi_n^{\square}\in\mathbb C^{(N_C+N_G)p\times (N_C+N_G)p}$ of $\Phi_n$ exist, and we will refer to them as poised submatrices.
We also consider
\begin{align*}
\varphi_n:=\begin{bmatrix}\boldsymbol{\mathcal J}_{C,P^{[1]}_{n+N_C}},
R_{n+N_C,0},\dots, R_{n+N_C,n-1}
\end{bmatrix} \in\mathbb C^{p\times (N_C+n)p}
\end{align*}
and the  \emph{square} submatrices $\varphi_n^\square\in\mathbb C^{p\times (N_C+N_G)p}$ corresponding to the selection of columns to built  the poised submatrix $\Phi^\square_n$.
We also consider
\begin{align*}
\varphi_n^K(y)&=
\begin{bmatrix}
	W_G(y) \boldsymbol{\mathcal J}_{C, K_{n-1}}(y), R^K_{n,0}(y),\dots,R^K_{n,n-1}(y)
\end{bmatrix}\in\mathbb C^{p\times (N_C+n)p}[y]
\end{align*}
and $(\varphi^K_n(y))^\square$.
\end{defi}

\begin{pro}\label{pro:poisedLST}
	If $A_{G,N}$ is nonsingular, then
	\begin{align*}
	\begin{bmatrix}
	\boldsymbol{\mathcal J}_{C,P^{[1]}_{n-N_G}}& R_{n-N_G,0}& \dots& R_{n-N_G,N_G-1}\\\vdots&\vdots& &\vdots\\
	\boldsymbol{\mathcal J}_{C,P^{[1]}_{n+N_C-1}}& 	R_{n+N_C-1,0}&\dots &R_{n+N_C-1,N_G-1}
	\end{bmatrix}
	\end{align*}
	is nonsingular.
\end{pro}

\begin{proof}
	From Proposition \ref{OmegaRLST} we deduce the system
	\begin{align*}
	\big[\omega_{n,n-N_G},\dots,\omega_{n,n+N_C-1}\big]\begin{bmatrix}
	R_{n-N_G,l}\\\vdots\\
	R_{n+N_C-1,l}
	\end{bmatrix}=-R_{n+N_C,l},
	\end{align*}
	for $l\in\{0,1,\dots,n-1\}$.
	In particular,  the resolvent vector $\big[\omega_{n,n-N},\dots,\omega_{n,n-1}\big]$ is a solution of the linear system
	\begin{multline}\label{eq:systemLST}
	\big[\omega_{n,n-N_G},\dots,\omega_{n,n+N_C-1}\big]\begin{bmatrix}
	\boldsymbol{\mathcal J}_{C,P^{[1]}_{n-N_G}}& R_{n-N_G,0}& \dots& R_{n-N_G,N_G-1}\\\vdots&\vdots& &\vdots\\
	\boldsymbol{\mathcal J}_{C,P^{[1]}_{n+N_C-1}}& 	R_{n+N_C-1,0}&\dots &R_{n+N_C-1,N_G-1}
	\end{bmatrix}\\=-\begin{bmatrix} \boldsymbol{\mathcal J}_{C,P^{[1]}_{n+N_C}},
	R_{n+N_C,0},\dots,R_{n+N_C,N_G-1}
	\end{bmatrix}.
	\end{multline}
	Let us discuss  the uniqueness of the solutions of this  linear system. Assume that we have another solution \begin{align*}
	\big[\tilde\omega_{n,n-N_G},\dots,\tilde\omega_{n,n+N_C-1}\big].
	\end{align*}
	Then, consider  the  monic matrix polynomial
	\begin{align*}
	Q_{n+N_C}(x)=P^{[1]}_{n+N_C}(x)+\tilde\omega_{n,n+N_C-1}P^{[1]}_{n+N_C-1}(x)+\dots+\tilde\omega_{n,n-N_G}P^{[1]}_{n-N_G}(x).
	\end{align*}
	Because $\big[\tilde\omega_{n,n-N},\dots,\tilde\omega_{n,n-1}\big]$ solves \eqref{eq:systemLST} we conclude the two important relations
	\begin{align}
	\label{eq:important1}	\boldsymbol{\mathcal J}_{C,Q_{n+N_C}}&= 0_{p,N_Cp},\\
	\label{eq:important2}		\langle Q_{n+N_C}(x), I_py^l\rangle_{\check u}&=0_p, & l\in\{0,\dots, N_G-1\}.
	\end{align}
	Using Lemma \ref{lemma:pair}, \eqref{eq:important1} can be expressed as follows
	\begin{align*}
	\boldsymbol{\mathcal J}_{C,Q_{n+N_C}}=\begin{bmatrix}
	\big(Q_{n+N_C}\big)_0,\dots, \big(Q_{n+N_C}\big)_{n+N_C}
	\end{bmatrix}\begin{bmatrix}
	X_C\\X_CJ_C\\\vdots\\X_C(J_C)^{n+N_C}
	\end{bmatrix}&= 0_{p,N_Cp},
	\end{align*}
	where $(X_C,J_C)$ is a Jordan pair for the perturbing polynomial $W_C(x)$. According to Corollary 3.8 in \cite{lan1}  this is a necessary and sufficient condition for $W_C(x)$ to be a right divisor of the polynomial $Q_{n+N_C}(x)$, so that we can write
	\begin{align*}
	Q_{n+N_C}(x)=\tilde P_n (x)W_C(x),
	\end{align*}
	where $\tilde P_n (x)$ is a degree $n$ monic polynomial. Then, \eqref{eq:important2} reads
	\begin{align*}
	\langle \tilde P_n(x), I_py^l\rangle_{\hat u}&=0_p, & l\in\{0,\dots, N_G-1\}.
	\end{align*}
	Let us  proceed as we did in Proposition \ref{pro:first_poised}. We first notice that 	Lemma  \ref{lemma:Euclides} can be applied again to get
	%implies  the following relations for $\deg \alpha_l<m$,
	\begin{align*}
	%R_{m,l}&=
	\langle P^{[1]}_m(x), I_p y^l\rangle_{\check  u}
	&=\langle P^{[1]}_m(x), \beta_l(y)\rangle_{\check u},
	\end{align*}
	for  $l<m+N_G$.
	Hence,  when  $l\in\{0,\dots,n-1\}$ we have
	\begin{align*}
	\langle\tilde P_n(x), I_p y^l\rangle_{\hat  u}&= \langle P^{[1]}_{N_C+n}(x), I_p y^l\rangle_{\check u}+\tilde\omega_{n,n+N_C-1}\langle  P^{[1]}_{n+N_C-1}(x), I_p y^l\rangle_{\check u}+\dots+\tilde\omega_{n,n-N_G}\langle P^{[1]}_{n-N_G}(x)
	, I_p y^l\rangle_{\check u}\\
	&= \langle P^{[1]}_{N_C+n}(x), \beta_l(y)\rangle_{\check u}+\tilde\omega_{n,n+N_C-1}\langle  P^{[1]}_{n+N_C-1}(x), \beta_l(y)\rangle_{\check u}+\dots+\tilde\omega_{n,n-N_G}\langle P^{[1]}_{n-N_G}(x)
	, \beta_l(y)\rangle_{\check u}\\
	&=\sum_{k=0}^{N_G-1}\big(R_{n+N_C,k}+\tilde{\omega}_{n,n+N_C-1}R_{n+N_C-1,k}+\dots+ \tilde{\omega}_{n,n-N_G}R_{n-N_G,k} \big)    (\beta_{l,k})^\top\\
	&=0_p.
	\end{align*}
	Now, the uniqueness of  the biorthogonal polynomial families implies
	\begin{align*}
	\tilde P_n(x)=\hat  P^{[1]}_n(x),
	\end{align*}
	and, considering  \eqref{conex2LST},  we infer  that there is a unique solution of \eqref{eq:systemLST}. Thus,
	\begin{align*}
	\begin{bmatrix}
	\boldsymbol{\mathcal J}_{C,P^{[1]}_{n-N_G}}& R_{n-N_G,0}& \dots& R_{n-N_G,N_G-1}\\\vdots&\vdots& &\vdots\\
	\boldsymbol{\mathcal J}_{C,P^{[1]}_{n+N_C-1}}& 	R_{n+N_C-1,0}&\dots &R_{n+N_C-1,N_G-1}
	\end{bmatrix}
	\end{align*}
	is nonsingular and, therefore,  it is a poised submatrix.
\end{proof}

\begin{pro}\label{pro:como_es_OmegaLST}
	 	For $n\geq N_G$ and a full rank matrix $\Phi_n$, let us take a poised submatrix $\Phi_n^\square$. Then,
	\begin{align*}
	\begin{bmatrix}
	\omega_{n,n-N_G},\dots,\omega_{n,n+N_C-1}
	\end{bmatrix}=-\varphi^\square_n \big(\Phi_n^\square\big)^{-1}.
	\end{align*}
\end{pro}
\begin{proof}
	From  Proposition \ref{OmegaRLST}  we get, for $n\geq N_G$,
		\begin{align*}
	\begin{bmatrix}
		\omega_{n,n-N_G},\dots,\omega_{n,n+N_C-1}
		\end{bmatrix}\begin{bmatrix}
	 R_{n-N_G,0} & \dots &  R_{n-N_G,n-1} \\
	 \vdots & &\vdots\\
	 	 R_{n+N_C-1,0} & \dots &  R_{n+N_C-1,n-1}
		\end{bmatrix}=-\begin{bmatrix}
		R_{n+N_C,0},\dots, R_{n+N_C,n-1}
		\end{bmatrix}.
		\end{align*}
		Using \eqref{eq:JP} we can extend this equation to
	\begin{align*}%\label{eq:JP}
	\big[\omega_{n,n-N_G},\dots,\omega_{n,n+N_C-1}\big]	
	\Phi_n=-\varphi_n,
	\end{align*}
	and the result follows.
\end{proof}

\begin{teo}[Mixed spectral/nonspectral matrix Christoffel--Geronimus--Uvarov formulas]\label{theorem:nonspectralLST}
	Given a  matrix Geronimus--Uvarov transformation the corresponding  perturbed polynomials can be expressed,
	for $n\geq  N_G$, as follows
	\begin{align*}
	\hat  P^{[1]}_n(x)W_C(x)&=\Theta_*\left[\begin{array}{c|c}
	\Phi^\square_n &\begin{matrix}P^{[1]}_{n-N_G}(x)\\\vdots \\P^{[1]}_{n+N_C-1}(x)\end{matrix}\\\hline
	\varphi^\square_n & P^{[1]}_{n+N_C}(x)
	\end{array}\right], & \big(	\hat  P_n^{[2]}(y)\big)^\top A_N=-\Theta_*\left[\begin{array}{c|c}\Phi_n^\square & \begin{matrix}
	H_{n-N_G}\\ 0_p\\\vdots \\0_p
	\end{matrix}\\\hline
	\big(\varphi^K_n(y)\big)^\square& 0_p
	\end{array}	
	\right].
	\end{align*}
	The corresponding quasitau matrices are
	\begin{align*}
	\hat H_n&=\Theta_*\left[\begin{array}{c|c}
	\Phi_n^\square&\begin{matrix}R_{n-N_G,n}\\\vdots \\R_{n+N_C-1,n}\end{matrix}\\\hline
	\varphi^\square_n& R_{n+N_C,n}
	\end{array}\right]=\Theta_*\left[\begin{array}{c|c}\Phi_n^\square & \begin{matrix}
			H_{n-N_G}\\ 0_p\\\vdots \\0_p
			\end{matrix}\\\hline
			\varphi_n^\square& 0_p
			\end{array}	
			\right].
	\end{align*}	
\end{teo}
\begin{proof}
	From the connection formula \eqref{conex2LST} we find
	\begin{align*}
	\hat  P^{[1]}_{n}(x)W_C(x)=	[\omega_{n,n-N_G},\dots,\omega_{n,n+N_C-1}]\begin{bmatrix}
	P^{[1]}_{n-N_G}(x)\\\vdots\\P^{[1]}_{n+N_C-1}(x)
	\end{bmatrix}+P^{[1]}_{n+N_C}(x),
	\end{align*}
	and from Proposition \ref{OmegaRLST} we deduce
	\begin{align*}
	\hat H_n=	\
	[\omega_{n,n-N_G},\dots,\omega_{n,n+N_C-1}]\begin{bmatrix}
	R_{n-N_G,n}\\\vdots\\R_{n+N_C-1,n}
	\end{bmatrix}+R_{n+N_C,n},
	\end{align*}
	we will also use \eqref{eq:HG2}. 
	Then, the result for $\hat P^{[1]}_n(x)$ and $\hat H_n$ follows from Proposition \ref{pro:como_es_OmegaLST}.
	
	From 	\eqref {eq:gammaLST} 	and formula
	%s \eqref{eq:KLST} and
	\eqref{eq:WKpcLST2} we get
	\begin{align*}%\label{eq:gammaLST}
	\varphi^K_n(y)=	\begin{bmatrix}
	\big(\hat P_{n-m_C}^{[2]}(y)\big)^\top(\hat H_{n-m_C})^{-1},\dots,\big(\hat P_{n+N_G-1}^{[2]}(y)\big)^\top(\hat H_{n+N_G-1})^{-1}
	\end{bmatrix}\begin{bmatrix}0_{m_Cp\times  N_G p} &
	-\Omega^C{[n]} \\
	\Omega^G{[n]}	& 0_{N_Gp\times N_Cp}
	\end{bmatrix}\Phi_n,
	\end{align*}
so that
	\begin{align}%\label{eq:gammaLST}
	\big(\varphi^K_n(y)\big)^\square\big(\Phi_n^\square\big)^{-1}=	\begin{bmatrix}
	\big(\hat P_{n-m_C}^{[2]}(y)\big)^\top(\hat H_{n-m_C})^{-1},\dots,\big(\hat P_{n+N_G-1}^{[2]}(y)\big)^\top(\hat H_{n+N_G-1})^{-1}
	\end{bmatrix}\begin{bmatrix}0_{m_Cp\times  N_G p} &
	-\Omega^C{[n]} \\
	\Omega^G{[n]}	& 0_{N_Gp\times N_Cp}
	\end{bmatrix}.
	\end{align}
	In particular, recalling \eqref{eq:omegaALST} we deduce that
	\begin{align*}
	(\hat P_{n}^{[2]}(y))^\top A_N=		\big(\varphi^K_n(y)\big)^\square\big(\Phi_n^\square\big)^{-1}\begin{bmatrix}
	H_{n-N}\\ 0_p\\\vdots \\0_p
	\end{bmatrix},
	\end{align*}
	and the expression for the perturbation of the second family of biorthogonal polynomials follows.
\end{proof}

\subsubsection{Discussion for $n<N_G$}
We proceed similarly as we did in \S \ref{susection:<N}. Let us notice that from \eqref{conex2LST} and \eqref{eq:RSMLST} we get
	\begin{align*}
	\omega\begin{bmatrix}
	\boldsymbol{\mathcal J}_{C,P^{[1]}},R
	\end{bmatrix}=\hat H\big(\hat S_2\big)^{-\top}\begin{bmatrix}
	0, I
	\end{bmatrix}.
	\end{align*}
Therefore, we conclude
		\begin{align*}%\label{eq:WomegaCPLST2}
		\omega_{[N_g],[N_C+N_G]}
		\begin{bmatrix}
		\boldsymbol{\mathcal J}_{C,P^{[1]}},R
		\end{bmatrix}_{[N_C+N_G]}	=
		\begin{bmatrix} 0_{N_Gp\times N_Cp},
		-\hat H_{[N_G]}\big(\hat S_2\big)_{[N_G]}^{-\top}
		\end{bmatrix}.
		\end{align*}

Given a block matrix $A$, we denote by $A_{[N],[M]}$ the truncation obtained by taking the first $N$ block rows and the first $M$ block columns.
Then, we easily conclude
\begin{ma} The following Gauss-Borel factorization is fulfilled
\begin{align*}
		\begin{bmatrix}
		\boldsymbol{\mathcal J}_{C,P^{[1]}},R
		\end{bmatrix}_{[N_C+N_G]}=	\left[\begin{array}{c|c} I_{[N_C]} &0_{pN_C\times pN_G}\\\hline
		\multicolumn{2}{c}{	\omega_{[N_G],[N_C+N_G]} }
		\end{array}\right]^{-1}\begin{bmatrix}
		\big(\boldsymbol{\mathcal J}_{C,P^{[1]}}\big)_{[N_C]} & 0_{pN_C\times pN_G}\\
		0_{pN_G\times pN_C} & \hat H_{[N_G]}
		\end{bmatrix}
		\begin{bmatrix} I_{[N_C]} & R_{[N_C],[N_G]}\\0_{N_Gp\times N_Cp}&\big(\hat S_2\big)_{[N_G]}^{-\top}
		\end{bmatrix}.
		\end{align*}
		Therefore, for  $n\in\{1,\dots,N_G-1\}$ we have
	\begin{align*}
	\hat H_{n}&=\Theta_*\begin{bmatrix} \boldsymbol{\mathcal J}_{C,P_0^{[1]}} &
	R_{0,0}&\dots &R_{0,n-1}\\
	\vdots &\vdots&&\vdots\\	\boldsymbol{\mathcal J}_{C,P_{N_C+n-1}^{[1]}}& 		R_{N_C+n-1,0}&\dots &R_{N_C+n-1,n-1}
	\end{bmatrix}, \\
	\omega_{n,k}&=\Theta_*\left[\begin{array}{c|c}
	\begin{matrix}
	\boldsymbol{\mathcal J}_{C,P_0^{[1]}} &	R_{0,0}&\dots &R_{0,n-1}\\
	\vdots&\vdots\\
	\boldsymbol{\mathcal J}_{C,P_{N_C+n}^{[1]}} &		R_{N_C+n,0}&\dots &R_{N_C+n,n-1}\\	
	\end{matrix}& e_k
	\end{array}\right],& 0&\leq k<N_C+n,\\	\big((\hat S_2)^\top\big)_{n,k}&=\Theta_*\begin{bmatrix}
	\boldsymbol{\mathcal J}_{C,P_0^{[1]}}  &	R_{0,0}&\dots &R_{0,n}\\
	\vdots &	\vdots\\
	\boldsymbol{\mathcal J}_{C,P_{N_C+n-1}^{[1]}}& 			R_{N_C+n,0}&\dots &R_{N_C+n,n}\\[5pt]
	\multicolumn{4}{c}{	(e_{N_C+k})^\top}
	\end{bmatrix},& 0&\leq k<N_G.
	\end{align*}
\end{ma}

\begin{teo}[Mixed spectral/nonspectral Christoffel--Geronimus--Uvarov formulas]
	 For $n< N_G$, the perturbed biorthogonal matrix polynomials have the following quasideterminatal expressions
	\begin{align*}
	\hat  P^{[1]}_{n}(x)W_C(x)&=
	\Theta_*\begin{bmatrix}
	\boldsymbol{\mathcal J}_{C,P_0^{[1]}} &R_{0,0}&\dots &R_{0,n-1}&P_n^{[1]}(x)\\
	\vdots&\vdots&&\vdots&\vdots\\
	\boldsymbol{\mathcal J}_{C,P_{N_C+n}^{[1]}} &	R_{N_C+n,0}&\dots &R_{N_C+n,n-1}&P^{[1]}_{n+N_C}(x)
	\end{bmatrix},
	\\
	\big(\hat P^{[2]}_n(y)\big)^\top&=\Theta_*\begin{bmatrix}
	\boldsymbol{\mathcal J}_{C,P_0^{[1]}}  &R_{0,0}&\dots &R_{0,n}\\
	\vdots &	&&\vdots\\
	\boldsymbol{\mathcal J}_{C,P_{N_C+n-1}^{[1]}}& 			R_{N_C+n-1,0}&\dots &R_{N_C+n-1,n}\\
	0_{p\times pN_C} &\multicolumn{3}{c}{\big(\chi(y)_{[N_G]}\big)^\top}
	\end{bmatrix}.
	\end{align*}
\end{teo}

\begin{coro}\label{coro:CPTR}
	If $A_{N,G}$ is nonsingular, then the matrix
	\begin{align}
		\begin{bmatrix}
		\boldsymbol{\mathcal J}_{C,P^{[1]}_{n-N_G}} &	\boldsymbol{\mathcal J}_{G,C^{[1]}_{n-N_G}}-\prodint{   P^{[1]}_{n-N_G} (x),(\xi)_x}\mathcal W_G\\ 	\vdots & \vdots \\ \boldsymbol{\mathcal J}_{C,P^{[1]}_{n+N_C-1}}&\boldsymbol{\mathcal J}_{G, C^{[1]}_{n+N_C-1}}-\prodint{   P^{[1]}_{n+N_C-1} (x),(\xi)_x}\mathcal W_G
		\end{bmatrix}
	\end{align}
	is nonsingular.
\end{coro}
\begin{proof}
	 Proposition 	 \ref{pro:specvsnon} implies for the Geronimus part of the Geronimus--Uvarov transformation that
	\begin{align}\label{eq:spectral non spectralLST}
	\begin{bmatrix}
	\boldsymbol{\mathcal J}_{ G,C^{[1]}_{n-N_G}}-\prodint{   P^{[1]}_{n-N_G} (x),(\xi)_x}\mathcal W_G\\
	\vdots\\
	\boldsymbol{\mathcal J}_{ G,C^{[1]}_{n+N_C-1}}-\prodint{   P^{[1]}_{n+N_C-1} (x),(\xi)_x}\mathcal W_G
	\end{bmatrix}
	=-\begin{bmatrix}
	R_{n-N_G,0}& & R_{n-N_G,N_G-1}\\\vdots& &\vdots\\
	R_{n+N_C-1,0}&\dots &R_{n+N_C-1,N_G-1}
	\end{bmatrix}
	\mathcal R_G,
	\end{align}
	which in turn gives
		\begin{multline}
		\begin{bmatrix}
		\boldsymbol{\mathcal J}_{C,P^{[1]}_{n-N_G}} &	\boldsymbol{\mathcal J}_{G,C^{[1]}_{n-N_G}}-\prodint{   P^{[1]}_{n-N_G} (x),(\xi)_x}\mathcal W_G\\ 	\vdots & \vdots \\ \boldsymbol{\mathcal J}_{C,P^{[1]}_{n+N_C-1}}&\boldsymbol{\mathcal J}_{G, C^{[1]}_{n+N_C-1}}-\prodint{   P^{[1]}_{n+N_C-1} (x),(\xi)_x}\mathcal W_G
		\end{bmatrix}\\=	\begin{bmatrix}
		\boldsymbol{\mathcal J}_{C,P^{[1]}_{n-N_G}}& R_{n-N_G,0}& \dots& R_{n-N_G,N_G-1}\\\vdots&\vdots& &\vdots\\
		\boldsymbol{\mathcal J}_{C,P^{[1]}_{n+N_C-1}}& 	R_{n+N_C-1,0}&\dots &R_{n+N_C-1,N_G-1}
		\end{bmatrix}
		\diag (I_{N_Gp }, - \mathcal R_G).
		\end{multline}
Finally, Proposition  \ref{pro:poisedLST} and Lemma \ref{lemma:triple} give the result.
\end{proof}

\subsection{Applications}

We discuss in the next three subsections some extensions of the previous techniques by using the adjugate matrix. For these we need the original matrix of generalized kernels to be Hankel and, therefore, we are leading with Hankel block matrices.

\subsubsection{Christoffel transformations with singular leading coefficients for Hankel generalized kernels}\label{S:chris}

In  \S \ref{s:unimodular}, we considered  unimodular Christoffel transformations, which have been  broadly studied in the matrix orthogonal polynomials community, under the light of Geronimus transformations. As shown there, despite the nonspectral condition,  the Geronimus transformation can be of help and provide Christoffel type formulas for the perturbed orthogonal matrix polynomials. We extend now these considerations for a  Christoffel transformation with singular leading coefficient and not necessarily unimodular. For the unimodular matrix polynomial perturbation we got results for  arbitrary matrix of generalized kernels $u_{x,y}$. However, for the singular but not unimodular case there is a ticket to pay, namely the non perturbed matrix of generalized kernels must be of Hankel type, i.e., $u_{x,y}=u_{x,x}$.

The idea is to use the adjugate or classical adjoint of   a matrix polynomial $W(x)$, see \cite{hor}, $\operatorname{adj}(W(x))$ defined as the transpose of the  matrix of cofactors. From the Laplace formula we know that
\begin{align*}
W(x)\operatorname{adj}(W(x))=\operatorname{adj}(W(x))W(x)=\det( W(x) )I_p.
\end{align*}
If  $N=\deg (W(x))$, we have that  $\deg\big( \det (W(x))\big)\leq Np$ and
$\deg\big(\operatorname{adj}(W(x))\big)\leq N(p-1)$.
The relations
\begin{align*}
\big(W(x)\big)^{-1} &=\frac{1}{\det (W(x))}\operatorname{adj}(W(x)),&
W(x)&=\det(W(x))\big( \operatorname{adj}(W(x))\big)^{-1}
\end{align*}
 will be instrumental in the sequel.

We study the Christoffel transformation\footnote{Notice that the matrix Christoffel transformation $\hat u= W(x)u$ is a   transposition of this one.} of a Hankel matrix of generalized kernels
\begin{align*}
\hat u_{x,x}= u_{x,x} W(x),
\end{align*}
as the following  massless Geronimus--Uvarov transformation
\begin{align*}
\hat u_{x,x} &= W_C(x) u_{x,x} \big(W_G(x)\big)^{-1}, & W_C(x)&:=I_p\det(W(x)), & W_G(x)&:= \operatorname{adj} (W(x)),
\end{align*}
with $N_C=\deg (W_C(x)) \leq Np$ and $N_G=\deg (W_G(x))\leq N(p-1)$. Let us  remark that is the presence of an non constant determinant $\det(W(x))$ that forces for Hankel matrices $u_{x,x}$.
Therefore, we could apply our results for Geronimus--Uvarov transformations, and  use the mixed  spectral-nonspectral Christoffel--Geronimus--Uvarov formula
of Theorem 	\ref{theorem:nonspectralLST}. That is,  the   perturbed polynomials can be expressed,
	for $n\geq  N_G$, as follows
	\begin{align*}
	\hat  P^{[1]}_n(x)\det(W(x))&=\Theta_*\left[\begin{array}{c|c}
	\Phi^\square_n &\begin{matrix}P^{[1]}_{n-N_G}(x)\\\vdots \\P^{[1]}_{n+N_C-1}(x)\end{matrix}\\\hline
	\varphi^\square_n & P^{[1]}_{n+N_C}(x)
	\end{array}\right], & \big(	\hat  P_n^{[2]}(y)\big)^\top A_N=-\Theta_*\left[\begin{array}{c|c}\Phi_n^\square & \begin{matrix}
	H_{n-N_G}\\ 0_p\\\vdots \\0_p
	\end{matrix}\\\hline
	\big(\varphi^K_n(y)\big)^\square& 0_p
	\end{array}	
	\right].
	\end{align*}
	The corresponding perturbed quasitau matrices are
	\begin{align*}
	\hat H_n&=\Theta_*\left[\begin{array}{c|c}
	\Phi_n^\square&\begin{matrix}R_{n-N_G,n}\\\vdots \\R_{n+N_C-1,n}\end{matrix}\\\hline
	\varphi^\square_n& R_{n+N_C,n}
	\end{array}\right].
	\end{align*}	
Let us recall that	$\Phi_n^{\square}\in\mathbb C^{(N_C+N_G)p\times (N_C+N_G)p}$, see Definition \ref{def:Phi},  is a nonsingular submatrix of
\begin{align*}
	\Phi_n:=	\begin{bmatrix}
	\boldsymbol{\mathcal J}_{C,P^{[1]}_{n-N_G}}&R_{n-N_G,0} & \dots &  R_{n-N_G,n-1} \\ 	\vdots&\vdots&&\vdots\\ \boldsymbol{\mathcal J}_{C,P^{[1]}_{n+N_C-1}}&	R_{n+N_C,0}&\dots &R_{n+N_C,n-1}
	\end{bmatrix}\in\mathbb C^{(N_C+N_G)p\times (N_C+n)p}
	\end{align*}
	and that   $\varphi_n^\square,(\varphi_n^K)^\square\in\mathbb C^{p\times (N_C+N_G)p}$ corresponds to the same selection of columns of
	\begin{align*}
	\varphi_n&=\begin{bmatrix}\boldsymbol{\mathcal J}_{C,P^{[1]}_{n+N_C}},
	R_{n+N_C,0},\dots, R_{n+N_C,n-1}
	\end{bmatrix} \in\mathbb C^{p\times (N_C+n)p},\\
	\varphi_n^K&=
	\begin{bmatrix}
	W_G(y) \boldsymbol{\mathcal J}_{C, K_{n-1}}(y), r^K_{n,0}(y),\dots,r^K_{n,n-1}(y)
	\end{bmatrix}\in\mathbb C^{p\times (N_C+n)p}[y].
	\end{align*}
Now, the  $R$'s are
\begin{align*}
R_{n,m}=\left\langle P^{[1]}_n(x), y^m  \frac{W(y)}{\det(W(y))}\right\rangle_u,
\end{align*}
while the root spectral jet $\boldsymbol {\mathcal J}_{P}$ is
\begin{align*}
\boldsymbol {\mathcal J}_{C,P}=\begin{bmatrix}
P(x_1), P'(x_1),\dots, \dfrac{P^{(\alpha_1-1)}(x_1)}{\alpha_1!},\dots,P(x_q), P'(x_q),\dots, \dfrac{P^{(\alpha_q-1)}(x_q)}{\alpha_q!}
\end{bmatrix},
\end{align*}
where $x_a$ are the eigenvalues, with corresponding multiplicities  $\alpha_a$ , of $W(x)$. Thus, we  have the more explicit expression
\begin{align*}
\Phi=\begin{bmatrix}
P_{n-N_G}^{[1]}(x_1)&\dots&\dfrac{(P_{n-N_G}^{[1]})^{(\alpha_q-1)}(x_q)}{\alpha_q!}& \left\langle P^{[1]}_{n-N_G}(x),  \frac{W(y)}{\det(W(y))}\right\rangle_u & \dots &\left\langle P^{[1]}_{n-N_G}(x),  \frac{y^{n-1}W(y)}{\det(W(y))}\right\rangle_u\\
\vdots &&\vdots&\vdots&&\vdots\\
P_{n+N_C-1}^{[1]}(x_1)&\dots&\dfrac{(P_{n+N_C-1}^{[1]})^{(\alpha_q-1)}(x_q)}{\alpha_q!}& \left\langle P^{[1]}_{n+N_C-1}(x),  \frac{W(y)}{\det(W(y))}\right\rangle_u & \dots &\left\langle P^{[1]}_{n+N_C-1}(x),  \frac{y^{n-1}W(y)}{\det(W(y))}\right\rangle_u\end{bmatrix}
\end{align*}
from where a poised submatrix, which exists for an appropriate selection of columns, must be picked.

\subsubsection{Symmetric perturbations of Hankel  matrices of generalized kernels}\label{s:symmetric}

For  symmetric Hankel matrices of generalized kernels $u$, $u=u^\top$, the biorthogonality becomes an orthogonality condition, and so we have orthogonal matrix polynomials. The perturbations we have considered so far do not respect this  condition.  For the study  of transformations preserving this symmetric condition,  we are lead to the study of transformations
\begin{align*}
\hat u_{x,x} &=W(x)u_{x,x} (W(x))^\top,
\end{align*}
which we call  of the symmetric Christoffel type,
or  (we omit the masses for sake of simplicity) the massless  symmetric Geronimus perturbations
\begin{align*}
\hat u_{x,x} &=(W(x))^{-1}u_{x,x} (W(x))^{-\top}.
\end{align*}
These transformations, that deserve further study,  can be understood using the adjugate technique at the light of the Geronimus--Uvarov transformations techniques we have discussed previously in this paper. We need to assume that the leading coefficient of $W(x)$ is nonsingular, and hence, that spectral techniques could be applied.
For example,  the symmetric Christoffel transformation can be written as the following Geronimus--Uvarov transformation
\begin{align*}
\hat u_{x,x} &=W_C(x)u_{x,x} (W_G(x))^{-1}, & W_C(x)&:=(\det W(x))W(x), & W_G(x):=(\operatorname{adj}(W(x)))^\top,
\end{align*}
with polynomial degrees $N_C= N(p+1)$
and $N_G= N(p-1)$.
The massless symmetric Geronimus transformation  can be  understood as the following Geronimus--Uvarov transformation
\begin{align*}
\hat u_{x,x} &=W_C(x)u_{x,x} (W_G(x))^{-1}, & W_C(x)&:=\operatorname{adj}(W(x)), & W_G(x):=(\det W(x))|(W(x))^\top.
\end{align*}
Now,  the degrees are $N_C= N(p-1)$, $N_G= N(p+1)$, respectively.
Observe that for the  symmetric Christoffel transformation we have $\det W_C(x)=(\det W(x))^{p+1}$, so that eigenvalues coincide but multiplicities are multiplied by $p$. The same happens for the massless symmetric Geronimus transformations and $W_G(x)$. Notice also that as $\det (\operatorname{adj} (W(x)))=(\det W(x))^{p-1}$,  the eigenvalues of $W_G(x)$ (of $W_C(x)$) in the symmetric Christoffel   (in the symmetric Geronimus) are those $W(x)$ but with multiplicities multiplied by $p+$. Then, the massless spectral Christoffel--Geronimus--Uvarov formulas of Theorem \ref{teo:SCGU} can be applied with the corresponding perturbing polynomials $W_C(x) $ and $W_G(x)$.

\subsubsection{Nonsymmetric perturbations of Hankel matrices of generalized kernels}\label{s:nsymmetric}

We could consider a more general situation  of the following nonsymmetric form, with the perturbing matrix polynomial $W(x)$ having a nonsingular leading coefficient,
\begin{align*}
\hat u &=W(x)u V(x)=W_C(x)u (W_G(x))^{-1}, &W_C(x)&=(\det V(x))W(x), & W_G(x) &=\operatorname{adj}(V(x)),
\end{align*}
or
\begin{align*}
\hat u& =(W(x))^{-1}u (V(x))^{-1}=W_C(x)u (W_G(x))^{-1}, &W_C(x)&=\operatorname{adj}(W(x)), & W_G(x)&= (\det W(x))V(x).
\end{align*}
In this case the polynomial $V(x)$ can  have  a singular leading coefficient, and when this happens we apply the mixed spectral/nonspectral Christoffel--Geronimus--Uvarov formulas.
	
	\subsubsection{Degree one Geronimus--Uvarov transformations}

We consider  degree one perturbing polynomials
\begin{align*}
W_C(x)&=xI_p-A_C, & W_G(x)&=xI_p-A_G,
\end{align*}
and no masses.
For the  Jordan pairs $(X_C,J_C)$ and $(X_G,J_G)$ we have  $A_C=X_CJ_C(X_C)^{-1}$ and $A_G=X_GJ_G(X_G)^{-1}$.
Now, from Theorem \ref{teo:SCGU}
	we deduce, for $n\geq N_G$,
	the following last quasideterminant expressions
	\begin{align*}
	\hat P^{[1]}_{n}(x)(xI_p-A_C)&=
	\Theta_*
	\begin{bmatrix}
	P^{[1]}_{n-1}(A_C) X_C&	C^{[1]}_{n-1}(A_G) X_G&P_{n-1}^{[1]}(x)\\
 P^{[1]}_{n}(A_C) X_C &	C^{[1]}_{n}(A_G) X_G &P_{n}^{[1]}(x)\\
 P^{[1]}_{n+1}(A_C) X_C &	C^{[1]}_{n+1}(A_G) X_G &P_{n+1}^{[1]}(x)
	\end{bmatrix},\\
	\hat H_{n}&=\Theta_*
		\begin{bmatrix}
	P^{[1]}_{n-1}(A_C) X_C&	C^{[1]}_{n-1}(A_G) X_G&H_{n-1}\\
 P^{[1]}_{n}(A_C) X_C &	C^{[1]}_{n}(A_G) X_G &0_p\\
 P^{[1]}_{n+1}(A_C) X_C &	C^{[1]}_{n+1}(A_G) X_G &0_p
	\end{bmatrix},\\
	\big(	\hat  P _{n}^{[2]}(y)\big)^\top&=-\Theta_*
	\begin{bmatrix}
	P^{[1]}_{n-1}(A_C) X_C&	C^{[1]}_{n-1}(A_G) X_G&H_{n-1}\\
 P^{[1]}_{n}(A_C) X_C &	C^{[1]}_{n}(A_G) X_G &0_p\\
 (yI_p-A_G)K_{n-1}(A_C,y) X_C &	\big(yI_p-A_G)K^{(pc)}_{n-1}(A_G,y) +I_p\big) X_G &0_p
	\end{bmatrix}.
	\end{align*}
If we expand the quasideterminant we get
\begin{align*}
\hat P^{[1]}_{n}(x)(xI_p-A_C)
=&P^{[1]}_{n+1}(x)
-C^{[1]}_{n+1}(A_G)\Big(C^{[1]}_{n}(A_G)- P^{[1]}_{n}(A_C)\big(P^{[1]}_{n-1}(A_C)\big)^{-1}C^{[1]}_{n-1}(A_G)\Big)^{-1}P^{[1]}_{n}(x)\\
&-C^{[1]}_{n+1}(A_G)\Big(C^{[1]}_{n-1}(A_G)- P^{[1]}_{n-1}(A_C)\big(P^{[1]}_{n}(A_C)\big)^{-1}C^{[1]}_{n}(A_G)\Big)^{-1}P^{[1]}_{n-1}(x)\\
&-P^{[1]}_{n+1}(A_C)\Big(P^{[1]}_{n}(A_C)- C^{[1]}_{n}(A_G)\big(C^{[1]}_{n-1}(A_G)\big)^{-1}P^{[1]}_{n-1}(A_C)\Big)^{-1}P^{[1]}_{n}(x)\\
&-P^{[1]}_{n+1}(A_C)\Big(P^{[1]}_{n-1}(A_C)- C^{[1]}_{n-1}(A_G)\big(C^{[1]}_{n}(A_G)\big)^{-1}P^{[1]}_{n}(A_C)\Big)^{-1}P^{[1]}_{n-1}(x),\\
\hat H_n=&-\Big(C^{[1]}_{n+1}(A_G)\Big(C^{[1]}_{n-1}(A_G)- P^{[1]}_{n-1}(A_C)\big(P^{[1]}_{n}(A_C)\big)^{-1}C^{[1]}_{n}(A_G)\Big)^{-1}\\&+P^{[1]}_{n+1}(A_C)\Big(P^{[1]}_{n-1}(A_C)- C^{[1]}_{n-1}(A_G)\big(C^{[1]}_{n}(A_G)\big)^{-1}P^{[1]}_{n}(A_C)\Big)^{-1}\Big)H_{n-1},\\
	\big(	\hat  P _{n}^{[2]}(y)\big)^\top=&\Big(	\big(\big(yI_p-A_G)K^{(pc)}_{n-1}(A_G,y) +I_p\big) \Big(C^{[1]}_{n-1}(A_G)- P^{[1]}_{n-1}(A_C)\big(P^{[1]}_{n}(A_C)\big)^{-1}C^{[1]}_{n}(A_G)\Big)^{-1}\\&+(yI_p-A_G)K_{n-1}(A_C,y)  \Big(P^{[1]}_{n-1}(A_C)- C^{[1]}_{n-1}(A_G)\big(C^{[1]}_{n}(A_G)\big)^{-1}P^{[1]}_{n}(A_C)\Big)^{-1}\Big)H_{n-1}.
\end{align*}

\section{Matrix Uvarov transformations}

Uvarov perturbations  for the scalar case and a number of Dirac deltas, has been considered first in \S 2 of \cite{Uva} in the context of orthogonal polynomials with respect to a Stieltjes  integral scalar product.    Then, for the matrix case it was studied in a series of papers
\cite{Yakhlef1,Yakhlef2,Yakhlef3} where we can find the corresponding Christoffel--Geronimus--Uvarov formula for the perturbed polynomials when a  solely Dirac delta is added in a point to a matrix of measures. Now  we present the general case for an additive perturbation,
that has a discrete finite support in the $y$-variable, of  a sesquilinear form. We allow, therefore,  for additive perturbations having, in the $y$-variable, an arbitrary finite number of  derivatives of the Dirac delta at several different points, and arbitrary linear independent generalized functions in the $x$-variable. Despite the result found largely
extends the Christoffel formulas in the papers of Yakhlef and coworkers, we have three more  reasons to discuss this material. Firstly to show how some of the tools, like spectral jets, used in previous sections of the paper also apply in this context,  secondly,
to achieve a more complete account of the family of transformations of Darboux type for matrix orthogonal polynomials and thirdly to apply it to orthogonalities  like that of matrix discrete Sobolev type, see \cite{Marcellan2014Sobolev}.
For the multivariate scenario these transformations have been discussed in \cite{Delgado} and \cite{Delgado2,Ariznabarreta-linear}.

\subsection{Additive perturbations}
We discuss here an interesting formula for additive perturbations of matrix generalized kernels. It will be instrumental when we discuss the matrix Uvarov transformation.
\begin{pro}[Additive perturbation and reproducing kernels]\label{pro:additive}
	Let us consider an additive perturbation of the matrix of  bivariate generalized functions  $u_{x,y}$ of the form
\begin{align*}
\hat u_{x,y}= u_{x,y}+v_{x,y},
\end{align*}
and let us assume that $u_{x,y}$ and $\hat u_{x,y}$ are quasidefinite.
Then,
\begin{align*}
\hat P^{[1]}_{n}(z)&=
P^{[1]}_{n}(z)-\big\langle\hat P^{[1]}_n(x),\big(K_{n-1}(z,y)\big)^\top\big\rangle_v\,,&
(\hat P^{[2]}_{n}(z))^\top&=(P^{[2]}_{n}(z))^\top-\big\langle K_{n-1}(x,z),\hat P^{[2]}_n(y)\big\rangle_v\,,\\
\end{align*}
and
\begin{align*}
\hat H_n&= H_n+\prodint{\hat P^{[1]}_n(x),P^{[2]}_n(y)}_{v}\\
&=H_n+\prodint{ P^{[1]}_n(x),\hat P^{[2]}_n(y)}_{v}.
\end{align*}
\end{pro}
\begin{proof}
	From \eqref{eq:orthogonality1} and \eqref{eq:orthogonality2}	we deduce
\begin{gather}
\label{eq:additive1}\begin{aligned}
\prodint{\hat P^{[1]}_n(x),P^{[2]}_m(y)}_{\hat{u}}&=0_p,&\prodint{P^{[1]}_m(x), \hat P^{[2]}_n(y)}_{\hat u}&= 0_p,  &m&\in\{1,\dots n-1\},
\end{aligned}\\
\label{eq:additive2}
\prodint{\hat P^{[1]}_n(x),P^{[2]}_n(y)}_{\hat u}= \hat H_n= \prodint{ P^{[1]}_n(x), \hat P^{[2]}_n(y)}_{\hat u}.
\end{gather}
Equation \eqref{eq:additive1} can be expressed as
\begin{align*}
\prodint{\hat P^{[1]}_n(x),P^{[2]}_m(y)}_{{u}}&=-\prodint{\hat P^{[1]}_n(x),P^{[2]}_m(y)}_{v},&\prodint{P^{[1]}_m(x), \hat P^{[2]}_n(y)}_{u}&=-\prodint{P^{[1]}_m(x), \hat P^{[2]}_n(y)}_{v},  &m&\in\{1,\dots n-1\}.
\end{align*}
Then, recalling \eqref{eq:CD kernel} we get
\begin{align*}
\prodint{\hat P^{[1]}_n(x),(K_{n-1}(z,y))^\top}_{{u}}&=-\prodint{\hat P^{[1]}_n(x),(K_{n-1}(z,y))^\top}_{v},&\prodint{K_{n-1}(x,z), \hat P^{[2]}_n(y)}_{u}&=-\prodint{K_{n-1}(x,z), \hat P^{[2]}_n(y)}_{v}.
\end{align*}
But, notice that $\hat P^{[1]}_n(x)- P^{[1]}_n(x)$ and $\hat P^{[2]}_n(y)- P^{[2]}_n(y)$ have degree $n-1$ and, therefore, recalling \eqref{eq:reproducing}
we deduce
\begin{align*}
\hat P^{[1]}_n(z)- P^{[1]}_n(z)&=	\prodint{ \hat P^{[1]}_n(x)- P^{[1]}_n(x),(K_{n-1}(z,y))^\top}_u\\
&=\prodint{ \hat P^{[1]}_n(x),(K_{n-1}(z,y))^\top}_u\\
&=-\prodint{ \hat P^{[1]}_n(x),(K_{n-1}(z,y))^\top}_v,\\
\big(\hat P^{[2]}_n(z)- P^{[2]}_n(z)\big)^\top&=	\prodint{ K_{n-1}(x,z),\hat P^{[2]}_n(y)- P^{[2]}_n(y)}_u\\
&=	\prodint{ K_{n-1}(x,z),\hat P^{[2]}_n(y)}_u\\
&=-	\prodint{ K_{n-1}(x,z),\hat P^{[2]}_n(y)}_v.
\end{align*}

Finally, from \eqref{eq:additive2} we get
\begin{align*}
\hat H_n&=\prodint{\hat P^{[1]}_n(x),P^{[2]}_n(y)}_{\hat u}\\
&=\prodint{\hat P^{[1]}_n(x),P^{[2]}_n(y)}_{ u}+\prodint{\hat P^{[1]}_n(x),P^{[2]}_n(y)}_{v}\\
&= H_n+\prodint{\hat P^{[1]}_n(x),P^{[2]}_n(y)}_{v},
\end{align*}
as well as
\begin{align*}
\hat H_n&=\prodint{ P^{[1]}_n(x),\hat P^{[2]}_n(y)}_{\hat u}\\
&=\prodint{ P^{[1]}_n(x),\hat P^{[2]}_n(y)}_{ u}+\prodint{ P^{[1]}_n(x),\hat P^{[2]}_n(y)}_{v}\\
&= H_n+\prodint{ P^{[1]}_n(x),\hat P^{[2]}_n(y)}_{v}.
\end{align*}
\end{proof}

\subsection{Matrix Christoffel--Uvarov formulas for Uvarov additive perturbations}

We consider the following additive Uvarov  perturbation
\begin{align}\label{eq:v_uvarov_general}
\hat u_{x,y}&=u_{x,y}+v_{x,y}, & 
v_{x,y}&=\sum_{a=1}^q\sum_{m=0}^{\kappa^{(a)}-1}\frac{(-1)^m}{m!}\big(\beta^{(a)}_m\big)_x\otimes\delta^{(m)}(y-x_a), & \big(\beta^{(a)}_m\big)_x&\in\big(\mathbb C^{p\times p}[x]\big)',
\end{align}
which has a finite support on the $y$ variable at the set $\{x_a\}_{a=1}^q$.
The set of  linear functionals $\big\{\big(\beta^{(a)}_m\big)_x\big\}_{\substack{a=1,\dots,q\\
		m=0,\dots,\kappa^{(a)}-1}}$ is supposed to be linearly independent in the bimodule $\big(\mathbb C^{p\times p}[x]\big)'$, i.e., the unique solution to
\begin{align*}
\sum_{a=1}^q\sum_{m=0}^{\kappa^{(a)}-1}\big(\beta^{(a)}_m\big)_xX^{(a)}_m &=0, & X^{(a)}_m &\in\mathbb C^{p\times p},
\end{align*}
is $X^{(a)}_m=0_p$,
and to
\begin{align*}
\sum_{a=1}^q\sum_{m=0}^{\kappa^{(a)}-1} Y^{(a)}_m\big(\beta^{(a)}_m\big)_x&=0,& Y^{(a)}_m &\in\mathbb C^{p\times p},
\end{align*}
is $Y^{(a)}_m=0_p$.

We will assume along this subsection that both $u_{x,y}$ and $u_{x,y}+v_{x,y}$ are quasidefinite matrices of bivariate generalized functions.

\begin{defi}
We will  say that the degree of the Uvarov perturbation is $N=\kappa^{(1)}+\dots+\kappa^{(q)}$.
	The spectral jet associated with the finite support matrix of linear functionals $v_{x,y}$ is, for any sufficiently smooth matrix function $f(x)$ defined in an open set in $\mathbb R$ containing the support $\{x_1,\dots,x_q\}$,  the following matrix
	\begin{align*}
	\mathcal J_f=\begin{bmatrix}
f(x_1),\dots,\dfrac{(f(x))_{x_1}^{(\kappa^{(1)}-1)}}{(\kappa^{(1)}-1)!},\dots,f(x_q),\dots,\dfrac{(f(x))_{x_q}^{(\kappa^{(q)}-1)}}{(\kappa^{(q)}-1)!}
	\end{bmatrix}\in\mathbb C^{p\times Np}.
	\end{align*}
	For a  matrix of kernels  $K(x,y)$, we have
		\begin{align*}
		\mathcal J^{[0,1]}_K(x)=\begin{bmatrix}
		K(x,x_1)\\\vdots\\\dfrac{(K(x,y))_{x,x_1}^{(0,\kappa^{(1)}-1)}}{(\kappa^{(1)}-1)!}\\\vdots\\K(x,x_q)\\\vdots\\\dfrac{(K(x,y))_{x,x_q}^{(0,\kappa^{(q)}-1)}}{(\kappa^{(q)}-1)!}
		\end{bmatrix}\in\mathbb C^{Np\times p}.
		\end{align*}
		We also require  the introduction of the following matrices
		 \begin{gather*}
 \prodint{
 	P(x),(\beta)_x}	:=
 \begin{bmatrix}
 \prodint{P(x),\big(\beta^{(1)}_{0}\big)_x },\dots,
 \prodint{P(x),\big(\beta^{(1)}_{\kappa^{(1)}-1}\big)_x},\dots, \prodint{P(x),\big(\beta^{(q)}_{0}\big)_x },\dots,
 \prodint{P(x),\big(\beta^{(q)}_{\kappa^{(q)}-1}\big)_x}
 \end{bmatrix}\in\mathbb C^{ p\times Np}.
 \end{gather*}
\end{defi}

\begin{pro}\label{pro:uvarov0}
	The following relations
				\begin{align*}
				\prodint{\hat P^{[1]}_n(x), (K_{n-1}(z,y))^\top}_v&= \prodint{\hat P_n^{[1]}(x),(\beta)_x}\mathcal  J^{[0,1]}_{K_{n-1}}(z),\\
				\prodint{K_{n-1}(x,z), \hat P^{[2]}_n(y)}_v&=
				\prodint{K_{n-1}(x,z),(\beta)_x}\big(\mathcal  J_{\hat P^{[2]}_n}\big)^\top
				\end{align*}
				are fulfilled.
\end{pro}
\begin{proof}
Direct substitution gives
	\begin{align*}
	\prodint{\hat P^{[1]}_n(x), (K_{n-1}(z,y))^\top}_v&= \sum_{a=1}^q\sum_{m=0}^{\kappa^{(a)}-1}\frac{1}{m!}
	\prodint{\hat P_n^{[1]}(x),\big(\beta^{(a)}_m\big)_x}\Big(K_{n-1}(z,x)\Big)^{(m)}_{x_a},\\
	\prodint{K_{n-1}(x,z), \hat P^{[2]}_n(y)}_v&= \sum_{a=1}^q\sum_{m=0}^{\kappa^{(a)}-1}\frac{1}{m!}	\prodint{K_{n-1}(x,z),\big(\beta^{(a)}_m\big)_x}
	\Big(\big(\hat P^{[2]}_n(y)\big)^\top\Big)^{(m)}_{x_a},
	\end{align*}
	and the result follows.	
\end{proof}

\begin{defi}\label{def:kernels}
	We define the left %and right
	kernel subspace
	\begin{align*}
	\operatorname{Ker}^R_\beta:=\big\{P(x)\in\mathbb C^{p\times p}[x]:
	\prodint{\big(\beta^{(a)}_m\big)_x,P(x)}=0_p, a\in\{1,\dots,q\},
	m\in\{1,\dots,
	\kappa^{(a)}-1\}\big\},
	\end{align*}
	and the bilateral ideal
	\begin{align*}
				\mathbb I :=(x-x_1)^{\kappa^{(1)}}\cdots (x-x_q)^{\kappa^{(q)}}\mathbb C^{p\times p}[x].
				\end{align*}
The corresponding	right and left
	orthogonal complements with respect to the sesquilinear form are
	\begin{align*}
\big(	\operatorname{Ker}^R_\beta\big)^{\perp_u^R}&:=\Big\{
Q(x)\in\mathbb C^{p\times p}[x]: \prodint{P(x),Q(y)}_u=0_p\,\forall P(x)\in\operatorname{Ker}^R_\beta\Big\},
\\
\mathbb I^{\perp_u^R}&:=\Big\{
Q(y)\in\mathbb C^{p\times p}[y]: \prodint{Q(x),P(y)}_u=0_p\,\forall P(x)\in\mathbb I\Big\}.
	\end{align*}
\end{defi}

\begin{teo}[Christoffel--Uvarov formulas]\label{teo:uvarov}
Whenever   one of the two conditions $\big(	\operatorname{Ker}^R_\beta\big)^{\perp_u^R}=\{0_p\}$ or $\mathbb I^{\perp_u^L}=\{0_p\}$ holds, the matrix $I_{Np}+\prodint{\mathcal  J^{[0,1]}_{K_{n}}(x),(\beta)_x}$ is nonsingular and the Uvarov perturbed matrix orthogonal polynomials and quasitau matrices can be expressed as follows
	\begin{align*}
	\hat P^{[1]}_n(x)&=\Theta_*\begin{bmatrix}
	I_{Np}+\prodint{\mathcal  J^{[0,1]}_{K_{n-1}}(x),(\beta)_x}&\mathcal  J^{[0,1]}_{K_{n-1}}(x)\\
\prodint{ P_n^{[1]}(x),(\beta)_x}& P_n^{[1]}(x)
	\end{bmatrix}, \\
	(\hat P^{[2]}_n(y))^\top&=\Theta_*\begin{bmatrix}
I_{Np}+\prodint{\mathcal  J^{[0,1]}_{K_{n-1}}(x),(\beta)_x}&\big(\mathcal J_{ P^{[2]}_{n}}\big)^\top\\
\prodint{K_{n-1}(x,y),(\beta)_x}	& (P_n^{[2]}(y ))^\top
	\end{bmatrix},\\
	\hat H_n&=\Theta_*\begin{bmatrix}
	I_{Np}+\prodint{\mathcal  J^{[0,1]}_{K_{n-1}}(x),(\beta)_x}&- (\mathcal J_{P^{[2]}_n}  )^\top\\
\prodint{ P_n^{[1]}(x),(\beta)_x}& H_n
	\end{bmatrix}.
	\end{align*}
\end{teo}

\begin{proof}
From Propositions \ref{pro:additive} and  \ref{pro:uvarov0} we get
\begin{align}\label{eq:uvarov11}
\hat P^{[1]}_{n}(z)&=
P^{[1]}_{n}(z)-\prodint{\hat P_n^{[1]}(x),(\beta)_x}\mathcal  J^{[0,1]}_{K_{n-1}}(z),\\
(\hat P^{[2]}_{n}(z))^\top&=(P^{[2]}_{n}(z))^\top-
\prodint{K_{n-1}(x,z),(\beta)_x}\big(\mathcal  J_{\hat P^{[2]}_n}\big)^\top,
\label{eq:uvarov12}
\end{align}
which in turn  imply
\begin{align}
\prodint{\hat P^{[1]}_{n}(x),(\beta)_x}&=
\prodint{ P^{[1]}_{n}(x),(\beta)_x}-\prodint{\hat P_n^{[1]}(x),(\beta)_x}\prodint{\mathcal  J^{[0,1]}_{K_{n-1}}(x),(\beta)_x},\\
\big(\mathcal J_{\hat P^{[2]}_{n}}\big)^\top&=\big(\mathcal J_{ P^{[2]}_{n}}\big)^\top
-\prodint{\mathcal J_{K_{n-1}}^{[0,1]},(\beta)_x)}\big(\mathcal  J_{\hat P^{[2]}_n}\big)^\top,\label{eq:nonsing2uvarov1}
\end{align}
and, consequently,
\begin{align}
\label{eq:nonsing1uvarov}
\prodint{\hat P^{[1]}_{n}(x),(\beta)_x}\Big(I_{Np}+\prodint{\mathcal  J^{[0,1]}_{K_{n-1}}(x),(\beta)_x}\Big)&=
\prodint{ P^{[1]}_{n}(x),(\beta)_x},\\\label{eq:nonsing2uvarov}
\Big(I_{Np}+\prodint{\mathcal  J^{[0,1]}_{K_{n-1}},(\beta)_x}\Big)\big(\mathcal  J_{\hat P^{[2]}_n}\big)^\top
&=\big(\mathcal  J_{ P^{[2]}_n}\big)^\top.
\end{align}

Let us check the nonsingularity of  the matrices $I_{Np}+\prodint{\mathcal  J^{[0,1]}_{K_{n-1}}(x),(\beta)_x}$.  If we assume the contrary, we can find a nonzero vector $ X\in \mathbb C^{Np}$ such that
\begin{align*}
\Big(I_{Np}+\prodint{\mathcal  J^{[0,1]}_{K_{n-1}}(x),(\beta)_x}\Big) X=0,
\end{align*}
is the zero vector in $\mathbb C^{Np}$, and, equivalently,  a covector $Y\in(\mathbb C^{Np})^*$ with
\begin{align*}
Y\Big(I_{Np}+\prodint{\mathcal  J^{[0,1]}_{K_{n-1}}(x),(\beta)_x}\Big)=0.
\end{align*}

Thus, using \eqref{eq:nonsing1uvarov} and \eqref{eq:nonsing2uvarov} we conclude that $\prodint{ P^{[1]}_{n}(x),(\beta)_x}X=0$ and $Y\big(\mathcal  J_{\hat P^{[2]}_n}\big)^\top=0$. The definition of the Christoffel--Darboux kernels \eqref{eq:CD kernel} implies
\begin{align*}
I_{Np}+\prodint{\mathcal  J^{[0,1]}_{K_{n}}(x),(\beta)_x}=I_{Np}+\prodint{\mathcal  J^{[0,1]}_{K_{n-1}}(x),(\beta)_x}+
(\mathcal J_{P_n^{[2]}})^\top
(H_n)^{-1}
\prodint{ P^{[1]}_{n}(x),(\beta)_x},
\end{align*}
and, consequently, we deduce that $\Big(I_{Np}+\prodint{\mathcal  J^{[0,1]}_{K_{n}}(x),(\beta)_x}\Big)X=0$ and $Y\Big(I_{Np}+\prodint{\mathcal  J^{[0,1]}_{K_{n}}(x),(\beta)_x}\Big)=0$, so that $\prodint{ P^{[1]}_{n+1}(x),(\beta)_x}X=0$ and $Y(\mathcal J_{P_{n+1}^{[2]}})^\top=0$, and so forth and so on.
Thus, we deduce that $\prodint{ P^{[1]}_{k}(x),(\beta)_x}X=0$ and $Y(\mathcal J_{P_k^{[2]}})^\top=0$ for $k\in\{n,n+1,\dots\}$ . If we write
\begin{align*}
X&=\begin{bmatrix}
X^{(1)}_{0}\\
\vdots\\
X^{(1)}_{\kappa^{(1)-1}}\\
\vdots\\
X^{(q)}_{0}\\
\vdots\\
X^{(q)}_{\kappa^{(q)-1}}
\end{bmatrix},\\
Y&=\begin{bmatrix}
Y^{(1)}_{0},\dots,
Y^{(1)}_{\kappa^{(1)-1}},\dots
Y^{(q)}_{0},\dots,
Y^{(q)}_{\kappa^{(q)-1}}
\end{bmatrix},
\end{align*}
 the linear functionals
\begin{align*}
\beta\cdot X&:=\sum_{a=1}^q\sum_{m=0}^{\kappa^{(a)}-1}\big(\beta^{(a)}_m\big)_xX^{(a)}_m,&
Y\cdot\delta&:=\sum_{a=1}^q\sum_{m=0}^{\kappa^{(a)}-1}Y^{(a)}_m\frac{(-1)^m}{m!}\delta^{(m)}(y-x_a),
\end{align*}
are  such that
\begin{align*}
\langle P^{[1]}_k(x),\beta\cdot X\rangle&=0,&
\prodint{Y\cdot\delta, \big(P^{[2]}_{k}(x)\big)^\top}&=0,
\end{align*}
for $k\in\{n,n+1,\dots\}$. We can say that
\begin{align*}
\beta\cdot X\in \big(\big\{P^{[1]}_k(x)\big\}_{k=n}^\infty\big)^{\perp_R}:=\Big\{\tilde u, \text{ a  matrix of linear functionals such that }  \prodint{P_k^{[1]}(x), \tilde u}=0_p: k\geq n\Big\},\\
Y\cdot \delta\in \Big(\Big\{\big(P^{[2]}_k(x)\big)^\top\Big\}_{k=n}^\infty\Big)^{\perp_L}:=\Big\{\tilde u, \text{ a  matrix of linear functionals such that }  \prodint{\tilde u, (P_k^{[2]}(x))^\top}=0_p: k\geq n\Big\}.
\end{align*}
It is convenient at this point to  recall that  the  topological and algebraic duals of the set of matrix polynomials coincide, $(\mathbb C^{p\times p}[x])'=(\mathbb C^{p\times p}[x])^*=\mathbb C^{p\times p}[\![x]\!]$, where we understand the set of matrix polynomials or matrix formal series as left or right modules over the ring of matrices. We also recall that the module of matrix polynomials of  degree less than or equal to $m$ is a the free module of rank $m+1$.  Therefore, for each positive integer $m$, we consider the  linear basis given by the following set of
matrices of linear functionals $\big\{(P^{[1]}_k)^*\big\}_{k=0}^m$,  dual to $\big\{P^{[1]}_k(x)\big\}_{k=0}^m$,  such that
$ \prodint{P^{[1]}_k(x), (P^{[1]}_l)^*}=\delta_{k,l} I_p$,  and
any linear functional can be written
$\sum\limits_{k=0}^m (P^{[1]}_l)^* C_k$,
where $C_k\in\mathbb C^{p\times p}$.  Then, $\big(\big\{P^{[1]}_k(x)\big\}_{k=n}^\infty\big)^{\perp_R}=\big\{(P^{[1]}_k)^*\big\}_{k=0}^{n-1}\mathbb C^{p\times p}\cong  \big(\mathbb C^{p\times p}\big)^n$ and, therefore, is a free right module of rank $m+1$.
But,  according to \eqref{eq:biorthogonal} and Schwartz kernel theorem, for
the set of matrices of linear functionals $\big\{\mathcal L_u\big(P^{[2]}_k(y)\big)\big\}_{k=0}^{n-1}\subseteq \big(\big\{P^{[1]}_k(x)\big\}_{k=n}^\infty\big)^{\perp_R}=\big\{(P^{[1]}_k)^*\big\}_{k=0}^{n-1}\mathbb C^{p\times p}$, and we conclude
\begin{align*}
\big\{\mathcal L_u\big(P^{[2]}_k(x)\big)\big\}_{k=0}^{n-1}\mathbb C^{p\times p}= \big(\big\{P^{[1]}_k(x)\big\}_{k=n}^\infty\big)^{\perp_R}.
\end{align*}
Thus, we can write $\beta\cdot X=\mathcal L_u \big(Q_X(x)\big)$, where $Q_X(x)$ is  a matrix polynomial with degree $\deg (Q_X(x))\leq n-1$. A similar argument leads us to write $Y\cdot\delta=\mathcal L_u' \big(_YQ(x)\big)$, with $\mathcal L_u'$ the transpose operator of $\mathcal L_u$ --see \cite{Schwartz1}--   and $Q_Y(x)$ is  a matrix polynomial with degree $\deg (Q_Y(x))\leq n-1$.
We obviously have
\begin{align*}
\big\langle\mathcal L_u \big(Q_X(y)\big),P(x)\big\rangle&=0,  &\forall P(x)&\in \operatorname{Ker}^R(\beta\cdot X):=\{P(x)\in\mathbb C^{p\times p}[x]: \prodint{\beta_X,P(x)}=0\},
\\
\big\langle P(y), \mathcal L_u'\big(Q_Y(x)\big)\big\rangle&=0, &\forall P(y)&\in \operatorname{Ker}^L(Y\cdot \delta):=\{P(y)\in\mathbb C^{p\times p}[y]: \prodint{P(y),Y\cdot\delta}=0\},
\end{align*}
and Schwartz kernel theorem gives
\begin{align*}
\langle P(x), Q_X(y) \rangle_u&=0, & \forall P(x)&\in\operatorname{Ker}^R(\beta\cdot X),\\
\langle Q_Y(x) ,P(y)\rangle_u&=0, & \forall P(y)&\in\operatorname{Ker}^L(Y\cdot\delta),
\end{align*}
which in turn implies
\begin{align*}
Q_X(x)&\in\big(\operatorname{Ker}^R(\beta\cdot X)\big))^{\perp_u^R}, &
Q_Y(x)&\in\big(\operatorname{Ker}^L(Y\cdot \delta)\big))^{\perp_u^L}.
\end{align*}
Notice that, $\operatorname{Ker}^R_\beta\subset \operatorname{Ker}^R(\beta\cdot X)$
and, consequently,
$(\operatorname{Ker}^R_\beta)^{\perp^R_u}\supset (\operatorname{Ker}^R(\beta\cdot X))^{\perp^R_u}$. But we have assumed that $\big(\operatorname{Ker}^R_\beta\big)^{\perp_u^R}=\{0_p\}$, so that $Q_X(x)=0_p$ and  $\beta\cdot X=0_p$. Consequently, since the set $\big\{\big(\beta^{(a)}_m\big)_x\big\}_{\substack{a=1,\dots,q\\
		m=0,\dots,\kappa^{(a)}-1}}$ is linearly independent in the bimodule $\big(\mathbb C^{p\times p}[x]\big)'$, we deduce that $X=0$, which contradicts  the initial assumption.
		Observe that
		$\mathbb I \subset \operatorname{Ker}^L(Y\cdot\delta)$ so that $\big(\operatorname{Ker}^L(Y\cdot\delta)\big)^{\perp_u^L}\subset \mathbb I^{\perp_u^L}$. If we assume $\mathbb I^{\perp_u^L}=\{0_p\}$ we get $Q_Y(x)=0_p$ and $Y=0$, in contradiction with the initial assumption and the matrix is
		$I_{Np}+\prodint{\mathcal  J^{[0,1]}_{K_{n-1}}(x),(\beta)_x}$. This also implies that $Y=0$

Hence, the matrix  $I_{Np}+\prodint{\mathcal  J^{[0,1]}_{K_{n}}(x),(\beta)_x}$ is nonsingular, allowing us to clean \eqref{eq:nonsing1uvarov} and \eqref{eq:nonsing2uvarov} and get
\begin{align*}
\prodint{\hat P^{[1]}_{n}(x),(\beta)_x}&=
\prodint{ P^{[1]}_{n}(x),(\beta)_x}\Big(I_{Np}+\prodint{\mathcal  J^{[0,1]}_{K_{n-1}}(x),(\beta)_x}\Big)^{-1},\\
\big(\mathcal  J_{ \hat P^{[2]}_n}\big)^\top
&=\Big(I_{Np}+\prodint{\mathcal  J^{[0,1]}_{K_{n-1}}(x),(\beta)_x}\Big)^{-1}\big(\mathcal  J_{ P^{[2]}_n}\big)^\top.
\end{align*}
These relations, when introduced in \eqref{eq:uvarov11} and \eqref{eq:uvarov12}, give
\begin{align*}
\hat P^{[1]}_{n}(x)&=
P^{[1]}_{n}(x)-\prodint{ P^{[1]}_{n}(x),(\beta)_x}\Big(I_{Np}+\prodint{\mathcal  J^{[0,1]}_{K_{n-1}}(x),(\beta)_x}\Big)^{-1}\mathcal  J^{[0,1]}_{K_{n-1}}(x),\\
(\hat P^{[2]}_{n}(y))^\top&=(P^{[2]}_{n}(y))^\top-
\prodint{K_{n-1}(x,y),(\beta)_x}
\Big(I_{Np}+\prodint{\mathcal  J^{[0,1]}_{K_{n-1}}(x),(\beta)_x}\Big)^{-1}\big(\mathcal  J_{ P^{[2]}_n}\big)^\top.
\end{align*}
and the result follows.
To complete the proof notice that, see Proposition  \ref{pro:additive}, 
\begin{align*}
\hat H_n&=H_n+\prodint{\hat P^{[1]}_n(x),\beta }\big(\mathcal J_{P_n^{[2]}}\big)^\top\\
&=H_n+\prodint{ P^{[1]}_{n}(x),(\beta)_x}\Big(I_{Np}+\prodint{\mathcal  J^{[0,1]}_{K_{n-1}}(x),(\beta)_x}\Big)^{-1}\big(\mathcal J_{P_n^{[2]}}\big)^\top.
\end{align*}
\end{proof}

\subsection{Applications}
We discuss two particular cases of the general additive Uvarov perturbation presented above.

\subsubsection{Total derivatives}
We take the perturbation, which is supported by the diagonal $y=x$, in the following way
\begin{align}\label{eq:v_uvarov_diagonal_I}
v_{x,x}&=\sum_{a=1}^q\sum_{m=0}^{\kappa^{(a)}-1}\frac{(-1)^m}{m!}\beta^{(a)}_m\delta^{(m)}(x-x_a), & \beta^{(a)}_m&\in\mathbb C^{p\times p}.
\end{align}

\begin{pro}\label{pro:diagonal_masses}
	The discrete Hankel  mass terms \eqref{eq:v_uvarov_diagonal_I} are particular cases of  \eqref{eq:v_uvarov_general}  with
	\begin{align*}
	\big(\beta^{(a)}_{k}\big)_x=\sum_{n=0}^{\kappa^{(a)}-1-k}(-1)^{n}\frac{\beta^{(a)}_{k+n}}{(n)!}\delta^{(n)}(x-x_a).
	\end{align*}
\end{pro}

To discuss this reduction it is convenient  to introduce some further notation.
\begin{defi}
	For the family of perturbation matrices $\beta^{(a)}_m\in\mathbb C^{p\times p}$, let consider
	\begin{align*}
	\beta^{(a)}:=\begin{bmatrix}
	\beta^{(a)}_0& \beta^{(a)}_1& \beta^{(a)}_2& &&&& \beta^{(a)}_{\kappa^{(a)}-1}\\
	\beta^{(a)}_1&  \beta^{(a)}_2&\iddots &&&& \beta^{(a)}_{\kappa^{(a)}-1}&0_p\\
	\beta^{(a)}_2&\iddots &&&& \beta^{(a)}_{\kappa^{(a)}-1}&0_p&0_p\\
	\\\iddots  & &\iddots & &  \\\\
	\beta^{(a)}_{\kappa^{(a)}-1}&0_p&0_p&&&0_p&0_p& 0_p
	\end{bmatrix}\in\mathbb C^{\kappa^{(a)}p\times \kappa^{(a)}p}
	\end{align*}
	and if $N:=\kappa^{(1)}+\dots+\kappa^{(q)}$, then  we introduce
	\begin{align*}
	\beta=\diag(\beta^{(1)},\dots,\beta^{(q)})\in\mathbb C^{ Np\times Np}.
	\end{align*}
		We also introduce some additional jets. First a spectral jet with respect to the first variable
				\begin{align*}
				\mathcal J^{[1,0]}_K(y)=\begin{bmatrix}
				K(x_1,y),\dots,\dfrac{(K(x,y))_{x_1,y}^{(\kappa^{(1)}-1,0)}}{(\kappa^{(1)}-1)!},\dots,K(x_q,y),\dots,\dfrac{(K(x,y))_{x_q,y}^{(\kappa^{(q)}-1,0)}}{(\kappa^{(q)}-1)!}
				\end{bmatrix}\in\mathbb C^{p\times Np},
				\end{align*}		
		and also a  double spectral jet of a matrix kernel
		\begin{align*}
		\mathcal J_K=\begin{bmatrix}
		K(x_1,x_1)&\dots&\dfrac{(K(x,y))_{x_1,x_1}^{(\kappa^{(1)}-1,0)}}{(\kappa^{(1)}-1)!}&\dots&K(x_q,x_1)&\dots&\dfrac{(K(x,y))_{x_q,x_1}^{(\kappa^{(q)}-1,0)}}{(\kappa^{(q)}-1)!}
		\\
		\vdots&&\vdots&&\vdots& &\vdots \\
		\dfrac{(K(x,y))_{x_1,x_1}^{(0,\kappa^{(1)}-1)}}{(\kappa^{(1)}-1)!}&\dots &\dfrac{(K(x,y))_{x_1,x_1}^{(\kappa^{(1)}-1,\kappa^{(1)}-1)}}{(\kappa^{(1)}-1)!(\kappa^{(1)}-1)!}&\dots&	\dfrac{(K(x,y))_{x_q,x_1}^{(0,\kappa^{(1)}-1)}}{(\kappa^{(1)}-1)!}&\dots &\dfrac{(K(x,y))_{x_1,x_q}^{(\kappa^{(q)}-1,\kappa^{(1)}-1)}}{(\kappa^{(q)}-1)!(\kappa^{(1)}-1)!}	\\
		\vdots&&\vdots&&\vdots& &\vdots \\
		K(x_1,x_q)&\dots&\dfrac{(K(x,y))_{x_1,x_q}^{(\kappa^{(1)}-1,0)}}{(\kappa^{(1)}-1)!}&\dots&K(x_q,x_q)&\dots&\dfrac{(K(x,y))_{x_q,x_q}^{(\kappa^{(q)}-1,0)}}{(\kappa^{(q)}-1)!}\\
		\vdots&&\vdots&&\vdots& &\vdots \\
		\dfrac{(K(x,y))_{x_1,x_q}^{(0,\kappa^{(q)}-1)}}{(\kappa^{(q)}-1)!}&\dots &\dfrac{(K(x,y))_{x_1,x_q}^{(\kappa^{(1)}-1,\kappa^{(q)}-1)}}{(\kappa^{(1)}-1)!(\kappa^{(q)}-1)!}&\dots&	\dfrac{(K(x,y))_{x_q,x_q}^{(0,\kappa^{(q)}-1)}}{(\kappa^{(q)}-1)!}&\dots &\dfrac{(K(x,y))_{x_q,x_q}^{(\kappa^{(q)}-1,\kappa^{(q)}-1)}}{(\kappa^{(q)}-1)!(\kappa^{(q)}-1)!}
		\end{bmatrix},
		\end{align*}
		which belongs to $\mathbb C^{Np\times Np}$.
		We have employed the compact notation
		\begin{align*}
		(K(x,y))^{(n,m)}_{a,b}=\frac{\partial^{n+m} K}{\partial x^n\partial y^m}\Big|_{x=a,y=b}.
		\end{align*}
\end{defi}

\begin{pro}	
When  the mass term $v_{x,y}$ is as \eqref{eq:v_uvarov_diagonal_I} the triviality of kernel subspace
$\operatorname{Ker}^R_\beta=\{0_p\}$
	is ensured whenever
$	\mathbb I^{\perp_u^R}=\{0_p\}$.
	Thus, the quasideterminantal expressions of Theorem \ref{teo:uvarov} hold whenever $\mathbb I^{\perp_u^L}=\mathbb I^{\perp_u^R}=\{0_p\}$.
	Moreover, if the generalized kernel $u_{x,y}$ is of Hankel type,  the quasidefiniteness of $u$ ensures the trivially  of the left and right  orthogonal complements of the bilateral ideal $\mathbb I$.
\end{pro}
\begin{proof}
For the diagonal case we  choose the mass terms as in Proposition \ref{pro:diagonal_masses}
\begin{align*}
\big(\beta^{(a)}_{k}\big)_x=\sum_{n=0}^{\kappa^{(a)}-1-k}(-1)^{n}\frac{\beta^{(a)}_{k+n}}{(n)!}\delta^{(n)}(x-x_a).
\end{align*}
Then, $\mathbb I\subset\operatorname{Ker}^R_\beta$,
and, consequently, $
\mathbb I^{\perp_u^R}\supset \big(\operatorname{Ker}^R_\beta\big)^{\perp_u^R}$.
If $u_{x,y}$ is of Hankel type, we have that $P(x)\in \mathbb I^{\perp_u^R}$ if
\begin{align*}
\prodint{
	P(x)(x-x_1)^{\kappa^{(1)}}\cdots (x-x_q)^{\kappa^{(q)}},x^nI_p}_u=0,
\end{align*}
for $n\in\{0,1,\dots\}$, so that, since  $u$ is quasidefinite then $P(x)=0_p$, and we deduce that $\mathbb I^{\perp_u^R}=\{0_p\}$.
A similar argument leads to  the triviality $\mathbb I^{\perp_u^L}=\{0_p\}$.
\end{proof}

\begin{pro}
When  the mass term $v_{x,y}$ is as \eqref{eq:v_uvarov_diagonal_I} the perturbed matrix orthogonal polynomials and quasitau matrices are
\begin{align*}
\hat P^{[1]}_n(x)&=\Theta_*\begin{bmatrix}
I_{Np}+\beta \mathcal J_K & \beta \mathcal J^{[0,1]}_{K_{n-1}}(x)\\
\mathcal J_{P^{[1]}_n} & P_n^{[1]}(x)
\end{bmatrix}, \\
(\hat P^{[2]}_n(y))^\top&=\Theta_*\begin{bmatrix}
I_{Np}+ \beta\mathcal J_K  &\beta (\mathcal J_{P^{[2]}_n}  )^\top\\
\mathcal J^{[1,0]}_{K_{n-1}}(y)	 & (P_n^{[2]}(y ))^\top
\end{bmatrix},\\
\hat H_n&=\Theta_*\begin{bmatrix}
I_{Np}+ \beta\mathcal J_K &- \beta(\mathcal J_{P^{[2]}_n}  )^\top\\\
\mathcal J_{P^{[1]}_n} & H_n
\end{bmatrix}.
\end{align*}	
\end{pro}

\subsubsection{Uvarov perturbations with finite discrete support}

We will assume that the matrices of generalized functions $(\beta^{(a)}_m)_x$ are supported on a discrete finite set, say $\{\tilde x_b\}_{b=1}^{\tilde q}$,
with multiplicities $\tilde\kappa^{(b)}$ such that $\tilde\kappa^{(1)}+\dots+\tilde\kappa^{(\tilde q)}=\tilde N$,  so that
\begin{align}\label{eq:discrete_support}
(\beta^{(a)}_m)_x=\sum_{b=1}^{\tilde q}\sum_{l=0}^{\tilde \kappa^{(b)}-1}\beta^{(b,a)}_{l,m}\frac{(-1)^l}{l!}\delta^{(l)}(x-\tilde x_b).
\end{align}
The additive perturbation is
\begin{align*}
v_{x,y}&=\sum_{b=1}^{\tilde q}\sum_{l=0}^{\tilde \kappa^{(b)}-1}\sum_{a=1}^q\sum_{m=0}^{\kappa^{(a)}-1}\frac{(-1)^{l+m}}{l!m!}\beta^{(b,a)}_{l,m}\delta^{(l)}(x-\tilde x_b)\otimes\delta^{(m)}(y-x_a), & \beta^{(b,a)}_{l,m}&\in\mathbb C^{p\times p}.
\end{align*}
This implies that there is a discrete support of the perturbing matrix of generalized functions, $\operatorname{supp}(v_{x,y})
=\{\tilde x_b\}_{b=1}^{\tilde q}\times  \{ x_a\}_{a=1}^{ q}$. The perturbed sesquilinear form  is
\begin{align*}
\prodint{P(x),Q(y)}_{u+v}=\prodint{P(x),Q(y)}_u+
\sum_{b=1}^{\tilde q}\sum_{l=0}^{\tilde \kappa^{(b)}-1}\sum_{a=1}^q\sum_{m=0}^{\kappa^{(a)}-1}\frac{1}{l!m!}\big(P(x)\big)^{(l)}_{\tilde x_b}\beta^{(b,a)}_{l,m}\big((Q(y))^\top\big)^{(m)}_{x_a}.
\end{align*}

\begin{defi}
	We introduce the following block rectangular matrix of couplings
	\begin{align*}
\beta:=\begin{bmatrix}
\beta^{(0,0)}_{1,1}&\dots &\beta^{(0,\kappa^{(1)}-1)}_{1,1}&\dots&\beta^{(0,0)}_{1,q}&\dots&\beta^{(0,\kappa^{(q)}-1)}_{1,q}\\
\vdots & & \vdots & & \vdots & & \vdots\\
\beta^{(\tilde \kappa^{(1)}-1,0)}_{1,1}&\dots &\beta^{(\tilde \kappa^{(1)}-1,\kappa^{(1)}-1)}_{1,1}&\dots&\beta^{(\tilde \kappa^{(1)}-1,0)}_{1,q}&\dots&\beta^{(\tilde \kappa^{(1)}-1,\kappa^{(q)}-1)}_{1,q}\\
\vdots & & \vdots & & \vdots & & \vdots\\
\beta^{(0,0)}_{\tilde q,1}&\dots &\beta^{(0,\kappa^{(1)}-1)}_{\tilde q,1}&\dots&\beta^{(\tilde \kappa^{(1)}-1,0)}_{\tilde q,q}&\dots&\beta^{(\tilde \kappa^{(1)}-1,\kappa^{(q)}-1)}_{\tilde q,q}\\
\vdots & & \vdots & & \vdots & & \vdots\\
\beta^{(\tilde\kappa^{(\tilde q)}-1,0)}_{\tilde q,1}&\dots &\beta^{(\tilde\kappa^{(\tilde q)}-1,\kappa^{(1)}-1)}_{\tilde q,1}&\dots&\beta^{(\tilde\kappa^{(\tilde q)}-1,0)}_{\tilde q,q}&\dots&\beta^{(\tilde \kappa^{(\tilde q)}-1,\kappa^{(q)}-1)}_{\tilde q,q}
\end{bmatrix}\in\mathbb C^{\tilde N p\times N p}
	\end{align*}
	and,  given any matrix of kernels $K(x,y)$, we introduce the mixed double jet $\tilde{	\mathcal J}_K\in\mathbb C^{Np\times \tilde Np}$
	\begin{align*}
\tilde{	\mathcal J}_K:=\begin{bmatrix}
	K(\tilde x_1, x_1)&\dots&\dfrac{(K(x,y))_{\tilde x_1, x_1}^{(\tilde  \kappa^{(1)}-1,0)}}{(\tilde \kappa^{(1)}-1)!}&\dots&K(\tilde x_q, x_1)&\dots&\dfrac{(K(x,y))_{\tilde  x_q,x_1}^{(\tilde \kappa^{(\tilde q)}-1,0)}}{(\tilde \kappa^{(\tilde q)}-1)!}
	\\
	\vdots&&\vdots&&\vdots& &\vdots \\
	\dfrac{(K(x,y))_{\tilde x_1,x_1}^{(0,\kappa^{(1)}-1)}}{(\kappa^{(1)}-1)!}&\dots &\dfrac{(K(x,y))_{\tilde x_1,x_1}^{(\tilde \kappa^{(1)}-1,\kappa^{(1)}-1)}}{(\tilde \kappa^{(1)}-1)!(\kappa^{(1)}-1)!}&\dots&	\dfrac{(K(x,y))_{\tilde x_q,x_1}^{(0,\kappa^{(1)}-1)}}{(\kappa^{(1)}-1)!}&\dots &\dfrac{(K(x,y))_{\tilde  x_q,x_1}^{(\tilde  \kappa^{(\tilde q)}-1,\kappa^{(1)}-1)}}{(\tilde  \kappa^{(\tilde q)}-1)!(\kappa^{(1)}-1)!}	\\
	\vdots&&\vdots&&\vdots& &\vdots \\
	K(\tilde x_1,x_q)&\dots&\dfrac{(K(x,y))_{\tilde x_1,x_q}^{(\tilde \kappa^{(1)}-1,0)}}{(\tilde \kappa^{(1)}-1)!}&\dots&K(\tilde x_q,x_q)&\dots&\dfrac{(K(x,y))_{\tilde x_q,x_q}^{(\tilde \kappa^{(\tilde q)}-1,0)}}{(\tilde \kappa^{(\tilde q)}-1)!}\\
	\vdots&&\vdots&&\vdots& &\vdots \\
	\dfrac{(K(x,y))_{\tilde x_1,x_q}^{(0,\kappa^{(q)}-1)}}{(\kappa^{(q)}-1)!}&\dots &\dfrac{(K(x,y))_{\tilde  x_1,x_q}^{(\tilde \kappa^{(1)}-1,\kappa^{(q)}-1)}}{(\tilde  \kappa^{(1)}-1)!(\kappa^{(q)}-1)!}&\dots&	\dfrac{(K(x,y))_{\tilde  x_q,x_q}^{(0,\kappa^{(q)}-1)}}{(\kappa^{(q)}-1)!}&\dots &\dfrac{(K(x,y))_{\tilde x_q,x_q}^{(\tilde \kappa^{(\tilde q)}-1,\kappa^{(q)}-1)}}{(\tilde \kappa^{(\tilde q)}-1)!(\kappa^{(q)}-1)!}
	\end{bmatrix}.
	\end{align*}
\end{defi}

With this election we get
\begin{pro}
For $\beta$'s as in \eqref{eq:discrete_support} we have
\begin{align*}
\prodint{P^{[1]}_n(x),\big(\beta\big)_x}&=\tilde{\mathcal J}_{P^{[1]}_n}\beta,\\
\prodint{\mathcal  J^{[0,1]}_{K_{n-1}}(x),(\beta)_x}&=\tilde{\mathcal J}_{K_{n-1}} \beta,
\end{align*}
in terms of the spectral jet $\tilde {\mathcal J}$ relative to the  set $\{\tilde x_b\}_{b=1}^{\tilde q}$.
\end{pro}

\begin{coro}\label{coro:Uvarov}
	Whenever   one of the two conditions $\tilde{\mathbb I}^{\perp_u^R}=\{0_p\}$ or $\mathbb I^{\perp_u^L}=\{0_p\}$ hold, the matrix
	$	I_{Np}+\tilde{\mathcal J}_{K_{n-1}} \beta$ is nonsingular and the perturbed matrix orthogonal polynomials and quasitau matrices has the following quasideterminantal expressions
	\begin{align*}
	\hat P^{[1]}_n(x)&=\Theta_*\begin{bmatrix}
	I_{Np}+\beta\tilde{\mathcal J}_{K_{n-1}} &\beta\mathcal  J^{[0,1]}_{K_{n-1}}(x)\\
	\tilde{\mathcal J}_{P^{[1]}_n},& P_n^{[1]}(x)
	\end{bmatrix}, \\
	(\hat P^{[2]}_n(y))^\top&=\Theta_*\begin{bmatrix}
	I_{Np}+\beta\tilde{\mathcal J}_{K_{n-1}} &\beta\big(\mathcal J_{ P^{[2]}_{n}}\big)^\top\\
	\tilde{\mathcal J}^{[1,0]}_{K_{n-1}}(y)	& (P_n^{[2]}(y ))^\top
	\end{bmatrix},\\
	\hat H_n&=\Theta_*\begin{bmatrix}
	I_{Np}+\beta\tilde{\mathcal J}_{K_{n-1}} &- \beta(\mathcal J_{P^{[2]}_n}  )^\top\\
	\tilde{\mathcal J}_{P^{[1]}_n}& H_n
	\end{bmatrix}.
	\end{align*}
\end{coro}
\begin{rem}[Discrete Sobolev sesquilinear forms]
If the support lays by the diagonal $\tilde x_a=x_a$ and $\tilde \kappa^{(a)}=\kappa^{(a)}$ we have
\begin{align*}
\prodint{P(x),Q(y)}_{u+v}=\prodint{P(x),Q(y)}_u+
\sum_{a,b=1}^q\sum_{m,l=0}^{\kappa^{(a)}-1}\frac{(-1)^{l+m}}{l!m!}\big(P(x)\big)^{(l)}_{ x_b}\beta^{(b,a)}_{l,m}\big((Q(y))^\top\big)^{(m)}_{x_a}.
\end{align*}
This a prototype of discrete Sobolev sesquilinear form. In particular, when $\beta^{(b,a)}_{m.l}=\delta_{a,b}\delta_{l,m}\beta^{(a)}_m$ and $\beta$ is a blok diagonal matrix we have
\begin{align*}
\prodint{P(x),Q(y)}_{u+v}=\prodint{P(x),Q(y)}_u+
\sum_{a=1}^q\sum_{m=0}^{\kappa^{(a)}-1}\frac{1}{l!m!}\big(P(x)\big)^{(m)}_{ x_a}\beta^{(a)}_{m}\big((Q(y))^\top\big)^{(m)}_{x_a}.
\end{align*}
 When $u$ is a Borel measure, we have the  discrete Sobolev sesquilinear form
\begin{align*}
\prodint{P(x),Q(y)}_{u+v}=\int P(x)\d\mu(x)(Q(x))^\top+
\sum_{a=1}^q\sum_{m=0}^{\kappa^{(a)}-1}\frac{1}{l!m!}\big(P(x)\big)^{(m)}_{ x_a}\beta^{(a)}_{m}\big((Q(y))^\top\big)^{(m)}_{x_a},
\end{align*}
for which the Christoffel--Uvarov formulas in Corollary \ref{coro:Uvarov} provides biorthogonal families of matrix polynomials  as well as its matrix \emph{norms}.
\end{rem}
For more on discrete Sobolev orthogonality see \cite{Marcellan2014Sobolev}.

\section{Applications to the  non-Abelian 2D Toda lattice and noncommutative KP hierarchies}

The non-Abelian Toda hierarchy can be understood as the theory of continuous transformations of Gram matrices, see for example \cite{alvarez2015Christoffel}. We have chosen  this simpler scenario, instead of the more general multicomponent one \cite{ueno-takasaki0,ueno-takasaki1,ueno,BtK2,kac,manas1}, as many of the essential facts regarding integrability are captured within it. 

The perturbing parameters are the times of the integrable hierarchy.
Given a semi-infinite covector of times $t=\begin{bmatrix}
t_1,t_2,\dots
\end{bmatrix}  $, where we only have a finite number of  entries different from zero, and  we also consider the semi-infinite matrix   $V^{t}_0:=\exp\Big(\sum\limits_{0< j\ll \infty} t_{j}\Lambda^j\Big)$, and given two time covectors
$t_i=\begin{bmatrix}
 t_{i,1},t_{i,2},\dots
\end{bmatrix}  $, $i\in\{1,2\}$,
the perturbed  Gram matrix is
\begin{align*}
G^t= V_0^{t_1} G \big(V_0^{t_2} \big)^{-\top},
\end{align*}
where we have used the notation $t=(t_1,t_2)$ to denote the perturbation parameters.
We also consider the polynomials
\begin{align*}
t_1(x)&:=\sum_{0< j\ll \infty}t_{1,j} x^j, &t_2(y)&:=\sum_{0< j\ll \infty}t_{2,j} y^j
\end{align*}
Then, we can check that
\begin{align*}
G^t&=V_0^{t_1} G \big(V_0^{t_2} \big)^{-\top}\\
&= V_0^{t_1} \big\langle\chi(x),\chi(y)\big\rangle_u \big(V_0^{t_2} \big)^{-\top}\\
&=\prodint{\Exp{\sum\limits_{j=1}^{\infty}t_{1,j}x^j}\chi(x),\Exp{-\sum\limits_{j=1}^{\infty}t_{2,j}y^j}\chi(y)}_u\\
&=\prodint{\chi(x),\chi(y)}_{u^t},
\end{align*}
where the deformed bivariate matrix of generalized functions is given by
\begin{align*}
u_{x,y}^t:= \Exp{t_{1}(x)-t_{2}(y)}u_{x,y}.
\end{align*}
Observe that if the initial Hankel matrix of generalized kernels $u_{x,y}$, then so is $u^t_{x,y}$. Thus, the Hankel symmetry is preserved under these continuous transformations.  In this case, the perturbation is
\begin{align*}
u_{x}^t:= \Exp{\sum\limits_{j=1}^{\infty}(t_{1,j}-t_{2,j})x^j}u_{x},
\end{align*}
and we can replace the two familiy of times $t_1$ and $t_2$ by just one family, leading to a reduction from the 2D non-Abelian Toda lattice hierarchy to the 1D non-Abelian Toda lattice hierarchy.

We will assume that $u_{x,y}^t$ is quasidefinite; i.e.,   the Gauss--Borel factorization
\begin{align}\label{eq:gbt}
G^t=(S_1^t)^{-1} H^t (S_2^t)^{-\top}
\end{align}
holds. Consequently, for the time-dependent matrix polynomials
\begin{align*}
P^{[1],t}(x)&=S_1^t\chi(x), & P^{[2],t}(y)&=S_2^t\chi(y),
\end{align*}
the  biorthogonality conditions hold
\begin{align*}
\prodint{P_n^{[1],t}(x),P_m^{[2],t}(y)}_{u^t}=\delta_{n,m}H_n^t.
\end{align*}
We also will need the second kind functions
\begin{align*}
C_{n}^{[1],t}(z)&=\left\langle P^{[1],t}_n(x),\frac{I_p}{z-y}\right\rangle_{u^t}, &\big(C_{n}^{[2],t}(z)\big)^\top&=\left\langle \frac{I_p}{z-x},P^{[2],t}_n(y)\right\rangle_{u^t}.
\end{align*}
Finally, we will also require of   the Christoffel--Darboux kernel and its mixed versions
\begin{align*}
K^t_{n}(x,y)&=\sum_{k=0}^{n}(P_k^{[2],t}(y))^\top (H^t_k)^{-1}P^{[1],t}_k(x),&
K^{(pc),t}_n(x,y)&=\sum_{k=0}^n\big(P_k^{[2],t}(y)\big)^\top (H^t_k)^{-1}C_k^{[1],t}(x).	
\end{align*}
In the case of Hankel  Gram matrices,  they   satisfy the Christoffel--Darboux formula.
\subsection{Non Abelian 2D-Toda hierarchy}
 Given a semi-infinite matrix $A$ we have unique splitting $A=A_++A_-$ where $A_+$ is an upper  triangular block matrix while is $A_-$ a  strictly lower triangular block matrix.
 The Gaussian factorization of the deformed Gram matrix has the following differential consequences
 \begin{pro}[Sato--Wilson equations]\label{pro:evolution S}
 	The following holds
 	\begin{align*}
 	\frac{\partial S_1}{\partial t_{1,j}}(S_1)^{-1}&=-\Big(S_1{\Lambda}^{j} (S_1)^{-1}\Big)_-, &
 	\frac{\partial S_1}{\partial t_{2,j}}(S_1)^{-1}&=\Big(\tilde S_2\big({\Lambda}^{\top}\big)^{j} (\tilde S_2)^{-1}\Big)_-,\\
 	\frac{\partial \tilde S_2}{\partial t_{1,j}}(\tilde S_2)^{-1}&=\Big(S_1{\Lambda}^{j} (S_1)^{-1}\Big)_+, &
 	\frac{\partial \tilde S_2}{\partial t_{2,j}}(\tilde S_2)^{-1}&=-\Big(\tilde S_2\big({\Lambda}^{\top}\big)^{j} (\tilde S_2)^{-1}\Big)_+.
 	\end{align*}
 \end{pro}
 \begin{proof}
 Taking into account convergence and no associative issues, we look for the derivatives along  the deformation parameters in \eqref{eq:gbt} 
\begin{align}\notag
-(S_1^t)^{-1}\frac{\partial S_1^t}{\partial t_{1,j}}(S_1^t)^{-1} \tilde S_2^t+(S_1^t)^{-1}\frac{\partial \tilde S^t_2}{\partial t_{1,j}}
&=\Lambda^j G^t\\&=\Lambda^j (S_1^t)^{-1} \tilde S_2^t,\\\notag
-(S_1^t)^{-1}\frac{\partial S_1^t}{\partial t_{2,j}}(S_1^t)^{-1} \tilde S_2^t+(S_1^t)^{-1}\frac{\partial \tilde S^t_2}{\partial t_{2,j}}
&=-G^t(\Lambda^j)^\top\\&=(S_1^t)^{-1} \tilde S_2^t(\Lambda^j)^\top,
\end{align}
so that
\begin{align*}
-\frac{\partial S_1^t}{\partial t_{1,j}}(S_1^t)^{-1} 
+\frac{\partial \tilde S^t_2}{\partial t_{1,j}}\big(\tilde S_2^t\big)^{-1}
&=S_1^t\Lambda^j (S_1^t)^{-1} ,\\\notag
-\frac{\partial S_1^t}{\partial t_{2,j}}(S_1^t)^{-1} 
+\frac{\partial \tilde S^t_2}{\partial t_{2,j}}\big(\tilde S_2^t\big)^{-1}
&=- \tilde S_2^t(\Lambda^j)^\top\big(\tilde S_2^t\big)^{-1},
\end{align*}
and the result follows.
 \end{proof}

 \begin{pro}[Non-Abelian 2D Toda lattice equations]Using the notation $t_{1,1}=\eta$ and $t_{2,1}=\zeta$,
   \begin{align}\label{eq:2DToda}
   \frac{\partial }{\partial \zeta}\Big(\frac{\partial H_k}{\partial \eta}(H_k)^{-1}\Big)+H_{k+1}(H_{k})^{-1}-H_k(H_{k-1})^{-1}=0,
   \end{align}
   
      For the Hankel case we find a reduction the non-Abelian 1D Toda lattice equation, where $\eta=\zeta$,
     \begin{align*}
     %\partial_{2,1}\big(\partial_{1,1}(H_k)\cdot (H_k)^{-1}\big)+H_{k+1}(H_{k})^{-1}-H_k(H_{k-1})^{-1}=0.\\
     \frac{\partial }{\partial \eta}\Big(\frac{\partial H_k}{\partial \eta}(H_k)^{-1}\Big)+H_{k+1}(H_{k})^{-1}-H_k(H_{k-1})^{-1}=0.
     \end{align*}
 \end{pro}
 \begin{proof}
 	From Proposition \ref{pro:evolution S} we get
 	\begin{align*}
 	\frac{\partial H_{k}}{\partial \eta}(H_{k})^{-1}&=U_{k}-U_{k+1},& k\in\{0,1,\dots\},\\
 	\frac{\partial U_{k}}{\partial \zeta}&=H_{k}(H_{k-1})^{-1},& k\in\{1,2,\dots\}
 	\end{align*}
  	where $U_0=0_p$ and $U_k:=(S^t_1)_{k,k-1}\in\mathbb C^{p\times p }$, $k\in\{1,2,\dots\}$.
 \end{proof}
    The non-Abelian Toda lattice was introduced in the context of string theory by Polyakov, \cite{polyakov1,polyakov2}, and then studied under the inverse spectral transform by Mikhailov \cite{mikhailov} and Riemann surface theory by Krichever \cite{krichever}. The Darboux transformations were considered in \cite{salle} and later in \cite{nimmo}. We refer the reader to \cite{alvarez2015Christoffel} for a deeper  treatment on these issues.

 These equations are just the first members  of an infinite set of nonlinear partial differential equations, an integrable hierarchy. Its elements are given by
  \begin{defi}\label{def:integrable}
  	The Lax and  Zakharov--Shabat matrices are given by
  	\begin{align*}
  	L_{1}&:=S_1\Lambda(S_1)^{-1}, &
  	L_{2}&:=\tilde S_2(\Lambda)^{\top} (\tilde S_2)^{-1},\\
  	B_{1,j}&:=\big(({ L}_{1})^j\big)_+, & B_{2,j}&:=\big(({ L}_{2})^j\big)_-.
  	\end{align*}
  \end{defi}
  \begin{pro}[The zero-curvature formulation of the integrable hierarchy]\label{pro:integrable}

  	The  Lax matrices are subject to the following  \emph{Lax equations}
  	\begin{align*}
  	\frac{\partial L_{i}}{\partial t_{j,k}}&=\big[B_{j,k}, L_{i}\big],
  	\end{align*}
  	and Zakharov--Shabat matrices fulfill the following  \emph{Zakharov--Shabat equations}
  	\begin{align*}
  	\frac{\partial B_{i',k'}}{\partial t_{i,k}}- \frac{\partial B_{i,k}}{\partial t_{i',k'}}+\big[B_{i,k},B_{i',k'}\big]=0.
  	\end{align*}
  \end{pro}
 Using the notation
  \begin{align*}
  a_k&:= H_k(H_{k-1})^{-1}, & b_k:= U_k-U_{k+1},
  \end{align*}
  for $k\in\{1,2,\dots\}$ and $a_0=0$ and $b_0=-(S_1)_{1,0}$.
  The Gauss--Borel factorization problem  implies
  \begin{align*}
  \frac{\partial b_k}{\partial\zeta}&=a_k-a_{k+1},\\
  \frac{\partial H_k}{\partial\eta}&=b_kH_k,
  \end{align*}
  from where \eqref{eq:2DToda} follows. This system is equivalent to the following nonlinear system of first order PDE
  \begin{align}\label{eq:2DToda_system}
  \begin{aligned}
  \frac{\partial b_k}{\partial\zeta}&=a_k-a_{k+1},\\
  \frac{\partial a_k}{\partial \eta}&= b_ka_k -a_k b_{k-1}.
  \end{aligned}
  \end{align}

\subsection{Baker functions}
 \begin{defi}
 We  consider the wave matrices
  \begin{align}\label{eq:def_wave}
  W_1^t:=&S_1^tV_0^{t_1}, &
  \tilde W_2^t:= & \big(\tilde S_2^t \big)^{-\top}V_{0}^{-t_2}=(H^t)^{-\top}S_2^tV_0^{-t_2} % \big(V_{0}^{-t_2}\big)^{\top}\big(\tilde S_2^t \big)^{-1},
  \end{align}
  where $\tilde S_2^t:= H^t\big(S_2^t\big)^{-\top}$. 
 \end{defi}
\begin{pro}[Zakharov--Shabat equations]\label{pro:linear_systems}
The wave matrices satisfy the linear systems
\begin{align*}
  \frac{\partial W_1^t}{\partial t_{1,j}}= &B_{1,j} W_1^t,&
    \frac{\partial W_1^t}{\partial t_{2,j}}
    =&B_{2,j} W_1^t,&
\frac{\partial \tilde W_2^t}{\partial t_{1,j}}=&-(B_{1,j})^\top\tilde W_2^t,&
\frac{\partial \tilde W_2^t}{\partial t_{2,j}}=& -(B_{2,j})^\top\tilde W_2^t.
\end{align*}
\end{pro}
\begin{proof}
It is deduced as follows
\begin{align*}
  \frac{\partial W_1^t}{\partial t_{1,j}}=&\frac{\partial S_1^t}{\partial t_{1,j}}V_0^{t_1}+ S_1^t\Lambda^jV_0^{t_1}
  =\Big(\frac{\partial S_1^t}{\partial t_{1,j}}(S_1^t)^{-1}+(L_1)^j\Big)W_1^{t_1}\\
  =&B_{1,j} W_1^t,\\
    \frac{\partial W_1^t}{\partial t_{2,j}}=&\frac{\partial S_1^t}{\partial t_{2,j}}V_0^{t_1}
    =\Big(\frac{\partial S_1^t}{\partial t_{2,j}}(S_1^t)^{-1}\Big)W_1^{t_1}\\
    =&B_{2,j} W_1^t,\\
\frac{\partial \tilde W_2^t}{\partial t_{1,j}}= & \frac{\partial\big(\tilde S_2^t \big)^{-\top}}{\partial t_{1,j}}V_{0}^{-t_2}
=-\big(\tilde S_2^t \big)^{-\top}\frac{\partial\big(\tilde S_2^t \big)^{\top}}{\partial t_{1,j}}\big(\tilde S_2^t \big)^{-\top}V_{0}^{-t_2}=-\Big(\frac{\partial\tilde S_2^t }{\partial t_{1,j}}\big(\tilde S_2^t \big)^{-1}\Big)^\top
\tilde W_2^t\\=&-(B_{1,j})^\top\tilde W_2^t,\\
\frac{\partial \tilde W_2^t}{\partial t_{2,j}}= & \frac{\partial\big(\tilde S_2^t \big)^{-\top}}{\partial t_{2,j}}V_{0}^{-t_2}-
\big(\tilde S_2^t \big)^{-\top}\Lambda^jV_{0}^{-t_2}
=-\big(\tilde S_2^t \big)^{-\top}
\frac{\partial\big(\tilde S_2^t \big)^{\top}}{\partial t_{2,j}}\big(\tilde S_2^t \big)^{-\top}V_{0}^{-t_2}-\big(\tilde S_2^t \big)^{-\top}\Lambda^jV_{0}^{-t_2}
\\&=-\Big(\frac{\partial\tilde S_2^t }{\partial t_{2,j}}\big(\tilde S_2^t \big)^{-1}+(L_2)^j
\Big)^\top\tilde W_2^t\\=&-(B_{2,j})^\top\tilde W_2^t.
\end{align*}
\end{proof}
The wave matrices satisfy
  $ I=W^t_1G(\tilde W^t_2)^\top$.
  	Notice that we can formally write $W_2^t:=(\tilde W_2^t)^{-\top}=  \tilde S_2^t  \big(V_{0}^{t_2}\big)^{-\top}$, which could happen  to  not exist as a product.
\begin{defi}\label{def:baker}
The  Baker functions are given by 
    	\begin{align*}
    	\Psi_1(t,z)&:= W^t_1\chi(z), &	\Psi^*_2(t,z)&:= \tilde W_2^t\chi(z),\\
    	 	\big(\Psi^*_1(t,z)\big)^\top&:= \big(\chi^*(z)\big)^\top G(\tilde W^t_2)^\top, &	\Psi_2(t,z)&:=  W^t_1G\chi^*(z),
    	 	\end{align*}
\end{defi}
    	 	We will promptly  show that the second pair of Baker functions do have a proper meaning as Cauchy transforms.

  \begin{rem}
    	 	Normally, in the literature on integrable systems they are  formally written as follows
    	 	    	 	\begin{align*}
    	 	    	 	 	\Psi_1(t,z)&:= W^t_1\chi(z), &	\Psi^*_2(t,z)&:=  (W_2^t)^{-\top}\chi(z),\\
    	 	    	 \Psi^*_1(t,z)&:= (W_1^t)^{-\top}\chi^*(z), &	\Psi_2(t,z)&:=  W_2^t\chi^*(z).
    	 	    	 	\end{align*}
  \end{rem}
  Extending from a Gram type definition to a Cauchy type, as we did with the second kind functions.
  \begin{pro}\label{pro:baker}
 The Baker functions and the biorthogonal polynomials are connected as follows
  \begin{align*}
      	\Psi_1(t,z)&= \Exp{t_1(z)}P^{[1],t}(z), \\	\Psi^*_2(t,z)&:= \Exp{-t_2(z)}(H^t)^{-\top}P^{[2],t}(z),\\
      	 	\big(\Psi^*_1(t,z)\big)^\top&:= \prodint{\frac{I_p}{z-x},\Exp{-t_2(y)}P^{[2],t}(y)}_u(H^t)^{-1}& z\not \in&\operatorname{supp}_x(u), \\	\Psi_2(t,z)&= \prodint{\Exp{t_1(x)}P^{[1],t}(x),\frac{I_p}{z-y}}_u,
      	 	& z\not \in&\operatorname{supp}_y(u),\\
      	 	\end{align*}
      	 	where $r_x$ and $r_y$ where introduced  in Definition  \ref{def:supports}.
  \end{pro}
  \begin{proof}
The first two  relations follow immediately, while for the second two we proceed as follows
  \begin{align*}
	\big(\Psi^*_1(t,z)\big)^\top&:= \big(\chi^*(z)\big)^\top \prodint{\chi(x),\chi(y)}_u(\tilde W^t_2)^\top\\
	&=\prodint{\big(\chi^*(z)\big)^\top \chi(x),\tilde W^t_2\chi(y)}_u\\
	&=\prodint{\frac{I_p}{z-x},\Psi^*_2(t,z)}_u,     	 	& z\not \in&\operatorname{supp}_x(u),
	\\ 	\Psi_2(t,z)&:=  W^t_1\prodint{\chi(x),\chi(y)}_u\chi^*(z)\\
&=  \prodint{W^t_1\chi(x),\big(\chi^*(z)\big)^\top\chi(y)}_u\\
&=  \prodint{\Psi_1(x,t),\frac{I_p}{z-y}}_u,&      	 	z\not \in&\operatorname{supp}_y(u).
      	 	\end{align*}
  \end{proof}
  \begin{pro}[Zakharov--Shabat equations] %\label{pro:linear_systems}
  The Baker functions satisfy the linear systems
  \begin{align*}
    \frac{\partial \Psi_1}{\partial t_{1,j}}= &B_{1,j} \Psi_1,&
      \frac{\partial \Psi_1}{\partial t_{2,j}}
      =&B_{2,j} \Psi_1,&
  \frac{\partial \Psi_2^*}{\partial t_{1,j}}=&-(B_{1,j})^\top\Psi_2^*,&
  \frac{\partial  \Psi_2^*}{\partial t_{2,j}}=& -(B_{2,j})^\top\Psi_2^*,\\
     \frac{\partial \big(\Psi_1^*\big)^\top}{\partial t_{1,j}}= &-\big(\Psi_1^*\big)^\top B_{1,j} ,&
        \frac{\partial \big(\Psi_1^*\big)^\top}{\partial t_{2,j}}
        =&- \big(\Psi_1^*\big)^\top B_{2,j},&
    \frac{\partial \Psi_2}{\partial t_{1,j}}=&(B_{1,j})^\top\Psi_2,&
    \frac{\partial  \Psi_2}{\partial t_{2,j}}=& (B_{2,j})^\top\Psi_2.
  \end{align*}
  \end{pro}
\begin{proof}
From Definition \ref{def:baker} we get
	\begin{align*}
    \frac{\partial	\Psi_1}{\partial t_{i,j}}&:=\frac{ \partial W^t_1}{\partial t_{i,j}}\chi(z), &	\frac{\partial \Psi^*_2}{\partial t_{i,j}}&:= \frac{\partial \tilde W_2^t}{\partial t_{i,j}}\chi(z),\\
    	 	\frac{\partial\big(\Psi^*_1(t,z)\big)^\top}{\partial t_{i,j}}&:= \big(\chi^*(z)\big)^\top G\Big(\frac{\partial \tilde W^t_2}{\partial t_{i,j}}\Big)^\top, &\frac{\partial	\Psi_2}{\partial t_{i,j}}&:=  \frac{\partial W^t_1}{\partial t_{i,j}}G\chi^*(z),
    	 	\end{align*}
    	 	and using Proposition \ref{pro:linear_systems} we obtain the stated result.
\end{proof}
For the first flows 
and the corresponding matrices $B_\eta:=B_{1,1}=(L_1)_+$ and $B_\zeta:=B_{2,1}=(L_2)_-$ we have  the following expressions
  \begin{align*}
  B_\eta&=
  \begin{bmatrix}
  b_0 & I_p &0_p& 0_p&0_p&\dots\\
  0_p & b_1 & I_p &0_p&0_p&\dots\\
  0_p &0_p &b_2& I_p &0_p&\\
  0_p&0_p &0_p &b_3& I_p &\ddots&\\
  \vdots&\vdots&\vdots&\,\ddots&&\ddots&
  \end{bmatrix}, &
  B_\zeta&=\begin{bmatrix}
  0_p & 0_p &0_p &0_p& \dots\\
  a_1& 0_p &0_p&0_p&\dots\\
  0_p & a_2& 0_p &0_p&\\
  0_p & 0_p & a_3& 0_p & \\
  \vdots &\vdots &\ddots &
  \end{bmatrix}.
  \end{align*}
  Then, we have the Zakharov--Shabat linear system
  \begin{align*}
  \frac{\partial\Psi_{1,k}}{\partial \zeta}&=a_k \Psi_{1,k-1},\\
  \frac{\partial\Psi_{1,k}}{\partial \eta }&=b_k\Psi_{1,k}+\Psi_{1,k+1},
  \end{align*}
  whose compatibility is \eqref{eq:2DToda_system}.
  
  The reader must be aware that despite the Baker functions  $\Psi_1(t,z)$ and $\Psi_2^*(t,z)$ are, as Proposition \ref{pro:baker} tell us,  essentially the biorthogonal polynomials $P^{[1],t}(z)$ and $P^{[2],t}(z)$, it is just not the  same for the other two Baker functions, $(\Psi^*)^\top(t,z)$ and $\Psi_2(z,t)$, as they are not the second kind functions, see \eqref{eq:second_kind}.
Indeed, these Baker functions are
   \begin{align*}
        	 \Psi_2(t,z)&= \prodint{P^{[1],t}(x)\Exp{t_1(x)},\frac{I_p}{z-y}}_u,     	 	& z\not \in&\operatorname{supp}_y(u),
        	         	\\	\big(\Psi^*_1(t,z)\big)^\top&:= \prodint{\frac{I_p}{z-x},\Exp{-t_2(y)}P^{[2],t}(y)}_u(H^t)^{-1},     	 	& z\not \in&\operatorname{supp}_x(u),
        	         	\end{align*}
while, according to Proposition \ref{pro:Cauchy1},	
for a matrix of generalized kernels $u^t_{x,y}\in\big((\mathcal O_c')_{x,y}\big)^{p\times p}$, 
   the second kind functions are
  \begin{align*}
  C_{n}^{[1],t}(z)&=\prodint{P^{[1],t}_n(x)\Exp{t_1(x)}, \frac{I_p\Exp{-t_2(y)}}{z-y}
  }_{u},      	 	& z\not \in&\operatorname{supp}_y(u),\\
   \big(C_{n}^{[2],t}(z)\big)^\top&=\left\langle \frac{I_p\Exp{t_1(x)}}{z-x},\Exp{-t_2(y)}P^{[2],t}_n(y)\right\rangle_{u}^t,     	 	& z\not \in&\operatorname{supp}_x(u).
  \end{align*}
  We see that only for particular situations they can be identified. For example,  whenever $t_2=0$ we find $C_{n}^{[1],t}(z)=\Psi_2(t,z)$. Notice that, as we immediately will explain, this corresponds to the noncommutative KP flows.
  This issue will reappear with the bilinear identities discussed below.
  
   \subsection{Noncommutative KP hierarchy}\label{S:KP}
   We will show how the noncommutative KP hierarchy appears within the non-Abelian 2D Toda lattce hierarchy.
For that aim we put all the times $t_{2,j}=0$ and consider only continuous deformations given by the times $t_{1,j}$, $j\in\{1,2,\dots\}$, and the first three times will be denoted by $\eta:=t_{1,1}$, $\rho:=t_{1,2}$ and $\theta:=t_{1,3}$, $U_k:=(S_1)_{k,k-1}$, $k\in\{1,2\dots\},$ %and $\gamma_k:=(S_1)_{k,k-2}$, $k\in\{2,3\dots\}$. 
will denote the blocks  on the first  subdiagonal of $S_1^t$, 
the corresponding matrix will be denoted by $U$.
%, for the first subdiagonal, and $\gamma$, for the second subdiagonal.
   \begin{defi}\label{def:asymptotic-module}
   	Given two semi-infinite matrices $Z_1(t)$ and $Z_2(t)$ we say that
   	\begin{itemize}
   		\item  $Z_1(t)\in\mathfrak{l}V_0^{t_1}$ if $Z_1(t)\big(V_0^{t_1}\big)^{-1}$ is a block strictly lower triangular matrix.
   		\item  $Z_2(t)\in\mathfrak{u}$ if $Z_2(t)$ is a block upper triangular matrix.
   	\end{itemize}
   \end{defi}
   Then, we can state the following 
   \begin{pro}\label{pro:asymptotic-module}
   	Given two semi-infinite matrices $Z_1(t)$ and $Z_2(t)$ such that
   	\begin{itemize}
   		\item  $Z_1(t)\in\mathfrak{l}V^{t_1}_0$,
   		\item  $Z_2(t)\in\mathfrak{u}$,
   		\item $Z_1(t)G=Z_2(t)$,
   	\end{itemize}
   	then
  $	Z_1(t)= Z_2(t)=0$.
   \end{pro}
   \begin{proof}
   	Observe that
      	\begin{align*}
   	Z_1(t)\big(V^{t_1}_0\big)^{-1}\big(S_1(t)\big)^{-1}=Z_2(t)\big(\tilde S _2(t)\big)^{-1},
   	\end{align*}
   	and, as in the LHS we have a strictly lower triangular block semi-infinite matrix while in the RHS we have an upper triangular block semi-infinite  matrix, both sides must vanish and the result follows.
   \end{proof}
  
   \begin{defi}
   	When $A-B\in\mathfrak{l}V^{t_1}_0$  we write $A=B+\mathfrak{l}V^{t_1}_0$ and if $A-B\in\mathfrak{u}$ we write $A=B+\mathfrak{u}$.
   \end{defi}
   
     \begin{pro}[Second and third order linear ODE]
     	Among others the  Baker function $\Psi_1$ satisfies the following linear differential equations
     	\begin{align}\label{eq: linear.wave}
     	\frac{\partial (\Psi_1)_k}{\partial \rho}&=\frac{\partial^2(\Psi_1)_k}{\partial \eta^2}- 2\frac{\partial U_k}{\partial\eta}(\Psi_1)_k,\\
     	  	\frac{\partial(\Psi_1)_k}{\partial \theta}&=\frac{\partial^3(\Psi_1)_k}{\partial \eta^3}
     	  	-3\frac{\partial U_k}{\partial\eta}\frac{\partial(\Psi_1)_k}{\partial \eta} 
     	-\frac{3}{2}\Big(\frac{\partial^2 U_k}{\partial\eta^2}+ \frac{\partial U_k}{\partial\rho}\Big)(\Psi_1)_k,
     	  	\end{align}
     \end{pro}
     \begin{proof}
     We only check the first relation, the second despite being  more complicated it follows from the same principles.
     	In the one hand,
     	\begin{align*}
    	\frac{\partial W_1}{\partial \rho}&=\Big(\frac{\partial S_1}{\partial\rho}+S_1\Lambda^2\Big)V^{t_1}_0\\
    \frac{\partial^2 W_1}{\partial \eta^2}&=
    	\Big(  \frac{\partial^2 S_1}{\partial \eta^2}+2\frac{\partial S_1}{\partial \eta}\Lambda
     	+S_1\Lambda^2\Big)V^{t_1}_0
     	\end{align*}
     so that
     	\begin{align*}
    \bigg( \frac{\partial }{\partial \rho}-	 \frac{\partial^2 }{\partial \eta^2}\bigg)(W_1)&=-2\bigg(
    \frac{\partial U}{\partial \eta }\Lambda\bigg)V^{t_1}_0+\mathfrak lV^{t_1}_0
     	\end{align*}
     	and, consequently,
     	   	\begin{align*}
     	   	Z_1:=\bigg( \frac{\partial }{\partial \rho}-	 \frac{\partial^2 }{\partial \eta^2}
     	   	+2 \frac{\partial U}{\partial \eta }\Lambda\bigg)(W_1)&\in\mathfrak lV^{t_1}_0.
     	   	\end{align*}
  
     	On the other hand,
        	\begin{align*}
       Z_2:= \frac{\partial \tilde S_2}{\partial \rho}-
    \frac{\partial^2 \tilde S_2}{\partial \eta^2}+2 \frac{\partial U}{\partial \eta }\Lambda\tilde S_2\in\mathfrak u.
        	\end{align*}
     	Now, we apply Proposition \ref{pro:asymptotic-module} 
     	to get the result.
     \end{proof}
     
    \begin{rem}[Noncommutative KP equation]
     The compatibility of both equations leads to 
     \begin{align}\label{eq:nckp}
    \frac{\partial}{\partial\eta}\Big( 4\frac{\partial U_k}{\partial\theta}+6\Big(\frac{\partial U_k}{\partial \eta}\Big)^2-\frac{\partial U_k}{\partial \eta^3}\Big)-\frac{\partial^2 U_k}{\partial\rho^2}+6\bigg[\frac{\partial U_k}{\partial\eta},\frac{\partial U_k}{\partial\rho}\bigg]=0.
     \end{align}
    \end{rem}
This matrix noncommutative KP equations has been considered for the first time in \cite{Athorne} and its relation with the matrix heat hierarchy, the  Hopf--Cole transformation and the construction of solutions was discussed in \cite{Guil}. In \cite{Kupershmidt} we can find an excellent treatment on the subject, see also \cite{Wang}. Recently,  in \cite{Gilson}  a discussion of solutions derived from the Darboux transformation technique is given, for further studies on its  solutions see \cite{Schiebold,Tacchella}. Notice, that the standard notation in the literature for the independent variables is $x,y$ and $t$,  as we have already used $x,y$ we decided to introduce the alternative notation $\eta,\rho$ and $\theta$ for the independent variables.
\begin{rem}
Notice that the linear systems for the Baker function as well as the noncommutative KP equation involve a single position in the lattice as the equations are for $U_k$, with $k$ fixed.
\end{rem}
\subsection{Transformation theory}
We now proceed to the study of how the previous Geronimus--Uvarov transformations (we recover  Geronimus  for $W_C(x)=I_p$)
affect the solutions of the non-Abelian 2D Toda lattice and the noncommutative KP hierarchies. That is, the interplay between the continuous Toda type perturbations and the discrete rational deformations dictated by the Geronimus--Uvarov transformations. It is important to notice that
$[\Lambda, V_0^t]=0$, to realize that both transformations, continuous flows of Toda type and Geronimus--Uvarov transformations,  commute. That is, given a Geronimus--Uvarov transformation
\begin{align*}
\hat u_{x,y} W_G(y) &= W_C(x) u_{x,y}, & \hat G W_G(\Lambda^\top)&=W_C(\Lambda) G,
\end{align*}
we will have the corresponding $t$-evolved  equations
\begin{align*}
\hat u^t_{x,y} W_G(y) &= W_C(x) u^t_{x,y}, & \hat G^t W_G(\Lambda^\top)&=W_C(\Lambda) G^t,
\end{align*}
where
\begin{align*}
\hat u^t_{x,y} &=\Exp{t_1(x)-t_2(y)} \hat u_{x,y}, &
\hat G^t &=  V_0^{t_1} \hat G \big(V_0^{t_2} \big)^{-\top}.
\end{align*}
Therefore, all previous results hold true. For example, Theorem \ref{teo:SCGU} gives , when $n\geq N_G$, with Hankel spectral massses as described in \eqref{uva}, the following Christoffel--Geronimus--Uvarov type formulas,

	\begin{align*}
	\hat P^{[1],t}_{n}(x)&=
	\Theta_*
	\begin{bmatrix}
	\boldsymbol{\mathcal J}_{C,P^{[1],t}_{n-N_G}} &	\boldsymbol{\mathcal J}_{G,C^{[1],t}_{n-N_G}}-
	\prodint{  P^{[1],t}_{n-N_G} (x),(\xi)_x}\mathcal W_G&P_{n-N_G}^{[1],t}(x)\\ 	\vdots & \vdots &\vdots \\ \boldsymbol{\mathcal J}_{C,P^{[1]}_{n+N_C}}&\boldsymbol{\mathcal J}_{G, C^{[1],t}_{n+N_C}}-\prodint{  P^{[1],t}_{n+N_C} (x),(\xi)_x}\mathcal W_G& P^{[1],t}_{n+N_C}(x)
	\end{bmatrix},\end{align*}
	\begin{align*}
	\hat H^t_{n}&=\Theta_*
	\begin{bmatrix}
	\boldsymbol{\mathcal J}_{C,P^{[1],t}_{n-N_G}} &	\boldsymbol{\mathcal J}_{G,C^{[1],t}_{n-N_G}}-
	\prodint{  P^{[1],t}_{n-N_G} (x),(\xi)_x}\mathcal W_G&H^t_{n-N_G}\\
	\boldsymbol{\mathcal J}_{C,P^{[1],t}_{n-N_G+1}} &	\boldsymbol{\mathcal J}_{G,C^{[1],t}_{n-N_G+1}}-\prodint{  P^{[1],t}_{n-N_G+1} (x),(\xi)_x}\mathcal W_G&0_p\\ 	\vdots & \vdots &\vdots \\ \boldsymbol{\mathcal J}_{C,P^{[1],t}_{n+N_C}}&\boldsymbol{\mathcal J}_{G, C^{[1],t}_{n+N_C}}-\prodint{  P^{[1],t}_{n+N_C} (x),(\xi)_x}\mathcal W_G & 0_p
	\end{bmatrix},\\
	\big(	\hat  P _{n}^{[2],t}(y)\big)^\top&=-\Theta_*\begin{bmatrix} 	\boldsymbol{\mathcal J}_{ C, P^{[1],t}_{n-N_G}}&
	\boldsymbol{\mathcal J}_{ G, C^{[1],t}_{n-N_G}}	-\prodint{  P^{[1],t}_{n-N_G} (x),(\xi)_x}\mathcal W_G& H^t_{n-N_G}\\
	\boldsymbol{\mathcal J}_{ C, P^{[1],t}_{n-N_G+1}}&
	\boldsymbol{\mathcal J}_{ G, C^{[1],t}_{n-N_G+1}}	-\prodint{  P^{[1],t}_{n-N_G+1} (x),(\xi)_x}\mathcal W_G& 0_p\\ 	\vdots & \vdots\\\boldsymbol{\mathcal J}_{C, P^{[1],t}_{n+N_C-1}}&	\boldsymbol{\mathcal J}_{G, C^{[1],t}_{n+N_C-1}}-\prodint{  P^{[1],t}_{n+N_C-1} (x),(\xi)_x}\mathcal W_G&0_p\\
	W_G(y) \boldsymbol{\mathcal J}_{C, K^t_{n-1}}(y) &  W_G(y)\big(	\boldsymbol{\mathcal J}_{ G,K^{(pc),t}_{n-1}}(y)-
	\prodint{  K^t_{n-1}(x,y),(\xi)_x}\mathcal W_G
	\big)+\boldsymbol{\mathcal J}_{G, \mathcal V}(y) &0_p
	\end{bmatrix}.
	\end{align*}

In particular, $\hat H_n^t$ is a new solution of the non-Abelian 2D Toda lattice equation \eqref{eq:2DToda} constructed in terms of the original  solution $H^t_n$ and the time dependent   spectral jets,  corresponding to both perturbing matrix polynomials, of $P^{[1],t}(x)$ and $C^{[1],t}(x)$.

Uvarov transformations can be applied similarly. In particular,  Theorem \ref{teo:uvarov} yields, whenever
 $\big(	\operatorname{Ker}^R_\beta\big)^{\perp^{u^t}_R}=\big(	\operatorname{Ker}^L_\beta\big)^{\perp^{u^t}_L}=\{0_p\}$, the perturbed matrix orthogonal polynomials and quasitau matrices
	\begin{align*}
	\hat P^{[1],t}_n(x)&=\Theta_*\begin{bmatrix}
	I_{Np}+\prodint{\mathcal  J^{[0,1]}_{K^t_{n-1}}(x),(\beta)_x}&\mathcal  J^{[0,1]}_{K^t_{n-1}}(x)\\
	\prodint{ P_n^{[1],t}(x),(\beta)_x}& P_n^{[1],t}(x)
	\end{bmatrix}, \\
	(\hat P^{[2],t}_n(y))^\top&=\Theta_*\begin{bmatrix}
		I_{Np}+\prodint{\mathcal  J^{[0,1]}_{K^t_{n-1}}(x),(\beta)_x}&\big(\mathcal J_{ P^{[2],t}_{n}}\big)^\top\\
		\prodint{K^t_{n-1}(x,z),(\beta)_x}	& (P_n^{[2],t}(y ))^\top
			\end{bmatrix},\\
	\hat H_n^t&=\Theta_*\begin{bmatrix}
	I_{Np}+\prodint{\mathcal  J^{[0,1]}_{K^t_{n-1}}(x),(\beta)_x}&- \beta(\mathcal J_{P^{[2],t}_n}  )^\top\\
	\prodint{ P_n^{[1],t}(x),(\beta)_x}& H^t_n
	\end{bmatrix}.
	\end{align*}
	
	\subsubsection{Application to the noncommutative KP hierarchy}
	We have seen is \S\ref{S:KP} how the noncommutative KP flows appear once we put  $t_{2}(y)=0$. 
	For the noncommutative KP flows the first subdiagonal coefficients $U_k$ of $S_1$ play a key role, as the hierarchy and the equations, see the noncommutative KP equation \eqref{eq:nckp}, are expressed in terms of them.  Let us recall that the
$U_k$'s is the coefficient multiplying the  term  $x^{n-1}$ in the polynomial $P_k^{[1],t}(z)$. Thus, it is of interest to have quasideterminantal expressions for the transformations of the $U$. From 	\eqref{eq:conexion}
we get
\begin{align}\label{eq:conexion2}
	\hat U_n = U_{n+N_C}-A_{C,N_C-1}+\omega_{n,n+N_C-1}. 
	\end{align}
	Now, we recall Proposition \ref{pro:resolventCPLST}, so that
		for   $n\geq N_G$, we have
\begin{multline*}
		\omega_{n,n+N_C-1}
			=\\
			-\begin{bmatrix} \boldsymbol{\mathcal J}_{C,P^{[1]}_{n+N_C}},
			\boldsymbol{\mathcal J}_{G,C^{[1]}_{n+N_C}}-		\prodint{  P^{[1]}_{n+N_C} (x),(\xi)_x} \mathcal W_G
			\end{bmatrix}\begin{bmatrix}
			\boldsymbol{\mathcal J}_{C,P^{[1]}_{n-N_G}} &	
			\boldsymbol{\mathcal J}_{G,C^{[1]}_{n-N_G}}-
			\prodint{  P^{[1]}_{n-N_G} (x),(\xi)_x} \mathcal W_G
			\\ 	\vdots & \vdots \\ \boldsymbol{\mathcal J}_{C,P^{[1]}_{n+N_C-1}}&\boldsymbol{\mathcal J}_{G, C^{[1]}_{n+N_C-1}}-	 \prodint{  P^{[1]}_{n+N_C-1} (x),(\xi)_x} \mathcal W_G		\end{bmatrix}^{-1}
			\begin{bmatrix}
			0_p\\\vdots\\0_p\\I_p
			\end{bmatrix}.
			\end{multline*}
	and we find the quasideterminantal expression
	\begin{align*}
	\hat U_n^t =\Theta_*\begin{bmatrix}
				\boldsymbol{\mathcal J}_{C,P^{[1],t}_{n-N_G}} &	
				\boldsymbol{\mathcal J}_{G,C^{[1],t}_{n-N_G}}-
				\prodint{  P^{[1],t}_{n-N_G} (x),(\xi)_x} \mathcal W_G &0_p
				\\ 	\vdots & \vdots \\ 	\boldsymbol{\mathcal J}_{C,P^{[1],t}_{n+N_C-2}}&\boldsymbol{\mathcal J}_{G, C^{[1],t}_{n+N_C-2}}-	 \prodint{  P^{[1],t}_{n+N_C-2} (x),(\xi)_x} \mathcal W_G& 0_p	\\
				\boldsymbol{\mathcal J}_{C,P^{[1],t}_{n+N_C-2}}&\boldsymbol{\mathcal J}_{G, C^{[1],t}_{n+N_C-1}}-	 \prodint{  P^{[1],t}_{n+N_C-1} (x),(\xi)_x} \mathcal W_G	&I_p	\\
				 \boldsymbol{\mathcal J}_{C,P^{[1],t}_{n+N_C}}&\boldsymbol{\mathcal J}_{G, C^{[1],t}_{n+N_C}}-	 \prodint{  P^{[1],t}_{n+N_C} (x),(\xi)_x} \mathcal W_G &U^t_n-A_{C,N_C-1}		\end{bmatrix}.
	\end{align*}
	
	For the Uvarov transformations we can also find quasideterminantal expressions for the transformed $U$'s. Just looking at the corresponding power in $x^{n-1}$, from 
	Theorem \ref{teo:uvarov}, whenever
	 $\big(	\operatorname{Ker}^R_\beta\big)^{\perp^{u^t}_R}=\big(	\operatorname{Ker}^L_\beta\big)^{\perp^{u^t}_L}=\{0_p\}$, we get
		\begin{align*}
		\hat U_n^t&=\Theta_*\begin{bmatrix}
		I_{Np}+\prodint{\mathcal  J^{[0,1]}_{K^t_{n-1}}(x),(\beta)_x}&\big(\mathcal  J_{P^{[2],t}_{n-1}}\big)^\top (H_{n-1}^t)^{-1}\\
		\prodint{ P_n^{[1],t}(x),(\beta)_x}& U^t_n
		\end{bmatrix}.
		\end{align*}
		
\subsection{Christoffel and Geronimus transformations, Miwa shifts, $\tau$-type matrix functions and Sato formulas}

We consider some interesting expressions of the biorthogonal polynomials and its second kind functions in terms of the matrix norms and its Christoffel and Geronimus transformations for monic matrix polynomials of degree one.
\begin{rem}We will use the following ordered  products 
$\prod\limits_{m=0}^{\stackrel{\curvearrowleft}{n} }A_m:=A_n A_{n-1}\cdots A_0$ and
$\prod\limits_{m=0}^{\stackrel{\curvearrowright}{n} }A_m:=A_0 A_{1}\cdots A_n$,
 for its use in integrable systems theory see, for example, \cite{Faddeev}.
\end{rem}

\begin{pro}\label{pro:deg1christoffelH}
For a Christoffel transformations with $W_C(x)=I_px-A$ we have 
\begin{align*}
P^{[1]}_{n+1}(A)&=(-1)^{n+1}\prod_{m=0}^{\stackrel{\curvearrowleft}{n} }\hat H_m (H_m)^{-1},&
\hat C^{[2]}_{n}(A^\top)&=(-1)^{n-1}( H_{n} )^{\top}\prod_{m=0}^{\stackrel{\mbox{\larger$\curvearrowleft$}}{n-1} }
(\hat  H_{m} )^{-\top}(H_{m})^\top.
\end{align*}
\end{pro}
\begin{proof}Degree one Christoffel transformation \cite{alvarez2015Christoffel}, fit in the discussion of  \S\ref{sGU} when $W_G(x)=I_p$ and $W_C(x)=I_p x-A$. From the connection formula 
\begin{align*}
\hat P^{[1]}_n(x)(I_px-A)&=P^{[1]}_{n+1}(x)+\omega_{n,n}P^{[1]}_n(x), &\omega_{n,n}&=\hat H_n (H_n)^{-1},
\end{align*}
recalling \eqref{eq:JP} for $N_G=0$ and $N_C=1,$ and also the form of the spectral jets, see the discussion in \S \ref{sGU} we get
\begin{align*}
P^{[1]}_{n+1}(A)=-\hat H_n (H_n)^{-1}P^{[1]}_n(A),
\end{align*}
from where the first relation follows.

For the second relation, we first observe that \eqref{eq:conexionC2LST} gives
\begin{align*}
\big(\hat C^{[2]}_{0}(x)\big)^{\top}( H_{0} )^{-1}&=
(I_p x-A)\big(C^{[2]}_0(x)\big)^\top(H_0)^{-1}-I_p,\\
\big(\hat C^{[2]}_{n-1}(x)\big)^{\top}(\hat H_{n-1} )^{-1}+\big(\hat C^{[2]}_{n}(x)\big)^{\top}( H_{n} )^{-1}&=
(I_p x-A)\big(C^{[2]}_n(x)\big)^\top(H_n)^{-1}, & n&\in\{1,2,\dots\}.
\end{align*}
Therefore,
\begin{align*}
\hat C^{[2]}_{0}(x)&=
C^{[2]}_0(x)(I_p x-A^\top)-( H_{0} )^{\top},\\
(\hat H_{k-1} )^{-\top}\hat C^{[2]}_{k-1}(x)
+( H_{k} )^{-\top}\hat C^{[2]}_{k}(x)&=
(H_k)^{-\top}C^{[2]}_k(x) (I_p x-A^\top), & k&\in\{1,2,\dots\},
\end{align*}
which implies
\begin{align*}
\hat C^{[2]}_{0}(A^\top)&=-( H_{0} )^{\top},\\
\hat C^{[2]}_{n}(A^\top)&=-( H_{n} )^{\top}(\hat  H_{n-1} )^{-\top}\hat C^{[2]}_{n-1}(A^\top), & n&\in\{1,2,\dots\}.
\end{align*}
\end{proof}

Similarly, 
\begin{pro}\label{pro:deg1geronimusH}
Given a  Geronimus transformation with $W_G(x)=I_px-A$ we have 
\begin{align*}
C^{[1]}_{n}(A)&=(-1)^{n-1}\check H_{n}\prod_{m=0}^{\stackrel{\mbox{\larger$\curvearrowleft$}}{n-1} } (H_{m})^{-1} \check H_{m},&
\check P^{[2]}_{n+1}(A^\top)&=(-1)^{n+1}\prod_{m=0}^{\stackrel{\curvearrowleft}{n} }(H_{m})^{\top}(\check H_m)^{-\top}.
\end{align*}
\end{pro}
\begin{proof}
For a Geronimus transformation with $W_G(x)=I_px-A$ and $W_C(x)=I_p$ we deduce, see \eqref{eq:conexionC1} and \eqref{eq:HG},
\begin{align*}
\check C^{[1]}_0(x)(I_px-A)&= C_0^{[1]}(x)+\check H_0,\\
\check C^{[1]}_n(x)(I_px-A)&= C_n^{[1]}(x)+\omega_{n,n-1}C_{n-1}^{[1]}(x), & \omega_{n,n-1}=\check H_n (H_{n-1})^{-1}.
\end{align*}
Now,  discussion in \S \ref{sGU} implies
\begin{align*}
C_0^{[1]}(A)&=-\check H_0,\\
 C_n^{[1]}(A)&=-\check H_n (H_{n-1})^{-1}C_{n-1}^{[1]}(A), & n&\in\{1,2,\dots\}.
\end{align*}
From \eqref{conex3'} we conclude
\begin{align*}
(I_px-A)\big( P^{[2]}_{n}(x)\big)^\top(H_n)^{-1}&=\big(\check P^{[2]}_{n}(x)\big)^\top (\check H_n)^{-1}+\big(\check P^{[2]}_{n+1}(x)\big)^\top (\check H_{n+1})^{-1} \omega_{n+1,n}\\&=\big(\check P^{[2]}_{n}(x)\big)^\top (\check H_n)^{-1}+\big(\check P^{[2]}_{n+1}(x)\big)^\top  (H_{n})^{-1}.
\end{align*}
Consequently, we get
\begin{align*}
(H_n)^{-\top}P^{[2]}_{n}(x)(I_px-A^\top)&=(\check H_n)^{-\top}\check P^{[2]}_{n}(x)+(H_{n})^{-\top}\check P^{[2]}_{n+1}(x),
\end{align*}
and again we are lead,  see \S \ref{sGU}, to
\begin{align*}
\check P^{[2]}_{n+1}(A^\top)=-(H_{n})^{\top}(\check H_n)^{-\top}\check P^{[2]}_{n}(A^\top).
\end{align*}
\end{proof}

\begin{defi}[Miwa shifts]
We  consider the following coherent time shift
\begin{align}
[z]:=\begin{bmatrix}
\dfrac{1}{z},\dfrac{1}{2z^2},\dfrac{1}{3z^3},\dots
\end{bmatrix}.
\end{align}
\end{defi}
\begin{rem}
The relevance of this Miwa shift comes from the   Taylor expansion of the logarithm
\begin{align*}
[z](x)&=\dfrac{x}{z}+\dfrac{x^2}{2z^2}+\dfrac{x^3}{3z^3}+\cdots\\
&=-\log\Big(1-\frac{x}{z}\Big), & |x|&<|z|,
\end{align*}
so that $\Exp{t_1(x)-[z]_1(x)}=\Big(1-\dfrac{x}{z}\Big)\Exp{t_1(x)}$, $|x|<|z|$, and  $\Exp{-(t_2(y)-[z]_2(y))}=\Big(1-\dfrac{y}{z}\Big)^{-1}\Exp{-t_2(y)}$, $|y|<|z|$.  The shifts are going to be taken inside the sesquilinear form. Consequently,  for the Miwa shift $[z]_1$ we request
$|z|>r_x$ and for $[z]_2$ that $|z|>r_y$, ensuring in this way the convergence of the series of the logarithm.

\end{rem}

We see that for the scalar case the composition of these Miwa shifts can generate any massless Geronimus--Uvarov transformation. However, for the matrix case, the Miwa shifts just introduced are more limited, as they are  scalar shifts, the matrix $A$ is a multiple of the identity $I_p$.  This is due, in part, to the simplification of considering the non-Abelian Toda flows instead of the multicomponent flows. However, in that more general scenario, the times are not any more scalar, but diagonal matrices in $\mathbb C^{p\times p}$ and, consequently, the corresponding composition of partial Miwa shifts will   achieve utmost diagonal massless Geronimus--Uvarov transformations. 
\begin{defi}[$\tau$-ratio matrices]\label{def:tau}
We introduce the matrices
\begin{align*}
\tau^{(1)}_n(t,s)&:=\prod_{m=0}^{\stackrel{\curvearrowleft}{n} } H^{t}_m (H^s_m)^{-1},&
\tau^{(2)}_n(t,s)&:=\prod_{m=0}^{\stackrel{\curvearrowleft}{n} }  (H^t_m)^{-1} H^{s}_m.
\end{align*}
\end{defi}
\begin{rem}
Observe that
\begin{align*}
(\tau^{(1)}_n(s,t))^{-1}&=\prod_{m=0}^{\stackrel{\curvearrowright}{n} } H^{t}_m (H^s_m)^{-1},&
(\tau^{(2)}_n(s,t))^{-1}&=\prod_{m=0}^{\stackrel{\curvearrowright}{n} }  (H^t_m)^{-1} H^{s}_m.
\end{align*}
\end{rem}
\begin{rem}[Why $\tau$-ratio?]
Notice that in the scalar case, $p=1$, we have, in terms of the the $\tau$-function $\tau_n(t)=\prod\limits_{m=0}^{n} H^{t}_m$,
\begin{align*}
\tau^{(1)}_n(t,s)&=\frac{\tau_n(t)}{\tau_n(s)}, & 
\tau^{(2)}_n(t,s)&=\frac{\tau_n(s)}{\tau_n(t)}.
\end{align*}
Thus, in the scalar case the $\tau$-ratio matrices are just the ratio of $\tau$ functions.
\end{rem}

\begin{teo}[Sato formulas]\label{teo:Sato}
In terms of $\tau$-ratio matrices and  Miwa shifts we have the following expressions for the biorthogonal matrix polynomials and its second kind functions
\begin{align*}
P^{[1],t}_{n}(z)&=z^{-n}\tau_{n-1}^{(1)}(t-[z]_1,t),&
\big( C^{[2],t}_{n}(z)\big)^\top(  H^{t}_{n} )^{-1}&=z^{n-1}\big(\tau_n^{(1)}(t,t+[z]_1)\big)^{-1}, & |z|&> r_x\\
(H^t_n)^{-1}C^{[1],t}_{n}(z)&=z^{n+1} \tau_n^{(2)}(t,t-[z]_2),&
\big( P^{[2],t}_{n}(z)\big)^\top&=z^{-n}\big(\tau_{n-1}^{(2)}(t+[z]_2,t)\big)^{-1}, & |z|&>r_y.
\end{align*}
\end{teo}
\begin{proof}
If we write $1-\dfrac{x}{z}=-\dfrac{1}{z}\Big(x-z\Big)$, we can understand the Miwa shifts $t-[z]_1$  ($t-[z]_2$) as a degree one Christoffel (Geronimus) transformations of the generalized kernel $u^t_{x,y}$. For the Gram matrices  we have $G^{t-[z]_1}=-z^{-1} \hat G$ and $G^{t-[z]_2}=-z \check G$, where $\hat G$ and $\check G$ denote the corresponding scalar Christoffel and massless Geronimus transformations with monic perturbing polynomials $I_p(x-z)$ and $I_p(y-z)^{-1}$ and $\hat u^t_{x,y}=(x-z)u^t_{x,y}$ and
$\check u_{x,y}=u_{x,y}(y-z)^{-1}$, respectively. Thus, the uniqueness of the Gauss--Borel factorization gives 
\begin{align*}
S_1^{t-[z]_1}&=\hat S_1^t, & H^{t-[z]_1}&=-z^{-1} \hat H^t, & S_2^{t-[z]_1}&=\hat S_2^t,\\
S_1^{t-[z]_2}&=\check S_1^t, & H^{t-[z]_2}&=-z \check H^t, & S_2^{t-[z]_2}&=\check S_2^t.
\end{align*}
Therefore, according to Definition \ref{defi:bio2kind} we get for $|z|>r_x$
\begin{align*}
P^{[1],t-[z]_1}(x)&=\hat P^{[1],t}(x), & P^{[2],t-[z]_1}(y)&=\hat P^{[2],t}(y), & C^{[1],t-[z]_1}(x)&=-z^{-1}\hat C^{[1],t}(x), & C^{[2],t-[z]_1}(y)&=-z^{-1}\hat C^{[2],t}(y),\end{align*}
and for $|z|>r_y$ we obtain that
\begin{align*}
P^{[1],t-[z]_2}(x)&=\check P^{[1],t}(x), & P^{[2],t-[z]_2}(y)&=\check P^{[2],t}(y), & C^{[1],t-[z]_2}(x)&=-z\check C^{[1],t}(x), & C^{[2],t-[z]_2}(y)&=-z\check C^{[2],t}(y).
\end{align*}
Then, Propositions \ref{pro:deg1christoffelH} and \ref{pro:deg1geronimusH}, with $A=I_pz$, give
\begin{align*}
P^{[1],t}_{n+1}(z)&=z^{-n-1}\prod_{m=0}^{\stackrel{\curvearrowleft}{n} } H^{t-[z]_1}_m (H^t_m)^{-1},&
 C^{[2],t-[z]_1}_{n}(z)&=z^{n-1}( H^t_{n} )^{\top}\prod_{m=0}^{\stackrel{\mbox{\larger$\curvearrowleft$}}{n-1} }
(  H^{t-[z]_1}_{m} )^{-\top}(H^t_{m})^\top,\\
C^{[1],t}_{n}(z)&=z^{n+1} H^{t-[z]_2}_{n}\prod_{m=0}^{\stackrel{\mbox{\larger$\curvearrowleft$}}{n-1} } (H^{t}_{m})^{-1}H^{t-[z]_2}_{m},&
 P^{[2],t-[z]_2}_{n+1}(z)&=z^{-n-1}\prod_{m=0}^{\stackrel{\curvearrowleft}{n} }(H^t_{m})^{\top}( H^{t-[z]_2}_m)^{-\top},
\end{align*}
so that, Miwa shifting some times we have
\begin{align*}
P^{[1],t}_{n+1}(z)&=z^{-n-1}\prod_{m=0}^{\stackrel{\curvearrowleft}{n} } H^{t-[z]_1}_m (H^t_m)^{-1},&
 C^{[2],t}_{n}(z)&=z^{n-1}( H^{t+[z]_1}_{n} )^{\top}\prod_{m=0}^{\stackrel{\mbox{\larger$\curvearrowleft$}}{n-1} }
(  H^{t}_{m} )^{-\top}(H^{t+[z]_1}_{m})^\top, & |z|&> r_x,\\
C^{[1],t}_{n}(z)&=z^{n+1} H^{t-[z]_2}_{n}\prod_{m=0}^{\stackrel{\mbox{\larger$\curvearrowleft$}}{n-1} } (H^{t}_{m})^{-1} H^{t-[z]_2}_{m},&
 P^{[2],t}_{n+1}(z)&=z^{-n-1}\prod_{m=0}^{\stackrel{\curvearrowleft}{n} }(H^{t+[z]_2}_{m})^{\top}( H^{t}_m)^{-\top}, & |z|&>r_y.
\end{align*}
and a transposition gives 
\begin{align*}
P^{[1],t}_{n+1}(z)&=z^{-n-1}\prod_{m=0}^{\stackrel{\curvearrowleft}{n} } H^{t-[z]_1}_m (H^t_m)^{-1},&
\big( C^{[2],t}_{n}(z)\big)^\top(  H^{t}_{n} )^{-1}&=z^{n-1}\prod_{m=0}^{\stackrel{\mbox{$\curvearrowright$}}{n} }
H^{t+[z]_1}_{m}(  H^{t}_{m} )^{-1}, & |z|&> r_x,\\
(H^t_n)^{-1}C^{[1],t}_{n}(z)&=z^{n+1} \prod_{m=0}^{\stackrel{\mbox{$\curvearrowleft$}}{n} }(H^{t}_{m})^{-1} H^{t-[z]_2}_{m},&
\big( P^{[2],t}_{n+1}(z)\big)^\top&=z^{-n-1}\prod_{m=0}^{\stackrel{\curvearrowright}{n} }( H^{t}_m)^{-1}H^{t+[z]_2}_{m}, & |z|&>r_y.
\end{align*}
\end{proof}
\begin{rem}
In the scalar case, $p=1$, recalling that $H^t_n=\dfrac{\tau_n(t)}{\tau_{n-1}(t)}$, we have the following standard Sato formulas, 
\begin{align*}
P^{[1],t}_{n}(z)&=z^{-n}\frac{\tau_n(t-[z]_1)}{\tau_n(t)},&
C^{[2],t}_{n}(z)&=z^{n-1}\frac{\tau_n(t+[z]_1)}{\tau_{n-1}(t)}, & |z|&> r_x\\
C^{[1],t}_{n}(z)&=z^{n+1} \frac{\tau_n(t-[z]_2)}{\tau_{n-1}(t)},&
P^{[2],t}_{n}(z)&=z^{-n}\frac{\tau_n(t+[z]_2)}{\tau_{n}(t)}, & |z|&>r_y.
\end{align*}
\end{rem}

\subsection{Bilinear identity and Geronimus--Uvarov and Uvarov transformations}
Bilinear identities \cite{date1,date2,date3,Miwa} are normally formulated for $\tau$-functions, but can be also written down for Baker functions \cite{ueno,ueno-takasaki0,ueno-takasaki1}.
They are useful when deriving the Hirota bilinear equations as well as a number of Fay identities. We have lengthy discussed about Geronimus--Uvarov transformations, and a question naturally arises. Does it hold a bilinear identity for the Baker functions in where  Geronimus--Uvarov transformations plays a role? and further: Do we have a similar bilinear identity for  the  biorthogonal matrix polynomials and its second kind functions? and for the $\tau$-matrices? The  answer to these questions follows.
\begin{teo}[Geronimus--Uvarov bilinear identities]\label{teo:bilinear}
Let us consider two matrix polynomials $W_C(x)$ and $W_G(x)$, with $W_C(x)$ monic, and  the corresponding Gerominus--Uvarov transformation as given in Definition \ref{def:Geronimus-Uvarov}. Then, for $r_1>r_x$ and $r_2>r_y$, the following bilinear identities holds
\begin{align}\label{eq:bilinealBaker}
\oint_{|z|=r_1}\hat \Psi_{1,k}(t',z) W_C(z)\big(\Psi^*_{1,l}(t,z)\big)^\top\,\d z&=
\oint_{|z|=r_2}\hat \Psi_{2,k}(t',z)W_G(z)\big(\Psi^*_{2,l}(t,z)\big)^\top\,\d z,\\\label{eq:bilinealC}
       \oint_{|z|=r_1} \Exp{t'_1(z)-t_1(z)}\hat P_k^{[1],t'}(z)W_C(z)
                   \big(C^{[2],t}_l(z)\big)^\top\,\d z&=\oint_{|z|=r_2}
                   \hat C^{[1],t'}_k(z)W_G(z)\big(P_l^{[2],t}(z)\big)^\top\Exp{t'_2(z)-t_2(z)} \,\d z,
\end{align} 
and in terms of the matrix functions $\tau^{(1)}(t,s)$ and $\tau^{(2)}(t,s)$ introduced in Definition \ref{def:tau} 
\begin{multline}\label{eq:bilineal_tau}
    \Big( \oint_{|z|=r_1} \Exp{t'_1(z)-t_1(z)}z^{l-k-1}\hat \tau_{k-1}^{(1)}(t'-[z]_1,t')W_C(z)
                   \big(\tau_l^{(1)}(t,t+[z]_1)\big)^{-1}\,\d z\Big)H^t_l\\=
                   \hat H^{t'}_n\Big(\oint_{|z|=r_2}
                   z^{k-l+1} \hat \tau_k^{(2)}(t',t'-[z]_2)W_G(z)\big(\tau_{l-1}^{(2)}(t+[z]_2,t)\big)^{-1}\Exp{t'_2(z)-t_2(z)} \,\d z\Big).
\end{multline} 
\end{teo}
\begin{proof}
Let us first prove \eqref{eq:bilinealBaker}. If in Proposition \ref{pro:stringLST} we put
\begin{align*}
P(x)&=\Exp{t'_1(x)}\hat P_k^{[1],t'}(x), & Q(y)&=\Exp{-t_2(y)} P_l^{[2],t}(y),
\end{align*}
we find
\begin{align}\label{eq:bilineal1}
\prodint{ \Exp{t'_1(x)}\hat P_k^{[1],t'}(x)W_C(x), \Exp{-t_2(y)} P_l^{[2],t}(y)}_u=\prodint{\Exp{t'_1(x)}\hat P_k^{[1],t'}(x),  \Exp{-t_2(y)} P_l^{[2],t}(y)\big(W_G(y)\big)^\top}_{\hat u}.
\end{align}
Now, given the entire character in the complex plane of all the functions involved, the Cauchy formula allows us to write
\begin{align*}
\Exp{t'_1(x)}\hat P_k^{[1],t'}(x)W_C(x)&=\frac{1}{2\pi\operatorname{i}}
\bigointsss_{|z|=r_1}
\frac{\Exp{t'_1(z)}\hat P_k^{[1],t'}(z)W_C(z)}{z-x}\,\d z, \\
\Exp{-t_2(y)} P_l^{[2],t}(y)
\big(W_G(y)\big)^\top&=\frac{1}{2\pi\operatorname{i}}
\bigointsss_{|z|=r_2}
\frac{\Exp{-t_2(z)} P_l^{[2],t}(z)\big(W_G(z)\big)^\top}{z-y}\,\d z, 
\end{align*}
for  $|x|<r_1$ and $|y|<r_2$. Thus, in order to apply it to \eqref{eq:bilineal1} we need $r_1>r_x$ and $r_2>r_y$, in this manner
the circles contain all the $x$ and $y$ projected supports, respectively. Consequently,
 we have
 \begin{align*} %\label{eq:bilineal2}
 \prodint{ \bigointsss_{|z|=r_1}
 \frac{\Exp{t'_1(z)}\hat P_k^{[1],t'}(z)W_C(z)}{z-x}\,\d z, \Exp{-t_2(y)} P_l^{[2],t}(y)}_u=\prodint{\Exp{t'_1(x)}\hat P_k^{[1],t'}(x),  \bigointsss_{|z|=r_2}
 \frac{\Exp{-t_2(z)} P_l^{[2],t}(z)\big(W_G(z)\big)^\top}{z-y}\,\d z}_{\hat u},
 \end{align*}
and, recalling  that the sesquilinear form is continuous in each of two entries,  we get
\begin{align*} %\label{eq:bilineal3}
\bigointsss_{|z|=r_1}\prodint{ 
 \frac{\Exp{t'_1(z)}\hat P_k^{[1],t'}(z)W_C(z)}{z-x}\,\d z, \Exp{-t_2(y)} P_l^{[2],t}(y)}_u=\bigointsss_{|z|=r_2}\prodint{\Exp{t'_1(x)}\hat P_k^{[1],t'}(x),  
 \frac{\Exp{-t_2(z)} P_l^{[2],t}(z)\big(W_G(z)\big)^\top}{z-y}\,\d z}_{\hat u}.
 \end{align*}
Now, from Definition \ref{def:sesquilinear} we deduce
 \begin{multline*} %\label{eq:bilineal4}
\bigointsss_{|z|=r_1} \Exp{t'_1(z)}\hat P_k^{[1],t'}(z)W_C(z)\prodint{ 
 \frac{I_p}{z-x}, \Exp{-t_2(y)} P_l^{[2],t}(y)}_u\,\d z\\=\bigointsss_{|z|=r_2}\prodint{\Exp{t'_1(x)}\hat P_k^{[1],t'}(x),  
  \frac{I_p}{z-y}}_{\hat u}W_G(z)\Exp{-t_2(z)} \big(P_l^{[2],t}(z)\big)^\top\,\d z,
  \end{multline*}
  and Proposition \ref{pro:baker} implies the result.
  
 We proceed with the proof of \eqref{eq:bilinealC}.  In this case, we use the Cauchy formulas
  \begin{align*}
\Exp{t'_1(x)-t_1(x)}\hat P_k^{[1],t'}(x)W_C(x)&=\frac{1}{2\pi\operatorname{i}}
  \bigointsss_{|z|=r_1}
  \frac{\Exp{t'_1(z)-t_1(z)}\hat P_k^{[1],t'}(z)W_C(z)}{z-x}\,\d z,\\
 \Exp{t'_2(y)-t_2(y)}P_l^{[2],t}(y)
  \big(W_G(y)\big)^\top&=\frac{1}{2\pi\operatorname{i}}
  \bigointsss_{|z|=r_2}
  \frac{ \Exp{t'_2(z)-t_2(z)} P_l^{[2],t}(z)\big(W_G(z)\big)^\top}{z-y}\,\d z, 
  \end{align*}
  so that we get
   \begin{multline*} %\label{eq:bilineal2}
   \prodint{ \Exp{t_1(x)}\bigointsss_{|z|=r_1}\Exp{t'_1(z)-t_1(z)}
   \frac{\hat P_k^{[1],t'}(z)W_C(z)}{z-x}\,\d z, \Exp{-t_2(y)} P_l^{[2],t}(y)}_u\\=\prodint{\Exp{t'_1(x)}\hat P_k^{[1],t'}(x),  \Exp{-t'_2(y)} \bigointsss_{|z|=r_2}
   \Exp{t'_2(z)-t_2(z)} \frac{P_l^{[2],t}(z)\big(W_G(z)\big)^\top}{z-y}\,\d z}_{\hat u}.
   \end{multline*}
   Then, we obtain
    \begin{multline*} %\label{eq:bilineal4}
     \bigointsss_{|z|=r_1} \Exp{t'_1(z)-t_1(z)}\hat P_k^{[1],t'}(z)W_C(z)\prodint{ 
    \frac{I_p\Exp{t_1(x)}}{z-x}, \Exp{-t_2(y)} P_l^{[2],t}(y)}_u\,\d z\\=\bigointsss_{|z|=r_2}\prodint{\Exp{t'_1(x)}\hat P_k^{[1],t'}(x),  
     \frac{I_p\Exp{-t'_2(y)} }{z-y}}_{\hat u}W_G(z)\big(P_l^{[2],t}(z)\big)^\top\Exp{t'_2(z)-t_2(z)} \,\d z,
     \end{multline*}
     which can be simplified to
        \begin{multline*} %\label{eq:bilineal4}
          \bigointsss_{|z|=r_1} \Exp{t'_1(z)-t_1(z)}\hat P_k^{[1],t'}(z)W_C(z)\prodint{ 
         \frac{I_p}{z-x},  P_l^{[2],t}(y)}_{u^t}\,\d z\\=\oint_{|z|=r_2}\prodint{\hat P_k^{[1],t'}(x),  
          \frac{I_p}{z-y}}_{\hat u^{t'}}W_G(z)\big(P_l^{[2],t}(z)\big)^\top\Exp{t'_2(z)-t_2(z)} \,\d z.
          \end{multline*}
          Now, Proposition \ref{pro:Cauchy1}	gives the result.
Finally, the $\tau$-version \eqref{eq:bilineal_tau} follows from Definition \ref{def:tau} and \eqref{eq:bilinealC}.
\end{proof}

A similar  technique leads to a corresponding result when we replace Geronimus--Uvarov by Uvarov transformations. We give no proof and consider only the relation between biorthogonal matrix polynomials and its second kind functions
	\begin{pro}[Uvarov bilinear identities]
	Let us consider 
	the Uvarov transformation  described in \eqref{eq:v_uvarov_general}. Then, the following bilinear identities holds
		\begin{multline*}
		    \bigointsss_{|z|=r_1} \Exp{t'_1(z)-t_1(z)}\hat P_k^{[1],t'}(z)
		                  \bigg( \big(C^{[2],t}_l(z)\big)^\top
		                   +\prodint{\frac{\Exp{t_1(x)}I_p}{z-x},(\beta)_x}\Big(\mathcal J_{\Exp{-t_2}P_2^{[2],t}}\Big)^\top\bigg)\,\d z\\=\bigointssss_{|z|=r_2}
		                   \hat C^{[1],t'}_k(z)\big(P_l^{[2],t}(z)\big)^\top\Exp{t'_2(z)-t_2(z)} \,\d z,
		\end{multline*} 
	for $r_1>r_x$ and $r_2>r_y$.
	\end{pro}
	
\begin{landscape}
\appendix*
	\section{Table with transformations:  chart of results}\label{appendix}
	\thispagestyle{plain} 
	%\begin{sideways}
	\begin{tabular}{@{}cc@{}}
		\toprule
		\multicolumn{2}{ c }{\textsc{Geronimus transformations}}\\	\cmidrule(r){1-2}	
		{$	
			\check u_{x,y}=u_{x,y}(W(y))^{-1}+\sum\limits_{a=1}^{q}\sum\limits_{j=1}^{s_a}\sum\limits_{m=0}^{\kappa_j^{(a)}-1}\dfrac{(-1)^{m}}{m!}\big(\xi^{[a]}_{j,m}\big)_x\otimes\delta^{(m)}(y-x_a)l_{j}^{(a)}(y)
			$}
		& {  Theorem \ref{teo:spectral} (spectral) \&  Theorem \ref{theorem:nonspectral} (nonspectral)}\\\cmidrule(r){1-1}\cmidrule(lr){2-2}
		\multicolumn{2}{ c }{\textsc{Unimodular Christoffel  transformations}}\\	\cmidrule(r){1-2}	
		$\check u_{x,y}=u_{x,y}W(y)$ with $\det(W(x))=\text{constant}$& \S\ref{s:unimodular}
		\\\cmidrule(r){1-1}\cmidrule(lr){2-2}	
		\multicolumn{2}{ c }{\textsc{Geronimus--Uvarov transformations}}\\
		\cmidrule(r){1-2}	
		{$	
			\hat u_{x,y}=W_C(x)u_{x,y}(W_G(y))^{-1}+\sum\limits_{a=1}^{q_G}\sum\limits_{j=1}^{s_{G,a}}\sum\limits_{m=0}^{\kappa_{G,j}^{(a)}-1}\dfrac{(-1)^{m}}{m!}W_C(x)\big(\xi^{[a]}_{j,m}\big)_x\otimes\delta^{(m)}(y-x_a)l_{G,j}^{(a)}(y)
			$}&  {  Theorem \ref{teo:SCGU} (spectral) \&  Theorem \ref{theorem:nonspectralLST}} (mixed)\\
		\cmidrule(r){1-1}\cmidrule(lr){2-2}\multicolumn{2}{ c }{\textsc{Christoffel transformations with singular leading coefficient and Hankel orthogonality}}\\	 
		{$	\hat  u_{x,x}=u_{x,x}W(x)$ with $\det A_N=0$}
		& { \S \ref{S:chris}}\\
		\cmidrule(r){1-1}\cmidrule(lr){2-2}	\multicolumn{2}{ c }{\textsc{Christoffel and Geronimus symmetric transformations with nonsingular leading coefficient and Hankel orthogonality}}\\	\cmidrule(r){1-2}	
		{$	\hat  u_{x,x}=W(x)u_{x,x}(W(x))^\top$ with $\det A_N\neq 0$}
		& \multirow{2}{*}{ \S  \ref{s:symmetric}}
		\\	
		{$	\hat  u_{x,x}=(W(x))^{-1}u_{x,x}(W(x))^{-\top}$ with $\det A_N\neq 0$}
		& 
		\\
		\cmidrule(r){1-1}\cmidrule(lr){2-2}	\multicolumn{2}{ c }{\textsc{Christoffel and Geronimus nonsymmetric transformations with nonsingular leading coefficient and Hankel orthogonality}}\\	\cmidrule(r){1-2}		
		%\noalign{\vskip 1pt} 
		{$	\hat  u_{x,x}=W(x)u_{x,x}(V(x))^\top$ with $\det A_N(W)\neq 0$}
		& \multirow{2}{*}{ \S  \ref{s:nsymmetric}}
		\\	
		{$	\hat  u_{x,x}=(W(x))^{-1}u_{x,x}(V(x))^{-\top}$ with $\det A_N(W)\neq 0$}
		& \\\cmidrule(r){1-1}\cmidrule(lr){2-2}	
		\multicolumn{2}{ c }{\textsc{Uvarov transformations}}\\	\hline	
		{$	
			\hat u_{x,y}=u_{x,y}+\sum\limits_{a=1}^{q}\sum\limits_{m=0}^{\kappa^{(a)}-1}\dfrac{(-1)^{m}}{m!}\big(\beta^{(a)}_{m}\big)_x\otimes\delta^{(m)}(y-x_a)
			$}
		& { Theorem \ref{teo:uvarov}}\\
		\bottomrule	\end{tabular}
	%\end{sideways}
\end{landscape}


\begin{thebibliography}{99}
	
	%\bibitem{aba} K. Abadir and J. Magnus, \emph{Matrix Algebra,} Cambridge University Press, Cambridge 2005.
	
	\bibitem{adler}M. Adler and P. van Moerbeke, \emph{Group factorization, moment matrices and Toda lattices}, International Mathematics Research Notices \textbf{12} (1997) 556-572.
	
	\bibitem{adler-van moerbeke} M. Adler and P. van Moerbeke,  \emph{Generalized orthogonal polynomials, discrete KP and
		Riemann--Hilbert  problems},  Communications in  Mathematical  Physics \textbf{207} (1999) 589-620.
	
		\bibitem{adler-vanmoerbeke 0} M. Adler and P. van Moerbeke, \emph{Vertex operator solutions to the discrete KP
		hierarchy},  Communications in  Mathematical Physics \textbf{203} (1999) 185-210.
	
	\bibitem{adler-van moerbeke 1} M. Adler and P. van Moerbekee,  \emph{The spectrum of coupled random matrices}, Annals of Mathematics  \textbf{149} (1999) 921-976.
	
	\bibitem{adler-van moerbeke 1.1} M. Adler and P. van Moerbeke, \emph{Hermitian, symmetric and symplectic
	random ensembles: PDEs for the distribution of the spectrum}, Annals of  Mathematics \textbf{153} (2001) 149-189
	
	\bibitem{adler-van moerbeke 2} M. Adler and P. van Moerbeke, \emph{Darboux transforms on band matrices, weights and associated polynomials}, International Mathematics Research Notices  \textbf{18} (2001) 935-984.
	
	\bibitem{adler 2} M. Adler, P. van Moerbeke, and P. Vanhaecke, \emph{Moment matrices and multi-component KP, with applications to random matrix theory}, Communications in  Mathematical Physics \textbf{286} (2009) 1-38.
	
	%\bibitem{ahi} N. I. Akhiezer, \emph{Classical moment problem and some related questions in analysis,} Translated by N. Kemmer Hafner Publishing Co., New York 1965.
	
	\bibitem{am2013} C. Álvarez-Fernández and M. Mañas, \emph{Orthogonal Laurent Polynomials on the unit circle, extended CMV ordering an 2D Toda type integrable hierarchies}, Advances in Mathematics \textbf{240} (2013) 132-193.
	
	\bibitem{manas3} C. Álvarez-Fernández, M. Mañas and U. Fidalgo Prieto, \emph{The multicomponent 2D Toda hierarchy: generalized matrix orthogonal polynomials, multiple orthogonal polynomials and Riemann--Hilbert problems}, Inverse Problems \textbf{26} (2010) 055009 (15pp).
	
	\bibitem{amu} C. Álvarez-Fernández, U. Fidalgo Prieto, and M. Mañas, \emph{Multiple orthogonal polynomials of mixed type: Gauss-Borel factorization and the multi-component 2D Toda hierarchy } Advances in Mathematics \textbf{227} (2011) 1451-1525.
	
	\bibitem{am} C. Álvarez-Fernández and M. Mañas, \emph{On the Christoffel-Darboux formula for generalized matrix orthogonal polynomials}, Journal of Mathematical Analysis and Applications \textbf{418} (2014) 238-247.
	
	\bibitem{alvarez2015Christoffel} C. Álvarez Fernández, G. Ariznabarreta, J. C. García-Ardila, M. Mañas,	and F.~Marcellán, \emph{Christoffel transformations for matrix orthogonal polynomials in the real line and the non-Abelian 2D Toda lattice hierarchy}, International Mathematics Research Notices, doi: 10.1093/imrn/rnw027.
	
	\bibitem{Al} R. Álvarez-Nodarse,  A. J. Durán,  and A. Mart\'inez de los Ríos, \emph{Orthogonal matrix polynomials satisfying second order difference equation}, Journal of Approximation Theory \textbf{169} (2013) 40-55.
	
	\bibitem{Ni} A. I. Aptekarev and  E. M. Nikishin, \emph{The scattering problem for a discrete Sturm-Liouville operator}, 
	%Mat. Sb \textbf{121} (163): 327-358 (1983); 
	Mathematics of the USSR-Sbornik \textbf{49} (1984) 325-355.
	
	\bibitem{ari} G. Ariznabarreta and M. Mañas, \emph{Matrix orthogonal Laurent polynomials on the unit circle and Toda type integrable systems,} Advances in Mathematics \textbf{264} (2014) 396-463.
	
	\bibitem{MVOPR} G. Ariznabarreta and M. Mañas, \emph{Multivariate orthogonal polynomials and integrable systems}, 	Advances in Mathematics 302 (2016) 628-739.
	
	\bibitem{ari0} G. Ariznabarreta and M. Mañas, \emph{Darboux transformations for multivariate orthogonal polynomials,}  arXiv:1503.04786.
	
	
	\bibitem{Ariznabarreta-linear} G. Ariznabarreta and M. Mañas, 	\emph{Linear spectral transformations for multivariate orthogonal polynomials and multispectral Toda hierarchies}, \texttt{	 arXiv:1511.09129}.
	
	
	\bibitem{ari1} G. Ariznabarreta and  M. Mañas, \emph{Multivariate orthogonal Laurent polynomials and integrable systems,} 	arXiv:1506.08708.
	
	\bibitem{Athorne} 	C. Athorne and A.  Fordy, \emph{Integrable equations in (2+1) dimensions associated with symmetric and homogeneous spaces}, Journal of Mathematical Physics \textbf{28}  (1987) 2018
	
	
	%\bibitem{berg} C. Berg, \emph{Fibonacci numbers and orthogonal polynomials},  J. Comput. Appl. Math.  \textbf{17} (2011)75-88.
	
%	\bibitem{BtK} M. J. Bergvelt and A. P. E. ten Kroode, \emph{Tau-functions and zero-curvature equations of Toda-AKNS type}, Journal of   Mathematical Physics \textbf{29} (1988) 1308-1320.
	
	\bibitem{BtK2} M. J. Bergvelt and A. P. E. ten Kroode, \emph{Partitions, Vertex Operators Constructions and Multi-Component KP Equations}, Pacific Journal of Mathematics \textbf{171} (1995) 23-88.
	
	\bibitem{Bogdanov} L. V. Bogdanov, \emph{Analytic-Bilinear Approach to Integrable Hierarchies}, Mathematics and its Applications \textbf{493}, Kluwer Academic Publishers, 1999. 
	
	%\bibitem{Bre} C. Brezinski,  \emph{Padé-type approximation and general orthogonal polynomials}. International Series of Numerical Mathematics, \textbf{50}. Birkhäuser Verlag, Basel-Boston, Mass. 1980.
	
	\bibitem{Bue1} M. I. Bueno and F. Marcellán, \emph{Darboux transformation and perturbation of linear functionals}, Linear Algebra and its Applications  \textbf{384} (2004) 215-242.
	
	\bibitem{Bueno} M. I. Bueno and F. Marcellán, \emph{Polynomial perturbations of bilinear functionals and Hessenberg matrices}, Linear Algebra and its Applications \textbf{414 }(2006) 64-83.
	
	\bibitem{Cachafeiro} A. Cachafeiro, F. Marcellán, \emph{ Modifications of Toeplitz matrices: jump functions}. Rocky Mountain Journal Mathematics  \textbf{23}  (1993),  no. 2, 521–531.
	
	\bibitem{Cantero}  M. J. Cantero, F. Marcellán, L. Moral, and L. Velázquez, \emph{Darboux transformations for CMV matrices}, Advances in Mathematics \textbf{298} (2016) 122-206.
	
	\bibitem{Cas1} M. M. Castro and F. A. Grünbaum, \emph{Orthogonal matrix polynomials satisfying first order differential equations: a collection of instructive examples,} Journal of Nonlinear Mathematical Physics \textbf{4} (2005) 63-76.
	%\bibitem{Chel} V. S. Chelyshkov, \emph{Alternative Jacobi polynomials and orthogonal exponentials}, arXiv:1105.1838.
	
	%\bibitem{Abdon_Luis} A. E. Choque-Rivero and  L. E. Garza, \emph{ Moment perturbation of matrix polynomials,} Integral Transforms Spec. Funct. \textbf{26} (2015) 177-191.
	%\bibitem{collar} A. R. Collar, \emph{On the reciprocation of certain matrices}, Proc. Roy. Soc. Edinburgh \textbf{59} (1939) 195-206.
	
	\bibitem{Castillo} K. Castillo, L.  Garza, F. Marcellán, \emph{Linear spectral transformations, Hessenberg matrices, and orthogonal polynomials}, Rendiconti del Circolo Matematico di Palermo (Special number) Supplement  No. \textbf{82}  (2010) 3–26.

	\bibitem{Chen} Y. Chen and J. Griffin,  \emph{Krall-type polynomials via the Heine formula}, Journal of Physics A: Mathematical \& General \textbf{35} (2002), 637-656.
	
			\bibitem{Chi} T. S. Chihara, \emph{An Introduction to Orthogonal Polynomials}. In :\emph{ Mathematics and its Applications Series}, Vol. \textbf{13}. Gordon and Breach Science Publishers, New York-London-Paris, 1978.
			
	
	\bibitem{christoffel} E. B. Christoffel, \emph{Über die Gaussische Quadratur und eine Verallgemeinerung derselben}, 
	Journal für die reine und angewandte Mathematik (Crelle's journal) \textbf{55 }(1858) 61-82 (in German).
	
	
	\bibitem{cooke} R. Cooke, \emph{Infinite Matrices and Sequences Spaces}, MacMillan, London, 1950. reprinted in  Dover Books on Mathematics, Dover Publications, 2014.
	
	\bibitem{DAS} D. Damanik, A. Pushnitski, and B. Simon, \emph{The analytic theory of matrix orthogonal polynomials}, Surveys in Approximation Theory \textbf{4} (2008) 1-85.
	
	%	\bibitem{darboux} G. Darboux, \emph{Mémoire sur l'approximation des fonctions de très-grands nombres, et sur une classé etendue de developpements en série},  J. Math. Pur. Appl.  (3) \textbf{4} (1878) 5-56; 377-16.
	
	\bibitem{darboux2} G. Darboux,\emph{ Sur une proposition relative aux équations linéaires},  Comptes Rendus hebdomadaires Academie des Sciences Paris \textbf{94} (1882) 1456-1459 (in French).
	
	\bibitem{date1}  E. Date, M. Jimbo, M. Kashiwara,  and T. Miwa, \emph{Operator approach to the  Kadomtsev--Petviashvili equation. Transformation groups for soliton equations. III},  Journal of the Physical Society of Japan  \textbf{50}  (1981) 3806-3812.
	
	\bibitem{date2}  E. Date, M. Jimbo, M. Kashiwara,  and T. Miwa, \emph{Transformation groups for soliton equations.  Euclidean Lie algebras and reduction of the KP hierarchy}, Publications  of the  Research Institute of Mathematical Sciences \textbf{18} (1982) 1077-1110.
	
	\bibitem{date3}  E. Date, M. Jimbo, M. Kashiwara,  and T. Miwa, \emph{Transformation groups for soliton equations} in  \emph{Nonlinear Integrable Systems-Classical Theory and Quantum Theory}  M. Jimbo and T. Miwa (eds.). World  Scientific, Singapore, 1983.
	
	%\bibitem{delgado} A. M. Delgado, J. S. Geronimo, P. Iliev, and F. Marcellán, \emph{Two variable orthogonal polynomials and structured matrices,} SIAM J. Matrix Anal. Appl. \textbf{28} (2006) 118-147.
	
			\bibitem{Delgado2}	A. M. Delgado, L. Fernández, T. E. Pérez, M. A. Piñar, \emph{Multivariate orthogonal polynomials and modified moment functionals}, \texttt{arXiv:1601.07194}.
			
		\bibitem{Delgado}	A. M. Delgado, L. Fernández, T. E. Pérez, M. A. Piñar, and Y. Xu,  \emph{Orthogonal polynomials in several variables for measures with mass points}, Numerical Algorithms \textbf{55} (2010) 245–264.
		

	\bibitem{Derevyagin} M. Derevyagin	and F. Marcellán,	\emph{A note on the Geronimus transformation and Sobolev orthogonal polynomials},   Numerical Algorithms \textbf{67} (2014)  271-287.
	
	\bibitem{DereM}  M. Derevyagin, J. C. García-Ardila, and F. Marcellán, \emph{Multiple Geronimus transformations}, Linear Algebra and its Applications \textbf{454} (2014) 158-183.
	
	\bibitem{Dickey} 	L. A. Dickey, \emph{On $\tau$-Functions of Zakharov—Shabat and other Matrix Hierarchies of Integrable Equations} in
	\emph{Algebraic Aspects of Integrable Systems},  eds. A. S. Fokas and I. M. Gelfand,
	   Progress in Nonlinear Differential Equations and Their Applications \textbf{26},  49-74, Birkhäuser, Boston 1997.

		\bibitem{dini}	J. Dini and  P. Maroni,  \emph{La multiplication d'une forme linéaire par une fraction rationnelle. Application aux formes de Laguerre-Hahn}, Annales Polonici Mathematici \textbf{52} (1990)  175–185  (in French).
	
	\bibitem{dsm} A. Doliwa, P. M. Santini, and M. Mañas, \emph{Transformations of Quadrilateral Lattices}, Journal of Mathematical Physics \textbf{41} (2000) 944-990.
	
		\bibitem{Dur5} A. J. Durán, \emph{On orthogonal polynomials with respect to a definite positive matrix of measures},  Canadian Journal of Mathematics \textbf{47} (1995) 88-112.
	
	 \bibitem{Dur4} A. J. Durán \emph{Markov's theorem for orthogonal matrix polynomials}, Canadian Journal of Mathematics \textbf{48} (1996) 1180-1195.
	
	\bibitem{Dur3} A. J. Durán, \emph{Matrix inner product having a matrix symmetric second order differential operator}, Rocky Mountain Journal of Mathematics \textbf{27} (1997) 585-600.
	
	\bibitem{Dur6} A. J. Durán and F. A. Grünbaum, \emph{A survey on orthogonal matrix polynomials satisfying second order differential equations}, Journal of Computational and Applied Mathematices \textbf{178} (2005) 169-190.
	
	%\bibitem{Dur7} A. J. Durán and F. A. Grünbaum, \emph{Structural formulas for orthogonal matrix polynomials satisfying second order differential equations,} Constr. Approx. Approx.  {\bf 22} (2005) 255-271.
	
	%\bibitem{Dur2} A. J. Durán and W. Van Assche, \emph{Orthogonal matrix polynomials and higher-order recurrence relations,} Linear Algebra Appl. \textbf{219} (1995) 261-280.
	
	
	\bibitem{eisenhart} L. P. Eisenhart, \emph{Transformations of Surfaces}, Princeton University Press, London, 1923. Reprinted by  Sagwan Press, 2015.

	\bibitem{Faddeev} L. D. Faddeev and  L.Thakhtajan,	\emph{Hamiltonian Methods in the Theory	of Solitons},  Classics  in Mathematics, Springer Verlag, 2007, reprinted from the 1987 version.
	
	\bibitem{Fuh} P. A. Fuhrmann, \emph{Orthogonal matrix polynomials and system theory}, Rendiconti del Seminario Matematico Università e Politecnico di Torino 1987, Special Issue, 68-124.
	
	\bibitem{Gaut1} W. Gautschi, \emph{An algorithmic implementation of the generalized Christoffel theorem} in Numerical Integration, edited by G. Hämmerlin, International Series of Numerical Mathematics, \textbf{57}, Birkhäuser, Basel, 1982. 89-106.
	
	\bibitem{Gaut2} W. Gautschi, \emph{Orthogonal Polynomials Computation and Approximation}, Numerical Mathematics and Scientific Computation, Oxford University Press, Oxford 2004.
	
	\bibitem{gelfand-distribu1} I. M. Gel'fand and G. E. Shilov, \emph{Generalized Functions. Volume I: Properties and Operations}, Academic Press, New York, 1964.
	Reprinted in the  AMS Chelsea Publishing,  American Mathematical Society, Providence, RI, 2016.
	
	\bibitem{gelfand-distribu2} I. M. Gel'fand and G. E. Shilov, \emph{Generalized Functions. Volume II: Spaces of Fundamental Solutions and Generalized Functions}, Academic Press, New York, 1968. 	Reprinted in the  AMS Chelsea Publishing,  American Mathematical Society, Providence, RI, 2016.
	
	\bibitem{gelfand} I. M. Gel'fand, S. Gel'fand,  V. S. Retakh, and R. Wilson,  \emph{Quasideterminants}, Advances in   Mathematics \textbf{193} (2005) 56-141.
	
	\bibitem{Ger} J. S. Geronimo, \emph{Scattering theory and matrix orthogonal polynomials on the real line}, Circuits Systems Signal Process \textbf{1} (1982) 471-495.
	
	\bibitem{Geronimus} J. Geronimus, \emph{On polynomials orthogonal with regard to a given sequence of numbers and a theorem by W. Hahn},
	%Comm. Inst. Sci. Math. Mec. Univ. Kharkoff [Zapiski Inst. Mat. Mech.] (4) \textbf{17} (1940) 3-18.
	Izvestiya Akademii Nauk SSSR \textbf{4} (1940) 215–228 (in Russian).
	
	\bibitem{Gilson} C. R. Gilson, J. J. C. Nimmo, and C. M. Sooman, \emph{Matrix Solutions of a Noncommutative KP Equation
	and a Noncommutative mKP Equation}, Theoretical and Mathematical Physics \textbf{159} (2009) 796–805. 

	\bibitem{Godoy1} E. Godoy and F. Marcellán,  \emph{An analog of the Christoffel formula for polynomial modification of a measure on the unit circle},
Bollettino dell'Unione Matematica Italiana  A (7)  \textbf{5} (1991)   1–12.
	
	\bibitem{Godoy2} E. Godoy and  F.  Marcellán, \emph{Orthogonal polynomials and rational modifications of measures}, Canadian Journal of Mathematics  \textbf{45}  (1993) 930–943.
	
	\bibitem{lan1} I. Gohberg, P. Lancaster,  and L. Rodman, \emph{Matrix Polynomials}, Computer Science and Applied Mathematics. Academic Press, Inc. [Harcourt Brace Jovanovich, Publishers], New York-London, 1982.
	
	
	\bibitem{Golinskii} L. Golinskii , \emph{On the scientific legacy of Ya.L. Geronimus (to the hundredth anniversary)},  in \emph{Self-Similar Systems} (Proceedings of the International Workshop (July 30 - August 7, Dubna, Russia, 1998)), 273-281, Edited by V.B. Priezzhev and V.P. Spiridonov, Publishing Department, 
	Joint Institute for Nuclear Research, 
	Moscow Region, Dubna. 
	
\bibitem{gru1} F. A. Grünbaum, \emph{The Darboux process and a noncommutative bispectral problem: some explorations and challenges,} Geometric Aspects of Analysis and Mechanics Progress in Mathematics {\bf 292}  (2011),  161-177, Birkhäuser/Springer, New York, 2011.

\bibitem{gru2} F. A. Grünbaum and L. Haine, \emph{Bispectral Darboux transformations: an extension of the Krall polynomials,} International Mathematics Research Notices \textbf{8} (1997) 359-392.
	
%	\bibitem{gru} F. A. Grünbaum and M. Yakimov, \emph{Discrete bispectral Darboux transformations from Jacobi operators}, Pacific J. Math. \textbf{204} (2002) 395-431.
	
	\bibitem{grupach} F. A. Grünbaum,  I. Pacharoni, and J.  Tirao, \emph{Matrix valued orthogonal polynomials of the Jacobi type}, Indagationes Mathematicae  \textbf{14} (2003) 353-366.
	
	\bibitem{grupach2} F. A. Grünbaum,  I. Pacharoni, and J.  Tirao, \emph{Matrix valued orthogonal polynomials of the Jacobi type: The role of group representation theory}, Annales de  l'Institute Fourier  \textbf{55} (2005) 2051-2068.
	
	\bibitem{Guil} F. Guil, M. Mañas, and G. Álvarez-Galindo, \emph{The Hopf-Cole Transformation and the KP Equation}, Physics Letters A \textbf{190 }(1994) 49-52.
	
	\bibitem{Hormander} L. Hörmander, \emph{The analysis of Partial Differential Operators I, Distribution Theory and Fourier Analysis}, second edition, Springer, New York 1990.
	
	\bibitem{hor} R. A. Horn and C. R. Johnson, \emph{Matrix Analysis}, Second Edition, Cambridge University Press, Cambridge, 2013.
	
\bibitem{Hann} W. 	Hahn, \emph{Über die Jacobischen Polynome und zwei verwandte Polynomklassen}, Mathematische Zeitschrift \textbf{39} (1935), 634-38 (in German).
	
	\bibitem{ismail} M. E. Ismail and  R. W. Ruedemann, \emph{The relation between Polynomials Orthogonal on the Unit Circle with Respect Different Weights}, Journal of Approximation Theory \textbf{71} (1994) 39-60.
	
	%	\bibitem{hor2} R. A. Horn and C. R. Johnson, \emph{Topics in Matrix Analysis}, First Edition, Cambridge University Press, Cambridge, 1991.
	
	\bibitem{kac} V. G. Kac and J. W. van de Leur, \emph{The $n$-component KP hierarchy and representation theory}, Journal of Mathematical Physics \textbf{44} (2003) 3245-3293.
	
	%\bibitem{Karl} L. Karlberg and  H. Wallin, \emph{Padé type approximants and orthogonal polynomials for Markov-Stieltjes functions,} J. Comput. Appl. Math. {\bf 32} (1990) 153-157.
	
	\bibitem{Krein} M. G. Krein,  \emph{The fundamental propositions of the theory of representations of Hermitian operators with deficiency index $(m,m)$}, Ukrainian Mathematical Journal {\bf 1} (1949) 3-66 (in Russian).
	
	\bibitem{AKrall} A. M. Krall,  \emph{Orthogonal polynomials satisfying fourth order differential equations}, Proceedings of the Royal Society of Edinburgh Section A  \textbf{87 } (1980/81) 271–288.
	
	\bibitem{HKrall} H. L. Krall,  \emph{On orthogonal polynomials satisfying a certain fourth order differential equation}, Pennsylvania State College Studies,  1940  (1940) no. 6, 24 pp.
	
	\bibitem{krichever} I. M. Krichever, \emph{The periodic non-Abelian Toda chain and its two-dimensional generalization} (Appendix to the paper B A Dubrovin \emph{Theta functions and non-linear equations}, Russian Mathematical Surveys \textbf{36} (1981) 11-92.
	
		\bibitem{Kupershmidt} B. A. Kupershmidt, \emph{KP or mKP. Noncommutative mathematics of Lagrangian, Hamiltonian, and integrable systems},  Mathematical Surveys and Monographs  \textbf{78}, American Mathematical Society, Providence, RI, 2000. 
	
	\bibitem{manas-1} M. Mañas, L. Martínez Alonso, and E. Medina, \emph{Dressing methods for geometric nets: I. Conjugate nets},	Journal of Physics A: Mathematical \&  General  \textbf{33} (2000) 2871-2894.
	
	\bibitem{manas0} M. Mañas, L. Martínez Alonso, and E. Medina, \emph{Dressing methods for geometric nets: II. Orthogonal and Egorov nets}, Journal of Physics A: Mathematical \&  General \textbf{33} (2000) 7181-7206.
	
	\bibitem{manas1} M. Mañas, L. Martínez Alonso, and C. Álvarez-Fernández, \emph{The multicomponent 2D Toda hierarchy: discrete flows and string equations}, Inverse Problems \textbf{25} (2009) 0650077 (31pp).
	
	\bibitem{manas2} M. Mañas and L. Martínez Alonso, \emph{The multicomponent 2D Toda hierarchy: dispersionless limit}, Inverse Problems \textbf{25} (2009) 115020 (22pp).
	
	%\bibitem{MarSan} F. Marcell\'an and G. Sansigre, \emph{On a class of matrix orthogonal polynomials on the real line},   Linear Algebra Appl. \textbf{ 181} (1993)  97-109.
	
	\bibitem{Marc1} F. Marcellán, P. Maroni, \emph{Sur l'adjonction d'une masse de Dirac à une forme régulière et semi-classique}. Annali di Matematica Pura ed Applicata  \textbf{162} (1992) 1–22 (in French). 
	
	\bibitem{Marcellan2014Sobolev} F. Marcellán and Y. Xu,  \emph{On Sobolev orthogonal polynomials}, Expositiones Mathematicae \textbf{33} (2015) 308–352.
	
	\bibitem{Mark2} A. S. Markus \emph{Introduction to the spectral theory of polynomials operator pencil,} Translated from the Russian by H. H. McFaden. Translation edited by Ben Silver. With an appendix by M. V. Keldysh. Translations of Mathematical Monographs, \textbf{71}. American  Mathematical Society, Providence, RI, 1988.
	
%	\bibitem{Mark1} A. S. Markus and I. V. Mereuca, \emph{On the complete $n$-tuple of roots of the operator equation corresponding to a polynomial operator bundle} IIzvestiya Akademii Nauk SSSR  Ser. Mat. {\bf 37}  (1973) 1108-1131.
	
	\bibitem{Maroni1985espaces} P. Maroni, \emph{Sur quelques espaces de distributions qui sont des formes
		linéaires sur l'espace vectoriel des polynômes}, in Orthogonal polynomials
	and their applications, vol. 1171 of Lecture Notes in Mathematics, eds. C. Brezinski et al,
	Springer-Verlag, 1985, pp. 184--194  (in French).
	
	\bibitem{Maroni1988calcul} P. Maroni, \emph{Le calcul des formes
		linéaires et les polynômes orthogonaux semiclassiques}, in Orthogonal
	polynomials and their applications, M. Alfaro et al, ed., vol.~1329 of Lecture
	Notes in Mathematics, Springer-Verlag, 1988, pp. 279--290 (in French).
	
	\bibitem{Maro} P. Maroni, \emph{ Sur la suite de polynômes orthogonaux associée à la forme $u = \delta_c +\lambda(x-c)^{-1}L$}, Periodica Mathematica Hungarica \textbf{21} (1990) 223-248 (in French).
	
	\bibitem{matveev} V. B. Matveev and  M. A. Salle, \emph{Differential-difference evolution equations. II: Darboux Transformation for the Toda lattice}, Letters in Mathematical Physics\textbf{3} (1979) 425-429.
	
	\bibitem{matveev-book} V.  B. Matveev and  M. A. Salle, \emph{Darboux Transformations and Solitons}, Springer Series in Nonlinear Dynamics, Springer-Verlag, Berlin, 1991.
	
	\bibitem{mikhailov} A. V. Mikhailov, \emph{The reduction problem and the inverse scattering method}, Physica D \textbf{3} (1981) 73-117.
	
	\bibitem{Nikiforov1991Discrete} A. F. Nikiforov,  S. K. Suslov,  and V. B. Uvarov,\emph{ Classical Orthogonal Polynomials of a Discrete Variable}, Springer, Berlin,
	1991.
	
	\bibitem{mir} L. Miranian, \emph{Matrix valued orthogonal polynomials on the real line: some extensions of the classical theory}, Journal of Physics A: Mathematical \&  General {\bf 38} (2005), 5731-5749.
	
	\bibitem{mir2} L. Miranian, \emph{Matrix valued orthogonal polynomials in the unit circle: some extensions of the classical theory}, Canadian Mathematical Bulletin \textbf{52} (2009) 95-104.
	
	\bibitem{Miwa} T. Miwa, M. Jimbo, and E. Date, \emph{Solitons: differential equations, infinite-dimensional algebras}, Cambridge Tracts in Mathematics \textbf{135}, Cambridge University Press, Cambridge, UK,  2000.
	
	\bibitem{moutard} Th. F. Moutard,
	\emph{Sur la construction des équations de la forme $(1/z)\partial^2 z/\partial x\partial y=\lambda(x,y)$
	qui admettenent une intégrale générale explicite},  Journal de l'École polytechnique \textbf{45} (1878) 1-11 (in French).
	
	\bibitem{mulase} M. Mulase, \emph{Complete integrability of the Kadomtsev--Petviashvili equation}, Advances  in Mathematics \textbf{54} (1984)  57-66.
	
	\bibitem{nimmo} J. J. C. Nimmo and R. Willox, \emph{Darboux Transformations for the Two-Dimensional Toda System}, Proceedings of Royal Society of  London Series A \textbf{453} (1997) 2497-2525.
	
	\bibitem{olver} P. J. Olver, \emph{On Multivariate Interpolation}, Studies in Applied  Mathematics  \textbf{116} (2006) 201-240.
	
	\bibitem{polyakov1} A. M. Polyakov, \emph{Quantum geometry of bosonic strings}, Physics Letters B \textbf{103} (1981) 207-210.
	
	\bibitem{polyakov2} A. M. Polyakov, \emph{Quantum geometry of fermionic strings}, Physics Letters B \textbf{103} (1981) 211-213.
	
	%\bibitem{redivo-zaglia} M. Redivo Zaglia, \emph{Pseudo-Schur complements and their properties}, Applied Numerical Mathematics \textbf{50} (2004) 511-519.
	
	\bibitem{Rod} L. Rodman, \emph{ Orthogonal matrix polynomials,} in Orthogonal Polynomials (Columbus OH 1989), P. Nevai Editor, NATO Adv. Sci. Inst. Ser. C. Math. Phys. Sci. 294, Kluwer, Dordrecht, 1990. 345-362.
	
	\bibitem{rogers-schief} C. Rogers  and W. K. Schief, \emph{Bäcklund and Darboux Transformations: Geometry and Modern Applications in Soliton Theory}, Cambridge University Press, Cambridge, 2002.
	
	%\bibitem{Ros} M. Rosenberg, \emph{The square integrability of matrix valued functions with respect to a nonnegative Hermitian measure}, Duke Math. J. \textbf{31} (1964) 291-298.
	
	\bibitem{rowen} L. Rowen, \emph{Ring Theory,} Vol. I,  Academic Press, San Diego CA, 1988.
	
	\bibitem{R. W. Ruedemann} R. W. Ruedemann, \emph{Some results on biorthogonal polynomials}, International Journal of Mathematics \& Mathematical Sciences \textbf{17} (1994) 625-636.
	
	\bibitem{salle} M. A. Salle, \emph{Darboux transformations for non-Abelian and nonlocal equations of the Toda chain type}, Theoretical and  Mathematical Physics \textbf{53} (1982) 1092-1099.
	
	\bibitem{sato0}  M. Sato, \emph{Soliton equations as dynamical systems on infinite dimensional Grassmann manifolds (Random Systems and Dynamical Systems)},   RIMS Kôkyûroku \textbf{439}, (1981) 30-46, Research Institute for Mathematical Sciences,  Kyoto University.
	
	\bibitem{sato} M. Sato and Y. Sato, \emph{Soliton equations as dynamical systems on infinite-dimensional Grassmann manifold}, Nonlinear partial differential equations in applied science (Tokyo, 1982) 259-271, Mathematical Studies \textbf{81}, North-Holland, Amsterdam, 1983.
	
		\bibitem{Schiebold} C. Schiebold, \emph{Explicit Solution Formulas for the Matrix-KP}, Glasgow Mathematical Journal \textbf{51A} (2009) 147–155.
	
	\bibitem{Schwartz1}  L. Schwartz,  \emph{Théorie des noyaux}, Proceedings of the International Congress of Mathematicians
	(Cambridge, MA, 1950), vol. 1 p. 220-230,  American Mathematical Society, Providence, RI, 1952.
	
	%\bibitem{Schwartz1}  L. Schwartz,  \emph{Théorie des distributions à valeurs vectorielles I}, Annales de l'Institut Fourier (Grenoble) \textbf{7}(1957) 1-141.
	
	\bibitem{Schwartz}  L. Schwartz, \emph{Théorie des distributions}, Hermann, Paris, 1978.
	
	\bibitem{Simon2} B. Simon,  \emph{Orthogonal polynomials on the unit circle. Part 1. Classical theory}. American Mathematical Society Colloquium Publications, \textbf{54}, Part 1, American Mathematical Society, Providence, RI, 2005.
	
	\bibitem{simon-cd} B. Simon, \emph{The Christoffel--Darboux Kernel} in \emph{Perspectives in partial differential Equations, harmonic Analysis and Applications}, D. Mitrea and M. Mitrea (editors),  Proceedings of Symposia in Pure Mathematics \textbf{79} (2008) 295-346.
	
	\bibitem{van} A. Sinap and W. Van Assche, \emph{Polynomial  interpolation  and Gaussian  quadrature  for matrix-valued  functions, } Linear Algebra and its Applications \textbf{207 }(1994) 71-114.
	
	\bibitem{van2} A. Sinap and W. Van Assche, \emph{Orthogonal matrix polynomials and applications,} Proceedings of the Sixth International Congress on Computational and Applied Mathematics (Leuven, 1994). Journal of Computational Applied Mathematics \textbf{66} (1996) 27-52.
	
	\bibitem{Sze} G. Szeg\H{o}, \emph{Orthogonal Polynomials.} 4th ed., American  Mathematical Society Colloquium Publication Series\textbf{23}, American Mathematical Society, Providence RI, (1975).
	
	\bibitem{Tacchella}  A.Tacchella,  \emph{On rational solutions of multicomponent and matrix KP hierarchies}
	Journal of Geometry and Physics
	 \textbf{61} (2011) 1319–1328.

	%\bibitem{tir} J. Tirao and I. Zurri\'an, \emph{Reducibility of matrix weights,} arXiv:1501.04059.
	
	\bibitem{Yoon} G. Yoon,  \emph{Darboux transforms and orthogonal polynomials}, Bulletin Korean Mathematical Society \textbf{39} (2002)  359-376.
	
	\bibitem{Yakhlef1} H. O. Yakhlef, F. Marcellán, and M. A. Piñar, \emph{Relative Asymptotics for Orthogonal Matrix Polynomials with Convergent Recurrence Coefficients},  Journal of Approximation Theory \textbf{111} (2001) 1-30.
	
	\bibitem{Yakhlef2} H. O. Yakhlef, F. Marcellán, and M. A. Piñar, \emph{Perturbations in the Nevai matrix class of orthogonal matrix polynomials},  Linear Algebra and its Applications \textbf{336} (2001) 231-254.
	
	\bibitem{Yakhlef3} H. O. Yakhlef and  F. Marcellán, \emph{Relative Asymptotics for  Matrix Orthogonal Polynomials for Uvarov Perturbations: The Degenerate Cases},  Mediterranean Journal of Mathematics (2016) published online.
	
	\bibitem{ueno-takasaki0} K. Ueno and K. Takasaki, \emph{Toda lattice hierarchy. I}, Proceedings  of the Japan Academy Series A, Mathematical Sciences  \textbf{59} (1984) 167-170.
	
	\bibitem{ueno-takasaki1} K. Ueno and K. Takasaki, \emph{Toda lattice hierarchy. II}, Proceedings  of the Japan Academy Series A, Mathematical Sciences  \textbf{59} (1984) 215-218.
	
	\bibitem{ueno} K. Ueno and K. Takasaki, \emph{Toda lattice hierarchy}, in Group Representations and Systems of Differential Equations, Advances  Studies in   Pure Mathematics \textbf{4} (1984) 1-95.
	
	\bibitem{Uva0}	V. B. Uvarov, \emph{Relation between polynomials orthogonal with different weights}, Doklady Akademii Nauk SSSR \textbf{126} (1),  (1959)  33–36 (in Russian).
	
	\bibitem{Uva} V. B. Uvarov, \emph{The connection between systems of polynomials that are orthogonal with respect to different distribution functions,} USSR Computational and Mathematical Physics \textbf{9} (1969) 25-36.
	
			\bibitem{Wang} N. Wang and M. Wadati,  \emph{Noncommutative KP hierarchy and Hirota triple-product relations}, Journal of the Physical Society of Japan \textbf{73} (2004)1689–1698.
	
	\bibitem{Zhe} A. Zhedanov, \emph{Rational spectral transformations and orthogonal polynomials}, Journal of  Computational and  Applied Mathematics \textbf{ 85} (1997) 67-86.
	
\end{thebibliography}
\end{document}